\documentclass[12pt]{article}
\usepackage{amsmath}
\usepackage{graphicx}
\usepackage{enumerate}
\usepackage[numbers]{natbib}
\usepackage{url} 
\usepackage{amsthm,amsmath,amsfonts,amssymb}

\newcommand{\blind}{1}

\usepackage{amsmath, mathrsfs, amsthm}
\usepackage{graphicx}
\usepackage{verbatim}
\usepackage{float}
\usepackage{subfigure}
\usepackage{booktabs}
\usepackage{multirow}
\usepackage{bbm}

\usepackage{hyperref}
\usepackage{extarrows}
\usepackage{multirow}
\usepackage{natbib}
\usepackage{stmaryrd}
\usepackage{color}
\usepackage{dsfont}
\usepackage{amsmath}
\usepackage{enumitem}
\usepackage{mathrsfs}
\usepackage[dvipsnames]{xcolor}
\usepackage{bbm}
\usepackage{bm}
\usepackage{comment}
\usepackage{caption}
\usepackage{booktabs}
\usepackage{graphicx}
\usepackage{float}
\usepackage{epstopdf}
\usepackage{arydshln}

\newtheorem{thm}{Theorem}[section]
\newtheorem{lem}{Lemma}[section]
\newtheorem{cor}{Corollary}[section]
\newtheorem{prop}{Proposition}[section]

\newtheorem{de}{Definition}[section]
\newtheorem{cond}{Condition}[section]
\newtheorem*{thm*}{Theorem}

\theoremstyle{remark}
\newtheorem{rmk}[thm]{Remark}

\newcommand{\bed}{\begin{definition}}
\newcommand{\eed}{\end{definition}}

\newcommand{\bitem}{\begin{itemize}}
\newcommand{\eitem}{\end{itemize}}

\newcommand{\beqn}{\begin{equation}}
\newcommand{\eeqn}{\end{equation}}
\newcommand{\balign}{\begin{align}}
\newcommand{\ealign}{\end{align}}

\newcommand{\diag}{\mathrm{diag}}

\newcommand{\beq}{\begin{equation}}
\newcommand{\eeq}{\end{equation}}

\usepackage{algorithm}
\usepackage{algpseudocode}

\allowdisplaybreaks
\usepackage{hyperref}
\hypersetup{
    colorlinks=true,
    linkcolor=blue,
    filecolor=magenta,      
    urlcolor=cyan,
}

\begin{document}

\def\spacingset#1{\renewcommand{\baselinestretch}%
{#1}\small\normalsize} \spacingset{1}


\if1\blind
{
  \title{\bf Optimal Network Membership Estimation under Severe Degree Heterogeneity}
  \author{Zheng Tracy Ke\\
    Department of Statistics, Harvard University
    \vspace{.2cm}\\
    and  \vspace{.2cm}\\
    Jingming Wang\\
    Department of Statistics, Harvard University}
  \maketitle
} \fi

\if0\blind
{
  \bigskip
  \bigskip
  \bigskip
  \begin{center}
    {\LARGE\bf Optimal Network Membership Estimation Under Severe Degree Heterogeneity}
\end{center}
  \medskip
} \fi

\bigskip 

\begin{abstract}
Real networks often have severe degree heterogeneity, with the maximum, average, and minimum node degrees differing significantly. This paper examines the impact of degree heterogeneity on statistical limits of network data analysis. Introducing the heterogeneity distribution (HD) under a degree-corrected mixed-membership network model, we show that the optimal rate of mixed membership estimation is an explicit functional of the HD. This result confirms that severe degree heterogeneity may decelerate the error rate, even when the overall sparsity remains unchanged.

To obtain a rate-optimal method, we modify an existing spectral algorithm, Mixed-SCORE, by adding a pre-PCA normalization step. This step normalizes the adjacency matrix by a diagonal matrix consisting of the $b$th power of node degrees, for some $b\in \mathbb{R}$. We discover that $b = 1/2$ is universally favorable. The resulting spectral algorithm is rate-optimal for networks with arbitrary degree heterogeneity. A technical component in our proofs is entry-wise eigenvector analysis of the normalized graph Laplacian.

\end{abstract}

\noindent%

\vfill

\newpage
\spacingset{1.7}

\tableofcontents

\section{Introduction} \label{sec:Intro}
Networks are widely observed in social sciences, machine learning, economics, physics, and bioinformatics. One of the core problems in network data analysis is community detection, which aims to cluster nodes into socially meaningful groups. Mixed membership estimation \citep{airoldi2009mixed} is a ``soft clustering" problem that allows a node to have memberships in multiple communities. Given an undirected network with $n$ nodes, let $A\in\mathbb{R}^{n\times n}$ be the adjacency matrix, where $A(i,j)\in \{0,1\}$ indicates if there is an edge between two nodes $i$ and $j$.  
Suppose there are $K$ communities. Each node has a Probability Mass Function (PMF)  
$\pi_{i} \in\mathbb{R}^K$   
such that
\[
\mbox{$\pi_i(k)$ is the fractional weight of node $i$ on community $k$}, \qquad 1\leq k\leq K.  
\]
The goal is estimating $\pi_1,\pi_2,\ldots,\pi_n$ from $A$. 
One example of networks with mixed memberships is the political blog network \citep{adamic2005political}. Each node is a political blog active during the 2004 U.S. Presidential Election, and each edge is a link between two blogs in a snapshot on one day before the election. There are two communities, conservative and liberal. However, for most bloggers, they are neither extremely conservative nor extremely liberal; 
their positions in-between can be captured by mixed membership \citep{Mixed-SCORE}.  
Mixed membership estimation has also been used to learn research interests of statisticians from co-citation networks \citep{ji2021co} and understand developmental brain disorders from gene co-expression networks \citep{liu2018global}.

Several methods have been proposed for estimation and inference of network mixed memberships. 
To name a few, the Bayesian approach in {\it Airoldi et al.} \cite{airoldi2009mixed} puts a Dirichlet prior on $\pi_i$'s and uses variational inference to get the posteriors, the spectral approaches by {\it Zhang et al.} \cite{JiZhuMM} and  {\it Jin et al.} \cite{Mixed-SCORE} estimate $\pi_i$'s from the leading eigenvectors of $A$, and {\it Fan et al.} \cite{fan2019simple} and {\it Bhattacharya et al.} \cite{bhattacharya2023inferences} studied the hypothesis testing and uncertainty quantification for individual membership vectors. 

Despite these encouraging progresses, there is an open problem: how to deal with severe {\it degree heterogeneity} in real-world networks. 
An observed static network is often a snapshot of a dynamic process involving node birth and edge formation. 
The preferential attachment (``rich get richer") phenomenon \citep{barabasi1999emergence} is quite common in the underlying dynamic process;
as a consequence, the degrees of a small set of nodes can be a magnitude larger than the degrees of other nodes. Meanwhile, there are a large number of nodes whose degrees are quite small, most of which are the ``young" nodes in the underlying dynamic process.  \cite{jin2021improvements} reported the maximum, minimum, and average degrees of some frequently-studied network datasets and found that the three quantities, $d_{\max}$, $d_{\min}$, and $d_{\mathrm{mean}}$, could be significantly different.

The severe degree heterogeneity in real-world networks inspires a question: How does degree heterogeneity affect the accuracy of mixed membership estimation?
Suppose we have two networks whose average node degrees are comparable with each other. 
In one network, the node degrees are at the same order. In the other network, node degrees are significantly different from each other. 
Do we achieve the same accuracy of mixed membership estimation on these two networks? 
We provide the first result that connects degree heterogeneity and the optimal rate of mixed membership estimation. 
Under a degree-corrected mixed membership (DCMM) model \citep{Mixed-SCORE}, we propose to summarize degree heterogeneity by a cumulative distribution function $F_n(\cdot)$, defined using normalized degree parameters in DCMM, and 
we show that the optimal rate of mixed membership estimation is determined by a baseline rate 
(which is related to the average degree) 
and $F_n(\cdot)$ (which captures degree heterogeneity). 

Knowing the optimal rate, the next question is if there exists an algorithm that achieves the optimal rate under arbitrary degree heterogeneity. 
Existing methods include the variational inference approach \citep{airoldi2009mixed}, the spectral approach \citep{mao2020estimating}, and nonnegative matrix factorization 
\citep{yang2013overlapping}. There are also works studying mixed membership estimation for dynamic networks \citep{fu2009dynamic}, 
networks with different node types \citep{huang2020mixed}, and tensor models \citep{agterberg2022estimating}. 
However, most of these works assume no degree heterogeneity, hence not applicable in our setting.  
OCCAM \citep{JiZhuMM} is a method that allows for degree heterogeneity, but it restricts the fraction of mixed nodes to be relatively small.  Mixed-SCORE \citep{Mixed-SCORE} is the only existing method that allows for severe degree heterogeneity and an arbitrary fraction of mixed nodes. We then focus on Mixed-SCORE and
compare its rate with the aforementioned lower bound. We discover that Mixed-SCORE may be non-optimal when  degree heterogeneity is severe. Inspired by some non-trivial theoretical insights, we improve Mixed-SCORE by adding a (pre-PCA) Laplacian normalization and a (post-PCA) node trimming. The resulting method is called Mixed-SCORE-Laplacian (MSL). We show that MSL is rate-optimal under arbitrary degree heterogeneity.

Another theoretical result we provide is the node-wise error of MSL. For each node $i$, we derive a large-deviation bound for $\|\hat{\pi}_i-\pi_i\|_1$. 
The bound is a monotone decreasing function of $\theta_i$. It shows that the estimation error is not evenly distributed among nodes: higher-degree nodes receive more accurate estimation of $\pi_i$. Such node-wise error bounds will be useful for downstream tasks such as using $\hat{\pi}_i$ to rank node preferences or build prediction models. 

The backbone of our analysis is entry-wise eigenvector analysis for normalized adjacency matrices. We consider a degree-normalization of $A$ to $L=H^{-b}AH^{-b}$ for any $b\geq 0$, where $H$ is a diagonal matrix containing regularized node degrees.  
We study each entry of the leading eigenvectors of $L$ and show how its large-deviation is related to $b$. 
Such analysis reveals the effects of different degree-normalizations and inspires us to use the Laplacian normalization. 
Fixing the Laplacian normalization, we derive entry-wise large-deviation bounds for leading eigenvectors, and use this result to derive the node-wise error and rate-optimality of MSL.

In summary, we offer the first rate-optimality study of mixed membership estimation for arbitrary degree heterogeneity. Our contributions include (i) a degree-heterogeneity-aware lower bound, (ii) a rate-optimal spectral algorithm, (iii) node-wise estimation error bounds, and (iv) entry-wise eigenvector analysis for the normalized graph Laplacian. To demonstrate the practical implications of our work, we apply our method to a political blog network \citep{adamic2005political} and a coauthorship network \citep{JiJin}.

%


\subsubsection*{Related literature}

Mixed membership estimation may be viewed as ``soft" community detection, but it is fundamentally different from community detection. 
Community detection is a clustering problem, and its optimality can be achieved by a post-spectral-clustering `majority vote' \citep{gao2015achieving}. However, optimality of mixed membership estimation has much higher requirement on the entrywise signal-to-noise ratio in eigenvectors, hence a more challenging problem. 

A lower bound for mixed membership estimation under moderate degree heterogeneity was given in a short note \citep{jin2017sharp}. However, this lower bound is not sharp for severe degree heterogeneity. Using more sophisticated proofs, we provide sharp lower bounds under arbitrary degree heterogeneity. 
We also have many other contributions - a new algorithm, node-wise error bounds, and entry-wise eigenvector analysis, but none of which were seen in \cite{jin2017sharp}.

One of our ideas in methodology development is using the normalized graph Laplacian. This matrix has been used in community detection \cite{qin2013regularized,jin2021improvements}. There are interesting theoretical studies explaining why it improves community detection perfromance, such as large-deviation bounds for matrix spectral norms \citep{chaudhuri2012spectral} and analysis of global clustering quality metrics \citep{sarkar2015role}. However, these results are insufficient to conclude that Laplacian is the correct normalization for optimal mixed membership estimation under arbitrary degree heterogeneity. In Section~\ref{sec:MSCORE-L}, we show that an optimal method has to attain the correct node-wise errors simultaneously for all nodes. Hence, any conclusion of optimality requires both {\it node-wise error analysis} and {\it degree-heterogeneity-aware lower bounds}. Neither results existed in the literature.

A technical component of our work is the entry-wise eigenvector analysis for normalized graph Laplacian.  
In the literature, entry-wise eigenvector analysis was conducted for eigenvectors of the adjacency matrix under a stochastic block model (SBM), such as in {\it Abbe et al.} \cite{abbe2020entrywise}. 
However, a direct extension of the analysis in \cite{abbe2020entrywise} does not work for our purpose, because: (i) the Laplacian matrix no longer has independent entries,  (ii) our model allows for severe degree heterogeneity, and (iii) we need tighter bounds than those in \cite{abbe2020entrywise}. 
We overcome these hurdles by new proof ideas, which will be explained in Section~\ref{sec:Entrywise-proof}. 

Our theory is for the asymptotic regime where the average node degree grows faster than $\log(n)$, and the rate-optimality of our proposed method is subject to a logarithmic factor. It is an interesting question to remove such a  logarithmic gap, but the study is very challenging (see Section~\ref{subsec:MSL-rate} for more discussions). We leave it to future work.


The remainder of this paper is organized as follows. Section~\ref{sec:LB} states the model, asymptotic settings, and lower bounds. Section~\ref{sec:MSCORE-L} introduces the new algorithm. Section~\ref{sec:UB} contains the main results, including entry-wise eigenvector analysis, node-wise errors, rate-optimality, and extensions to general loss functions. Section~\ref{sec:Entrywise-proof} explains the techniques used in entry-wise eigenvector analysis. Section~\ref{sec:Simu} and Section~\ref{sec:real_data} present the simulations and real-data results, respectively. Discussions are in Section~\ref{sec:Discuss}, and all proofs are in the supplementary material.



\spacingset{1}
\section{The model, asymptotic settings, and lower bounds} \label{sec:LB}
\spacingset{1.7}


We model the network with a degree-corrected mixed membership (DCMM) model \citep{Mixed-SCORE}. DCMM generalizes the stochastic block model (SBM) by incorporating both degree heterogeneity and mixed membership, and its special cases include the well-known mixed membership stochastic block model (MMSBM) \citep{airoldi2009mixed} and degree-corrected block model (DCBM) \citep{DCBM}. 
We present an information-theoretic lower bound under DCMM. 
This lower bound depends on the distribution of node degree parameters, so it reveals the fundamental impact of degree heterogeneity on mixed membership estimation. 


\subsection{The DCMM model and the asymptotic settings} \label{subsec:DCMM}

Let $K$ be the number of communities and let $\Pi=[\pi_1,\pi_2,\ldots,\pi_n]'\in\mathbb{R}^{n\times K}$ be the membership matrix. DCMM uses a symmetric non-negative matrix $P\in\mathbb{R}^{K\times K}$ to model the community structure and a positive vector $\theta=(\theta_1,\theta_2,\ldots,\theta_n)'$ to model the degree heterogeneity among nodes. 
The upper triangular entries of $A$ are independent Bernoulli variables satisfying 
\beq \label{mod-DCMM} 
\mathbb{P}\big(A(i,j) = 1\big) =  \theta_i \theta_j \cdot \pi_i' P\pi_j, \qquad 1\leq i<j\leq n. 
\eeq
To guarantee parameter identifiability, we follow  \cite{Mixed-SCORE} to require that $P$ is non-singular and has unit diagonals. 
Write $\Theta=\diag(\theta_1,\theta_2,\ldots,\theta_n)$.\spacingset{1}\footnote{Given a vector $v\in\mathbb{R}^n$, we use $\diag(v)$ or $\diag(v_1,v_2,\ldots,v_n)$ to denote the $n\times n$ diagonal matrix whose $i$th diagonal entry is equal to $v_i$, for $1\leq i\leq n$.}\spacingset{1.7} It is seen that 
\beq \label{mod-DCMM2}
A = \Omega - \diag(\Omega)+W, \qquad\mbox{where}\quad \Omega:=\Theta\Pi P\Pi'\Theta\quad \mbox{and}\quad W:=A-\mathbb{E}[A]. 
\eeq
We call node $i$ a {\it pure node} if $\pi_i$ is degenerate (i.e., it has only one nonzero coordinate that is equal to $1$, with the other coordinates being zero) and a {\it mixed node} otherwise. The DCMM model is identifiable provided that each community has at least one pure node \citep{Mixed-SCORE}. 

Given $A$ and $K$, we are interested in estimating $\Pi$. 
We use $n$ as the driving asymptotic parameter and study the optimal error rate as $n\to\infty$. A challenge is that DCMM has many nuisance parameters, all of which affect the error rate. We hope to define a few summarizing quantities so that the optimal rate only depends on these quantities.

\medskip
\noindent
{\bf An illustrating example} {\it (Random-parameter DCMM)}. Let $g(\cdot)$ be a distribution on the standard simplex of $\mathbb{R}^K$ such that $\mathbb{E}_{g}[\pi]=\frac{1}{K}{\bf 1}_K$ and $\lambda_{\min}(\mathbb{E}_{g}[\pi\pi'])\geq \frac{c}{K}$, for a constant $c>0$. Let $F(\cdot)$ be a distribution with support in $(0,\infty)$ and mean equal to $1$. 
For $a_n, p_n\in (0,1)$, 
\[
P=(1-p_n)I_K + p_n {\bf 1}_K{\bf 1}_K', \qquad \pi_i\overset{iid}{\sim}g(\cdot), \qquad \theta_i\overset{iid}{\sim} a_n\cdot F(\cdot). 
\]
Here, $(1-p_n)$ captures the `dissimilarity' between communities. The average node degree is at the order $na^2_n$ (if $K$ is finite), so $a_n$ captures the sparsity of the network. The distribution $F(\cdot)$ controls degree heterogeneity (e.g., when $F$ is a point mass, there is no degree heterogeneity). We expect that the optimal rate  depends on $(1-p_n)$, $a_n$ and $F(\cdot)$. 

\medskip

Using insights from this example, we now define the parameter class for a general DCMM, where we introduce a few quantities that serve as the counterpart of $(1-p_n, a_n, F(\cdot))$. 
Write $\bar{\theta}=n^{-1}\sum_{i=1}^n\theta_i$ and $\eta_i=\theta_i/\bar{\theta}$, $1\leq i\leq n$.   
\begin{de} \label{Def-CDF}  
The heterogeneity distribution (HD) is the empirical distribution associated with $\eta_1,\eta_2,\ldots,\eta_n$, whose CDF is $F_n(t)=n^{-1}\sum_{i=1}^n1\{\theta_i\leq t\cdot \bar{\theta}\}$. 
\end{de}

\noindent
Since the diagonals of $P$ are fixed at $1$ for identifiability, the average node degree is $\asymp n\bar{\theta}^2$, implying that $\bar{\theta}$ captures network sparsity. 
In addition, $F_n(\cdot)$ captures degree heterogeneity. Here, $\bar{\theta}$ and $F_n(\cdot)$ serve as the counterpart of $a_n$ and $F(\cdot)$ for a general DCMM model.  

Let $d_i=\sum_{j}A(i,j)$ be the degree of node $i$ and $\bar d = n^{-1} \sum_{i} d_i$ be the average node degree. Write $D_\theta=\diag(\mathbb{E}d_1, \mathbb{E}d_2, \ldots, \mathbb{E}d_n) + (\mathbb{E}\bar{d})I_n$.
We define a ``community-overlap" matrix 
\beq \label{def:G}
G= K \cdot \Pi'\Theta D_{\theta}^{-1}\Theta \Pi\;\; \in\;\; \mathbb{R}^{K\times K}.
\eeq
Note that $G(k,\ell)=\frac{K}{n}\sum_{i=1}^n \frac{n\theta_i^2}{\mathbb{E}(d_i+\bar{d})}\pi_i(k)\pi_j(\ell)$, where for most nodes $i$, $\frac{n\theta_i^2}{\mathbb{E}(d_i+\bar{d})}\asymp 1$. Therefore, when $k=\ell$, a large value of $G(k,k)$ implies that a significant fraction of nodes have nonzero membership on community $k$, so $\frac{n}{K}G(k,k)$ is the ``effective size" of this community; when $k\neq \ell$, 
a large value of $G(k,\ell)$ indicates a big ``overlap" between community $k$ and community $\ell$. 
Our study suggests that the counterpart of $(1-p_n)$ for a general DCMM should be defined through eigenvalues of $PG$. In detail, let $\lambda_k(PG)$ denote the $k$th largest (in magnitude) right eigenvalue of $PG$, for $1\leq k\leq K$. 
Define 
\beq \label{delta_n}
\beta_n:=|\lambda_K(PG)|. 
\eeq
We use $\beta_n$ as the counterpart of $(1-p_n)$ for a general DCMM. In the aforementioned random-parameter DCMM example, $G\sim K\cdot \mathbb{E}_g[\pi\pi']$, and $\beta_n\asymp 1-p_n$
(this example also suggests that $G$ is a global quantity that is insensitive to degree heterogeneity).

Later in this section, we will present an information-theoretical lower bound for mixed membership estimation, which depends on $n$, $K$,  $\beta_n$ (capturing community dis-similarity), $\bar{\theta}$ (capturing network sparsity), and $F_n(\cdot)$ (capture degree heterogeneity).

\subsection{Regularity conditions} \label{subsec:RegConds}

Our results require some regularity conditions. Let $G$ and $\lambda_k(PG)$ be the same as in Section~\ref{subsec:DCMM}. For each $1\leq k\leq K$, denote by $\eta_k\in\mathbb{R}^K$ the $k$th right eigenvector of $PG$ (associated with the eigenvalue $\lambda_k(PG)$). Fix positive constants $c_1$-$c_4$. 

\begin{cond} \label{reg-conds}
The DCMM parameters $(\theta,\Pi,P)$ satisfy the following requirements:
\begin{itemize} \itemsep -5pt
\item[(a)] $\Vert G\Vert\leq c_1$, $\Vert G^{-1}\Vert\leq c_1$, and $\min_{1\leq k\leq K}\bigl\{\sum_{i=1}^n\theta_i\pi_i(k)\bigr\}\geq c_1K^{-1}\|\theta\|_1$. 
\item[(b)] $\max_{k\neq 1}\{\lambda_k(PG)\}\leq (1-c_2)\cdot \lambda_1(PG)$, and $\min_{1\leq k\leq K} \{\sum_{1\leq \ell\leq K}P(k, \ell) \} \geq c_2 K $. 
\item[(c)] $\eta_1$ is a positive vector, satisfying that $\min_{1\leq k\leq K}\{\eta_1(k)\}\geq c_3\max_{1\leq k\leq K}\{\eta_1(k)\}$. 
\item[(d)] Each community has at least one pure node whose $\theta_i$ is lower bounded by $c_4\bar{\theta}$. 
\end{itemize} 
\end{cond}

Condition (a) requires that the matrix $G$ defined in \eqref{def:G} is well-conditioned and the degree parameters are balanced across communities. These are commonly assumed in the literature \citep{Mixed-SCORE}. In Condition (b), the first is an eigen-gap condition. 
Since $PG$ is a nonnegative matrix, by Perron's theorem \cite{HornJohnson}, $\max_{k\neq 1}\{\lambda_k(PG)\}$ is strictly smaller than $\lambda_1(PG)$. Here, we further assume that $\max_{k\neq 1}\{\lambda_k(PG)\}$ is smaller than $(1-c_2)\cdot \lambda_1(PG)$, which is mild. 
The second requirement in (b) restricts that the sum of each row of $P$ is at the order of $K$. This is for convenience of presentation and can be easily relaxed (see Section~\ref{supp:relaxP} of the supplement).
For Condition (c), note that $\eta_1$ is a positive vector when $PG$ is irreducible (by Perron's theorem \cite{HornJohnson}). Hence, this condition is only slightly stronger than requiring the irreducibility of $PG$. \spacingset{1}\footnote{When $PG$ is reducible, with probability $1-o(1)$, the network splits into multiple disconnected components. In that case, we will conduct mixed membership estimation separately on each component, and each sub-problem still has an irreducible $PG$.}\spacingset{1.7}
Condition (d) is related to the identifiability of DCMM, which requires each community to have at least one pure node. 
Here we put a slightly stronger requirement: there is at least one ``moderate-degree" pure node per community.

We also define a class of $\theta$. 
We hope this class to be broad enough to cover various kinds of degree heterogeneity. We note that the EHD (see Definition~\ref{Def-CDF}) can have many different situations: First, $F_n(\cdot)$ may be discrete (i.e., $\theta_i$'s take only finitely many values), or converge to a continuous CDF as $n\to\infty$. Second, $F_n(\cdot)$ may have an unbounded support, to allow for $\theta_{\max}\gg\bar{\theta}$. Finally, the density of $F_n(\cdot)$ may have different shapes in the neighborhood of zero, to control the fraction of extremely-low-degree nodes. We introduce the following class of $\theta$. Despite that its definition is technical, this class indeed covers all the above cases. 

\begin{de} \label{def:thetaClass}
Given constants $\varrho\in (0,1)$ and $a_0\in (0,1)$ and any sequence $x_n\to 0$, let ${\cal G}_n(\varrho,a_0, x_n)$ be the collection of $\theta\in\mathbb{R}^n$ such that there exists $c_n>0$ satisfying $\varrho c_n\geq x_n^2$, $F_n(c_n)\leq 1- a_0$, and $\int_{\varrho c_n}^{c_n} \frac{1}{\sqrt{t\wedge 1}}d F_n(t) \geq a_0 \int_{x_n^2}^{\infty}\frac{1}{\sqrt{t\wedge 1}}d F_n(t)$.
\end{de}

In Section~\ref{subsec:theta-class} of the supplementary material, we justify the broadness of ${\cal G}(\varrho, a_0, x_n)$. We show that the requirements in Definition~\ref{def:thetaClass} are satisfied with high-probability if $\theta_i\overset{iid}{\sim}a_n\cdot F$, where 
$a_n>0$ is a scalar and $F(\cdot)$ is a fixed, finite-mean distribution which has its support in $(0,\infty)$ and belongs to one of the following cases: (i) $F(\cdot)$ is a discrete distribution; (ii) $F(\cdot)$ is a continuous distribution with support in $[c,\infty)$, for some $c>0$; (iii) $F(\cdot)$ is a continuous distribution supported in $(0,\infty)$, and its density $f(t)$ satisfies that $\lim_{t\to 0+}\{t^{\omega}f(t)\}=C
$, for some constants $\omega \neq 1/2$ and $C>0$. 

\subsection{A $\theta$-dependent lower bound} \label{subsec:LB}

For any estimator $\hat{\Pi}=[\hat{\pi}_1,\hat{\pi}_2,\ldots,\hat{\pi}_n]'$, we measure its performance by the $\ell^1$-loss: 
\beq \label{Loss}
{\cal L}(\hat{\Pi},\Pi) = \min_{T}\Bigl\{ \frac{1}{n}\sum_{i=1}^n \|T\hat{\pi}_i-\pi_i\|_1\Bigr\}, 
\eeq
where the minimum is taken over column permutation of $\hat{\Pi}$. Given any $\theta\in\mathbb{R}_+^n$, $K\geq 2$, and $\beta_n\in (0,1)$, let ${\cal Q}_n(\theta)={\cal Q}_n(\theta; K, \beta_n)$ be the collection of $(\Pi, P)$ such that $|\lambda_K(PG)|\geq \beta_n$ and that $(\theta,\Pi, P)$ satisfies Condition~\ref{reg-conds}. The meaning of $\beta_n$ here is slightly different from the one in \eqref{delta_n} ---we use $ \beta_n$  as the respective lower bound of $|\lambda_K(PG)|$, by the convention of minimax analysis. 
Define the minimax error for a particular $\theta$ as
\beq \label{minmaxLoss}
{\cal L}^*_n( \theta):= \inf_{\hat{\Pi}}\sup_{(\Pi, P)\in {\cal Q}_{n}(\theta)}\mathbb{E}{\cal L}(\hat{\Pi},\Pi). 
\eeq
Define the `baseline' rate as 
\beq \label{BaseRate}
err_n = K^{3/2}\beta_n^{-1}(n\bar{\theta}^2)^{-1/2}
\eeq
and the `degree-heterogeneity-aware' rate as 
\beq \label{OptRate}
err_n(\theta):= \int \min \Bigl\{\frac{err_n}{\sqrt{t\wedge 1}}, 1\Bigr\}d F_n(t) = \frac{1}{n}\sum_{i=1}^n \min\left\{ \frac{K\sqrt{K}}{\beta_n \sqrt{n\bar{\theta} (\bar{\theta}\wedge \theta_i)}},\;\; 1\right\}. 
\eeq

\begin{thm}[A $\theta$-dependent lower bound]\label{thm:LB}
Fix constants $(c_1, c_2, c_3, c_4, \varrho, a_0)$ and positive sequences $(K, \beta_n)$ such that  $K^{3/2}\beta_n^{-1}(n\bar{\theta}^2)^{-1/2}\to 0$.   
There exists a constant $C_1>0$ such that simultaneously for all $\theta\in {\cal G}_n(\varrho,a_0, err_n)$,   ${\cal L}_n^*(\theta) \geq C_1 err_n(\theta)$.   
\end{thm}

Delving into $err_n(\theta)$, the individual contribution of node $i$ in the lower bound is $\tau_n(\theta_i)\equiv \min\bigl\{ \frac{K\sqrt{K}}{\beta_n \sqrt{n\bar{\theta} (\bar{\theta}\wedge \theta_i)}},\, 1\bigr\}$. In Figure~\ref{fig:entrywise-illustration}, we plot the curve of $\tau_n(\theta_i)$ versus $\theta_i$.  The `minimum with 1' reflects the naive bound $\|\hat{\pi}_i-\pi_i\|_1\leq 2$. When $\frac{K\sqrt{K}}{\beta_n \sqrt{n\bar{\theta} (\bar{\theta}\wedge \theta_i)}}\leq 1$, $\tau_n(\theta_i)$ depends on the minimum of $\theta_i$ and $\bar{\theta}$. 
This is because the error at $\pi_i$ comes from (i) noise in the $i$th row of $A$ (controlled by $\theta_i$) and (ii) the errors of estimating global parameters such as $P$ (controlled by $\bar{\theta}$).   When $\theta_i$ is small, (i) dominates (ii); but when $\theta_i$ is large, (ii) dominates (i).

\spacingset{1}
\begin{figure}[tb!]
\centering
\includegraphics[width=.5\textwidth]{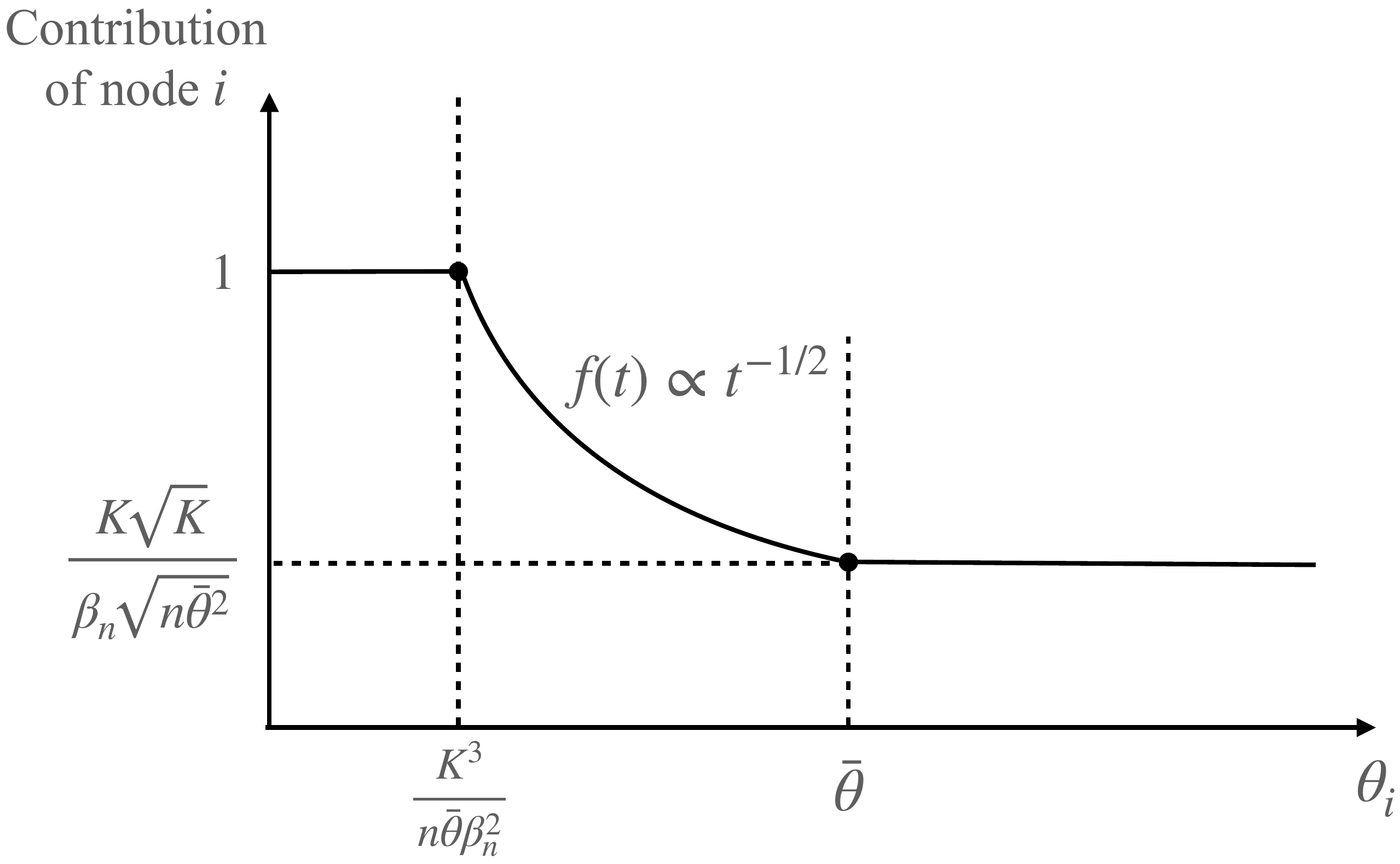}
\caption{Visualization of the contributions of individual nodes in the minimax lower bound. The x-axis represents $\theta_i$, and the y-axis corresponds to the summand in \eqref{OptRate}.} \label{fig:entrywise-illustration}
\end{figure}
\spacingset{1.7}

We revisit the example of random-parameter DCMM in Section~\ref{subsec:DCMM}, where $P$ has equal off-diagonal entries, $\pi_i\overset{iid}{\sim}g(\cdot)$, and $\theta_i\overset{iid}{\sim}a_n\cdot F(\cdot)$. Then, $
err_n = \frac{K\sqrt{K}}{(1-p_n)\sqrt{na_n^2}}$, and $err_n(\theta) = \int \min \bigl\{\frac{err_n}{\sqrt{t\wedge 1}}, 1\bigr\}d F(t)$. 
The effect of degree heterogeneity is seen from comparing $err_n(\theta)$ with the baseline rate. 
We consider a few examples of $F(\cdot)$: 
\begin{prop} \label{prop:rate}
Let $\epsilon\in (0,1)$, $\alpha>0$ and $c_{\min}>0$ be constants, and let $L\geq 1$ be a constant integer. Let $\delta_x$ denote the point mass at $x$. Suppose $x_1<x_2<\ldots <x_L$, $\sum_{\ell=1}^L\epsilon_\ell x_\ell=1$, and as $n\to\infty$, $x_1\gg err_n^2$. 
\[
\int \min \Bigl\{\frac{err_n}{\sqrt{t\wedge 1}}, 1\Bigr\}d F(t)  \asymp
\begin{cases}
err_n, & \mbox{when }F=\mathrm{Uniform}([1-\epsilon, 1+\epsilon]),\cr
err_n, & \mbox{when }F=\frac{\alpha-1}{\alpha c_{\min}}\mathrm{Pareto}(c_{\min}, \alpha), \mbox{ with } \alpha>1,\cr
err_n^{(1\wedge 2\alpha)}, & \mbox{when }F=\frac{\beta}{\alpha}\mathrm{Gamma}(\alpha,\beta), \mbox{ with }\alpha\neq 1/2,\cr
err_n\cdot \max_{1\leq \ell\leq L}\bigl\{\frac{\epsilon_\ell}{\sqrt{x_\ell\wedge 1}}\bigr\}, &\mbox{when }F=\sum_{\ell=1}^{L} \epsilon_\ell \delta_{x_\ell}. 
\end{cases}
\]
\end{prop}

Among the examples in Proposition~\ref{prop:rate}, 
the uniform distribution corresponds to moderate degree heterogeneity, where $\theta_{\max}\asymp \bar{\theta}\asymp \theta_{\min}$.  The Pareto distribution is an example of severe degree heterogeneity, where $\theta_{\max}\gg\bar{\theta}\asymp \theta_{\min}$. 
The Gamma distribution indicates even more severe heterogeneity: $\theta_{\max}\gg\bar{\theta}\gg\theta_{\min}$; and the optimal rate is slower than the baseline rate when $\alpha<1/2$. 
The last example is a discrete distribution, for which the optimal rate is slower than the baseline rate if there are a considerable fraction of extremely small $\theta_i$'s.


\section{The Mixed-SCORE-Laplacian (MSL) algorithm} \label{sec:MSCORE-L}

We shall introduce an algorithm to achieve the lower bound in Theorem~\ref{thm:LB}. Most methods for mixed membership estimation \citep{airoldi2009mixed} assume that $\theta_i$'s are equal (i.e., no degree heterogeneity). Mixed-SCORE \citep{Mixed-SCORE} is one of the few methods that deal with degree heterogeneity, but its error rate only matches with the lower bound in Theorem~\ref{thm:LB} for some special $\theta$. We modify Mixed-SCORE to achieve the lower bound for any $\theta$ in the regular class in Definition~\ref{def:thetaClass}. 

\subsection{A general node embedding, and review of Mixed-SCORE}
 
We define a general node embedding approach, which extends the SCORE embedding \citep{SCORE}.  
Given constants $\tau>0$ and $b\geq 0$, let 
\beq \label{Def:Laplacian}
L= H^{-b} AH^{-b}, \qquad \mbox{where}\quad H = \diag(d_1,d_2,\ldots,d_n)+ \tau\bar{d}\cdot I_n. 
\eeq
Here $L$ is a normalized version of $A$, where the entries of $A$ are re-weighted by node degrees. $\tau$ is a ridge-regularization parameter to avoid extreme entries in $H^{-b}$; we usually set $\tau=1$. The most important parameter is $b$, which controls the level of re-weighting. When $b=0$, $L$ is the adjacency matrix; when $b=1/2$, $L$ is the normalized graph Laplacian. 
Let $\hat{\lambda}_1, \cdots, \hat{\lambda}_K$ be the $K$ largest eigenvalues (in magnitude) of $L$, and let  $\hat{\xi}_1, \cdots, \hat{\xi}_K\in\mathbb{R}^{n}$ be the corresponding eigenvectors.
Define an $n\times (K-1)$ matrix $\hat{R}$ by 
\beq \label{SCORE}
\hat{R}(i,k) = \hat{\xi}_{k+1}(i)/\hat{\xi}_1(i), \qquad 1\leq i\leq n,\, 1\leq k\leq K-1. 
\eeq
When the network is connected, by Perron's theorem, $\hat{\xi}_1$ is always a strictly positive vector, so $\hat{R}$ is well-defined. Let $\hat{r}_1',\hat{r}_2',\ldots,\hat{r}_n'$ denote the rows of $\hat{R}$. We use $\hat{r}_i$ as a $(K-1)$-dimensional embedding of node $i$. 
We call them the (general) SCORE embeddings of nodes. The original SCORE embedding \citep{SCORE} is a special case with $b=0$ in \eqref{Def:Laplacian}, but we allow a general $b$ here. 

\spacingset{1}
\begin{figure}[tb!]
\centering
\includegraphics[width=.75\textwidth]{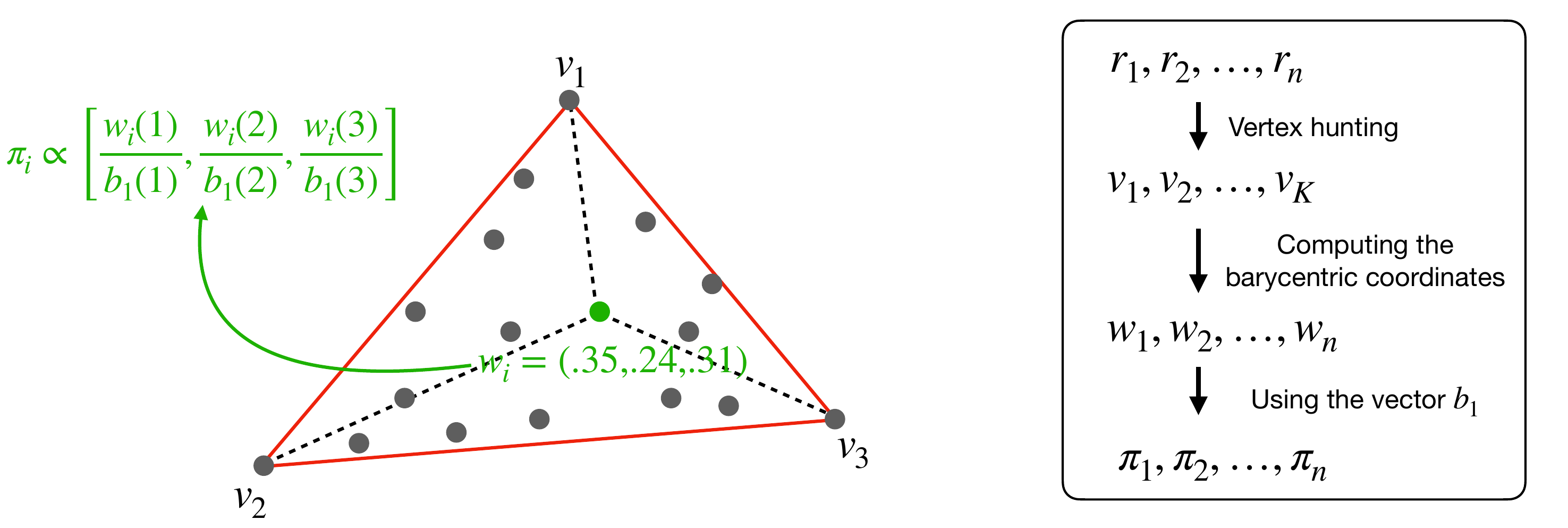} 
\caption{The main insight of mixed membership estimation. In the left panel, the black dots are $r_1,r_2,\ldots,r_n$. The red triangle is the Ideal Simplex, with $K$ vertices $v_1,v_2,\ldots,v_K$. Each $r_i$ has a barycentric coordinate $w_i$ in the simplex, and $w_i$ is related to the target membership vector $\pi_i$ through a vector $b_1$. In the right panel, we present the steps of estimating $\pi_i$'s from the node embeddings.} \label{fig:MSCORE-idea}
\end{figure}
\spacingset{1.7}

The main idea of Mixed-SCORE \citep{Mixed-SCORE} is to estimate $\pi_i$ from the simplex geometry associated with node embeddings. 
For illustration, we consider an oracle case where $A=\Omega$. The eigen-pairs $(\hat{\lambda}_k, \hat{\xi}_k)$ and the embedded vectors $\hat{r}_i$ now become non-stochastic, and we re-write them as $(\lambda_k, \xi_k)$ and $r_i$. Figure~\ref{fig:MSCORE-idea} displays the point cloud $\{r_i\}_{i=1}^n$. We observe that this point cloud is contained in a simplex (called the Ideal Simplex) with vertices $v_1, v_2, \ldots, v_K$. For any pure node $i$, the corresponding $r_i$ falls on one vertex. We can recover the Ideal Simplex as long as there is a pure node for each community. Such a step is called Vertex Hunting \citep{Mixed-SCORE}. 
By definition, each point in the simplex can be uniquely written as a convex combination of its $K$ vertices, and the vector containing the combination coefficients is called the {\it barycentric coordinate} of this point in the simplex. We use $w_i\in\mathbb{R}^K$ to denote the barycentric coordinate of $r_i$ in the Ideal Simplex.  After vertex hunting, we can obtain $w_i$ immediately from $r_i$ and the recovered vertices. \cite{Mixed-SCORE} observed that $\pi_i$ is linked to $w_i$ via $\pi_i\propto [\diag(b_1)]^{-1} w_i$, where $b_1\in\mathbb{R}^K$ is defined by $b_1(k)=[\lambda_1+v_k'\diag(\lambda_2,\ldots,\lambda_K)v_k]^{-1/2}$. We thus estimate $\pi_i$'s as follows: First, we compute $b_1$ from the $K$ vertices and the eigenvalues of $L$. Next, for each $1\leq i\leq n$, let $\pi_i^*= [\diag(b_1)]^{-1} w_i$ and $\hat{\pi}_i = \pi_i^*/\|\pi_i^*\|_1$. The whole estimation procedure is summarized in Figure~\ref{fig:MSCORE-idea}. 
In Section~\ref{subsec:simplex} of the supplement, we show that when $A=\Omega$,  $\hat{\Pi}=\Pi$ (up to a column permutation of $\hat{\Pi}$). The proof is similar to that in \cite{Mixed-SCORE}, except that they fixed $b=0$ in \eqref{Def:Laplacian} while we allow for an arbitrary $b$.

In the real case, $A$ is a noisy version of $\Omega$.  
The above procedure has a natural extension to the real case. When $b=0$, it gives rise to the conventional Mixed-SCORE \citep{Mixed-SCORE}. However, the rate-optimality of Mixed-SCORE was unknown (due to the lack of a $\theta$-dependent lower bound in the literature). 
In Sections~\ref{subsec:Methods}-\ref{subsec:algorithm}, we first use the lower bound in Theorem~\ref{thm:LB} to conclude with non-optimality of Mixed-SCORE under severe degree heterogeneity, and then we modify the conventional Mixed-SCORE to obtain a rate-optimal algorithm.

\subsection{A theory-guided approach towards rate-optimality} \label{subsec:Methods}

The lower bound in Theorem~\ref{thm:LB} 
decomposes into node-wise contributions:
\beq \label{NodewiseRate}
err_n(\theta)=\frac{1}{n}\sum_{i=1}^n \tau_n(\theta_i), \qquad \tau_n(\theta_i):= \min\left\{1,\quad \frac{K\sqrt{K}}{\beta_n \sqrt{n\bar{\theta} (\bar{\theta}\wedge \theta_i)}}\right\}. 
\eeq
Therefore, if an estimate $\hat{\Pi}$ satisfies $\|\hat{\pi}_i-\pi_i\|_1\leq C\tau_n(\theta_i)$ simultaneously for all nodes $i$, then ${\cal L}(\hat{\Pi},\Pi)\leq Cerr_n(\theta)$. In other words, ``rate-optimality for arbitrary $\theta$" reduces to the task of achieving the correct node-wise errors. 

We can use \eqref{NodewiseRate} to check if an existing method is optimal. 
Let $\hat{\pi}_i^{\text{MSCORE}}$ be the estimated $\pi_i$ from the conventional Mixed-SCORE.  \cite{Mixed-SCORE} showed that with high probability ($\theta_{\min}$ and $\theta_{\max}$ denote the respective minimum and maximum of $\theta_1,\theta_2,\ldots,\theta_n$), 
\beq \label{Nodewise-MSCORE}
\Vert \hat{\pi}_i^{\text{MSCORE}}- \pi_i\Vert_1 \leq C\min\left\{1, \quad \frac{K\sqrt{K}}{\beta_n\sqrt{n\bar{\theta}\theta_{\min}}} \sqrt{\frac{\theta_{\max}^3\bar{\theta}^2\log(n)}{\theta_{\min}(n^{-1}\|\theta\|^2)^2}}\right\}. 
\eeq
When $\theta_i$'s are at the same order (i.e., moderate degree heterogeneity),  the right hand sides of \eqref{NodewiseRate} and \eqref{Nodewise-MSCORE} are the same up to a logarithmic factor. However, the two rates do not match for a general $\theta$ vector, implying non-optimality under severe degree heterogeneity.  

How to improve Mixed-SCORE to achieve the node-wise error rate in \eqref{NodewiseRate}?
Recall that the conventional Mixed-SCORE fixed $b=0$ in \eqref{Def:Laplacian}. Our proposal is using a theory-guided choice of $b$. 
This problem is connected to many attempts in network data analysis that try to find a good normalization of the adjacency matrix. 
To name a few, \cite{chaudhuri2012spectral} proposed to apply spectral clustering algorithms to $H^{-1}A$ (where $A$ and $H$ are the same as in \eqref{Def:Laplacian}) and they justified their approach by providing a sharp concentration inequality for the spectral norm of $H^{-1}A$ under a DCBM model. \cite{sarkar2015role} studied the leading eigenvectors of a graph Laplacian (a special case of \eqref{Def:Laplacian} with $b=1/2$ and $\tau=0$) under the SBM model, and they showed that the Laplacian normalization can improve an eigenvector-based clustering-quality metric by a constant factor. 
\cite{qin2013regularized} studied a spectral clustering algorithm under the DCBM model, which used the matrix in \eqref{Def:Laplacian} with $b=1/2$ and a constant $\tau$, and they provided a large-deviation bound for $\|L-\mathbb{E}L\|$. 
However, these results are insufficient to provide insights for choosing $b$ in our problem, because our target rate-optimality is hinged on the {\it node-wise errors}. The previous works focused on analyzing either the spectral norm of $L$ \citep{chaudhuri2012spectral,qin2013regularized} or some quantities that aggregate all errors in a cluster \citep{sarkar2015role}. 
Instead, we consider the Mixed-SCORE algorithm with a general value of $b$ and study the impact of $b$ on {\it node-wise errors}. 

Our approach has two main components: (i) Perturbation analysis of the Mixed-SCORE algorithm, to link the node-wise error with the noise level in $\hat{r}_i$. This part is similar to the analysis in \cite{Mixed-SCORE} and is kept brief. (ii) Sharp eigenvector analysis, to show the influence of $b$ on the node-wise noise level in $\hat{r}_i$. 
In this part, a key step is deriving entry-wise large-deviation bounds for leading eigenvectors of $L$. Our analysis starts from the leave-one-out trick \citep{abbe2020entrywise} and applies new strategies to handle the degree normalization in $L$. The analysis for $b=1/2$ is presented in Section~\ref{sec:Entrywise-proof} in full detail. The analysis for a general $b$ is discussed in Section~\ref{sec:choose-b} of the supplementary material. 
In this section, due to space constraint, we only describe the high-level ideas and then present the final conclusions. 

Introduce $H_0 = \diag(\mathbb{E} d_i + \tau \mathbb{E} \bar{d})$ and $L_0=H_0^{-b}\Omega H_0^{-b}$ as the respective population counterpart of $H$ and $L$. Let $\lambda_k$ be the $k$th largest eigenvalue (in magnitude) of $L_0$, and let 
$\xi_k\in\mathbb{R}^n$ be the corresponding eigenvector. Write $\Xi=[\xi_1,\xi_2,\ldots,\xi_K]$ and $R=[\diag(\xi_1)]^{-1}[\xi_2,\ldots,\xi_K]$. Denote by $r_i'$ and $\hat{r}_i'$ the $i$th row of $R$ and $\hat{R}$, respectively. If we supply $R$ to Mixed-SCORE (see Figure~\eqref{fig:MSCORE-idea}), we obtain $\hat{\Pi}=\Pi$. Alternatively, if we supply $\hat{R}$, we obtain a noisy estimate of $\Pi$. Mimicking the perturbation analysis in \cite{Mixed-SCORE}, we can show that under mild conditions, with overwhelming probability, there exists an $K\times K$ orthogonal matrix $O$ such that 
\beq \label{SNR_i}
\|\hat{\pi}_i - \pi_i\|_1 \leq  C \|\hat{r}_i - r_i \|  \leq C \frac{\|e_i'(\hat{\Xi}O-\Xi)\|}{\xi_1(i)}, \qquad \mbox{simultaneously for all } 1\leq i\leq n. 
\eeq
As $b$ changes, both $\xi_1(i)$ and $\|e_i'(\hat{\Xi}O-\Xi)\|$ will change. The dependence of $\xi_1(i)$ on $b$ follows from elementary linear algebra. The dependence of $\|e_i'(\hat{\Xi}O-\Xi)\|$ on $b$ is difficult to study. It requires knowing row-wise large-deviation bounds  for $\hat{\Xi}$ (see Section~\ref{subsec:entrywise} and Section~\ref{sec:Entrywise-proof}). 
A common strategy \citep{abbe2020entrywise} is approximating each $\hat{\xi}_k$ by its first-order proxy defined as follows: Note that $L\hat{\xi}_k=\hat{\lambda}_k\hat{\xi}_k$ implies $\hat{\xi}_k(i) = \hat{\lambda}_k^{-1}e_i'L\hat{\xi}_k$.  When the signal-to-noise ratio is sufficiently large, $\|\hat{\xi}_k-\xi_k\|$ and $\|H-H_0\|$ are appropriately small. It follows that $
\hat{\xi}_k(i) \approx  \hat{\lambda}_k^{-1} e_i'L\xi_k=   \hat{\lambda}_k^{-1} e_i'H^{-b}AH^{-b}\xi_k \approx  \lambda_k^{-1} e_i'H_0^{-b}AH_0^{-b}\xi_k$. 
We define $\hat{\xi}_k^*={\color{red}\lambda_k^{-1}} H_0^{-b}AH_0^{-b}\xi_k$ as a proxy to $\hat{\xi}_k$ and write $\hat{\Xi}^*=[\hat{\xi}_1^*, \hat{\xi}_2^*, \ldots, \hat{\xi}_K^*]$. Now, $\hat{\Xi}^*$ is a linear mapping of $A$, easier to study. We relegate detailed discussions to Section~\ref{sec:choose-b} of the supplement but only present the final results: 
\beq \label{SNR-main}
\|\hat{\pi}_i-\pi_i\|_1\lesssim  \frac{C\sqrt{\log(n)}}{\sqrt{n\theta_i\bar{\theta}}} \cdot \delta(b,F_n), \qquad\mbox{where } \delta(b, F_n)= \frac{\sqrt{\int t^{3-4b}dF_n(t)}}{\int t^{2-2b}dF_n(t)}. 
\eeq

Given \eqref{SNR-main}, we have two important observations: First, the influence of $b$ on all $n$ nodes is the same, through a function $\delta(b, F_n)$ that is independent of $i$.  Therefore, we can select the best $b$ by simply minimizing $\delta(b, F_n)$. 
Second, there exists a universal choice of $b$ that minimizes $\delta(b, F_n)$ simultaneously for all $F_n(\cdot)$. To see this, we note that $\int tF_n(t)=1$, by Definition~\ref{Def-CDF}. The Cauchy-Schwarz inequality implies $\int t^{2-2b}dF_n(t)\leq \sqrt{\int t^{3-4b}dF_n(t)}\sqrt{\int tdF_n(t)}=\sqrt{\int t^{3-4b}dF_n(t)}$, so $\delta(b, F_n)$ can never be smaller than $1$. Meanwhile, if we set $b=1/2$, $\delta(b, F_n)$ becomes exactly $1$, regardless of $F_n(\cdot)$. 
Therefore, $b=1/2$ is the universally best choice.

From now on, we fix $b=1/2$. Then, $L$ is the normalized graph Laplacian, and $\delta(b, F_n)=1$. 
However, the right hand side of \eqref{SNR-main} is still slightly different from $\tau_n(\theta_i)$. In Figure~\ref{fig:entrywise-illustration} we have seen that $\tau_n(\theta_i)$ has three regions; in fact, \eqref{SNR-main} matches $\tau_n(\theta_i)$ only in the middle region (which is also the most subtle region). To match $\tau_n(\theta_i)$ in all three regions, we need some minor refinements of the estimation procedure, as described in Section~\ref{subsec:algorithm} below.

\subsection{The proposed algorithm} \label{subsec:algorithm}

Let $L$ be the normalized Laplacian matrix, corresponding to $b=1/2$ in \eqref{Def:Laplacian}. Let $\hat{r}_1,\hat{r}_2, \ldots,\hat{r}_n$ be the node embeddings as in \eqref{SCORE}. By \eqref{SNR-main}, the noise level in $\hat{r}_i$ is closely related to $\theta_i$. When $\theta_i$ is too small, we should delete such $\hat{r}_i$, as it brings in more noise than signals; additionally, we can get information of the order of $\theta_i$ from observed node degrees. To this end, given any constants $c>0$ and $\gamma\in (0,1)$, 
define:
\beq \label{keepSet-VS}
\hat{S}_n(c)= \bigl\{i: \; d_i \geq c\cdot K\hat{\lambda}_K^{-2}\log(n) \bigr\}, \qquad 
\hat{S}_n^*(c, \gamma) = \hat{S}_n(c)\cap \bigl\{i:\; d_i\geq \gamma\cdot\bar{d}\bigr\}. 
\eeq
Now, all nodes partition into three non-overlapping subsets: $M_1=\hat S_n^*(c, \gamma)$, which consists of high-degree and moderate-degree nodes,  $M_2=\hat{S}_n(c)\setminus \hat S_n^*(c, \gamma)$, which consists of low-degree nodes, and $M_3=\{1,2,\ldots,n\}\setminus \hat{S}_n(c)$, consisting of extremely low-degree nodes. 
For nodes in $M_3$, their $\hat{r}_i$'s are too noisy to contain useful information of $\pi_i$. We delete all such nodes from the embedded point cloud. 
For nodes in $M_2$, we exclude them from the vertex hunting step. From Figure~\ref{fig:MSCORE-idea}, we can see that not all $\hat{r}_i$'s are needed for estimating the simplex. We apply a vertex hunting algorithm (e.g., \citep{araujo2001successive}) on those $\hat{r}_i$'s of nodes in $M_1$. After vertex hunting, we estimate $\pi_i$ from $\hat{r}_i$, for all nodes in $M_1\cup M_2$.  See Algorithm~\ref{alg:MSL}. The pseudo-code, together with details of the vertex hunting algorithm, is in Section~\ref{sec:MSL} of the supplementary material. 

\spacingset{1}
\begin{table}[htb!]
\centering
\caption{A summary of node trimming. All nodes are used to obtain $\hat{R}$. All nodes receive an estimate $\hat{\pi}_i$, but for extremely low-degree nodes,  we assign a trivial $\hat{\pi}_i$ rather than estimating it. The three cases here correspond to the three regions in Figure~\ref{fig:entrywise-illustration}.} \label{tb:trimming}
\scalebox{.9}{
\begin{tabular}{cl|ccc}
\hline
Subset& Degree range& Node embedding & Vertex hunting & Non-trivial $\hat{\pi}_i$ \\[0.2cm]
\hline
$M_1$ & $\big[\gamma\bar{d},\, \, n \big]$ & yes & yes & yes \\[0.16cm]
$M_2$ & $\big[c K\hat{\lambda}_K^{-2}\log(n), \, \, \gamma\bar{d} \,  \big]$ & yes & no & yes  \\[0.16cm]
$M_3$ & $\big[0, \,\, c K\hat{\lambda}_K^{-2}\log(n) \big]$ & yes & no & no\\
\hline
\end{tabular}
}
\end{table}
\spacingset{1.7}


\spacingset{1}
\section{Entywise eigenvector analysis, node-wise errors, and rate-optimality}  \label{sec:UB}
\spacingset{1.7}

We present the theoretical properties of Mixed-SCORE-Laplacian (MSL). The backbone of our analysis is an entry-wise large-deviation bound for leading eigenvectors of the normalized graph Laplacian, which is presented in Section~\ref{subsec:entrywise}. The node-wise errors and rate-optimality are in Section~\ref{subsec:MSL-rate}. We also extend all results to a more general loss function in Section~\ref{subsec:otherLoss}.

\spacingset{1}
\begin{algorithm}[tb!]
\caption{Mixed-SCORE-Laplaccian (a high-level description).} \label{alg:MSL}
\medskip
\noindent
{\bf Input:} $K$, $A$, and tuning parameters $\tau=1$, $c=0.5$, and $\gamma=0.05$ (default values). 

\begin{enumerate}
\item Node Embedding: Fix $b=1/2$ in \eqref{Def:Laplacian} and obtain $\hat{R}=[\hat{r}_1, \hat{r}_2, \ldots,\hat{r}_n]'$ as in \eqref{SCORE}. 
\item Trimming: Let $\hat{S}_n(c)$ be as in \eqref{keepSet-VS}. For any $i\notin \hat{S}_n(c)$, set $\hat{\pi}_i=K^{-1}{\bf 1}_K$ and delete such nodes from all the remaining steps.  
\item (De-noised) Vertex Hunting: Let $\hat{S}_n^*(c, \gamma)$ be as in \eqref{keepSet-VS}. Run the successive projection algorithm \citep{araujo2001successive} on $\{\hat{r}_i: i\in \hat{S}_n^*(c, \gamma)\}$ to obtain $\hat{v}_1,\hat{v}_2,\ldots,\hat{v}_K$. 

\item Membership Estimation: Obtain $\hat{b}_1\in\mathbb{R}^K$ by $\hat{b}_1(k)=[\hat{\lambda}_1+\hat{v}_k'\diag(\hat{\lambda}_2,\ldots,\hat{\lambda}_K)\hat{v}_k]^{-1/2}$.  For each $i\in \hat{S}_n(c)$, solve its barycentric coordinate $\hat{w}_i\in\mathbb{R}^K$. 
Let $\hat{\pi}_i^*\in\mathbb{R}^K$ be such that $\hat{\pi}_i^*(k)=\max\{\hat{w}_i(k)/\hat{b}_1(k),\, 0\}$. Obtain $\hat{\pi}_i$ by normalizing $\hat{\pi}_i^*$ to have a unit $\ell^1$-norm. 
\end{enumerate}
{\bf Output:} $\hat{\Pi}$.
\end{algorithm}
\spacingset{1.7}

\subsection{Entry-wise eigenvector analysis for graph Laplacian} \label{subsec:entrywise}

Fix $b=1/2$ and $\tau=1$ in the definition of $L$ in \eqref{Def:Laplacian}.  For $1\leq k\leq K$, let $\hat{\lambda}_k$ be the $k$th largest eigenvalue (in magnitude) of $L$,  and let $\hat{\xi}_k$ be the associated eigenvector. Define
\beq \label{Def:Laplacian0}
L_0 = H_0^{-1/2}\Omega H_0^{-1/2}, \qquad \mbox{where}\quad H_0 = \mathbb{E}[H]. 
\eeq
For $1\leq k\leq K$, let $\lambda_k$ be the $k$th largest eigenvalue (in magnitude) of $L_0$, and let ${\xi}_k$ be the corresponding eigenvector.
Write $\Xi_1=[\xi_2,\ldots,\xi_K]$, $\Xi=[\xi_1, \Xi_1]$, $\hat{\Xi}_1=[\hat{\xi}_2,\ldots,\hat{\xi}_K]$, and $\hat{\Xi}=[\hat{\xi}_1, \hat{\Xi}_1]$. 
By Condition~\ref{reg-conds}(b), there is a gap between $\lambda_1$ and the other eigenvalues; hence,   
by sin-theta theorem \citep{sin-theta}, we can obtain upper bounds for $
\|\hat{\xi}_1-\xi_1\|$ and $\min_{O_1}\|\hat{\Xi}_1O_1-\Xi_1\|_F$, 
where the minimum is taken over all orthogonal matrices $O_1\in\mathbb{R}^{(K-1)\times (K-1)}$. However, these bounds are insufficient for controlling  node-wise errors of MSL. We now bound each entry of $(\hat{\xi}_1-\xi_1)$ and each row of $(\hat{\Xi}_1O_1-\Xi_1)$.
Let $e_i\in\mathbb{R}^n$ denote the $i$th standard basis of $\mathbb{R}^n$. 

\begin{thm}[Row-wise large-deviation bounds for $\hat{\Xi}$]\label{thm:eigenvector}
Consider the DCMM model \eqref{mod-DCMM}-\eqref{mod-DCMM2}, where Condition~\ref{reg-conds}(a)-(c) are satisfied. 
Suppose $K^3 \log(n)/(n\bar{\theta}^2\beta_n^2)\to 0$ as $n\to\infty$. 
With probability $1-o(n^{-3})$,  there exists $\omega\in \{1, -1\}$ and  an orthogonal matrix $O_1\in \mathbb{R}^{(K-1)\times(K-1)}$ such that simultaneously for all $1\leq i\leq n$, 
\begin{align}
|\omega\hat{\xi}_1(i)-\xi_1(i)| & \leq C \sqrt{\frac{K\theta_i\log(n)}{n^2\bar{\theta}^3 }}\Biggl(1+ \sqrt{\frac{\log(n)}{n\bar{\theta}\theta_i}}\, \Biggr),\label{result-1-mainpaper}\\
\Vert e_i' (\hat{\Xi}_1 O_1 - \Xi_1)\Vert&\leq  C \sqrt{\frac{K^3 \theta_i\log(n)}{n^2\bar{\theta}^3 \beta_n^2}}\Biggl(1+ \sqrt{\frac{\log(n)}{n\bar{\theta}\theta_i}} \, \Biggr).  \label{result-2-mainpaper}
\end{align}
\end{thm}

Theorem~\ref{thm:eigenvector} is one of our major technical contributions. 
We sketch the proof ideas in Section~\ref{sec:Entrywise-proof} and relegate the full proof to the supplementary material. In this theorem, we make an assumption: $K^3 \log(n)/(n\bar{\theta}^2\beta_n^2)\to 0$. In a simple case where $K$ is finite and $\beta_n^{-1}=O(1)$, this assumption reduces to $n\bar{\theta}^2\gg\log(n)$. It means the average degree only needs to grow faster than $\log(n)$. Hence, Theorem~\ref{thm:eigenvector} can handle sparse networks up to logarithmic degree regime. Similarly, all other theorems in this section cover sparse networks. 

\subsection{Node-wise errors and rate-optimality} \label{subsec:MSL-rate}

We apply Theorem~\ref{thm:eigenvector} to establish the theoretical properties of MSL. First, we give a large-deviation bound for each row of $\hat{R}$. 
Define $R\in\mathbb{R}^{n\times (K-1)}$ by 
\beq \label{oracle1}
R(i,k) = \xi_{k+1}(i)/\xi_1(i), \qquad1\leq i\leq n,1\leq k\leq K-1. 
\eeq
Let $r'_1,r'_2,\ldots,r_n'$ be the rows of $R$. It is seen that $\hat{r}_i'=[\hat{\xi}_1(i)]^{-1}\cdot e_i'\hat{\Xi}_1$ and $r_i'=[\xi_1(i)]^{-1}\cdot e_i' \Xi_1$. Therefore, we can control the difference between $\hat{r}_i$ and $r_i$ by applying Theorem~\ref{thm:eigenvector}: 

\begin{lem} \label{lem:hatR}
Consider the DCMM model in \eqref{mod-DCMM}-\eqref{mod-DCMM2}, where Condition~\ref{reg-conds}(a)-(c) are satisfied. Fix $b=1/2$ and $\tau=1$ in \eqref{Def:Laplacian}.  
Suppose $K^3 \log(n)/(n\bar{\theta}^2\beta_n^2)\to 0$ as $n\to\infty$. For any constant $c_0>0$, define $
S_n(c_0):=\big\{1\leq i\leq n:\; n\bar{\theta} \theta_i \beta_n^2\geq c_0 K^3 \log (n) \big\}$. 
With probability $1-o(n^{-3})$,  there exists an orthogonal matrix $O_1\in \mathbb{R}^{K-1, K-1}$ such that, simultaneously for $ i \in S_n(c_0)$,  
\beq \label{bound-hatR}
\Vert O_1' \hat{r}_i - r_i \Vert \leq C \sqrt{\frac{K^3\log (n)}{n\bar{\theta} (\bar{\theta}\wedge \theta_i)\beta_n^2}}.   
\eeq
The statement holds for any $c_0>0$, except that the constant $C$ will depend on $c_0$. 
\end{lem}

Lemma~\ref{lem:hatR} only considers nodes in $S_n(c_0)$. For a node $i\notin S_n(c_0)$, its degree is so small that $\hat{r}_i$ is too noisy to contain useful information of $\pi_i$. In MSL, the set $S_n(c_0)$ is estimated by $\hat{S}_n(c)$, and those $\hat{r}_i$'s with $i\notin \hat{S}_n(c)$  are discarded.

Next, we present the node-wise error of MLS. 
\begin{thm}[Node-wise errors and rate-optimality] \label{thm:uppbd_pi}
Consider the DCMM model in \eqref{mod-DCMM}-\eqref{mod-DCMM2}, where Condition~\ref{reg-conds}(a)-(c) hold, and additionally, Condition~\ref{reg-conds}(d) is satisfied. Fix $b=1/2$ and $\tau=1$ in \eqref{Def:Laplacian}.   
Suppose $K^3 \log(n)/(n\bar{\theta}^2\beta_n^2)\to 0$ as $n\to\infty$. Let $\hat{\Pi}$ be the output of Algorithm~\ref{alg:MSL}, in which $(c,\gamma)$ are positive constants such that $\gamma<c_4$ ($c_4$ is the same as in Condition~\ref{reg-conds}(d)). With probability $1-o(n^{-3})$, there exists a permutation $T$ on $\{1,2,\ldots,K\}$, such that simultaneously for all $1\leq i\leq n$, 
\beq \label{Nodewise-Err-mainpaper}
\Vert T\hat{\pi}_i- \pi_i\Vert_1 \leq  C \min\left\{ \sqrt{\frac{K^3\log (n)}{n\bar{\theta} (\bar{\theta}\wedge \theta_i)\beta_n^2}},\;\; 1\right\}, 
\eeq
In addition, let 
${\cal L}(\hat{\Pi},\Pi)$ be the $\ell^1$-loss in \eqref{Loss}, and $err_n(\theta)$ be as in \eqref{OptRate}. Then, 
\[
\mathbb{E}{\cal L}(\hat{\Pi},\Pi) \leq C err_n(\theta)\sqrt{\log(n)}.
\] 
\end{thm}


Up to a logarithmic factor, the node-wise error in \eqref{Nodewise-Err-mainpaper} matches with the right hand side of \eqref{NodewiseRate} for every $\theta_i$, and the optimality of MSL follows immediately. 

An interesting question is if we can further remove the $\sqrt{\log(n)}$-factor in the upper bound. This factor arises from the entry-wise eigenvector analysis in Section~\ref{subsec:entrywise}. We conjecture that it is not easy to remove, due to the nature of the leave-one-out trick in eigenvector analysis (see Section~\ref{sec:Entrywise-proof}). Another question is if the requirement $K^3 \log(n)/(n\bar{\theta}^2\beta_n^2)\to 0$ can be further relaxed to allow for bounded degrees. We conjecture that it is impossible. In the community detection literature, there are many results for the bounded-degree regime. However, mixed membership estimation is a more challenging task than community detection. To our best knowledge, there has been no result for the bounded-degree regime.    

Theorem~\ref{thm:uppbd_pi} not only shows the optimality of MSL under a global loss but also provides insights on local errors. 
For example, if we only want to estimate one particular $\pi_i$, 
should we pre-process the network by removing low-degree nodes or grouping nodes of similar degrees?
This is answered by \eqref{Nodewise-Err-mainpaper}. Suppose $i$ is a moderate-degree node, with $\theta_i=O(\bar{\theta})$ and $n\bar{\theta}\theta_i\beta_n^2\gg K^3\log(n)$. Then, \eqref{Nodewise-Err-mainpaper} becomes $\Vert T\hat{\pi}_i- \pi_i\Vert_1\leq C(n\bar{\theta}\theta_i\beta_n^2)^{-1/2}\sqrt{K^3\log(n)}$. The total effect of other nodes at $\hat{\pi}_i$ is carried by $n\bar{\theta}=\theta_i+\sum_{j:j\neq i}\theta_j$. As a result, to achieve the best accuracy of $\hat{\pi}_i$, we should not remove any node or group nodes of similar degrees. \spacingset{1}\footnote{This argument only says that we should use {\it all} nodes to compute $\hat{R}$, but it does not imply we should use all rows of $\hat{R}$ in the subsequent steps. There is no conflict with the node trimming strategy in Section~\ref{subsec:algorithm}.} \spacingset{1.7}

\subsection{Extension to other loss functions} \label{subsec:otherLoss}
We extend the results in Section~\ref{sec:LB} and Section~\ref{subsec:MSL-rate} to a general loss function: 
\beq \label{generalLoss}
 {\cal L}(\hat{\Pi},\Pi; p, q) =  \min_{T}\biggl\{\Bigl(\frac{1}{n}\sum_{i=1}^n (\theta_i/\bar{\theta})^{p} \|T\hat{\pi}_i-\pi_i\|_q^q\Bigr)^{1/q}\biggr\}, \quad\mbox{for $p\geq 0$ and $q\geq 1$}. 
\eeq
The $\ell^1$-loss in \eqref{Loss} is a special case with $p=0$ and $q=1$. When $p>0$, the estimation errors are re-weighted by degree parameters. In many real applications, the $\pi_i$ of high-degree nodes are more interesting, so this general loss metric is relevant. We are particularly interested in the case of $p=1/2$ and $q=1$. We write ${\cal L}^w(\hat{\Pi},\Pi):= {\cal L}(\hat{\Pi},\Pi; 1/2, 1)$ and call it the {\it weighted $\ell^1$-loss}. 
Same as before, let $err_n$ denote the baseline rate in \eqref{BaseRate}

\begin{cor} \label{cor:rates-generalLoss}
Suppose conditions of Theorem \ref{thm:uppbd_pi} hold. Let $\hat{\Pi}$ be the estimator from Mixed-SCORE-Laplacian and ${\cal L}(\hat{\Pi},\Pi; p,q)$ be the loss metric in \eqref{generalLoss}, for $p\geq 0$ and $q\geq 1$. Then,  
\beq \label{L1-generalLoss-rate}
 {\cal L}(\hat{\Pi},\Pi; p, q)\leq C\sqrt{\log(n)}  \Big(\int t^p\min\Bigl\{ \frac{err^q_n}{(t\wedge 1)^{q/2}},\, 1\Bigr\}  dF_n(t) \Big)^{1/q}. 
\eeq
Furthermore, in the special case of $p=1/2$ and $q=1$, 
$\mathbb{E}{\cal L}^w(\hat{\Pi},\Pi) \leq C\sqrt{\log(n)}\, err_n$. 
\end{cor}

For the loss metric ${\cal L}^w(\hat{\Pi},\Pi)$, we provide a matching lower bound as follows. It suggests that 
Mixed-SCORE-Laplacian is still rate-optimal (for arbitrary $\theta$) under this loss metric.

\begin{thm}\label{thm:LB2}
Fix constants $c_1$-$c_4$. 
Given $(n,K)$ and $\theta\in \mathbb{R}^n$ such that $F_n(err_n^2)\leq \check{c}$, for a constant $\check{c}\in (0,1)$. Let ${\cal Q}_n(\theta)$ be the collection of $(\Pi, P)$ satisfying Condition~\ref{reg-conds}. There is  a constant  $C>0$ such that, for all sufficiently large $n$, 
$\inf_{\hat{\Pi}}\sup_{(\Pi, P)\in {\cal Q}_n(\theta)}\mathbb{E}{\cal L}^w(\hat{\Pi},\Pi) \geq C err_n$.  
\end{thm}

\section{Proof of Theorem~\ref{thm:eigenvector}: Challenges, and our strategy} \label{sec:Entrywise-proof}


Theorem~\ref{thm:eigenvector} is connected to the literature of ``entry-wise eigenvector analysis" for random graphs \citep{abbe2020entrywise,chen2021spectral,erdHos2013spectral,fan2019simple,Mixed-SCORE,mao2020estimating,tang2018limit} but has significant differences. Most existing works (a) study eigenvectors of the adjacency matrix but not the normalized Laplacian, (b) consider a network model that has no degree heterogeneity or only mild degree heterogeneity ($\theta_{\min}\asymp \theta_{\max}$), and (c) derive the ``2-to-infinity" bounds, i.e., upper bounds for $\max_{i}|\omega\hat{\xi}_1(i)-\xi_1(i)|$ or $\max_{i}\|e_i'(\hat{\Xi}_1O_1-\Xi_1)\|$, instead of obtaining different bounds for different $i$. As a result, Theorem~\ref{thm:eigenvector} cannot be deduced from any of the existing results. 
In what follows, Section~\ref{subsec:leave-one-out} reviews the leave-one-out strategy for entry-wise eigenvector analysis, Section~\ref{subsec:why-difficult} explains the technical challenges, and Section~\ref{subsec:Entrywise-proof-sketch} contains a proof sketch of Theorem~\ref{thm:eigenvector}.

\subsection{A review of the leave-one-out approach} \label{subsec:leave-one-out}

{\it Abbe et al.} \cite{abbe2020entrywise} introduced the leave-one-out trick to study the eigenvectors of the adjacency matrix $A$ when the network follows an SBM model (a special case of DCMM  when $\theta_i$'s are equal to each other and when $\pi_i$'s are all degenerate). 
Let $\hat{\lambda}_k^*$ be the $k$th largest eigenvalue (in magnitude) of $A$ and let $\hat{\xi}_k^*\in\mathbb{R}^n$ be the corresponding eigenvector. Let $\lambda_k^*$ be the $k$th largest eigenvalue (in magnitude) of $\Omega$ and let $\xi^*_k\in\mathbb{R}^n$ be the associated eigenvector. Write $\hat{\Xi}^*=[\hat{\xi}_1^*,\ldots,\hat{\xi}_K^*]$ and $\Xi^*=[\xi_1^*,\ldots,\xi_K^*]$. Let $\|\cdot\|_{2\to\infty}$ denote the maximum row-wise $\ell^2$-norm of a matrix. 
\cite{abbe2020entrywise} derived a large-deviation bound for $\min_{O}\|\hat{\Xi}^*O-\Xi^*\|_{2\to\infty}$, where the minimum is over $K\times K$ orthogonal matrices.  
To explain their proof strategies, we consider $K=1$. Their goal reduced to obtaining a bound for $\|\hat{\xi}^*_1-\xi^*_1\|_\infty$, up to a sign flip of $\hat{\xi}_1$. By definition, $A\hat{\xi}_1^* = \hat{\lambda}^*_1\hat{\xi}_1^*$. 
It is seen that 
\beq \label{leave-one-out-1}
\hat{\xi}_1^*(i) = (1/\hat{\lambda}_1^{*}) \, e_i'A\hat{\xi}_1^* = (1/\hat{\lambda}_1^*) \sum_{j: j\neq i}\hat{\xi}_1^*(j) A(i,j).
\eeq
If $\hat{\xi}_1^*$ is independent of the $i$th row of $A$, then the right hand side is nothing but a weighted sum of independent Bernoulli variables and can be easily analyzed. 
However, $\hat{\xi}_1^*$ is dependent of the $i$th row of $A$. The leave-one-out trick creates a proxy of $\hat{\xi}_1^*$ that is independent of the $i$th row of $A$. Fix $i$ and let $\tilde{\xi}_1^{*}$ be the leading eigenvector of the matrix $\tilde{A}$, obtained by setting the $i$th row and the $i$th column of $A$ to zero. If $\tilde{\xi}_1^*$ is sufficiently close to $\hat{\xi}_1^*$, then 
\beq \label{leave-one-out-2}
\hat{\xi}_1^*(i)\approx (1/\hat{\lambda}_1^*) \sum_{j: j\neq i}\tilde{\xi}_1^*(j) A(i,j). 
\eeq
Since $\tilde{\xi}^*_1$ is independent of the $i$th row of $A$, the right hand side of \eqref{leave-one-out-2} is now easy to analyze. 
Although this leave-one-out trick is easy to describe, it does require tedious analysis to control the ``approximation" in \eqref{leave-one-out-2}. 
\cite{abbe2020entrywise} proposed to control the difference between the left and right hand sides of \eqref{leave-one-out-2} using $\|\hat{\xi}_1^*-\tilde{\xi}_1^*\|$ and then apply the sin-theta theorem to bound $\|\hat{\xi}_1^*-\tilde{\xi}_1^*\|$.

\subsection{Challenges in proving Theorem~\ref{thm:eigenvector} and how to overcome them} \label{subsec:why-difficult}

Suppose we want to extend the analysis to study eigenvectors of the normalized Laplacian matrix $L$, assuming that the network follows a DCMM model. 
Recall that $
L = H^{-\frac12}AH^{-\frac12}$ and $L_0 = H_0^{-\frac 12}\Omega H_0^{-\frac 12}$, 
where $H$ is a diagonal matrix constructed from degrees. The first challenge is that entries in the upper triangle of $L$ are no longer independent of each other. If we mimick \cite{abbe2020entrywise} to define a proxy of $\hat{\xi}_1$ as the leading eigenvector of the matrix obtained by setting the $i$th row \& column of $L$ to zero, this vector is still dependent of the $i$th row of $A$. 

Our solution is to create a proxy to $L$ such that the effect of the $i$th row of $A$ is removed. 
Recall that $W=A-\mathbb{E}[A]$. 
Let $W^{(i)}$ be the matrix obtained by setting the $i$-th row/column of $W$ to zero. Introduce $A^{(i)}=\Omega-{\rm diag}(\Omega) + W^{(i)}$ and $
\tilde{H}^{(i)}=\diag\bigl( A^{(i)}{\bf 1}_n\bigr) + n^{-1}({\bf 1}_n' A^{(i)}{\bf 1}_n)I_n$. 
It can be shown that $(A^{(i)}, \tilde{H}^{(i)})$ are independent of the $i$th row of $A$, and they are appropriately close to $(A, H)$. Define
\beq \label{Intermediate-Matrices}
\overline{L}^{(i)}: =(\tilde{H}^{(i)})^{-\frac 12} A^{(i)} (\tilde{H}^{(i)})^{-\frac 12}. 
\eeq
Let $(\overline{\lambda}_1^{(i)}, \overline{\xi}^{(i)}_1)$ be the first eigen-pair of $\overline{L}^{(i)}$. Then, $\overline{\xi}_1^{(i)}$ is the proxy of $\hat{\xi}_1$ we seek for.

Given $\overline{\xi}_1^{(i)}$, is it easy to study $\hat{\xi}_1(i)-\xi_1(i)$ by mimicking the analysis in \cite{abbe2020entrywise}? Unfortunately, the answer is no. Since $\hat{\xi}_1 = (1/\hat{\lambda}_1)H^{-1/2}AH^{-1/2}\hat{\xi}_1$, we can re-write $\hat{\xi}_1(i)$ as
\beq \label{leave-one-out-new}
\hat{\xi}_1(i) = \frac{1}{\hat{\lambda}_1\sqrt{H(i,i)}}\Biggl[\sum_{j:j\neq i}\frac{\overline{\xi}^{(i)}_1(j)}{\sqrt{H(j,j)}}A(i,j)+\overline{e}_i \Biggr], 
\eeq
where $\overline{e}_i$ is the approximation error to $(1/\hat{\lambda}_1)e_i'H^{-1/2}AH^{-1/2}\hat{\xi}_1$ by replacing $\hat{\xi}_1$ by $\overline{\xi}^{(i)}_1$ in this expression. To analyze the right hand side of \eqref{leave-one-out-new}, we face two more challenges:
\begin{itemize}
\item On the right hand side of \eqref{leave-one-out-new}, $H(j,j)$ is still dependent of the $i$th row of $A$. We may replace it with $\tilde{H}^{(i)}(j,j)$. However, this replacement affects every term in the sum, and its effect is not easy to control. Even if we can control it, we are still unable to show that the resulting quantity is close enough to $\xi_1(i)$, because the effect of $|\tilde{H}^{(i)}(j,j)-H_0(j,j)|$ is still non-negligble. 
\item The control of $\overline{e}_i$ is much more challenging than similar steps in \cite{abbe2020entrywise}, for three reasons: First, since \cite{abbe2020entrywise} studies the eigenvectors of $A$ instead of those of $L$, the proxy eigenvector is defined in a more straightforward way; consequently, the proxy error has a simpler form. Second,  
in the setting of \cite{abbe2020entrywise}, the network model  has no degree heterogeneity, and the target bound for each entry of $\hat{\xi}_1^*$ is at the same order; then the desirable bound for $\overline{e}_i$ is also the same for all $i$.
However, in our setting, the desirable bound for $\overline{e}_i$ may be significantly different for different $i$, so we must have {\it better} control of  each $\overline{e}_i$. 
Third, the idea in \cite{abbe2020entrywise} for controlling $\overline{e}_i$  is based on studying the $\ell^2$-error between the empirical eigenvector and its proxy. However, for our problem, it is impossible to control $\overline{e}_i$ from studying $\|\hat{\xi}_1-\overline{\xi}^{(i)}_1\|$, because the entries of $\hat{\xi}_1-\overline{\xi}^{(i)}_1$ can be at different orders due to degree heterogeneity, and analysis based on $\|\hat{\xi}_1-\overline{\xi}^{(i)}_1\|$ is not sharp. 
\end{itemize}

To overcome the first challenge, we introduce another proxy eigenvector. Let $\tilde{H}^{(i)}$ be the same as in \eqref{Intermediate-Matrices}. Let $(\tilde{\lambda}_1^{(i)}, \tilde{\xi}^{(i)}_1)$ be the first eigen-pair of the following matrix:
\beq \label{Intermediate-Matrices2}
\tilde{L}^{(i)}:= (\tilde{H}^{(i)})^{-\frac 12} \Omega (\tilde{H}^{(i)}) ^{-\frac 12}. 
\eeq 
Our idea is to use $\tilde{\xi}_1^{(i)}$ as a proxy to $\xi_1$, and study $|\hat{\xi}_1(i) - \tilde{\xi}^{(i)}_1(i)|$ rather than $|\hat{\xi}_1(i)-\xi_1(i)|$. 
In the first bullet point above, we have mentioned that the effect of $|\tilde{H}^{(i)}(j,j)-H_0(j,j)|$ is non-negligible when we try to bound $|\hat{\xi}_1(i)-\xi_1(i)|$ directly. 
Comparing $\tilde{L}^{(i)}$ and $L_0$, the difference is that $H_0$ is replaced by $\tilde{H}^{(i)}$. Hence, when we bound 
$|\hat{\xi}_1(i) - \tilde{\xi}^{(i)}_1(i)|$, there is no longer any term caused by $|\tilde{H}^{(i)}(j,j)-H_0(j,j)|$; and the issue is partially resolved. 
It still remains to bound $|\tilde{\xi}^{(i)}_1(i)-\xi_1(i)|$. This seems to involve $|\tilde{H}^{(i)}(j,j)-H_0(j,j)|$ again. Fortunately, since $L_0$ and $\tilde{L}^{(i)}$ are both low-rank matrices, we can study the difference between $\tilde{\xi}_1^{(i)}$ and $\xi_1$ in a different way to avoid using $|\tilde{H}^{(i)}(j,j)-H_0(j,j)|$ explicitly (see Lemma~\ref{lem:upbd_key1} below).


To overcome the second challenge, we notice that under severe degree heterogeneity, the noise level is different at different entries of $\hat{\xi}_1-\overline{\xi}_1^{(i)}$, and the same happens for $\overline{\xi}_1^{(i)}-\tilde{\xi}_1^{(i)}$. The key question is how to ``track" such heterogeneous noise levels in every step of our analysis of $\overline{e}_i$ to ensure that the resulting bounds are good for all $1\leq i\leq n$. 
We map out the analysis by two key technical lemmas, Lemmas~\ref{lem:upbd_key11}-\ref{lem:techB1}.  
Recall that we focus on studying $|\hat{\xi}_1(i)-\tilde{\xi}^{(i)}_1(i)|$ (see the previous paragraph). 
These two lemmas together establish an inequality:
\begin{align} 
 |\hat{\xi}_1(i)-\tilde{\xi}^{(i)}_1(i)| &\leq c_{1n}(\theta_i)+ c_{2n}(\theta_i) \cdot \|(\tilde{H}^{(i)})^{-\frac12}(\hat{\xi}_1-\tilde{\xi}^{(i)}_1)\|_\infty \label{Illustrate-Proof-1} \\ 
 &\leq c_{1n}(\theta_i)+ c_{2n}(\theta_i) \cdot \|(\tilde{H}^{(i)})^{-\frac12} H_0^{\frac12}\|\cdot \|H_0^{-\frac12}(\hat{\xi}_1-\tilde{\xi}^{(i)}_1)\|_\infty, 
 \label{Illustrate-Proof-2}
\end{align}
where $c_{1n}(\theta_i)$ and $c_{2n}(\theta_i)$ are some explicit sequences.
We note that \eqref{Illustrate-Proof-2} follows immediately from \eqref{Illustrate-Proof-1}. 
Given \eqref{Illustrate-Proof-2}, we first multiply $H_0^{-\frac 12}(i,i)$ on both hand sides to obtain a bound for $\|H_0^{-\frac12}(\hat{\xi}_1-\tilde{\xi}^{(i)}_1)\|_\infty$ and then plug this bound into \eqref{Illustrate-Proof-2} again to get the bound for $|\hat{\xi}_1(i)-\tilde{\xi}^{(i)}_1(i)|$. 
The non-trivial part is deriving \eqref{Illustrate-Proof-1} --- the backbone of our leave-one-out analysis.  
The proof of \eqref{Illustrate-Proof-1} ``tracks" the heterogeneous noise levels by repeatedly using the idea of ``re-weighting": For example, Lemma~\ref{lem:upbd_key11} introduces a re-weighted noise matrix $\Delta^{(i)}:= (\tilde{H}^{(i)})^{-\frac{1}{2}} W {H}^{-\frac{1}{2}}$, 
and Lemma~\ref{lem:techB1} controls every term by using $(\tilde{H}^{(i)})^{-\frac12}(\hat{\xi}_1-\tilde{\xi}^{(i)}_1)$ rather than the unweighted vector $\hat{\xi}_1-\tilde{\xi}^{(i)}_1$. In fact, in the proof of these lemmas, we have properly ``re-weighted" almost every intermediate vector and matrix by either $H^{-\frac12}$, or $H^{-\frac12}_0$, or $(\tilde{H}^{(i)})^{-\frac12}$.

\subsection{Proof sketch of Theorem~\ref{thm:eigenvector}} \label{subsec:Entrywise-proof-sketch}
We prove the claim about $\hat{\xi}_1$ in Theorem~\ref{thm:eigenvector}. The proof of the claim for $\hat{\Xi}_1$ follows the same vein but is slightly more complicated, which is relegated to the supplementary material.   

Let $\overline{\xi}_1^{(i)}$ and $\tilde{\xi}_1^{(i)}$ be as defined in \eqref{Intermediate-Matrices} and \eqref{Intermediate-Matrices2}, respectively. As mentioned, $\overline{\xi}_1^{(i)}$ is a proxy to $\hat{\xi}_1$, and $\tilde{\xi}_1^{(i)}$ is a proxy to $\xi_1$. 
The next lemma controls the difference between $\xi_1$ and $\tilde{\xi}_1^{(i)}$. 
\begin{lem}\label{lem:upbd_key1}
Suppose the conditions of Theorem \ref{thm:eigenvector} hold. We pick the sign of $\xi_1$ such that $\xi_1(1)\geq 0$. For each $1\leq i\leq n$, we pick the sign of $\tilde{\xi}^{(i)}_1$ such that $\tilde{\xi}_1^{(i)}(1)\geq 0$. With probability $1-o(n^{-3})$, simultaneously for all $1\leq i, j\leq n$,  
\begin{align}\label{eq:t1}
|\tilde{\xi}_1^{(i)}(j)-\xi_1(j)| \leq  CK^{\frac 12} \kappa_j\cdot \Big( \sqrt{\frac{\theta_j}{\bar{\theta}}} \wedge 1\, \Big), \qquad\mbox{where}\quad \kappa_j:=\sqrt{\frac{\log (n)}{n\bar{\theta}^2 }} \sqrt{\frac{\theta_j}{n\bar{\theta}}}. 
\end{align}
\end{lem}

Next, we study the difference between $\tilde{\xi}^{(i)}_1(i)$ and $\hat{\xi}_1(i)$. 
We use $\overline{\xi}_1^{(i)}$ as a bridge quantity. The following lemma is proved in the supplementary material:

\begin{lem}\label{lem:upbd_key11}
Suppose the conditions of Theorem \ref{thm:eigenvector} hold. For $1\leq i\leq n$, let $\kappa_i$ be as in \eqref{eq:t1}, and define $\Delta= \Delta^{(i)}:= (\tilde{H}^{(i)})^{-\frac{1}{2}} W {H}^{-\frac{1}{2}} $. We pick the signs of $\xi_1$ and $\tilde{\xi}_1^{(i)}$ in the same way as in Lemma~\ref{lem:upbd_key1}, and we pick the sign of $\overline{\xi}^{(i)}_1$ such that ${\rm sgn} ( (\tilde{\xi}^{(i)}_1)' \overline{\xi}^{(i)}_1) = 1$.  With probability $1-o(n^{-3})$, there exists some $w\in \{1, -1\}$ such that, simultaneously for all $1\leq i\leq n$, 
\begin{align}
&|w\hat{\xi}_1(i)- \tilde{\xi}^{(i)}_1(i)|\leq  C\sqrt K \kappa_i  + C |e_i'\Delta\hat{\xi}_1|,\label{Laplacian-entry-1}\\
&|e_i'\Delta\hat{\xi}_1| 
\leq  |e_i'\Delta \tilde{\xi}^{(i)}_1| + |e_i'\Delta(\overline{\xi}^{(i)}_1- \tilde{\xi}^{(i)}_1)| +\frac{C}{\sqrt{n\bar{\theta}^2}} \| w\hat{\xi}_1-\overline{\xi}^{(i)}_1\|.  \label{Laplacian-entry-2}
\end{align}
\end{lem}

Write $\overline{\xi}_1^{(i)}=\overline{\xi}_1$,  $\tilde{\xi}_1^{(i)}=\tilde{\xi}_1$  and $\tilde{H}^{(i)}=\tilde{H}$ for short. The role of Lemma~\ref{lem:upbd_key11} is to reduce the analysis of $|\hat{\xi}_1(i)-\tilde{\xi}_1(i)|$ to the analyses of $|e_i'\Delta\tilde{\xi}_1|$, $|e_i'\Delta(\overline{\xi}_1- \tilde{\xi}_1)|$ and $\|\hat{\xi}_1-\overline{\xi}_1\|$. Among the three quantities, the first two are relatively easy to bound, because $e_i'\Delta \approx e_i'H_0^{-1/2}WH_0^{-1/2}$, which is independent of $(\overline{\xi}_1,\tilde{\xi}_1)$; 
To control the third term $\|\hat{\xi}_1-\overline{\xi}_1\|$, we need the following lemma, whose proof can be found in the supplementary material:  

\begin{lem}\label{lem:techB1}
Under the assumptions in Lemma \ref{lem:upbd_key11}, with probability $1-o(n^{-3})$, for the same $w\in \{1, -1\}$ in Lemma \ref{lem:upbd_key11}, simultaneously for all $1\leq i\leq n$,
\begin{align}  
|e_i'\Delta\tilde{\xi}_1^{(i)}|&  \leq C\widetilde{\kappa}_i, \qquad \mbox{where}\quad \widetilde{\kappa}_i:=\frac{ 1}{n\bar{\theta} }\sqrt{\frac{\log (n)}{n\bar{\theta}^2}} \sqrt{n\bar{\theta}\theta_i \vee \log (n)}, \label{Laplacian-entry-3-1}\\
|e_i\Delta(\overline{\xi}^{(i)}_1-\tilde{\xi}^{(i)}_1)| & \leq C\widetilde{\kappa}_i \Bigl(1+n\bar{\theta}\bigl\| (\tilde{H}^{(i)})^{-\frac12}( w\hat{\xi}_1-\tilde{\xi}^{(i)}_1)\big\|_\infty\Bigr) +  \frac{C\log(n)}{n\bar{\theta}^2} \| w\hat{\xi}_1-\overline{\xi}^{(i)}_1\|,\label{Laplacian-entry-3-2}\\
\| w\hat{\xi}_1 - \overline{\xi}_1^{(i)}\| \leq  C&\,\widetilde{\kappa}_i \Bigl(1+n\bar{\theta}\|(\tilde{H}^{(i)})^{-\frac12}(w\hat{\xi}_1-\tilde{\xi}^{(i)}_1)\|_\infty\Bigr) + \frac{C  }{\sqrt{n\bar{\theta}^2}}|w\hat{\xi}_1(i)-\tilde{\xi}^{(i)}_1(i)|. \label{Laplacian-entry-3-3}
\end{align}
\end{lem}

We now use Lemmas~\ref{lem:upbd_key1}-\ref{lem:techB1} to prove the first claim, \eqref{result-1-mainpaper}, in Theorem~\ref{thm:eigenvector}. In Lemma~\ref{lem:upbd_key11}, although the vector $\tilde{\xi}_1^{(i)}$ depends on $i$, the scalar $w\in \{\pm 1\}$ is shared by $1\leq i\leq n$; similarly, the $w$ in Lemma~\ref{lem:techB1} is also shared by all $1\leq i\leq n$. Therefore, we can assume $w=1$ in all claims, without loss of generality. 
When there is no confusion, we write $\overline{\xi}_1^{(i)}=\overline{\xi}_1$,  $\tilde{\xi}_1^{(i)}=\tilde{\xi}_1$  and $\tilde{H}^{(i)}=\tilde{H}$ for short. 
We plug \eqref{Laplacian-entry-3-1}-\eqref{Laplacian-entry-3-3} into \eqref{Laplacian-entry-1}-\eqref{Laplacian-entry-2} and note that $\kappa_i\leq \widetilde{\kappa}_i$. It gives
\begin{align*}
|\hat{\xi}_1(i)-\tilde{\xi}_1(i)|\leq C  \sqrt K\,  \widetilde{\kappa}_i + C\widetilde{\kappa}_i n\bar{\theta}\|\tilde{H}^{-\frac12}(\hat{\xi}_1-\tilde{\xi}_1)\|_\infty + \frac{C  \sqrt{\log (n)}}{n\bar{\theta}^2}|\hat{\xi}_1(i)-\tilde{\xi}_1(i)|. 
\end{align*}
Since $ \sqrt{\log (n)}\ll n\bar{\theta}^2$, we can rearrange the above inequality to get
\beq \label{eq:ht11}
 |\hat{\xi}_1(i)-\tilde{\xi}_1(i)|  \leq C\sqrt K\, \widetilde{\kappa}_i+ C\widetilde{\kappa}_i n\bar{\theta}\|\tilde{H}^{-\frac 12}(\hat{\xi}_1-\tilde{\xi}_1)\|_\infty. 
\eeq
We further apply Lemma~\ref{lem:upbd_key1} and note that $\kappa_j\leq \widetilde{\kappa}_j$ for all $j$ and that with high probability, $\|\tilde{H}^{-\frac12}H_0^{\frac12}\|\leq C$ (see the supplementary material). It yields that $|\hat{\xi}_1(i)-\xi_1(i)| \leq |\hat{\xi}_1(i)-\tilde{\xi}_1(i)|+ |\tilde{\xi}_1(i)-\xi_1(i)|\leq C\sqrt K\, \widetilde{\kappa}_i+ C \widetilde{\kappa}_i n\bar{\theta}\|H_0^{-\frac 12}(\hat{\xi}_1-\tilde{\xi}_1)\|_\infty$, which leads to  
\begin{align*} 
|\hat{\xi}_1(i)-\xi_1(i)|  &\leq C\sqrt K\,  \widetilde{\kappa}_i+ C \widetilde{\kappa}_i n\bar{\theta}\bigl(\|H_0^{-\frac 12}(\hat{\xi}_1-\xi_1)\|_\infty + \|H_0^{-\frac 12}(\tilde{\xi}_1-\xi_1)\|_\infty\bigr). 
\end{align*}
To bound the second term in the bracket, we apply Lemma~\ref{lem:upbd_key1} again and note that $H_0(j,j)\asymp n\bar{\theta} (\bar{\theta}\vee \theta_j)$ (see Section~\ref{subsec:property_L_0} of the supplementary material) 
and $\kappa_j[H_0(j,j)]^{-\frac 12}\leq \widetilde{\kappa}_j[H_0(j,j)]^{-\frac 12}\leq  \sqrt{\frac{\log (n)}{n\bar{\theta}^2}}\frac{1}{n\bar{\theta}}$. It gives $
\|H_0^{-\frac 12}(\tilde{\xi}_1-\xi_1)\|_\infty \leq C\sqrt{\frac{K\log (n)}{n\bar{\theta}^2}}\frac{1}{n\bar{\theta}}$. 
Plugging it into the above inequality, we obtain that
\beq \label{temp1}
 |\hat{\xi}_1(i)-{\xi}_1(i)| \leq C\sqrt K\, \widetilde{\kappa}_i+ C \widetilde{\kappa}_i n\bar{\theta}\|{H}_0^{-\frac 12}(\hat{\xi}_1-{\xi}_1)\|_\infty. 
\eeq
We multiply $H_0^{-\frac 12}(i,i)$ on both hand sides and use $\widetilde{\kappa}_j[H_0(j,j)]^{-\frac 12}\leq \sqrt{\frac{\log (n)}{n\bar{\theta}^2}}\frac{1}{n\bar{\theta}}$ again. It leads to that 
$\big|e_i' H_0^{-\frac 12}(\hat{\xi}_1- \xi_1)\big| \leq  CK^{\frac 12} \sqrt{\frac{\log (n)}{n\bar{\theta}^2 }} \frac{1}{n\bar{\theta}} + C\sqrt{\frac{\log (n)}{n\bar{\theta}^2 }} \, \|H_0^{-\frac 12}(\hat{\xi}_1- \xi_1)\| _\infty$. 
This inequality holds for every $1\leq i\leq n$. 
Since $\hat{\xi}_1-\xi_1$ does not depend on $i$, if we take a maximum over $i$, then both the left and right hand sides contain a term related to $\|H_0^{-\frac 12}(\hat{\xi}_1- \xi_1)\|_\infty$. Noticing that $\sqrt{\frac{\log (n)}{n\bar{\theta}^2 }}= o(1)$, we re-arrange the terms to get
$\|H_0^{-\frac 12}(\hat{\xi}_1- \xi_1)\|_\infty \leq CK^{\frac 12} \sqrt{\frac{\log (n)}{n\bar{\theta}^2 }} \frac{1}{n\bar{\theta}} \leq CK^{\frac 12} \frac{1}{n\bar{\theta}}$.
Plugging 
it into \eqref{temp1} gives
$|\hat{\xi}_1(i)-{\xi}_1(i)| \leq CK^{\frac 12}  \widetilde{\kappa}_i$.
The claim \eqref{result-1-mainpaper} follows immediately by plugging in the definition of $\widetilde{\kappa}_i$ in Lemma~\ref{lem:techB1}. \qed

\section{Simulations} \label{sec:Simu}
We conduct two experiments: Experiment 1 compares Mixed-SCORE-Laplacian (MSL) with Mixed-SCORE \citep{Mixed-SCORE}. Recall that MSL is a modification of the conventional Mixed-SCORE. We hope to demonstrate that our proposed modification improves the numerical performance under severe degree heterogeneity. Experiment 2 investigates the node-wise errors of MSL. We hope to show that the node-wise error indeed varies with $\theta_i$ as specified in Theorem~\ref{thm:uppbd_pi}. 
In addition, we also study the effect of $K$ in Section~\ref{supp:Simu-add} of the supplementary material. 
Both MSL and Mixed-SCORE require plugging in a vertex hunting (VH) algorithm as the input. We use successive projection \citep{araujo2001successive} as the default VH algorithm,  except in Experiment~2. Since this experiment studies the very delicate node-wise errors, we follow \cite{Mixed-SCORE} to use an improved VH algorithm; see Section~\ref{supp:VH} of the supplementary material for details.

\spacingset{1}
\begin{figure}[tb!]
\centering
\includegraphics[height=.24\textwidth, width=.242\textwidth]{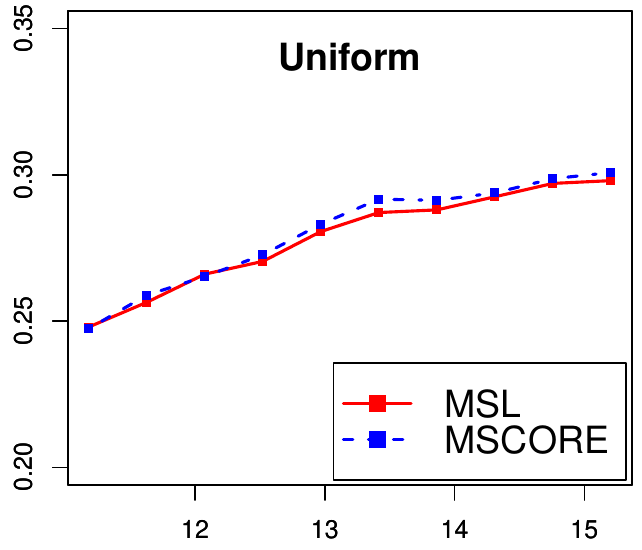} 
\includegraphics[height=.24\textwidth, width=.242\textwidth]{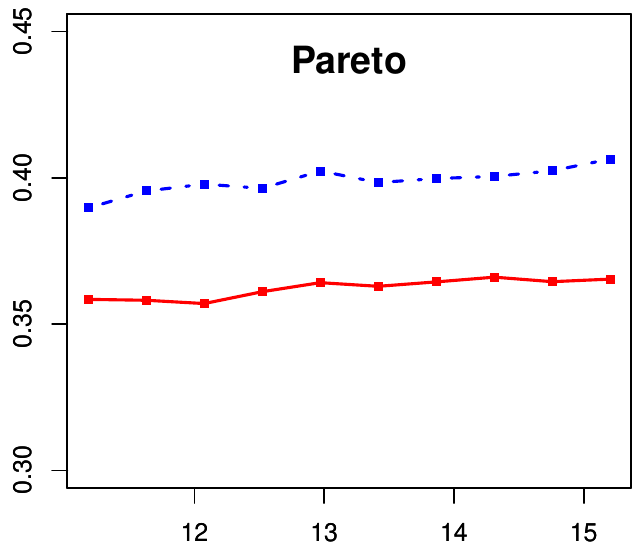}
\includegraphics[height=.24\textwidth, width=.242\textwidth]{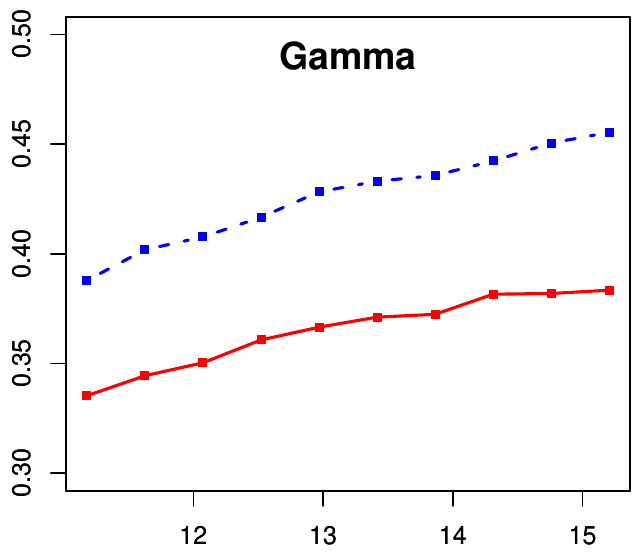}
\includegraphics[height=.24\textwidth, width=.242\textwidth]{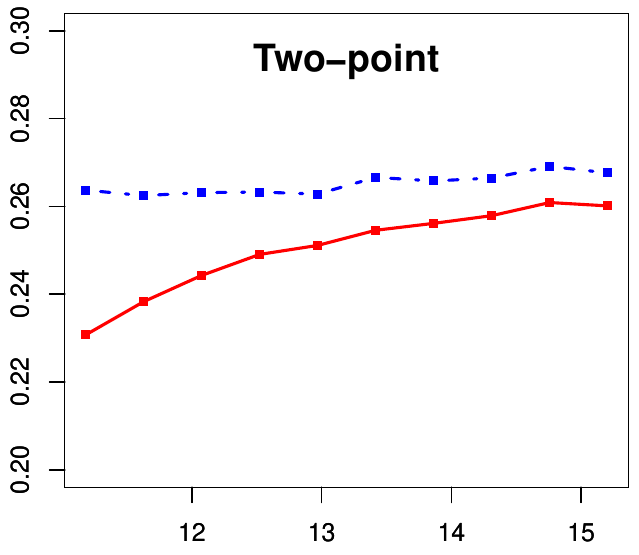}
\caption{MSL v.s. MSCORE ($n=2000$, $K=2$, $x$-axis is $\sqrt{n\bar{\theta}^2}$, and $y$-axis is the $\ell^1$-loss in \eqref{Loss}, averaged over 100 repetitions). 
We consider four different cases of degree heterogeneity.}
\label{fig:compare}
\end{figure}
\spacingset{1.7}

\smallskip\noindent
{\bf Experiment 1: Comparison with the conventional Mixed-SCORE.}
Fix $(n, K)=(2000, 2)$. Let $P=\beta_nI_2 + (1-\beta_n){\bf 1}_2{\bf 1}_2'$. To generate $\theta$, given a distribution $F(\cdot)$ and $b_n>0$, we draw $\theta^0_1, \theta^0_2,\ldots,\theta^0_n \overset{iid}{\sim} F(\cdot)$ and then let $\theta_i= b_n \cdot n\theta^0_i/\Vert\theta^0\Vert_1$ for $1\leq i\leq n$.   
To generate $\Pi$,  we first set $\pi_i=(1,0)'$ and $\pi_i=(0,1)'$ each for $15\%$ of nodes, and then let $\pi_i=(t_i, 1-t_i)'$ for the remaining $70\%$ of nodes, with $t_i\overset{iid}{\sim}\mathrm{Uniform}([0, 1])$. 
The distribution $F(\cdot)$ controls degree heterogeneity. 
In light of Proposition~\ref{prop:rate}, we consider four choices of $F(\cdot)$: in Experiment 1.1,  
$F=\mathrm{Uniform}([0.3, 5])$; in Experiment 1.2, $F=\mathrm{Pareto}(10,0.3)$, with a truncation at $5000$; in Experiment 1.3, $F=\mathrm{Gamma}(1/3, 1)$; in Experiment 1.4, $F=0.05\delta_9 + 0.95\delta_{0.1}$, where $\delta_x$ is a point mass at $x$. The parameter $b_n$ ($=\bar{\theta}$) controls network sparsity, and $\beta_n$ controls `community dis-similarity'. Write $\mathrm{SNR}:=\sqrt{n\bar{\theta}^2}(1-P(1,2))= b_n\beta_n\sqrt{n}$. 
We let $nb_n$ range from 500 to 680 with a grid of 20, and change $\beta_n$ accordingly by fixing $\mathrm{SNR}=10$. 
Figure~\ref{fig:compare} reports the $\ell^1$-loss (averaged over 100 repetitions) for four cases of degree heterogeneity and various levels of network sparsity. In Section~\ref{supp:Simu-add} of the supplementary material, we also report the  
weighted $\ell^1$-loss. Except in Experiment 1.1, degree heterogeneity is severe in the other three cases, and we always observe a significant improvement of MSL over MSCORE.


\spacingset{1}
\begin{figure} [tb!]
\centering
\includegraphics[height=.26\textwidth, width=0.35\textwidth]{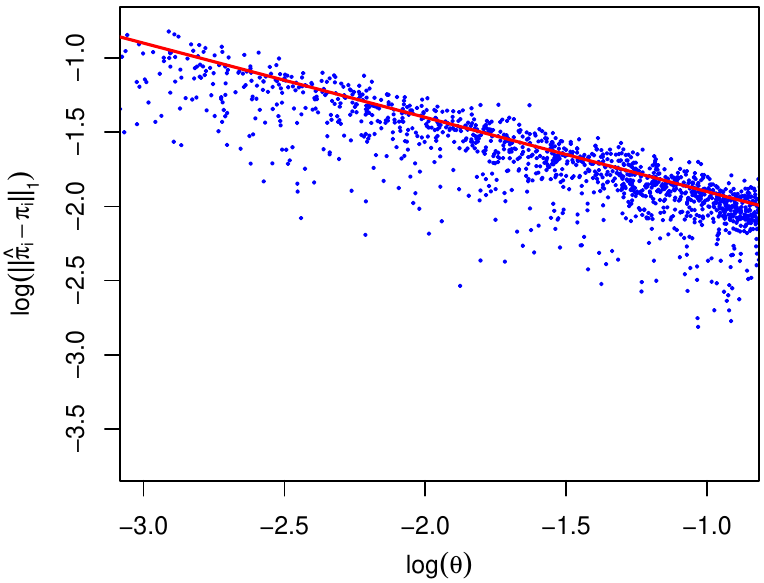}
\includegraphics[height=.26\textwidth, width=0.35\textwidth]{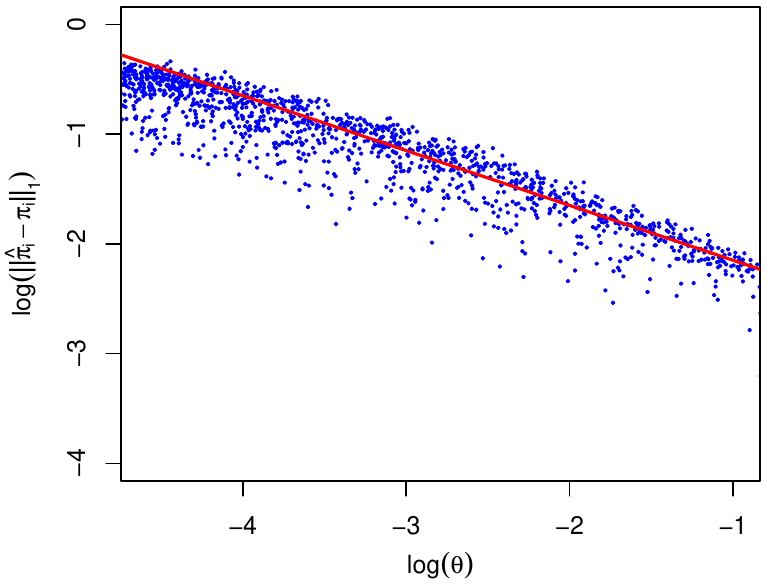}
\caption{Node-wise errors of MSL. $\theta_i$'s are generated from $\mathcal{U}([0.3, 5])$ and Pareto$(10,0.3)$, respectively  for the left and the right. In each plot, the slope of the red line is $-1/2$.
} 
\label{fig:theta}
\end{figure}
\spacingset{1.7}

\smallskip\noindent
{\bf Experiment 2: Node-wise errors of Mixed-SCORE-Laplacian.}
Fix $(n,K)=(2000,2)$. We first generate $(\theta, \Pi)$ as in Experiment 1. Fixing $(\theta,\Pi)$, we then generate 100 networks, and compute the average of $\|\hat{\pi}_i-\pi_i\|_1$ on these 100 repetitions, for each $1\leq i\leq n$. 
In Experiment 2.1, $F$ is $\mathrm{Uniform}([0.3, 5])$; in Experiment 2.2, $F$ is $\mathrm{Pareto}(10,0.3)$. Fix $\sqrt{n}\,  b_n = 15$ and $\mathrm{SNR}=b_n\beta_n\sqrt n=10$. 
Theorem~\ref{thm:uppbd_pi} says that the error at $\hat{\pi}_i$ is (approximately) proportional to $(\theta_i\wedge \bar{\theta})^{-1/2}$. In Figure~\ref{fig:theta}, we plot $\log(\Vert \hat{\pi}_i-\pi_i\Vert_1)$ versus $\log(\theta_i)$, for $\theta_i\leq \bar{\theta}$. 
Under both moderate degree heterogeneity (left panel) and severe degree heterogeneity (right panel), the plot fits a straight line with slope $-1/2$ well. This verifies the claim of Theorem~\ref{thm:uppbd_pi}. 

\section{Real data} \label{sec:real_data}

We use real-world examples to denomstrate the practical implications of our study. 
The first data set, the political blog network \citep{adamic2005political}, contains 
``true" community labels. 
We use this data set to compare the node embeddings with and without pre-PCA normalization.
 The second data set is a co-authorship network of statisticians \citep{JiJin}.  We apply both the conventional Mixed-SCORE and our proposed MSL with $K=2$ and compare their performances.

\smallskip \noindent
{\bf The political blog network \citep{adamic2005political}.}
Each node is a blog during 2004 U.S. presidential election, and there is an edge between two nodes if there is a link (one-way or bilateral) between two blogs.
A manual label of ``conservative" or ``liberal" was given to each blog by  \cite{adamic2005political}. 
We use the giant component of the network, which has $n=1222$ nodes.  In addition, we construct a new network by adding two outlier nodes; for each of them, we connect it to $80\%$ of the original nodes (randomly selected). The original network already shows severe degree heterogeneity ($d_{\min}=1$, $d_{\mathrm{mean}}=27.35$, and $d_{\max}=351$).  Degree heterogeneity in the new network is even more severe ($d_{\max}=978$). 
For each network, we compute the SCORE embeddings in \eqref{SCORE} with $K=2$, where $\hat{\xi}_1$ and $\hat{\xi}_2$ are either eigenvectors of $A$ (no pre-PCA normalization) or $L$ (with pre-PCA normalization). 
The histograms of $\hat{r}_i$'s in these four cases (i.e., two networks, with/without pre-PCA normalization) are displayed in Figure~\ref{plot:polblogs}.  
Since most bloggers are not extremely conservative/liberal, this network has mixed memberships. However, the manual labels by \cite{adamic2005political} are binary. 
To compare different embeddings, we calculate (a) Rayleight quotient $Rq : =\frac{ \pi^*(1- \pi^*)(\bar{r}_{(1)} - \bar{r}_{(2)})^2}{\pi^* v_{(1)} + (1- \pi^*)v_{(2)}}$, where $\pi^*$ is the fraction of nodes in the ``conservative" group, $\bar{r}_{(1)}$ and $v_{(1)}$ are the mean and variance of $\hat{r}_i$'s within this group, and $\bar{r}_{(2)}$ and $v_{(2)}$ are those within the ``liberal" group, and (b) the k-means clustering errors on $\hat{r}_i$'s (treating the manual labels as the ground truth) in each of the four cases. 
When calculating RQ and clustering error, we exclude the outliers and only use $\hat{r}_i$'s of the original nodes. 
The left two panels of Figure~\ref{plot:polblogs} are for the original network (no outlier). The  SCORE embeddings with pre-PCA normalization yield better RQs and smaller clustering errors. The advantage is more significant in the new network with outliers. 
These outliers have extremely high degrees and bring a lot of noise into the eigenvectors of $A$.
Consequently, the RQ drops significantly and the clustering error has a big increase.
In comparison, the pre-PCA normalization can properly ``down-weight" high-degree nodes in PCA and make the SCORE embeddings more robust to outliers.  

\spacingset{1}
\begin{figure}[tb!]
\center
\includegraphics[scale=0.42]{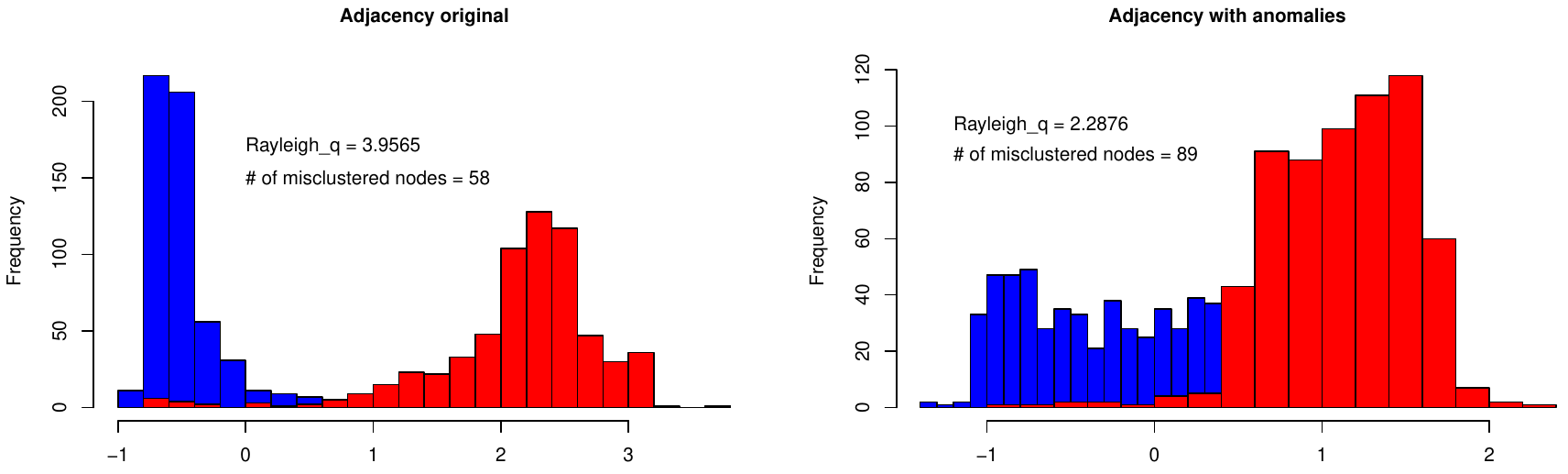} 
\includegraphics[scale=0.42]{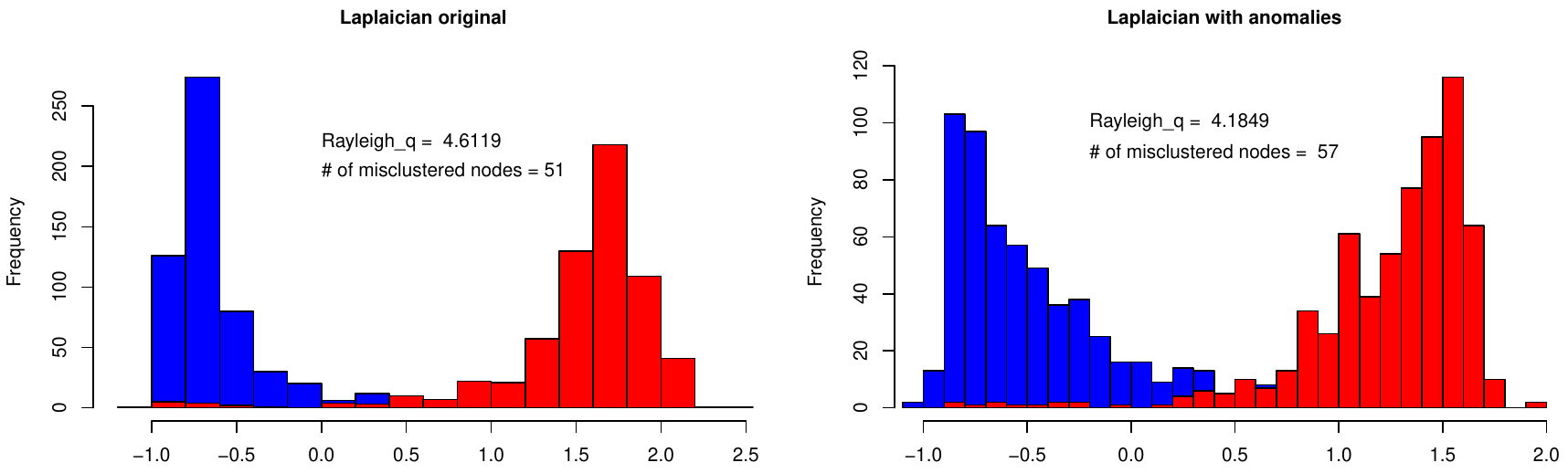}
\caption{Comparison of the SCORE embeddings without pre-PCA normalization (top) and with pre-PCA normalization (bottom). Since $K=2$, each embedding is a point in $\mathbb{R}$. The colors correspond to the true community labels. The left two panels are based on the original political blog network, and the right two panels are based on the network with outliers. 
}\label{plot:polblogs}
\end{figure}
\spacingset{1.7}

\smallskip\noindent
{\bf The co-authorship network of statisticians \citep{JiJin}.}
Several co-authorship networks were constructed by \cite{JiJin}. One of them is a 236-author network, in which two authors (nodes) are connected by an edge if they co-authored $\geq 2$ papers during 2003-2012 in four core statistical journals. \cite{JiJin} discorvered that there are two communities, ``Carroll-Hall" and ``North-Carolina", in this network. \cite{Mixed-SCORE} further studied this network and found strong evidence of mixed memberships: many members of the Fan-group (Jianqing Fan and collaborators) have nearly half-half memberships in two communities.  Figure~\ref{plot:coauthor} shows the estimated membership in the ``Carroll-Hall" for all 236 authors, where the x-axis is the Mixed-SCORE estimate, the y-axis corresponds to the MSL estimate, and the dashed line is $y=x$. 
We make several observations. First, the membership of Jianqing Fan is still near 50\%. Second, most pure nodes of the ``North-Carolina" community are still pure. Last, Raymond Carroll and Peter Hall are no longer 100\% pure in the ``Carroll-Hall" community. Meanwhile, authors such as Jing Qin are now pure nodes of this community. This does not conflict with the discoveries in \cite{JiJin, Mixed-SCORE}: In fact, they claimed that this is a group of authors interested in nonparametric and semi-parametric statistics and ``Carroll-Hall" was used as a short name. We note that both Raymond Carroll and Peter Hall have high degrees. MSL properly down-weights the contributions of high-degree nodes before conducting PCA, so we expect it to yield more accurate estimates. Indeed, the results are more consistent with our common insights: highly collaborative authors usually have diversified interests, hence ``mixed" memberships. 


%
%

\spacingset{1}
\begin{figure}[tb!]
\center
\includegraphics[height=.35\textwidth, trim=0 0 0 50, clip=true]{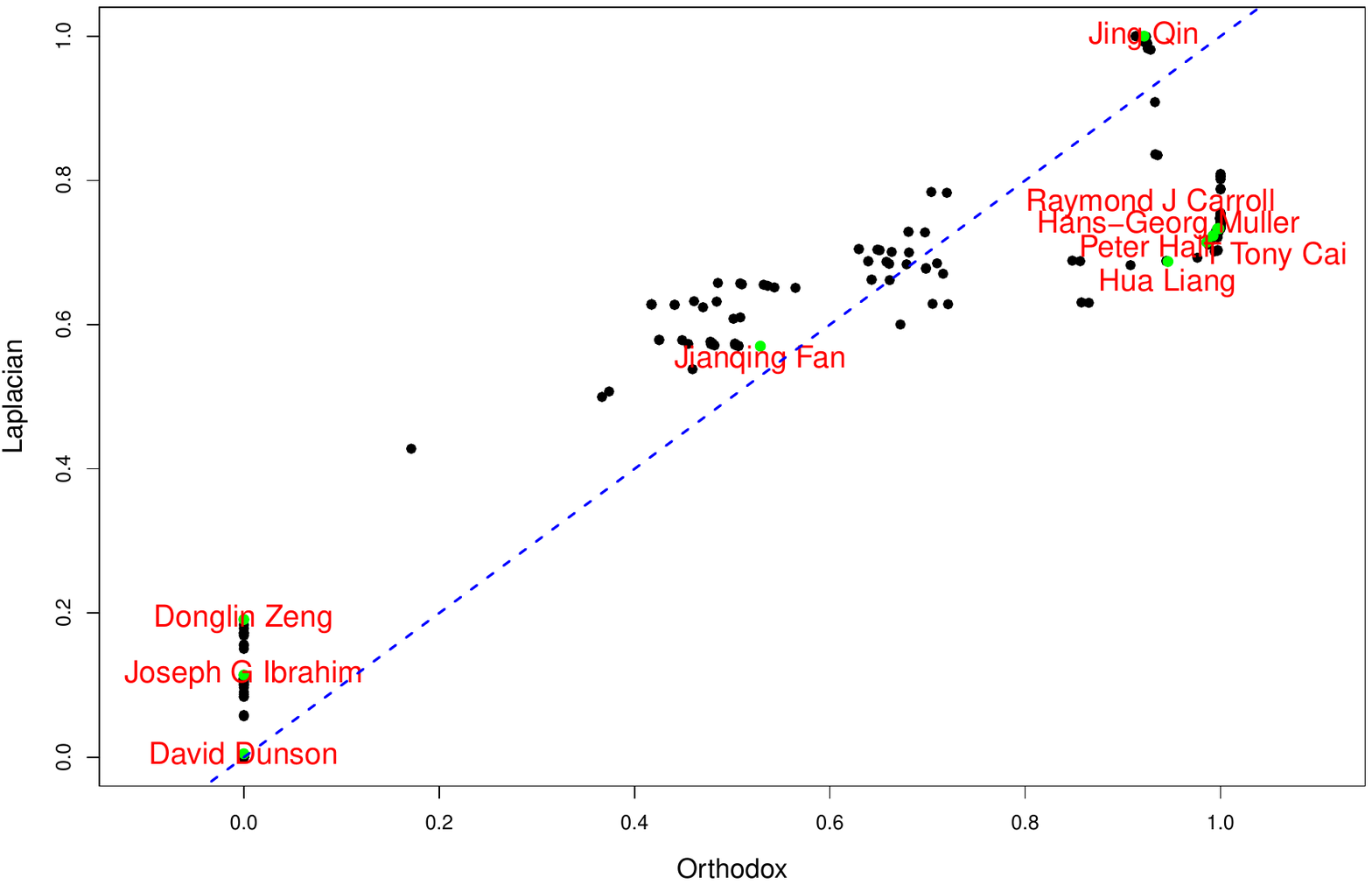} 
\caption{The estimated membership in the ``Caroll-Hall" community.  The $x$-axis and $y$-axis correspond to MSCORE and MSL, respectively. The $10$ highest-degree nodes are marked in green with the author names displayed by side.  }\label{plot:coauthor}
\end{figure}
\spacingset{1.7}

\section{Discussions} \label{sec:Discuss}
Severe degree heterogeneity has been widely observed in real networks \citep{barabasi1999emergence}, but its effect on network model estimation and inference is largely unclear. This paper provides the first result of optimal mixed membership estimation under (nearly) arbitrary degree heterogeneity. Our contributions include (i) a degree-heterogeneity-aware lower bound, (ii) an optimal spectral algorithm, (iii) node-wise error rates, and (iv) entry-wise eigenvector analysis.

In this paper, we assume $K$ is known. If $K$ is unknown, how to estimate $K$ from data is an interesting problem. 
Existing works of estimating $K$ (e.g., \cite{wang2017likelihood,jin2023optimal,dall2021unified,hwang2023estimation}) focus on the community detection setting and cannot be applied directly. Our theoretical results on eigenvalues of $L$ (see Section~\ref{supp:AuxiliaryLemma-L} of the supplementary material) suggest an estimator of $K$:
\[
\hat{K}=\#\bigl\{1\leq j\leq n:\; \hat{\lambda}_j^2\geq \bar{d}^{-1}\log^2(n)\bigr\},  \quad\mbox{where $\bar{d}$ is avarege node degree}. 
\]
We can show that when $K^3\log^2(n)/(n\bar{\theta}^2\beta_n^{2})\to\infty$, $\hat{K}=K$ with probability $1-o(1)$. Hence, all the results in this paper continue to hold when $K$ is estimated from data.

We focus on undirected networks, but our results are extendable to directed networks. We can extend the DCMM model to directed networks by introducing two degree parameters, $\theta_i^{in}$ and $\theta_i^{out}$, and two mixed membership vectors $\pi_i^{in}$ and $\pi_i^{out}$, for each node $i$. 
The degree heterogeneity in this model is captured by two CDFs, $F_n^{in}$ and $F_n^{out}$. We conjecture that the optimal rates in estimating $\Pi^{in}$ and $\Pi^{out}$ are different, and they are both functions of $n$, $K$, $\bar{\theta}^{in}$, $\bar{\theta}^{out}$, and  $F_n^{in}$ and $F_n^{out}$. We also conjecture that the optimal spectral method requires a normalization on the adjacency matrix as $L=(H^{out})^{-b_1}A(H^{in})^{-b_2}$, where $H^{in}$ and $H^{out}$ are diagonal matrices consisting of incoming and outgoing node degrees of nodes, and the choice of $(b_1, b_2)$ is inspired by node-wise error analysis. We leave this to future work. 

Uncertainty quantification of $\hat{\Pi}$ is another interesting problem. \cite{bhattacharya2023inferences} derived a novel finite-sample expansion for the $\hat{\pi}_i$ from conventional Mixed-SCORE, which yields valid confidence interval of $\pi_i$. They also proposed a ranking inference procedure for individual membership's profiles. 
Our results in this paper suggest that when there is severe degree heterogeneity, MSL may significantly improve the conventional Mixed-SCORE. 
It is thus interesting to develop uncertainty quantification for the $\pi_i$'s from the MSL algorithm. 
This requires second-order expansions for the leading eigenvectors of graph Laplacian. We also leave it to future work. 

The SCORE-family of algorithms \citep{SCORE,Mixed-SCORE} can be extended to weighted networks, because the rationale of such methods is hinged on the structure of $\Omega$. For example, let $\Omega=\Theta\Pi P\Pi'\Theta$ as before. If $A(i,j)\sim \mathrm{Poisson}(m\Omega(i,j))$ or $A(i,j)\sim N(\Omega(i,j), \sigma^2)$, we can still apply Mixed-SCORE (with or without normalizations) to get a consistent estimate of $\Pi$. 
The question is how to achieve the optimal error rate. The rate-optimality result in this paper is tied to the Bernoulli edge generation (e.g., the mean and variance of a Bernoulli variable are of the same order). For Poisson edges, we conjecture that similar optimality result can be achieved. For normal edges, we conjecture that no degree normalization is needed, when the variances of all edges are the same.

\spacingset{1.7} 



\appendix

\numberwithin{equation}{section}
\numberwithin{figure}{section}
\numberwithin{table}{section}

\section{Supplementary lemmas for the main text}

In the main paper, owing to the space constraints, we stated some arguments without giving detailed proofs. In this section, we revisit and prove  these arguments. 

\subsection{The simplex geometry and the oracle procedure} \label{subsec:simplex}

In Section~\ref{sec:MSCORE-L}, we considered the oracle case where $A=\Omega$ and claimed that there is a simplex geometry associated with the rows of $\hat{R}$. The next lemma makes this argument rigorous: 

\begin{lem}[The simplex geometry]\label{lem:simplex}
Consider a DCMM model, where each community $k$ has at least one pure node.  Let $H_0=\mathbb{E}[H]$ and $L_0=H_0^{-\frac 12}\Omega H_0^{-\frac 12}$. Let $\lambda_k$ be the $k$th largest eigenvalue (in magnitude) of $L_0$, and let $\xi_k$ be the corresponding eigenvector. If we pick the sign of $\xi_1$ such that $\sum_{i=1}^n\xi_1(i)>0$, then $\xi_1$ is a strictly positive vector. Furthermore, consider the matrix $R\in\mathbb{R}^{n\times (K-1)}$, where $R(i,k) = \xi_{k+1}(i)/\xi_1(i)$, $1\leq i\leq n,1\leq k\leq K-1$. 
Write $R=[r_1,r_2,\ldots,r_n]'$.
\begin{itemize}
\item There exists a simplex ${\cal S}\subset\mathbb{R}^{K-1}$ with $K$ vertices $v_1,v_2,\ldots,v_K$, such that $r_1,r_2,\ldots,r_n$ are contained in ${\cal S}$. If node $i$ is a pure node, then $r_i$ falls on one vertex of this simplex; if node $i$ is a mixed node, then $r_i$ is in the interior of the simplex (it can be on an edge or a face, but cannot be on any of the vertices). 
\item Each $r_i$ is a convex combination of the $K$ vertices, $r_i = \sum_{k=1}^K w_i(k)v_k$. The combination coefficient vector is $w_i=\|\pi_i\circ b_1\|_1^{-1}(\pi_i\circ b_1)$, where $\circ$ is the Hardarmart product and $b_1$ is a $K$-dimensional vector with $b_1(k)=1/\sqrt{\lambda_1+v_k'\diag(\lambda_2,\ldots,\lambda_K)v_k}$, $1\leq k\leq K$. 
\end{itemize}
\end{lem}

\begin{proof}[Proof of Lemma \ref{lem:simplex}] \label{sec:pf_simplex}

The proof largely follows the one in \cite{Mixed-SCORE}, except that they considered a special case of $H=I_n$ while we allow for a general diagonal matrix $H$ here.    

Recall that $L_0 = H_0^{-\frac12}\Theta \Pi (P\Pi'\Theta H_0^{-\frac12})$. Under the condition that each community has at least one pure node, 
$L_0$ has a rank $K$. It follows that $L_0$ has the same column space as $H_0^{-\frac12}\Theta\Pi$. Meanwhile, $\Xi=[\xi_1,\xi_2,\ldots,\xi_K]$ also has the same column space as $L_0$. Therefore, there exists a non-singular matrix $B\in\mathbb{R}^{K\times K}$ such that 
\[
\Xi = H_0^{-\frac12}\Theta\Pi B.  
\]
Write $B=[b_1,b_2,\ldots,b_K]$. Define $v_1,v_2,\ldots,v_K\in\mathbb{R}^{K-1}$ by $v_{k}(\ell)=b_{\ell+1}(k)/b_1(k)$, for $1\leq k\leq K$, $1\leq \ell\leq K-1$. Write $V=[v_1,v_2,\ldots,v_K]'\in\mathbb{R}^{K\times (K-1)}$. It follows that
\[
B = \diag(b_1)[{\bf 1}_K, V]. 
\] 
By definition of $R$, $[{\bf 1}_n, R] = [\diag(\xi_1)]^{-1}\Xi$. It follows that
\begin{align*}
[{\bf 1}_n, R] &= [\diag(\xi_1)]^{-1}\Xi =  [\diag(\xi_1)]^{-1} H_0^{-\frac12}\Theta\Pi B \cr
&=  [\diag(\xi_1)]^{-1} H_0^{-\frac12}\Theta\Pi\diag(b_1)[{\bf 1}_K, V]. 
\end{align*}
Define $W= [\diag(\xi_1)]^{-1} H_0^{-\frac12}\Theta\Pi\diag(b_1)$. The above equation implies that ${\bf 1}_n=W{\bf 1}_K$ and $R=WV$. Denote by $w_i'$ the $i$th row of $W$. 
It follows that $w_i'{\bf 1}_K=1$ and $r_i=\sum_{k=1}^K w_i(k)v_k$. Furthermore, under Condition~\ref{reg-conds}(c), we can show that both $\xi_1$ and $b_1$ are strictly positive vectors; the proof is similar to the proof of Lemma B.4 of \cite{Mixed-SCORE}, which we omit.  
It suggests that $W$ is also a nonnegative matrix. Combining the above, each $r_i$ is a convex combination of $v_1,v_2,\ldots,v_K$. This proves the simplex structure. 

We now derive the connection between $w_i$ and $\pi_i$. Write $\alpha_i=\xi_1^{-1}(i) H^{-\frac12}_0(i,i)\theta_i$.  Then, $w'_i=\alpha_i\cdot \pi_i'\diag(b_1)=\alpha_i\cdot(\pi_i\circ b_1)$. Since $\|w_i\|_1=1$, we immediately have $\alpha_i=1/\|\pi_i\circ b_1\|_1$. This proves that $w_i = \frac{1}{\|\pi_i\circ b_1\|_1}(\pi_i\circ b_1)$. To get the expression of $B_1$, we notice that 
\begin{align*}
\Lambda &= \Xi'L_0\Xi = (H_0^{-\frac12}\Theta\Pi B)' (H_0^{-\frac12}\Theta \Pi P\Pi'\Theta H_0^{-\frac12})(H_0^{-\frac12}\Theta\Pi B) \cr
&= B'(\Pi'\Theta D_{\theta}^{-1}\Theta\Pi )P (\Pi'\Theta D_{\theta}^{-1}\Theta\Pi )B' = K^{-2}\cdot B'GPGB,
\end{align*}
where $D_\theta$ and $G$ are as defined in Section~\ref{sec:LB} and we note that $D_{\theta}$ is actually $H_0$. Moreover, $G = K\cdot (H_0^{-\frac12}\Theta \Pi )'(H_0^{-\frac12}\Theta \Pi )=K\cdot (\Xi B^{-1})'(\Xi B^{-1})=K\cdot (BB')^{-1}$. It follows that
\[
B\Lambda B' = K^{-2}\cdot BB'GPGBB = P. 
\]
Write $\Lambda=\diag(\lambda_1,\Lambda_1)$, where $\Lambda_1=\diag(\lambda_2,\ldots,\lambda_K)$. Also, recall that $B=\diag(b_1)[{\bf 1}_K, V]$. We plug them into the above expression to get 
\[
P = \diag(b_1)[{\bf 1}_K, V]\begin{bmatrix}\lambda_1\\ &\Lambda_1\end{bmatrix}\begin{bmatrix}{\bf 1}_K'\\V'\end{bmatrix}\diag(b_1). 
\]
It follows that $P(k,k)=b_1(k)\cdot [\lambda_1+v_k'\Lambda_1v_k]\cdot b_1(k)$. The identifiability condition of DCMM model in Section \ref{subsec:DCMM} says that $P(k,k)=1$. Therefore, $b_1(k)=1/\sqrt{\lambda_1+v_k'\Lambda_1v_k}$. 
\end{proof}

\subsection{Broadness of the $\theta$-class ${\cal G}(\varrho, a_0)$} \label{subsec:theta-class}

In Section \ref{sec:LB}, we introduced a technical condition on $F_n(\cdot)$ (see Definition \ref{def:thetaClass}) and defined ${\cal G}(\varrho, a_0)$, a class of $\theta$. We claimed that this class is broad enough to include most interesting cases of degree heterogeneity. This is justified by the following lemma:  

\begin{lem} \label{prop:VerifyThetaClass}
The requirements in Definition~\ref{def:thetaClass} are satisfied if $\theta_i$'s are i.i.d. drawn from $\kappa_nF(\cdot)$, where $\kappa_n>0$ is a scalar and $F(\cdot)$ is a fixed, finite-mean distribution which has its support in $(0,\infty)$ and satisfies one of the following conditions: 
\begin{itemize}
\item $F(\cdot)$ is a discrete distribution; 
\item $F(\cdot)$ is a continuous distribution with support in $[c,\infty)$, for some $c>0$;
\item $F(\cdot)$ is a continuous distribution supported in $(0,\infty)$, and its density $f(t)$ satisfies that $\lim_{t\to\infty}t^{b}f(t)=C$, for some $b\neq 1/2$ and $C>0$. 
\end{itemize}
\end{lem}

\begin{proof}[Proof of Lemma~\ref{prop:VerifyThetaClass}] \label{subsec:Remark-F-condition}
Recall that we assume
 $\theta_i$'s i.i.d. generated from $\kappa_nF(\cdot)$, where $\kappa_n>0$, and $F(\cdot)$ is fixed distribution that is either continuous or discrete with finite mean $m$.

First, we consider the case that $F(\cdot)$ is a discrete distribution, i.e., $F= \sum_{\ell=1}^L \epsilon_\ell \delta_{x_\ell}$ where $L$ is a fixed constant and  $0< x_1<x_2<\ldots <x_L$, $\epsilon_\ell$'s  are all fixed, $\delta_x$ is a point mass at $x$, and $\sum_{\ell=1}^L\epsilon_\ell x_\ell=m$. 
In this case, we simply set $c_n= x_{L-1}/m$, $\rho = x_{1}/x_{L-1}$ and $a_0=\min_{\ell}{ \epsilon_\ell}$. One can easily check that with high probability, 
\begin{align*}
&F_n(c_n) = F(x_{L-1})= 1- \epsilon_L \leq 1 - a_0, \notag\\
& \sum_{\ell=1}^{L-1} \frac{\epsilon_\ell}{\sqrt{m^{-1}x_\ell\wedge 1}} \geq (1-\epsilon_L) \sum_{\ell=1}^{L} \frac{\epsilon_\ell}{\sqrt{m^{-1} x_\ell\wedge 1}} \geq a_0 \sum_{\ell=1}^{L} \frac{\epsilon_\ell}{\sqrt{m^{-1} x_\ell\wedge 1}} 
\end{align*}
which indeed verify the condition in  Definition \ref{def:thetaClass} for the chosen $(\rho, a_0)$. This proves the first bullet point of Lemma~\ref{prop:VerifyThetaClass}. 

Next, we consider the case that $F(\cdot)$ is a continuous  distribution with density $f(\cdot)$ and  ${\rm supp(f)}\subset [0, +\infty)$. Since $\int t d F_n(t)= 1$, it is not hard to see that $d F_n(t)=m f(mt) dt $. We can rewrite 
\begin{align} \label{2022042501}
\int_{err_n^2}^{\infty}\frac{1}{\sqrt{t\wedge 1}}d F_n(t)= \int_{err_n^2 m }^{\infty}\frac{f(t)}{\sqrt{ t/m\wedge 1}}d t
\end{align}
The singularity of the  integral on the RHS of (\ref{2022042501}) lies in the neighborhood of $0$, or $err_n^2m$.

If $f(t)t^{\frac 12- \epsilon_0} \to C$ as $t\to 0$ for some $\epsilon_0>0$, $C>0$, then the integral on the RHS of (\ref{2022042501})  converges and can be bounded by some constant   $C_1>0$. Since $F(\cdot)$ is a fixed continuous distribution with finite mean $m$, we can always find $\tilde{c}>  0$, $\tilde{a}\in (0,1)$ and $\rho\in (0,1)$ such that $F(\tilde{c}) - F(\rho \tilde{c})> C_2$ and $F(\tilde{c})\leq 1- \tilde{a}$ for some constant $0<C_2<C_1$. We then set $c_n = \tilde{c}/m$ and $a_0 = \min\{ \tilde{a}, C_2/C_1\}$.  As a result,  
\begin{align*}
&F_n(c_n) = F(\tilde{c}) \leq 1 - \tilde{a}\leq 1- a_0, \notag\\
& \int_{\rho c_n}^{c_n}\frac{1}{\sqrt{t\wedge 1}}d F_n(t) \geq \frac{F_n(\tilde{c}_n) - F_n(\rho \tilde{c}_n)}{\sqrt{c_n \wedge 1}} \geq C_2\geq a_0 \int_{err_n^2}^{\infty}\frac{1}{\sqrt{t\wedge 1}}d F_n(t) .
\end{align*}
Here $err_n$ can be replaced by any other sequence $x_n\to 0$. 
We  remark  that the case $F(\cdot)$ has a  support bounded  below from zero is also included in the current discussion. This proves the second bullet point in Lemma~\ref{prop:VerifyThetaClass} and part of the third bullet point.

If $f(t)t^{\frac 12 + \epsilon_0} \to C $ as $t\to 0$ for some $\epsilon_0>0$ and $C>0$, then the RHS of (\ref{2022042501}) is of the order $err_n^{-2\epsilon_0}$ and its mass is located in the neighborhood of $err_n^2\, m$. Therefore, we can simply set $c_n = C_3 \, err_n^2$ for some large $C_3>1$ such that $F(C_3err_n^2 \, m)\leq 1- \tilde{a}$ for some $\tilde{a}>0$ (this can be always achieved since $F(\cdot)$ is a fixed distribution with mean $m$). Let $\varrho= C_3^{-1}$. Then,  
\begin{align*}
&F_n(c_n) = F(C_3\, err_n^2\, m)\leq 1-\tilde{a}, \cr
& \int_{\varrho c_n}^{c_n}\frac{1}{\sqrt{t\wedge 1}}d F_n(t)  = \int_{err_n^2\, m}^{C_3\, err_n^2\, m}\frac{f(t)}{\sqrt{t/m}}d t > C_ 4 \int_{err_n^2 m }^{\infty}\frac{f(t)}{\sqrt{ t/m\wedge 1}}d t, 
\end{align*}
for some $C_4>0$. We thus take $a_0=\min\{ \tilde{a}, C_4\}$. The arguments also hold if we replace $err_n$ by any other sequence $x_n\to 0$. The condition in Definition \ref{def:thetaClass} is satisfied for the chosen $(\rho, a_0)$. This proves the remaining part of the third bullet point. 
\end{proof}

\section{The Mixed-SCORE-Laplacian (MSL) algorithm} \label{sec:MSL}

In Section~\ref{sec:MSCORE-L}, we explained the membership estimation steps in Figure~\ref{fig:MSCORE-idea} and gave a high-level description of the MSL algorithm in Algorithm~\ref{alg:MSL}. We now present Algorithm~\ref{alg:MSL-full}, a detailed version of Algorithm~\ref{alg:MSL}. In this algorithm, we assume there is a given vertex hunting (VH) algorithm. The choices of the VH algorithm are discussed in Section~\ref{supp:VH}. 

\spacingset{1}
\begin{algorithm}[tb!]
\caption{Mixed-SCORE-Laplacian.}\label{alg:MSL-full}

\medskip
\noindent
{\bf Input:} $K$, $A$, tuning parameters $(\tau, c, \gamma)=(1, 0.5, 0.05)$ (default), and a given VH algorithm. 
 
\begin{enumerate} \itemsep 4pt
\item Let $L$ be the normalized graph Laplacian in \eqref{Def:Laplacian}. Let $\hat{\lambda}_k$ be the $k$th largest eigenvalue (in magnitude) of $L$, and let $\hat{\xi}_k$ be the associated eigenvector, $1\leq k\leq K$. Define an $n\times (K-1)$ matrix $\hat{R}$ by 
\[
\hat{R}(i,k) = \hat{\xi}_{k+1}(i)/\hat{\xi}_1(i), \qquad 1\leq i\leq n,\, 1\leq k\leq K-1. 
\]
Denote by $\hat{r}'_1,\hat{r}'_2,\ldots,\hat{r}'_n$ the rows of $\hat{R}$.  
\item Let $\hat{\delta}_n =  K|\hat{\lambda}_K| $ and $
\hat{S}_n(c) = \{1\leq i\leq n:\, d_i\hat{\delta}_n^2 \geq c K^{3}\log(n)\}$. 
For any $i\notin \hat{S}_n(c)$, set $\hat{\pi}_i=K^{-1}{\bf 1}_K$.  
\item Let $\hat{S}_n^*(c, \gamma) = \hat{S}_n(c)\cap \{1\leq i\leq n:\, d_i\geq \gamma\bar{d}\}$. 
Run the given VH algorithm on the point cloud $\{\hat{r}_i\}_{i\in \hat{S}_n^*(c,\gamma)}$. Denote the output by $\hat{v}_1, \hat{v}_2, \ldots,\hat{v}_K$.


\item Let $\hat{\Lambda}_1=\diag(\hat{\lambda}_2,\ldots,\hat{\lambda}_K)$ and obtain $\hat{b}_1\in\mathbb{R}^K$ from 
\[
\hat{b}_1(k)=[\hat{\lambda}_1+\hat{v}_k'\hat{\Lambda}_1\hat{v}_k]^{-1/2}, \qquad 1\leq k\leq K, 
\]
For each $i\in \hat{S}_n(c)$, solve $\hat{w}_i\in\mathbb{R}^K$ from the linear equation set: 
\[
\sum_{k=1}^K \hat{w}_i(k)\hat{v}_k=\hat{r}_i, \qquad\mbox{and}\qquad \sum_{k=1}^K \hat{w}_i(k)=1.
\] 
Let $\hat{\pi}_i^*\in\mathbb{R}^K$ be such that $
\hat{\pi}_i^*(k)=\max\{\hat{w}_i(k)/\hat{b}_1(k),\, 0\}$, for $1\leq k\leq K$. 
Output $\hat{\pi}_i=\hat{\pi}_i^*/\|\hat{\pi}_i^*\|_1$, for each $i\in \hat{S}_n(c)$. 
\end{enumerate}
{\bf Output:} $\hat{\Pi}$.
\end{algorithm}
\spacingset{1.7}

\subsection{Choices of the plug-in VH algorithm} \label{supp:VH}

In our theoretical analysis and most simulations, we use successive projection (SP) \citep{araujo2001successive} as the plug-in VH algorithm. The details of SP are presented in Algorithm~\ref{alg:SP}. 
When plugging this algorithm into Algorithm~\ref{alg:MSL-full}, we need to pay attention to the dimension:
Algorithm~\ref{alg:SP} requires that the input point cloud is in a dimension $d\geq K$. However, each $\hat{r}_i$ is in dimension $K-1$. To resolve this issue, we follow \cite{Mixed-SCORE} to let 
\[
x_i = (1, \; \hat{r}_i')', \qquad 1\leq i\leq n. 
\] 
Now, each $x_i$ is in dimensional $K$. 
We input $x_i$'s to Algorithm~\ref{alg:SP}. The output $\hat{v}_1,\hat{v}_2,\ldots,\hat{v}_K$ will also be in dimension $K$. 
Since the first entry of each $x_i$ is $1$, the first entry of each $\hat{v}_k$ is also $1$. 
We then remove this first entry and output the $(K-1)$-dimensional sub-vectors of $\hat{v}_1,\hat{v}_2,\ldots,\hat{v}_K$. 
As argued in \cite{Mixed-SCORE}, this has no effect on the vertex estimation accuracy.

\spacingset{1}
\begin{algorithm}[tb!]
\caption{Successive projection (SP).}\label{alg:SP}
\medskip
\noindent
{\bf Input:} $K$ and $x_1, x_2, \ldots, x_m\in\mathbb{R}^d$ ($d\geq K$). 

\medskip

For each $1\leq k\leq K$, run the following steps: 
 
\begin{itemize} \itemsep 0pt
\item If $k\geq 2$, compute ${\cal P}=Y(Y'Y)^{-1}Y'$, where $Y=[x_{i_1}, x_{i_2},\ldots,x_{i_{k-1}}]\in\mathbb{R}^{d\times (k-1)}$.  \item If $k\geq 1$, find $i_1=\mathrm{argmax}_{i}\|x_i\|$; otherwise, find $i_k=\mathrm{argmax}_{i}\|(I_d-{\cal P})x_i\|$.   
\end{itemize}
{\bf Output:} $\hat{v}_k=x_{i_k}$, for $1\leq k\leq K$. 
\end{algorithm}
\spacingset{1.7}

\spacingset{1}
\begin{algorithm}[tb!]
\caption{Sketched vertex search (SVS).}\label{alg:SVS}
\medskip
\noindent
{\bf Input:} $K$, $x_1, x_2, \ldots, x_m\in\mathbb{R}^d$, and a tuning integer $L\geq K$. 

\medskip

\begin{itemize} \itemsep 0pt
\item Run k-means clustering on $x_1, x_2, \ldots, x_m$, assuming there are $L$ clusters. Denote by $\hat{y}_1, \hat{y}_2,\ldots,\hat{y}_L$ the estimated cluster centers. 
\item Input $\hat{y}_1, \hat{y}_2,\ldots,\hat{y}_L$ to Algorithm~\ref{alg:SP} to obtain $\hat{v}_1, \hat{v}_2,\ldots,\hat{v}_K$. 
\end{itemize}
{\bf Output:} $\hat{v}_1, \hat{v}_2,\ldots,\hat{v}_K$. 
\end{algorithm}
\spacingset{1.7}

SP performs well when the noise level in $x_1,x_2,\ldots,x_m$ is relatively low. When the noise level is relatively high, \cite{Mixed-SCORE} recommended to `de-noise' before running SP. They proposed the sketched vertex search (SVS) algorithm, which uses k-means to denoise. The details of SVS can be found in Algorithm~\ref{alg:SVS}. 
We also refer the readers to \cite[Section 3.4]{SCOREreview} and \cite{jinimproved} for more options of VH algorithms. In our simulation studies, we use SVS only in Experiment~2. This experiment studies the node-wise errors. Compared to the $\ell^1$-loss used in other experiments, node-wise errors are more sensitive to the noise level. 
This motivates us to replace SP by SVS. SVS has one tuning integer $L$, which is set as $L=5$ in Experiment~2.

{\bf Remark}: In Algorithm~\ref{alg:MSL-full}, we apply the plug-in VH algorithm on the trimmed point cloud: $
\{\hat{r}_i:  d_i\geq \gamma\bar{d}\}$. 
Our rationale is that the noise level on low-degree nodes is too high, so these points should be trimmed to improve performance. 
This trimming can also be viewed as a `denoising' step. Therefore, when we plug SVS into Algorithm~\ref{alg:MSL-full}, we actually conduct two rounds of denoising (trimming and k-means) before running SP. 

\section{Additional simulation results} \label{supp:Simu-add}
In this section, we present the additional simulation results, which are omitted in the main paper due to the space constraint. 

\subsection{The weighted $\ell^1$-loss}
We recall that most results in the main text are for the unweighted $\ell^1$-loss. In Section~\ref{subsec:otherLoss}, we extend the theoretical results to a general loss function parametrized by $p$ and $q$. A special case of $p=1/2$ and $q=1$ is called the weighted $\ell^1$-loss:  
\beq \label{unweightedLoss}
 {\cal L}^w(\hat{\Pi},\Pi) =  \min_{T}\Biggl\{\frac{1}{n}\sum_{i=1}^n \sqrt{\frac{\theta_i}{\bar{\theta}}}\, \|T\hat{\pi}_i-\pi_i\|_1\Biggr\}. 
\eeq
Compared to the unweighted $\ell^1$-loss, this performance metric  down-weights those errors in low-degree nodes. 
In Figure~\ref{fig:compare}, we report the unweighted $\ell^1$-loss of MSL and Mixed-SCORE in four different cases of degree heterogeneity and various levels of network sparsity. We now provide more results of these simulations by reporting the weighted $\ell^1$-loss in Figure~\ref{fig:compare-wLoss}. 

The conclusions for the weighted $\ell^1$-loss are similar to those for the unweighted $\ell^1$-loss: MSL greatly improves the conventional MSCORE in the last three cases, Pareto, Gamma, and Two-point mixture, which are the cases of severe degree heterogeneity. Furthermore, if we compare each panel in Figure~\ref{fig:compare-wLoss} with the corresponding panel in Figure~\ref{fig:compare}, we find that the value of the weighted loss is significantly smaller than the value of the unweighted loss in the case of Pareto and Gamma. This is because the node-wise error is a decreasing function of $\theta_i$; when low-degree nodes are down-weighted and high-degree nodes are up-weighted, the loss will decrease.

\spacingset{1}
\begin{figure}[tb!]
\centering
\includegraphics[height=.24\textwidth, width=.242\textwidth]{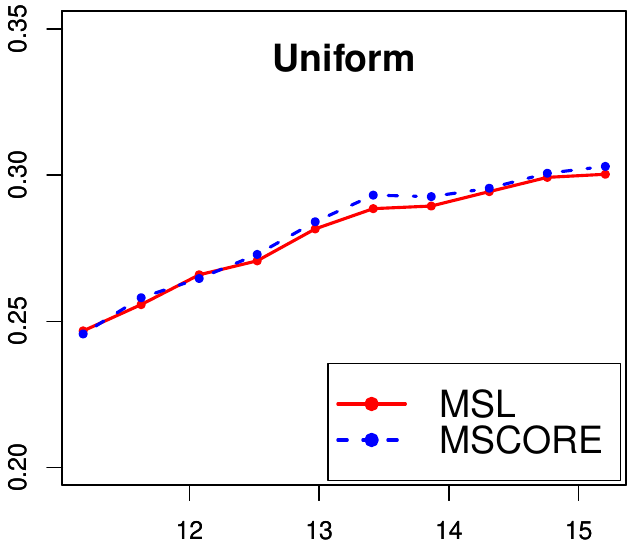} 
\includegraphics[height=.24\textwidth, width=.242\textwidth]{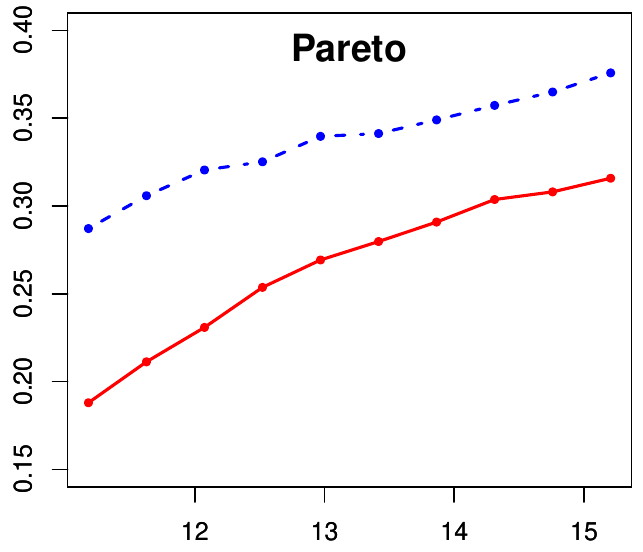}
\includegraphics[height=.24\textwidth, width=.242\textwidth]{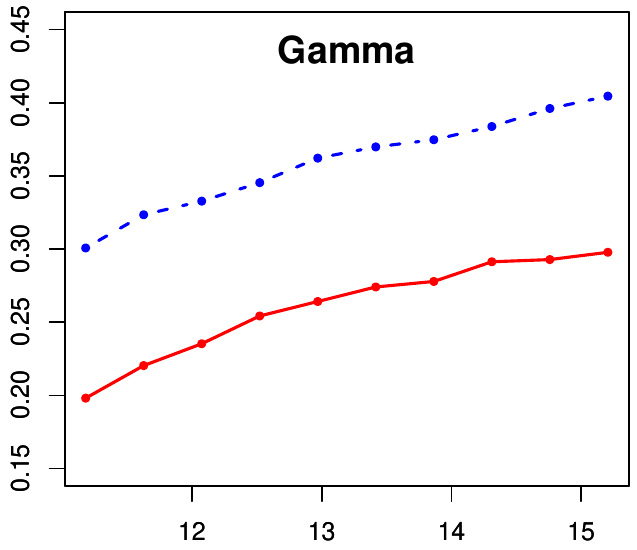}
\includegraphics[height=.24\textwidth, width=.242\textwidth]{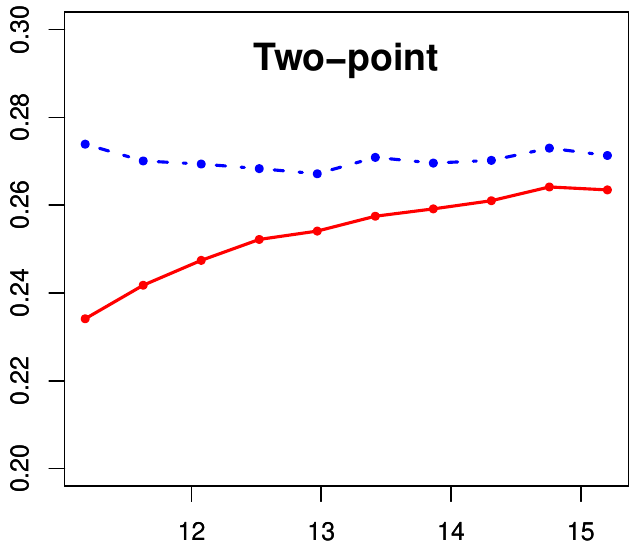}
\caption{MSL v.s. MSCORE ($n=2000$, $K=2$, $x$-axis is $\sqrt{n\bar{\theta}^2}$, and $y$-axis represents the unweighted $\ell^1$-loss in \eqref{unweightedLoss}). This figure complements Figure~\ref{fig:compare} in the main text by reporting the unweighted $\ell^1$-loss.}
\label{fig:compare-wLoss}
\end{figure}
\spacingset{1.7}

\subsection{Other values of $K$}
The simulation experiments in Section~\ref{sec:Simu} focus on $K=2$. We now consider other values of $K$ and investigate how the performance of MSL changes with $K$.
 
Fix $n=5000$ and let $K$ range in $\{3, 4, \ldots, 10, 11\}$. Given any $\beta_n\in (0,1)$ and $b_n>0$, we let $P=\beta_nI_K + (1-\beta_n){\bf 1}_K{\bf 1}_K'$ and generate $\theta$ as follows: Draw $\theta^0_1, \theta^0_2,\ldots,\theta^0_n \overset{iid}{\sim} \mathrm{Uniform}(0, 1)$ and let $\theta_i= b_n \cdot n\theta^0_i/\Vert\theta^0\Vert_1$ for $1\leq i\leq n$.  
We remark that this is a severe-degree-heterogeneity case: Since the support of $\mathrm{Uniform}(0,1)$ contains the neighborhood near zero, $\theta_{\max}/\theta_{\min}$ can be potentially large. 
We generate $\Pi$ in the same way as in Experiments 1-2: Set $\pi_i=(1,0)'$ and $\pi_i=(0,1)'$ each for $15\%$ of nodes, and let $\pi_i=(t_i, 1-t_i)'$ for the remaining $70\%$ of nodes, with $t_i\overset{iid}{\sim}\mathrm{Uniform}(0, 1)$. 
Write $\mathrm{SNR}:=\sqrt{n\bar{\theta}^2}(1-P(1,2))= b_n\beta_n\sqrt{n}$. We set $b_n (=\bar{\theta})$ such that $nb_n^2=800$
and select $\beta_n$ accordingly such that $\mathrm{SNR}=\sqrt{500}$. 
Figure~\ref{fig:largeK} reports the $\ell^1$-loss (averaged over 100 repetitions) of MSL for different $K$.

The results suggest that the $\ell^1$-loss of MSL increases with $K$. 
We recall that the minimax rate is proportional to $K\sqrt{K}$ ``asymptotically." However, for these ``finite" $K$ here, we observe that the error grows with $K$ approximately linearly.

\spacingset{1}
\begin{figure}[tb!]
\centering
\includegraphics[width=.5\textwidth]{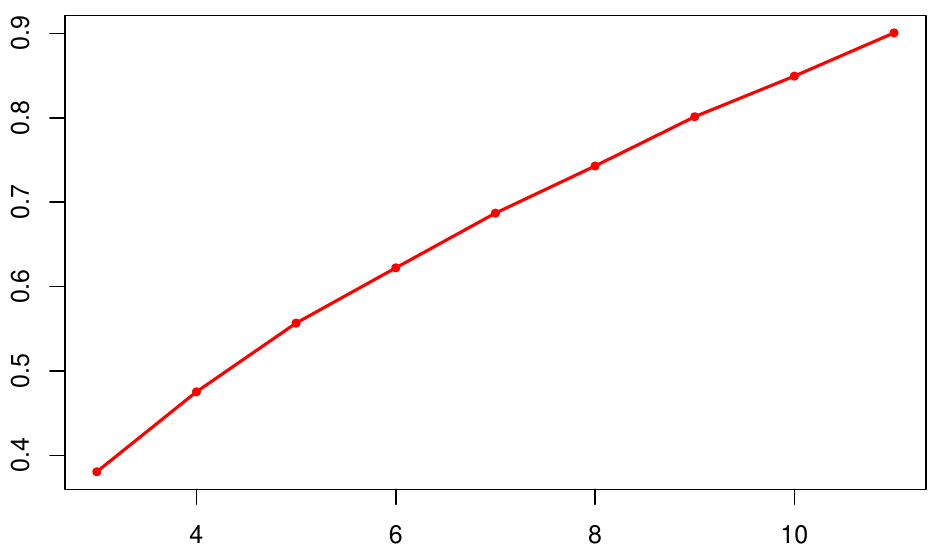} 
\caption{The $\ell^1$-loss of MSL under different values of $K$ ($n=5000$).}
\label{fig:largeK}
\end{figure}
\spacingset{1.7}


\section{Auxiliary lemmas on regularized graph  Laplacian} \label{supp:AuxiliaryLemma-L}
Given $\theta =(\theta_1, \theta_2, \ldots, \theta_n) $, recall that $\bar \theta = n^{-1} \sum_{i=1}^n \theta_i $. We introduce two disjoint index sets:
\beq\label{def:S1S2}
S_1 = \{1\leq i\leq n: \theta_i\geq \bar{\theta}\}, \qquad S_2 = \{1\leq i\leq n: \theta_i<\bar{\theta}\}.
\eeq

\subsection{Properties of $L_0$} \label{subsec:property_L_0}
Recall that $L_0= H_0^{-\frac 12} \Omega H_0^{-\frac 12}$. We state the following three lemmas  which give the spectrum properties of $L_0$ and also the estimates of the degree regularization matrix $H_0$. 
\begin{lem} \label{lem:order-eigenval}
Under the conditions of Theorem \ref{thm:eigenvector}, 
\begin{align} \label{order-lambda}
\lambda_1 >0, \quad \lambda_1 \asymp 1, \quad |\lambda_K|\asymp K^{-1}\lambda_K(PG), \quad \lambda_1-\max_{2\leq k\leq K}|\lambda_k|\geq c \lambda_1. 
\end{align}
\end{lem}

\begin{lem} \label{lem:order}
Under the conditions of Theorem \ref{thm:eigenvector}, 
\beq \label{order-xi}
\xi_1(i)\asymp \frac{1}{\sqrt{n}}\begin{cases}
\sqrt{\theta_i/\bar{\theta}}, & i\in S_1,\\
\theta_i/\bar{\theta}, & i\in S_2,
\end{cases}\qquad 
 \|\Xi(i)\|  \leq \frac{C\sqrt K}{\sqrt{n}} 
\begin{cases}
\sqrt{\theta_i/\bar{\theta}}, & i\in S_1,\\
\theta_i/\bar{\theta}, & i\in S_2,
\end{cases}
\eeq
and
\beq \label{order-H}
H_0(i,i) \asymp 
\begin{cases}
n\theta_i\bar{\theta}, & i\in S_1,\\
n\bar{\theta}^2, & i\in S_2.
\end{cases}
\eeq
\end{lem}

\begin{lem} \label{lem:noise}
Under the conditions of Theorem~\ref{thm:eigenvector}, with probability $1-o(n^{-3})$,
\beq \label{order-noise}
\|I_n-H_0^{-1}H\|\leq \frac{C\sqrt{\log(n)}}{\sqrt{n\bar{\theta}^2}}, \qquad \|H_0^{-1/2}(A-\Omega)H_0^{-1/2}\|\leq \frac{C}{\sqrt{n\bar{\theta}^2}}. 
\eeq
\end{lem}

The proof of Lemma \ref{lem:order-eigenval} is straightforward by noting that $D_{\theta}= H_0$ (see the definition of $D_\theta$ in  Section~\ref{sec:LB})  and therefore $L_0$ share the same eigenvalues as $K^{-1} PG$. Immediately, one can conclude (\ref{order-lambda}) from Condition~\ref{reg-conds}(b) and also the fact that $\lambda_1(PG) \asymp K$ which is derived from Condition~\ref{reg-conds}(a) and (b). 
In the sequel, we show the proof of the Lemmas \ref{lem:order} and \ref{lem:noise}. Before that, we introduce the Bernstein inequality which we will use frequently to bound sum of independent Bernoulli entries.
\begin{thm}[Bernstein inequality]\label{thm:Bern_ineq}
Let $X_1, \cdots, X_n$ be independent zero-mean random variables. Suppose that $|X_i|\leq M$ almost surely, for all $i$. Then for all $t>0$, 
\begin{align*}
\mathbb{P}\Big(\Big|\sum_{i=1}^n X_i\Big| \geq t \Big)\leq 2\exp\Big( -\frac{t^2/2}{\sigma^2 + Mt/3}\Big),
\end{align*}
with $\sigma^2:=\sum_{i=1}^n \mathbb{E}(X_i^2)$.
In particular, taking $t = C (\sigma \sqrt{\log (n)} + M \log (n) )$ for properly large $C$, then $$\Big|\sum_{i=1}^n X_i \Big|\leq C (\sigma \sqrt{\log (n)} + M \log (n) ) \quad \text{ with probability $1-o(n^{-5})$}.$$
\end{thm}

\begin{rmk}
The exponent in the high probability $1- o(n^{-5}) $ can be replaced by $-\tilde{C}$ for any large positive constant  integer $\tilde{C}$ by appropriately adjusting the constant $C$, which depends on $\tilde{C}$. Therefore, as long as we employ Bernstein inequality at most polynomial times in $n$ to obtain the ultimate upper bound, the result still holds with high probability by adjusting the constant factor in the bound. Specifically, we can achieve a high probability of $1- o(n^{-3})$ by choosing a  sufficiently large constant $C$ in the upper bound. Throughout the supplement, we sometimes omit specifying the high probability when applying Bernstein inequality for simplicity. It should be noted that in our analysis, Bernstein inequality is applied approximately $O(Kn)$ times. 
\end{rmk}

\begin{proof}[Proof of Lemma \ref{lem:order}]
First, we show (\ref{order-H}).  Uniformly for all $1\leq i \leq n$,
\begin{align}\label{210831031}
\mathbb{E} d_i   = \theta_i \sum_{j\neq i}\theta_j \pi_j' P \pi_i   =  \theta_i \sum_{j\neq i}\theta_j \sum_{k}\pi_j(k) e_k' P\pi_i\geq c_1 n\theta_i\bar{\theta}/K \cdot {\bf 1}_K' P \pi_i\geq c_1c_2 n\theta_i\bar{\theta}
\end{align}
by the last inequalities  in Condition~\ref{reg-conds}(a) and (b).  On the other hand, $ \pi_j' P \pi_i \leq \max_{t,s}P(t,s)$ for all $i, j$, then $\mathbb{E} d_i  = \theta_i \sum_{j\neq i}\theta_j \pi_j' P \pi_i \leq c n\theta_i\bar{\theta}$ for all $i$. As a result, $\mathbb{E} d_i\asymp n\theta_i\bar{\theta}$, $\mathbb{E} \bar{d}\asymp n\bar{\theta}^2$; and further
\begin{align*}
H_0(i,i)= \mathbb{E} d_i + \mathbb{E} \bar{d}  \asymp \begin{cases}
n\theta_i\bar{\theta}, & i\in S_1,\\
n\bar{\theta}^2, & i\in S_2.
\end{cases}
\end{align*}
This completes the proof of (\ref{order-H}).
Next, we turn to prove (\ref{order-xi}). By the definition $L_0= H_0^{-\frac 12} \Omega H_0^{-\frac 12}$, there exists a non-singular matrix $B\in \mathbb{R}^{K\times K}$ satisfying  
\begin{align*}
\Xi = H_0^{-\frac 12} \Theta \Pi  B, \qquad BB' = (\Pi' \Theta H_0^{-1} \Theta \Pi)^{-1}.
\end{align*}
Using Condition~\ref{reg-conds}(a) and $H_0= D_{\theta}$, one gets $ \Vert BB'\Vert \leq Kc$, $\lambda_{\min} (BB')\geq K c^{-1}$. Write $B= ({b}_1, \cdots, {b}_K)$. We have $K c^{-1} \leq \Vert b_i\Vert^2 \leq Kc$ for $1\leq i\leq K$. Taking the $i$-th row of $\Xi$, 
\begin{align*}
\Vert \Xi(i) \Vert= \frac{\theta_i}{\sqrt{H_0(i,i)}} \Vert \pi_i' B \Vert  \leq C\sqrt{\frac{\theta_i}{n\bar{\theta}}} \|\pi_i\| \Vert BB'\Vert^{\frac 12} \leq C\sqrt K \, \sqrt{\frac{\theta_i}{n\bar{\theta}}}.
\end{align*}
For the leading eigenvector $\xi_1$,  we have 
\begin{align*}
\xi_1(i) =  \frac{\theta_i}{\sqrt{H_0(i,i)}} \pi_i' b_1
\end{align*}
It follows from  $L_0 \Xi= \Xi \Lambda$ that $H_0^{-\frac 12} \Theta \Pi P \Pi' \Theta H_0^{-1} \Theta \Pi B= H_0^{-\frac 12}\Theta \Pi B \Lambda $, which implies that $P\Pi'\Theta H_0^{-1} \Theta \Pi B= B \Lambda$. As a consequence, $b_1$ is the first right  eigenvector of $P\Pi'\Theta H_0^{-1} \Theta \Pi$, and equivalently, the first right  eigenvector of $PG$. Using Condition~\ref{reg-conds}(c), we easily conclude that  $b_1(k)>0, b_1(k)\asymp 1$ for all $1\leq k\leq K$. Then, $\pi_i' b_1\asymp 1$ for all $1\leq i\leq n$, and the entrywise estimate of $\xi_1$ simply follows from (\ref{order-H}).
\end{proof}

\begin{proof}[Proof of Lemma \ref{lem:noise}]
Recall the definition of $H_0, H$. We write 
\begin{align*}
\frac{H(i,i)}{H_0(i,i)} -1  = \frac{d_i - \mathbb{E} d_i + \bar{d}- \mathbb{E} \bar{d}}{H_0(i,i)}, \qquad d_i - \mathbb{E} d_i  = \sum_{j\neq i} A_{ij} - \mathbb{E} A_{ij}.
\end{align*}
By (\ref{order-H}), we easily see that $H_0(i,i)\asymp n\bar{\theta}(\theta_i \vee \bar{\theta})$. What remains is to  estimate the numerator, or $d_i -\mathbb{E} d_i$ for all $1\leq i\leq n$. This actually can be achieved by employing Bernstein  inequality. Applying the Bernstein inequality (Theorem \ref{thm:Bern_ineq}) to $d_i -\mathbb{E} d_i$, we see that 
\begin{align*}
\mathbf{P}\Big( \big|\sum_{j\neq i} A_{ij} - \mathbb{E} A_{ij}\big| \geq t \Big) \leq 2\exp \bigg({- \frac{ \frac 12 t^2}{\sum_{j\neq i} {\rm var} A_{ij} + \frac 13 M t}}\bigg)
\end{align*}
where $M= \sup_{j} | A_{ij} - \mathbb{E} A_{ij}| \leq 2$.  Moreover, we have the crude bound 
\begin{align*}
\sum_{j\neq i} {\rm var} A_{ij} \leq c n \theta_i \bar{\theta}.
\end{align*}
Taking $t= C\sqrt{\log (n)} \sqrt{n\bar{\theta}\theta_i \vee \log (n) }$, it gives that $\big|\sum_{j\neq i} A_{ij} - \mathbb{E} A_{ij}\big|\leq C\sqrt{\log (n)( n\bar{\theta}\theta_i \vee \log (n) )}$ with probability $1- o(n^{-5})$.
Consider all $i$'s together,  one gets 
 \begin{align*}
\mathbf{P}\Big( \bigcup_{i=1}^n \Big\{ \big|\sum_{j\neq i} A_{ij} - \mathbb{E} A_{ij}\big| \geq C\sqrt{\log (n)}\, \big(\sqrt{ n \theta_i\bar{\theta} } \vee \sqrt{\log (n)}\big) \, \Big\}\Big) \leq c n^{-4}.
\end{align*}
This, combined with $H_0(i,i)\asymp n\bar{\theta}(\theta_i\vee \bar{\theta})$, implies that 
\begin{align*}
\Big| 1- H(i,i)/H_0(i,i)\Big| &\leq  \frac{|d_i - \mathbb{E} d_i|  }{H_0(i,i)} + \frac1n  \sum_{j=1}^n \frac{|d_j - \mathbb{E} d_j|  }{H_0(i,i)} \notag\\
&\leq C\sqrt{\frac{\log (n)}{n \bar{\theta}^2}} + C \sqrt{\frac{\log (n)}{n\bar{\theta}^2 }} \frac 1n \sum_{j=1}^n \max\Big\{ \sqrt{\frac{\theta_j}{\bar{\theta}}}, \sqrt{\frac{\log (n)}{n\bar{\theta}^2}} \, \Big\} \notag\\
&\leq C\sqrt{\frac{\log (n)}{n \bar{\theta}^2}} \Big(1+  \frac 1n \sum_{j=1}^n \sqrt{\frac{\theta_j}{\bar{\theta}}}+\sqrt{\frac{\log (n)}{n\bar{\theta}^2}} \,  \Big) 
\leq C \sqrt{\frac{\log (n)}{n\bar{\theta}^2}}
\end{align*}
with probability $1- o(n^{-3})$ uniformly for all $1\leq i \leq n$. Here in the last step, we used the Cauchy-Schwarz inequality $\sum_{j=1}^n \sqrt{\theta_j}\leq \sqrt n \sqrt{\sum_{j=1}^n \theta_j} = n \sqrt{\bar{\theta}}$. This finished the first estimate of (\ref{order-noise}). Now, we proceed to the second estimate. We crudely bound
$\|H_0^{-1/2}(A-\Omega)H_0^{-1/2}\|$ by 
$\Vert H_0^{-1/2}W H_0^{-1/2}\Vert + \Vert H_0^{-1/2}{\rm diag}(\Omega) H_0^{-1/2}\Vert $. First, it is easy to get the bound
\begin{align*}
\Vert H_0^{-1/2}{\rm diag}(\Omega) H_0^{-1/2}\Vert \leq C\max_{1\leq i \leq n }  \frac{\theta_i^2 \pi_i'P\pi_i}{H_0(i,i)}\leq \frac{C}{n\bar{\theta}} \leq \frac{C}{ \sqrt{n\bar{\theta}^2}}.
\end{align*}
Next, we apply the non-asymptotic bounds for random matrices in \cite{bandeira2016sharp} to bound the operator norm of $\widehat{W}:=  H_0^{-1/2}W H_0^{-1/2}$. Note that $\widehat{W}$ is a symmetric random matrix with independent upper triangular entries. Using Corollary 3.12 of \cite{bandeira2016sharp} with Remark 3.13, we bound 
\begin{align*}
\mathbb{P} (\Vert \widehat{W}\Vert \geq C \tilde{\sigma} + t )\leq n e^{-t^2/c\tilde{\sigma}_*^2}
\end{align*} 
for some constant $C, c>0$, with 
\begin{align*}
\tilde{\sigma}= \max_{i} \sqrt{\sum_{j} \mathbb{E} \widehat{W}(i,j)^2}\leq 1/\sqrt{n\bar{\theta}^2}, \quad \tilde{\sigma}_*= \max_{i,j} \Vert \widehat{W}(i,j)\Vert_\infty \leq C/n\bar{\theta}^2.
\end{align*}
Then, we take $t= c/\sqrt{n\bar{\theta}^2}$ for properly large $c>0$ and use the assumption $n\bar{\theta}^2\gg \log (n)$. It follows that $\Vert \widehat{W}\Vert \leq C/\sqrt{n\bar{\theta}^2} $ with probability $1- o(n^{-3})$.
We thus complete the proof of Lemma \ref{lem:noise}.
\end{proof}

\subsection{Properties of $\tilde{L}^{(i)}$} \label{subsec:property_tL}

In this section, for an arbitrary  fixed index $i$ and  the intermediate matrix $\tilde{L}^{(i)}$, we collect the spectrum properties of $\tilde{L}^{(i)}$ and estimate on  $\tilde{H}^{(i)}$ in the lemmas below. Let $E$ be the event that Lemma \ref{lem:noise} holds.

\begin{lem} \label{lem:est.etL}
Under the conditions in Theorem \ref{thm:eigenvector}. Over the event $E$, for any fixed $1\leq i\leq n$,   the eigenvalues $\tilde{\lambda}^{(i)}_1,\cdots, \tilde{\lambda}^{(i)}_K$ of $\tilde{L}^{(i)}$ satisfy
\begin{align}\label{prop:eig.tL}
\tilde{\lambda}^{(i)}_1 >0, \quad \tilde{\lambda}^{(i)}_1\asymp 1, \quad |\tilde{\lambda}^{(i)}_K|\asymp K^{-1}|\lambda_K(PG)|, \quad \tilde{\lambda}^{(i)}_1-\max_{2\leq k\leq K}|\tilde{\lambda}^{(i)}_k|\geq C^{-1}\tilde{\lambda}^{(i)}_1 ;
\end{align}
and for the associated eigenvectors, 
\begin{align}\label{prop:eigv.tL}
\tilde{\xi}^{(i)}_1(j)\asymp \frac{1}{\sqrt{n}}\begin{cases}
\sqrt{\theta_j/\bar{\theta}}, & j\in S_1,\\
\theta_j/\bar{\theta}, & j\in S_2,
\end{cases}\qquad 
 \|\tilde{\Xi}^{(i)}(j)\| \leq \frac{C\sqrt K}{\sqrt{n}} 
\begin{cases}
\sqrt{\theta_j/\bar{\theta}}, & j\in S_1,\\
\theta_j/\bar{\theta}, & j\in S_2.
\end{cases}
\end{align}
\end{lem} 

\begin{lem}\label{lem:noisetildeH}
 Under the conditions of  Theorem~\ref{thm:eigenvector}.  Over the event $E$, for any fixed $1\leq i \leq n$ and  $\tilde{H}^{(i)}$, 
 \begin{align} \label{est.tildeYY}
 \Vert I_n - H_0^{-1}\tilde{H}^{(i)} \Vert  \leq \frac{C\sqrt{\log(n)}}{\sqrt{n\bar{\theta}^2}},\quad  \Vert I_n - (\tilde{H}^{(i)})^{-1}{H} \Vert  \leq \frac{C\sqrt{\log(n)}}{\sqrt{n\bar{\theta}^2}}.
  \end{align}

 \end{lem}
 \noindent In addition to the above lemma, by Theorem \ref{thm:Bern_ineq} and after elementary computations, we also have that for each $1\leq j \leq n, j\neq i$, over the event $E$,
 \begin{align}\label{est:HHtd1}
\big|\tilde{H}^{(i)}(j,j) - H(j,j)\big| &= \Big|- W(j,i) - \frac 2n \Big(\sum_{s\neq i} W(i, s)\Big) \Big|  \notag\\
& \leq \Big|- A(j,i) + \Omega(j,i)- \frac 2n \Big(\sum_{s\neq i} W(i, s)\Big) \Big| \notag\\
&\leq A(j,i) + \theta_j\theta_i +  C\frac{\sqrt{ \theta_i \bar{\theta} \log n}}{\sqrt n}+  \frac{C\log (n)}{n};
\end{align}
 and 
\begin{align}\label{est:HHtd2}
\big|\tilde{H}^{(i)}(i,i) - H(i,i)\big|\leq (1+ 2/n) \Big|\sum_{s\neq i} W(i, s)\Big| \leq  C\sqrt{n\theta_i\bar{\theta} \log (n)}+ C \log (n).
\end{align}
Applying (\ref{est:HHtd1}) and (\ref{est:HHtd2}) with Lemma \ref{lem:noise}, it is easy to deduce the estimates in Lemma \ref{lem:noisetildeH}. To show the eigen-properties of $\tilde{L}^{(i)}$ in Lemma \ref{lem:est.etL}, one only need to rely on the estimate 
$$\Vert \tilde{L}^{(i)} - L_0 \Vert\asymp  \Vert I_n - H_0^{-1}\tilde{H}^{(i)} \Vert \lambda_1(L_0)\leq C \sqrt{\frac{\log (n)}{n\bar{\theta}^2}}\ll |\lambda_K|$$
under the assumption of Theorem~\ref{thm:eigenvector}, then (\ref{prop:eig.tL}) can be derived simply by further applying Lemma 
\ref{lem:order-eigenval}. Moreover, (\ref{prop:eigv.tL}) follows from  Lemmas \ref{lem:upbd_key1}, \ref{lem:upbd_key2K} and \ref{lem:order}. Thereby, we omit the proofs of  Lemmas \ref{lem:est.etL} and \ref{lem:noisetildeH}.  We comment here that the proof of Lemmas \ref{lem:upbd_key1}, \ref{lem:upbd_key2K} only depends on the lemmas in Section \ref{subsec:property_L_0}, i.e., the properties of $L_0$, not the properties of $\tilde{L}^{(i)}$. There is no circular logic for the  lemmas presenting in this subsection.

\subsection{Variants of Davis-Kahan sin$\theta$ Theorem}
In our analysis, we heavily rely on the use of Davis-Kahan sin$\theta$ theorem under different versions. For readers' convenience, we collect all the variants we employed in our theory below.

\begin{thm}[Davis-Kahan sin$\Theta$ Theorem and its variants]
Let $\Sigma, \hat{\Sigma}\in \mathbb{R}^{p, p}$ be symmetric, with eigenvalues $\lambda_1\geq \cdots \geq \lambda_p$ and  $\hat{\lambda}_1\geq \cdots \geq \hat{\lambda}_p$ respectively.   Fix $1\leq r\leq s\leq p$, let $U = (u_r, \cdots, u_s)$ and $U^\perp = (u_1, \cdots, u_{r-1}, u_{s+1}, \cdots, u_p)$ be the orthonormal eigenvectors such that $\Sigma u_j = \lambda_j u_j$, similarly we define $\hat{U}, \hat{U}^\perp $ for $\hat{\Sigma}$. Denote $\delta: = \inf \{ |\hat{\lambda} - \lambda|: \lambda\in [\lambda_r, \lambda_s], \hat{\lambda} \in (-\infty, \hat{\lambda}_{r-1}]\cup [\hat{\lambda}_{s+1}, \infty)  \}$ where we take the convention $\hat{\lambda}_ 0= - \infty$ and $\hat{\lambda}_{p+1} = \infty$. Then, 
\begin{align} \label{sine-theta:perp}
\Vert (\hat{U}^{\perp})' U\Vert = \Vert \sin\Theta (\hat{U}, U)\Vert \leq \frac{\Vert \hat{\Sigma} - \Sigma \Vert }{\delta}
\end{align}
for some constant $C>0$. 
Moreover, there exists an orthogonal matrix $O$ of dimension ${s-r+1}$ such that
\begin{align} 
&\Vert \hat{U}' U - O'\Vert \leq C  \bigg( \frac{\Vert \hat{\Sigma} - \Sigma \Vert }{\delta}
\bigg)^2, \label{sine-theta:perturb1}\\
& \Vert \hat{U} - UO\Vert \leq C \frac{\Vert \hat{\Sigma} - \Sigma \Vert }{\delta} \label{sine-theta:perturb2}
\end{align}
for some constant $C>0$. 
\end{thm}

Note that (\ref{sine-theta:perp}) is the version of sin$\Theta$ theorem proved by Davis and Kahan's original paper \cite{sin-theta}. The proof of  (\ref{sine-theta:perturb1}) can be referred to Lemma B.2 in the Supplementary of \cite{abbe2020entrywise}. More specifically, $O' = {\rm sgn}( \hat{U}' U) = \bar{U}\bar{V}'$ where the SVD of $ \hat{U}' U$ is given by $ \bar{U}\bar{\Lambda} \bar{V}'$.  By Chp I, Cor 5.4 of \cite{StewartSun1990}, the singular values in $\bar{\Lambda}$ are the cosines of canonical angles $0\leq \bar{\theta}_1 \leq \cdots \leq \bar{\theta}_{s-r+1}\leq \pi/2$ between $\hat{U}$ and $U$.  It follows that 
\begin{align*}
\Vert \hat{U}' U - O'\Vert = 1-  \cos \bar{\theta}_{s - r + 1} \leq 1- \cos^2 \bar{\theta}_{s - r + 1} = \sin ^2 \bar{\theta}_{s - r + 1}  =  \Vert \sin\Theta (\hat{U}, U)\Vert^2
\end{align*}
Thus, (\ref{sine-theta:perturb1}) follows directly from (\ref{sine-theta:perp}).  (\ref{sine-theta:perturb2}) is implied by the simple derivations
\begin{align*}
\Vert \hat{U} - UO\Vert^2 = \Vert 2I_{s-r+1} - \hat{U}' U O - O'U' \hat{U}\Vert \leq 2 \Vert \hat{U}' U - O'\Vert \, . 
\end{align*}
%

\section{Entrywise eigenvector analysis} \label{sec:EntrywiseAnalysis-in-supp}
Here we show the complete proof of Theorem  \ref{thm:eigenvector} in our manuscript. In Sections \ref{subsec:1steigen1}-\ref{subsec:1steigen2}, we state the proofs of key lemmas for proving (\ref{result-1}), while the claim of (\ref{result-1}) is already presented in the manuscript. Section \ref{sub:2-Keig} collects the proof of the second claim in Theorem~\ref{thm:eigenvector} (i.e.,  (\ref{result-2})) which provides the entry-wise estimates for the $2$- to $K$-th 
eigenvectors. Similarly to the proof of  the first claim in Theorem~\ref{thm:eigenvector} (i.e., (\ref{result-1}) ), we introduce three key lemmas, Lemmas \ref{lem:upbd_key2K}-\ref{lem:techB1_2K},  counterpart to Lemmas \ref{lem:upbd_key1}-\ref{lem:techB1}. The  proofs  of Lemmas \ref{lem:upbd_key2K}-\ref{lem:techB1_2K} are provided  correspondingly in Section \ref{subsec:pr_Lem_upbd_key2K}-\ref{subsec:pr_Lem_techB1_2K}.


\subsection{Proof of Lemma \ref{lem:upbd_key1}} \label{subsec:1steigen1}
In this subsection, we show the proof of Lemma \ref{lem:upbd_key1} using the eigen-properties of $L_0$ in Section \ref{subsec:property_L_0}.

Fix the index $i$, we study the perturbation from $L_0=H_0^{-1/2}\Omega H_0^{-1/2}$ to $\tilde{L}^{(i)}=(\tilde{H}^{(i)})^{-1/2}\Omega (\tilde{H}^{(i)})^{-1/2}$. 
By definition, 
\[
L_0=H_0^{-1/2}\Omega H_0^{-1/2}=\sum_{k=1}^K\lambda_k\xi_k\xi_k', \qquad (\tilde{H}^{(i)})^{-1/2}\Omega (\tilde{H}^{(i)})^{-1/2}\tilde{\xi}^{(i)}_1  = \tilde{\lambda}^{(i)}_1\tilde{\xi}^{(i)}_1. 
\]
Write $\tilde{Y}\equiv \tilde{Y}^{(i)}:={H_0^{1/2}(\tilde{H}^{(i)})^{-1/2}}$. Then, we have
\[
\tilde{Y}\bigl(\sum_{k=1}^K\lambda_k \xi_k\xi_k'\bigr)\tilde{Y}\tilde{\xi}^{(i)}_1= \tilde{\lambda}^{(i)}_1\tilde{\xi}^{(i)}_1. 
\]
It follows that, for each $1\leq j\leq n$, 
\beq \label{easyperturb-1}
\frac{1}{\tilde{Y}(j,j)}\tilde{\xi}^{(i)}_1(j)  = \frac{\lambda_1(\xi_1'\tilde{Y}\tilde{\xi}^{(i)}_1)}{\tilde{\lambda}^{(i)}_1}\xi_1(j)  + \sum_{k=2}^K  \frac{\lambda_k(\xi_k'\tilde{Y}\tilde{\xi}^{(i)}_1)}{\tilde{\lambda}^{(i)}_1}\xi_k(j). 
\eeq
As a result, 
\beq \label{easyperturb-2}
 |\tilde{\xi}^{(i)}_1(j)-\xi_1(j)|\leq \Bigl|\frac{1}{\tilde{Y}(j,j)}-1\Bigr||\tilde{\xi}^{(i)}_1(j)| + \Bigl| \frac{\lambda_1(\xi_1'\tilde{Y}\tilde{\xi}^{(i)}_1)}{\tilde{\lambda}^{(i)}_1}-1 \Bigr||\xi_1(j)| +  
 \sum_{k=2}^K  \Bigl|\frac{\lambda_k(\xi_k'\tilde{Y}\tilde{\xi}^{(i)}_1)}{\tilde{\lambda}^{(i)}_1}\Bigr||\xi_k(j)|.
\eeq
By Lemma~\ref{lem:order-eigenval}, $\|L_0\|\leq CK^{-1}\lambda_1(PG)\leq C$. And 
using the first estimate in (\ref{est.tildeYY}), it is easy to conclude that 
\begin{align} \label{est:tY-In}
\|\tilde{Y}-I_n\| = \big\Vert I_n - (H_0^{-1} \tilde{H}^{(i)})^{-\frac 12} \big\Vert  \leq \frac{C\sqrt{\log(n)}}{\sqrt{n\bar{\theta}^2}},
\end{align}
over the event $E$ where Lemma \ref{lem:noise} holds.
 As a result,  we have $\|\tilde{L}^{(i)}- L_0\|\leq\Vert (\tilde{Y} -I_n) L_0\Vert \leq CK^{-1} \lambda_1(PG)\|\tilde{Y} -I_n\|$ since $\tilde{L}^{(i)} = \tilde{Y}  L_0 \tilde{Y} $.  Using Weyl's inequality, we then see that 
 \begin{align*}
 \max_{1\leq k\leq K} {|\tilde{\lambda}^{(i)}_k-\lambda_k|} & \leq \|\tilde{L}^{(i)}- L_0\|\leq CK^{-1}\lambda_1(PG) \|\tilde{Y}-I_n\| \leq  C \|\tilde{Y}-I_n\|
 \end{align*}
 since $\lambda_1(PG) \leq CK$ under our model assumption.
Furthermore, by Lemma~\ref{lem:order-eigenval}, the eigen-gap between the largest eigenvalue and the other nonzero eigenvalues of $L_0$ is at the  order $K^{-1}\lambda_1(PG)$. Hence, the eigengap between $\lambda_1$ and $\tilde{\lambda}_2^{(i)}, \cdots, \tilde{\lambda}_K^{(i)}$ is still of the order $K^{-1}\lambda_1(PG)$.  It follows from the sin-theta theorem (\ref{sine-theta:perturb1}) that
\begin{align*}
|\xi_1'\tilde{Y}\tilde{\xi}^{(i)}_1- 1| &\leq |\xi_1' \tilde{\xi}^{(i)}_1- 1| + |\xi_1' (\tilde{Y}-I_n) \tilde{\xi}^{(i)}_1| \notag\\
  & \leq C (K \lambda_1^{-1}(PG)\Vert \tilde{L}^{(i)}- L_0\Vert )^2 + \Vert \tilde{Y} - I_n\Vert \leq C\|\tilde{Y}-I_n\|.
\end{align*}
Here ${\rm sgn}(\xi_1' \tilde{\xi}^{(i)}_1) = 1$ since we fix our choices of $\xi_1,  \tilde{\xi}^{(i)}_1$ with positive first components and they are both from the positive matrices. Then this will be claimed by Perron's theorem.

Using  Cauchy-Schwarz inequality, we bound $\sum_{k=2}^K |\xi_k'\tilde{Y}\tilde{\xi}^{(i)}_1| |\xi_k(j)| \leq \Vert  \Xi_1' \tilde{Y} \tilde{\xi}^{(i)}_1\Vert \Vert \Xi_1(j)\Vert $. And by sine-theta theorem (\ref{sine-theta:perp}),   
\begin{align*}
 \Vert  \Xi_1' \tilde{Y} \tilde{\xi}^{(i)}_1\Vert &\leq   \Vert (\tilde{\xi}^{(i)}_1)' \Xi_1\Vert  + \Vert \tilde{Y} -I_n \Vert  \notag\\
&\leq C\Big(K\lambda_1^{-1}(PG) \|\tilde{L}^{(i)}- L_0\| + \|\tilde{Y} -I_n\| \Big) \cr 
&\leq C  \|\tilde{Y} - I_n\|.
\end{align*}
Plugging the above  estimates, we have
\begin{align*}
|\tilde{\xi}^{(i)}_1(j)-\xi_1(j)| & \leq C\|\tilde{Y} - I_n\||\tilde{\xi}^{(i)}_1(j)| + C\|\tilde{Y} - I_n\|\|\Xi(j)\|\cr
&\leq C\|\tilde{Y} - I_n\| |\tilde{\xi}^{(i)}_1(j)-\xi_1(j)| +  C  \|\tilde{Y} - I_n\|\|\Xi(j)\|. 
\end{align*}
Since $\|\tilde{Y}-I_n\|= o(1)$ over the event $E$, rearranging the terms gives
\beq \label{easyperturb-3}
|\tilde{\xi}^{(i)}_1(j)-\xi_1(j)|\leq C  \|\tilde{Y}  -I_n\|\|\Xi(j)\|,  \qquad \mbox{for all }1\leq t\leq n. 
\eeq
We plug (\ref{est:tY-In}) into \eqref{easyperturb-3} and use the bound for $\|\Xi(j)\|$ in \eqref{order-xi}. It follows that over the event $E$, for all $1\leq j\leq n$, 
\beq  \label{easyperturb-5}
|\tilde{\xi}^{(i)}_1(j)- \xi_1(j)| \leq C \sqrt{K} \sqrt{\frac{\log(n)}{n\bar{\theta}^2}} \sqrt{\frac{\theta_j}{n\bar{\theta}}}
\bigg( \sqrt{\frac{\theta_j}{\bar{\theta}}} \wedge 1\, \bigg)
\eeq
Then, consider all $i$'s together, we conclude  (\ref{eq:t1}) with probability $1 - o(n^{-3})$ simultaneously for all $1\leq i, j\leq n$. 

\subsection{Proof of Lemma \ref{lem:upbd_key11}}
In this subsection, we state the proof of Lemma \ref{lem:upbd_key11} which heavily relies on the eigen-properties of $\tilde{L}^{(i)}$ in Section \ref{subsec:property_tL}.

Fix the index $i$, we first show (\ref{Laplacian-entry-2}) which is based on the decomposition
\begin{align*}
w\hat{\xi}_1 = \tilde{\xi}^{(i)}_1 + (\overline{\xi}^{(i)}_1-\tilde{\xi}^{(i)}_1) + (w\hat{\xi}_1- \overline{\xi}^{(i)}_1)
\end{align*}
where $w= {\rm sgn} ({\xi}_1' \hat{\xi}_1)$ will be claimed later. It is not hard to derive
\begin{align*}
|e_i' \Delta (w\hat{\xi}_1- \overline{\xi}^{(i)}_1)| \leq \Vert\Delta \Vert \,\| w\hat{\xi}_1-\overline{\xi}^{(i)}_1\| & \leq \Vert H_0^{-\frac 12} W H_0^{-\frac 12} \Vert \Vert H_0^{\frac 12} H^{-\frac 12} \Vert \Vert H_0^{\frac 12} (\tilde{H}^{(i)})^{-\frac 12} \Vert \,\| w\hat{\xi}_1-\overline{\xi}^{(i)}_1\| \notag\\
&\leq \frac{C}{\sqrt{n\bar{\theta}^2}} \,\| w\hat{\xi}_1-\overline{\xi}^{(i)}_1\|
\end{align*} 
over the event $E$,
in light of Lemmas \ref{lem:noise} and \ref{lem:noisetildeH}. We thus end up with (\ref{Laplacian-entry-2}) . We now turn to prove (\ref{Laplacian-entry-1}). 
We study the perturbation from $\tilde{L}^{(i)} = (\tilde{H}^{(i)})^{-1/2}\Omega(\tilde{H}^{(i)})^{-1/2}$ to $L= H^{-1/2}AH^{-1/2}$. Write $X\equiv X(i):= (\tilde{H}^{(i)})^{1/2} H^{-1/2}$. We can rewrite
\begin{align*}
L =  X(\tilde{H}^{(i)})^{-\frac 12}A (\tilde{H}^{(i)})^{-\frac 12} X= X\tilde{L}^{(i)}X - X(\tilde{H}^{(i)})^{-\frac 12}{\rm diag} (\Omega) (\tilde{H}^{(i)})^{-\frac 12}X + X\Delta 
\end{align*}
with $\Delta= (\tilde{H}^{(i)})^{-1/2} W H^{-1/2}$.
By definition, $\tilde{L}^{(i)}=( \tilde{H}^{(i)}) ^{-1/2}\Omega(\tilde{H}^{(i)})^{-1/2}=\sum_{k=1}^K\tilde{\lambda}^{(i)}_k\tilde{\xi}^{(i)}_k(\tilde{\xi}^{(i)}_k)'$ and $H^{-1/2}AH^{-1/2}\hat{\xi}_1=\hat{\lambda}_1\hat{\xi}_1$. It follows that
\[
\sum_{k=1}^K  \tilde{\lambda}^{(i)}_k(\hat{\xi}_1'X\tilde{\xi}_k^{(i)}) X\tilde{\xi}^{(i)}_k- X(\tilde{H}^{(i)})^{-\frac 12}{\rm diag} (\Omega) (\tilde{H}^{(i)})^{-\frac 12}X\hat{\xi}_1 + X\Delta  \hat{\xi}_1 = \hat{\lambda}_1\hat{\xi}_1. 
\]
As a result,
\begin{align} \label{Laplacian-step2-1}
\hat{\xi}_1(i) =& \frac{\tilde{\lambda}^{(i)}_1(\hat{\xi}_1'X\tilde{\xi}^{(i)}_1)}{\hat{\lambda}_1}X(i,i)\tilde{\xi}_1^{(i)}(i) +\sum_{k=2}^K \frac{\tilde{\lambda}^{(i)}_k(\hat{\xi}_1'X \tilde{\xi}^{(i)}_k)}{\hat{\lambda}_1}X(i,i)\tilde{\xi}^{(i)}_k(i)  \notag\\
&-  \frac{X^2(i,i)\Omega(i,i)}{\hat{\lambda}_1\tilde{H}^{(i)}(i,i)} \hat{\xi}_1(i) + \frac{X(i,i)}{\hat{\lambda}_1} e_i'\Delta \hat{\xi}_1. 
\end{align}
By Lemma \ref{lem:noisetildeH}, it is easy to deduce that 
\begin{align}\label{est:X-I}
\Vert X - I_n\Vert \leq C\frac{\sqrt{\log (n)}}{\sqrt{n\bar{\theta}^2}},
\end{align}
Since (\ref{prop:eig.tL}),  by Weyl's inequality, 
\begin{align} \label{2023030501}
\max_{1\leq k\leq K}\bigl\{{|\hat{\lambda}_k - \tilde{\lambda}^{(i)}_k|}\bigr\}\leq C\|L-\tilde{L}^{(i)}\|;
\end{align}
and over the event $ E$,
\begin{align}\label{l2norm_LtL}
\|L-\tilde{L}^{(i)}\| &\leq C\|X- I_n\| \Vert H^{-\frac 12} AH^{-\frac12}\Vert+\|(\tilde{H}^{(i)})^{-1}H_0\|\|{H}_0^{-1/2}(A-\Omega){H}_0^{-1/2}\| \notag\\
&\leq C\,\|X- I_n\| \Vert H_0^{-\frac 12} \Omega H_0^{-\frac12}\Vert+\|(\tilde{H}^{(i)})^{-1}H_0\|\|{H}_0^{-1/2}(A-\Omega){H}_0^{-1/2}\|, \notag\\
& \leq C \, \frac{K^{-1} \lambda_1(PG)\sqrt{\log (n)}+ 1}{\sqrt{n\bar{\theta}^2}} \leq C\sqrt{\frac{\log(n) }{n\bar{\theta}^2}}\ll  |\tilde{\lambda}^{(i)}_K|
\end{align}
since Lemmas \ref{lem:order-eigenval}, \ref{lem:noise} and (\ref{prop:eig.tL}) with the condition $K\beta_n^{-1}\sqrt{\log (n)}/\sqrt{n\bar{\theta}^2}\ll1$. Therefore, $\hat{\lambda}_1, \cdots, \hat{\lambda}_K$ share the same asymptotics as $\tilde{\lambda}^{(i)}_1, \cdots, \tilde{\lambda}^{(i)}_K$. The eigengap between $\hat{\lambda}_1$ and $\tilde{\lambda}^{(i)}_2, \cdots, \tilde{\lambda}^{(i)}_K$ is $K^{-1} \lambda_1(PG)$. Let $w(i) = {\rm sgn}(\hat{\xi}_1' \tilde{\xi}_1^{(i)})$. It follows that
\begin{align} \label{21090601}
|\hat{\xi}_1' X\tilde{\xi}_1^{(i)}- w(i) | &\leq \Vert X- I_n\Vert  + |\hat{\xi}_1' \tilde{\xi}_1^{(i)}- w(i) | \notag\\
&\leq \Vert X- I_n\Vert + C\Big( K \lambda_1^{-1}(PG)\|L-\tilde{L}^{(i)}\| \Big)^2 
\end{align}
where the last step is due to (\ref{sine-theta:perturb1}) in sin$\Theta$ Theorem . 
In particular,   further by (\ref{sine-theta:perturb2})
\begin{align}\label{210906011}
\Vert \hat{\xi}_1' X(\tilde{\xi}^{(i)}_2, \cdots, \tilde{\xi}^{(i)}_K)\Vert \leq  \Vert X- I_n\Vert + CK \lambda_1^{-1}(PG)\|L-\tilde{L}^{(i)}\| . 
\end{align}
We can actually claim that $w(i)\equiv w:= {\rm sgn} ({\xi}_1'\hat{\xi}_1)$ as follows. First notice $|\hat{\xi}_1' {\xi}_1  -\hat{\xi}_1' \tilde{\xi}^{(i)}_1|\leq \Vert \tilde{\xi}^{(i)}_1 - {\xi}_1 \Vert = o(1)$. Next,  $|\hat{\xi}_1' \tilde{\xi}^{(i)}_1|>c$ for some constant $c\in (0,1)$. It follows immediately that $w(i)={\rm sgn}(\hat{\xi}_1' \tilde{\xi}^{(i)}_1)  ={\rm sgn} (\hat{\xi}_1' {\xi}_1) = w$.
%
 In the sequel, we directly write $w$ instead of $w(i)$.  We plug in (\ref{2023030501}), (\ref{21090601}) and (\ref{210906011}) into (\ref{Laplacian-step2-1}).  By some elementary simplifications, it arrives at 
 \begin{align}\label{20230503}
 |w\hat{\xi}_1(i)-\tilde{\xi}_1^{(i)}(i)| 
 &\leq  \Big( \Vert X - I_n \Vert + K\lambda_1^{-1}(PG) \Vert L-\tilde{L}^{(i)} \Vert \Big) \Big( \big| \tilde{\xi}^{(i)}_1(i)\big| + \Vert \tilde{\Xi}_1^{(i)} (i)\Vert \Big) \notag\\
 & \quad +  \Big|\frac{X^2(i,i)\Omega(i,i)}{\hat{\lambda}_1\tilde{H}^{(i)}(i,i)} \hat{\xi}_1(i) \Big| + \Big|\frac{X(i,i)}{\hat{\lambda}_1} e_i'\Delta \hat{\xi}_1 \Big|\, . 
 \end{align}
We can further derive
\begin{align} \label{20230502}
 \bigg|\frac{X(i,i)^2\Omega(i,i)}{\hat{\lambda}_1\tilde{H}^{(i)}(i,i)} \hat{\xi}_1(i)\bigg| \leq \frac{CK\theta_i^2}{n\bar{\theta} (\bar{\theta} \vee \theta_i) \lambda_1(PG)} \leq C K \lambda_1^{-1} (PG){\kappa_i}\Big(  \sqrt{\frac{\theta_i}{\bar{\theta}}}\wedge 1 \Big), 
\end{align}
by the estimate $\tilde{H}^{(i)}(i,i) \asymp n\bar{\theta} (\bar{\theta} \vee \theta_i)$ following from (\ref{order-H}) and the first estimate in (\ref{est.tildeYY}), with $ \kappa_i = \frac{\sqrt{\log(n)}}{n\bar{\theta}}\cdot \frac{\sqrt{\theta_i}}{\sqrt{\bar{\theta}}}$.
Then, plugging (\ref{20230502}), (\ref{est:X-I}) and (\ref{l2norm_LtL}) into (\ref{20230503}) gives 
\begin{align*}
|w\hat{\xi}_1(i)-\tilde{\xi}_1^{(i)}(i)| 
 &\leq C \Big(\sqrt{\frac{\log n}{n\bar{\theta}^2}} + \frac{K}{\lambda_1(PG)\sqrt{n\bar{\theta}^2}}\Big)   \Big( \big| \tilde{\xi}^{(i)}_1(i)\big| + \Vert \tilde{\Xi}_1^{(i)} (i)\Vert \Big)  \notag\\
&\qquad +C K \lambda_1^{-1} (PG){\kappa_i}\Big(  \sqrt{\frac{\theta_i}{\bar{\theta}}}\wedge 1 \Big) + CK \lambda_1^{-1}(PG)|e_i'\Delta\hat{\xi}_1|\notag\\
& \leq C \kappa_i\Big(  \sqrt{\frac{\theta_i}{\bar{\theta}}}\wedge 1 \Big) K^{\frac 32} \lambda_1^{-1}(PG) + CK \lambda_1^{-1}(PG)|e_i'\Delta\hat{\xi}_1|
\end{align*}
where in the last step, we plugged 
 the bound of $ \big| \tilde{\xi}^{(i)}_1(i)\big|$ and $\Vert \tilde{\Xi}^{(i)}(i)\Vert$ in (\ref{prop:eigv.tL}). Since the assumption $\lambda_1(PG) \geq CK$, it gives that over the event $E$,
\begin{align*} 
|w\hat{\xi}_1(i)-\tilde{\xi}^{(i)}_1(i)|
&\leq  C\sqrt K \kappa_i  + C |e_i'\Delta\hat{\xi}_1|.  
\end{align*}
This concludes our proof by considering all $i$'s together.

\subsection{Proof of  Lemma \ref{lem:techB1} } \label{subsec:1steigen2}
In this section, we prove Lemma \ref{lem:techB1}. We separate the proofs into three parts corresponding to the three estimates (\ref{Laplacian-entry-3-1})-(\ref{Laplacian-entry-3-3}).
\subsubsection{Proof of (\ref{Laplacian-entry-3-1})}
\label{app:subsub_techproof1}
For any fixed $i$,
recall that  $X\equiv X(i):= (\tilde{H}^{(i)})^{1/2}H^{-1/2}$. We rewrite $\Delta\equiv \Delta(i)= (\tilde{H}^{(i)})^{-\frac 12} W (\tilde{H}^{(i)})^{-\frac 12}X$. It follows that 
\begin{align} \label{2112080000}
e_i' \Delta \tilde{\xi}^{(i)}_1 = \frac{W(i)  (\tilde{H}^{(i)})^{-\frac 12} X\tilde{\xi}^{(i)}_1}{\sqrt{\tilde{H}^{(i)}(i,i)}} =  \frac{W(i)  (\tilde{H}^{(i)})^{-\frac 12} \tilde{\xi}^{(i)}_1}{\sqrt{\tilde{H}^{(i)}(i,i)}}   +
 \frac{W(i)  (\tilde{H}^{(i)})^{-\frac 12} (X-I_n)\tilde{\xi}^{(i)}_1}{\sqrt{\tilde{H}^{(i)}(i,i)}} 
\end{align}

First, we study the term $|W(i)(\tilde{H}^{(i)})^{-1/2}\tilde{\xi}^{(i)}_1|$. Write
\[
W(i)(\tilde{H}^{(i)})^{-1/2}\tilde{\xi}^{(i)}_1 = \sum_{1\leq j\leq n: j\neq i}\frac{W(i,j)}{\sqrt{\tilde{H}^{(i)}(j,j)}}\tilde{\xi}^{(i)}_1(j). 
\]
In the sequel, we only consider the randomness of the $i$-th row of $W$. Note that the mean is $0$.  The variance is bounded by  (up to some constant $C$)
\[
\sum_{j\neq i }\frac{\theta_i\theta_j}{n\bar{\theta}(\theta_j\vee \bar{\theta})}\big( \tilde{\xi}^{(i)}_1(j)\big)^2\leq \sum_{j}\frac{\theta_i\theta_j}{n\theta_j\bar{\theta}}\big( \tilde{\xi}^{(i)}_1(j)\big)^2 \leq \frac{\theta_i}{n\bar{\theta}}. 
\]
Recall the definition of index sets $S_1, S_2 $ in (\ref{def:S1S2}) . Each term in the sum is bounded by 
\[
\frac{|\tilde{\xi}^{(i)}_1(j)|}{\sqrt{\tilde{H}^{(i)}(j,j)}}\leq \frac{C}{n\bar{\theta}}
\begin{cases}
1, & j\in S_1,\cr
\theta_j/\bar{\theta}, & j\in S_2,
\end{cases} 
\]
following from (\ref{prop:eigv.tL}) and the estimate $\tilde{H}^{(i)}(i,i)\asymp n\bar{\theta}(\bar{\theta}\vee \theta_i)$. 
Applying Theorem \ref{thm:Bern_ineq}, one see that over the event $E$,
\begin{align*}
\big| W(i)(\tilde{H}^{(i)})^{-1/2}\tilde{\xi}^{(i)}_1 \big| \leq C\sqrt{\frac{\theta_i\log (n)}{n\bar{\theta}}} + C\, \frac{\log (n)} {n\bar{\theta}}
\end{align*}
Hence, over the event $E$,
\begin{align}\label{2112080100}
 \frac{W(i)  (\tilde{H}^{(i)})^{-\frac 12} \tilde{\xi}^{(i)}_1}{\sqrt{\tilde{H}^{(i)}(i,i)}} \leq C 
 \frac{\sqrt{\log (n)}}{n\bar{\theta}} \Big( 1\wedge \frac{\sqrt{\theta_i}}{\sqrt{\bar{\theta}}}+ \frac{\sqrt{\log (n)}}{\sqrt{n\bar{\theta}^2}}\Big)\leq C\, \widetilde{\kappa}_i
\end{align}
by using the estimate $\tilde{H}^{(i)}(i,i)\asymp n\bar{\theta}(\bar{\theta}\vee \theta_i)$ and the definition of $\widetilde{\kappa}_i$ in (\ref{Laplacian-entry-3-1}).

%

Next, we study the term $|W(i)  (\tilde{H}^{(i)})^{-\frac 12} (X-I_n)\tilde{\xi}^{(i)}_1|$. Over the event $E$, by (\ref{est:HHtd1}), we have
\begin{align}\label{est.X-I.entry}
 |X(j,j)-1|\leq C\frac{|H(j,j)-\tilde{H}^{(i)}(j,j)|}{\tilde{H}^{(i)}(j,j)}\leq C\, 
 \frac{A(i,j) +\theta_i\theta_j+ \theta_i \bar{\theta} + \log (n)/n}{\tilde{H}^{(i)}(j,j)}, \qquad  j\neq i;
 \end{align}
 It follows that

 \begin{align*}
 |W(i)  (\tilde{H}^{(i)})^{-\frac 12} (X-I_n)\tilde{\xi}^{(i)}_1|&= \Bigl|     \sum_{1\leq j\leq n: j\neq i} W(i,j) \frac{[X(j,j)-1]\tilde{\xi}^{(i)}_1(j)}{\sqrt{\tilde{H}^{(i)}(j,j)}} \Bigr|\cr
& \leq  \sum_{1\leq j\leq n: j\neq i} |W(i,j)| \frac{|X(j,j)-1||\tilde{\xi}^{(i)}_1(j)|}{\sqrt{\tilde{H}^{(i)}(j,j)}}\cr
&\leq C \sum_{1\leq j\leq n: j\neq i} \big(A(i,j) + \theta_i\theta_j + \theta_i\bar{\theta} + \frac{\log (n)}{n}\big)\frac{|\tilde{\xi}^{(i)}_1(j)|}{[\tilde{H}^{(i)}(j,j)]^{3/2}},
 \end{align*}
where in the last line we have used the fact that $|W(i,j)|\leq 1$. We apply Bernstein's inequality. The mean is bounded by (up to some constant $C$)
\[
\sum_{j\neq i }\frac{\theta_i\theta_j+ \theta_i\bar{\theta}+ \log(n) /n}{(n\bar{\theta} (\theta_j \vee \bar{\theta}))^{3/2}}\cdot \frac{\sqrt{\theta_j}}{\sqrt{n\bar{\theta}}} \leq \sum_{j}\frac{\theta_i+ log(n)/(n\bar{\theta})}{(n\bar{\theta})^2}\leq \frac{\theta_i}{n\bar{\theta}^2}\Big(1+ \frac{\log (n)}{n\bar{\theta}\theta_i}\Big). 
\]
The variance is bounded by (up to some constant $C$)
\[
\sum_{j\neq i}\frac{\theta_i\theta_j}{(n\bar{\theta} (\theta_j \vee \bar{\theta}))^3}\big(\tilde{\xi}^{(i)}_1(j)\big)^2 \leq \frac{1}{(n\bar{\theta}^2)^2} \sum_{j}\frac{\theta_i\theta_j}{n\theta_j\bar{\theta}}\big(\tilde{\xi}^{(i)}_1(j)\big)^2\leq \frac{1}{(n\bar{\theta}^2)^2}\cdot \frac{\theta_i}{n\bar{\theta}}. 
\]
Each individual term is bounded by (up to some constant $C$)
\[
\frac{|\tilde{\xi}_1^{(i)}(j)|}{[\tilde{H}^{(i)}(j,j)]^{3/2}}\leq \frac{C}{n\bar{\theta}}
\left. \begin{cases}
1/(n\theta_j\bar{\theta}), & j\in S_1\cr
\theta_j/(n\bar{\theta}^2), & j\in S_2.
\end{cases}\right\}\leq \frac{C}{n\bar{\theta}}\cdot \frac{1}{n\bar{\theta}^2}. 
\]
We then have 
\begin{align*}
|W(i)  (\tilde{H}^{(i)})^{-\frac 12} (X-I_n)\tilde{\xi}^{(i)}_1| &\leq C 
\frac{\theta_i}{n\bar{\theta}^2} \Big(1+ \frac{\log (n)}{n\bar{\theta}\theta_i}\Big)
\end{align*}
As a result,  over the event $E$,
\begin{align}\label{2112080101}
 \frac{W(i)  (\tilde{H}^{(i)})^{-\frac 12} (X-I_n)\tilde{\xi}^{(i)}_1}{\sqrt{\tilde{H}^{(i)}(i,i)}} \leq 
 \frac{C}{n\bar{\theta}^2} \sqrt{\frac{\theta_i}{n\bar{\theta}}} \Big(1\wedge \frac{\sqrt{\theta_i}}{\sqrt{\bar{\theta}}} + \sqrt{\frac{\log (n)}{n\bar{\theta}\theta_i}} \sqrt{\frac{\log (n)}{n\bar{\theta}^2}} \,  \Big)
 \leq C \frac{\widetilde{\kappa_i}}{\sqrt{n\bar{\theta}^2 \log (n)}}
\end{align}
%
We plug \eqref{2112080100} and \eqref{2112080101} into \eqref{2112080000}, and consider all $i$'s over the event $E$, then we conclude the proof of (\ref{Laplacian-entry-3-1}).
 
 \subsubsection{Proof of (\ref{Laplacian-entry-3-2})}
\label{app:subsub_techproof2}

Similarly to (\ref{2112080000}), we have 
\begin{align}\label{21120901}
|e_i'\Delta(\overline{\xi}^{(i)}_1 -\tilde{ \xi}^{(i)}_1)|  &\leq  \frac{|W(i) (\tilde{H}^{(i)})^{-1/2}(\overline{\xi}^{(i)}_1 -\tilde{ \xi}^{(i)}_1)|}{\sqrt{\tilde{H}^{(i)}(i,i)}} + \frac{C|W(i)(\tilde{H}^{(i)})^{-1/2}( X-I_n)(\overline{\xi}^{(i)}_1 -\tilde{ \xi}^{(i)}_1)|}{\sqrt{\tilde{H}^{(i)}(i,i)}} 
\end{align}
We first study the term $|W(i)(\tilde{H}^{(i)})^{-1/2}(\overline{\xi}^{(i)}_1 -\tilde{ \xi}^{(i)}_1)|$. Write
\[
W(i)(\tilde{H}^{(i)})^{-1/2}(\overline{\xi}^{(i)}_1 -\tilde{ \xi}^{(i)}_1)= \sum_{1\leq j\leq n: j\neq i} \frac{W(i,j)[\overline{\xi}^{(i)}_1(j)-\tilde{\xi}^{(i)}_1(j)]}{\sqrt{\tilde{H}^{(i)}(j,j)}}.
\]
We shall apply Bernstein's inequality  since $(\tilde{H}^{(i)})^{-1/2}(\overline{\xi}^{(i)}_1 -\tilde{ \xi}^{(i)}_1)$ is independent of $W(i)$. The variance is bounded by (up to some constant)
\[
\sum_{j\neq i }\frac{\theta_i\theta_j}{n\bar{\theta}(\theta_j\vee \bar{\theta})}[\overline{\xi}^{(i)}_1(j)-\tilde{\xi}^{(i)}_1(j)]^2 \leq \sum_{j}\frac{\theta_i\theta_j}{n\theta_j\bar{\theta}}[\overline{\xi}^{(i)}_1(j)-\tilde{\xi}^{(i)}_1(j)]^2 \leq \frac{4\theta_i}{n\bar{\theta}}
\]
Each individual term is bounded by $\| (\tilde{H}^{(i)})^{-1/2}(\overline{\xi}^{(i)}_1-\tilde{\xi}^{(i)}_1)\|_\infty$. As a result,
\begin{align*}
& \quad |W(i)(\tilde{H}^{(i)})^{-1/2}(\overline{\xi}^{(i)}_1-\tilde{\xi}^{(i)}_1)| \notag\\
& \leq C\frac{\sqrt{\theta_i\log(n)}}{\sqrt{n\bar{\theta}}} + C\log(n)\|(\tilde{H}^{(i)})^{-1/2}(\overline{\xi}^{(i)}_1-\tilde{\xi}^{(i)}_1)\|_\infty\cr
&\leq C\frac{\sqrt{\theta_i\log(n)}}{\sqrt{n\bar{\theta}}} + C\log(n)\|(\tilde{H}^{(i)})^{-1/2}(w\hat{\xi}_1-\tilde{\xi}^{(i)}_1)\|_\infty
+ C\log(n)\|(\tilde{H}^{(i)})^{-1/2}(\overline{\xi}^{(i)}_1-w\hat{\xi}_1)\|\cr
&\leq C\frac{\sqrt{\theta_i\log(n)}}{\sqrt{n\bar{\theta}}} + C \log(n)\|(\tilde{H}^{(i)})^{-1/2}(w\hat{\xi}_1-\tilde{\xi}^{(i)}_1)\|_\infty
+ \frac{C\log(n)}{\sqrt{n\bar{\theta}^2}}\|\overline{\xi}^{(i)}_1- w\hat{\xi}_1\|. 
\end{align*} 
Further, 
\begin{align} \label{conditioningterm-4}
\frac{|W(i)(\tilde{H}^{(i)})^{-1/2}(\overline{\xi}^{(i)}_1-\tilde{\xi}^{(i)}_1)|}{\sqrt{\tilde{H}^{(i)}(i,i)}}  & \leq C\kappa_i +C\kappa_i \frac{\sqrt{n\bar{\theta}\log(n)}}{\sqrt{\theta_i}} \|(\tilde{H}^{(i)})^{-1/2}(w\hat{\xi}_1-\tilde{\xi}^{(i)}_1)\|_\infty
+ \frac{C\log(n)}{n\bar{\theta}^2}\|\overline{\xi}^{(i)}_1-w\hat{\xi}_1\| \notag\\
&\leq C\widetilde{\kappa}_i +C\widetilde{\kappa}_i n\bar{\theta} \, \|(\tilde{H}^{(i)})^{-1/2}(w\hat{\xi}_1-\tilde{\xi}^{(i)}_1)\|_\infty
+ \frac{C\log(n)}{n\bar{\theta}^2}\|\overline{\xi}^{(i)}_1-w\hat{\xi}_1\| 
\end{align}
by the definition of $\widetilde{\kappa}_i= \frac{ 1}{n\bar{\theta} }\sqrt{\frac{\log (n)}{n\bar{\theta}^2}} \sqrt{n\bar{\theta}\theta_i \vee \log (n)}$.

We then study the term $|W(i)(\tilde{H}^{(i)})^{-1/2}( X-I_n)(\overline{\xi}^{(i)}_1 -\tilde{ \xi}^{(i)}_1)|$. On the event $E$, recall (\ref{est.X-I.entry}). It follows that 
\begin{align*}
& \quad |W(i)(\tilde{H}^{(i)})^{-1/2}( X-I_n)(\overline{\xi}^{(i)}_1 -\tilde{ \xi}^{(i)}_1)| \notag\\
&= \Bigl|     \sum_{1\leq j\leq n: j\neq i} W(i,j) \frac{[X(j,j)-1][\overline{\xi}^{(i)}_1(j)-\tilde{\xi}^{(i)}_1(j)]}{\sqrt{\tilde{H}^{(i)}(j,j)}} \Bigr|\cr
& \leq  \sum_{1\leq j\leq n: j\neq i} |W(i,j)| \frac{|X(j,j)-1||\overline{\xi}^{(i)}_1(j)-\tilde{\xi}^{(i)}_1(j)|}{\sqrt{\tilde{H}^{(i)}(j,j)}}\cr
&\leq C\sum_{1\leq j\leq n: j\neq i} \frac{A(i,j)+\theta_i\theta_j + \theta_i\bar{\theta} + \log (n)/n }{\tilde{H}^{(i)}(j,j)}\, \frac{|\overline{\xi}^{(i)}_1(j)-\tilde{\xi}^{(i)}_1(j)|}{\sqrt{\tilde{H}^{(i)}(j,j)}}.
\end{align*}
We now decompose the RHS above by $\mathcal I_1 + \mathcal I_2$, where 
\begin{align*}
&\mathcal I_1 : =  \sum_{1\leq j\leq n: j\neq i} \frac{A(i,j)+\theta_i\theta_j+ \theta_i\bar{\theta} + \frac{\log (n)}{n} }{\tilde{H}^{(i)}(j,j)}\, \frac{|\overline{\xi}^{(i)}_1(j)-w\hat{\xi}_1(j)|}{\sqrt{\tilde{H}^{(i)}(j,j)}}  ,\notag\\
&\mathcal I_2:=  \sum_{1\leq j\leq n: j\neq i} \frac{A(i,j)+\theta_i\theta_j + \theta_i\bar{\theta} + \frac{\log (n)}{n} }{\tilde{H}^{(i)}(j,j)}\, \frac{|w\hat{\xi}_1(j)-\tilde{\xi}^{(i)}_1(j)|}{\sqrt{\tilde{H}^{(i)}(j,j)}}\,. 
\end{align*}
We bound the sub-terms separately as below. By Cauchy-Schwarz inequality, 
\begin{align}\label{211125sub1}
\mathcal I_1 \leq 
\Big(\sum_{1\leq j\leq n: j\neq i} \frac{(A(i,j) +\theta_i\theta_j + \theta_i\bar{\theta} + {\log (n)}/{n})^2}{\tilde{H}^{(i)}(j,j)^3}\Big)^{\frac 12} \Vert \overline{\xi}^{(i)}_1-w\hat{\xi}_1\Vert \,. 
\end{align} 
And we crudely bound
\begin{align}\label{211125sub2}
\mathcal I_2 \leq \sum_{1\leq j\leq n: j\neq i} \frac{A(i,j) +\theta_i\theta_j  + \theta_i\bar{\theta}+ {\log (n)}/{n}}{\tilde{H}^{(i)}(j,j)}\, \Vert (\tilde{H}^{(i)})^{-\frac 12} (\tilde{\xi}^{(i)}_1-w\hat{\xi}_1)\Vert_\infty.
\end{align}
Applying Bernstein inequality (i.e.,  Theorem \ref{thm:Bern_ineq}), similarly to the analysis of the term $ |W(i)  (\tilde{H}^{(i)})^{-\frac 12} (X-I_n)\tilde{\xi}^{(i)}_1|$ in Section~\ref{app:subsub_techproof1}, we can have the estimates
\begin{align}\label{211125sub3}
\Big(\sum_{1\leq j\leq n: j\neq i} \frac{(A(i,j) +\theta_i\theta_j + \theta_i\bar{\theta} + {\log (n)}/{n})^2}{\tilde{H}^{(i)}(j,j)^3}\Big)^{\frac 12} & \leq C\Big(\sum_{1\leq j\leq n: j\neq i} \frac{A(i,j) +\theta_i\theta_j + \theta_i\bar{\theta} + {\log (n)}/{n}}{\tilde{H}^{(i)}(j,j)^3}\Big)^{\frac 12} \notag\\
&\leq C\Big( \frac{n\bar{\theta}\theta_i + \log (n)}{(n\bar{\theta}^2)^3}\Big)^{\frac 12} \notag\\
& \leq C\Big(\frac{1}{n\bar{\theta}^2} \sqrt{\frac{\theta_i}{\bar{\theta}}}+ \frac{1}{n\bar{\theta}^2} \sqrt{\frac{\log (n)}{ n\bar{\theta}^2 }}\, \Big)
\end{align}
and 
\begin{align}\label{211125sub4}
\sum_{1\leq j\leq n: j\neq i} \frac{A(i,j) +\theta_i\theta_j + \theta_i\bar{\theta}+ {\log (n)}/{n}}{\tilde{H}^{(i)}(j,j)} \leq C\,\Big(  \frac{\theta_i}{\bar{\theta}} +  \frac{\log (n)}{n\bar{\theta}^2} \Big) 
\end{align}
over the event $E$. We thus conclude that 
\begin{align*}
&\quad |W(i)(\tilde{H}^{(i)})^{-1/2}( X-I_n)(\overline{\xi}^{(i)}_1 -\tilde{ \xi}^{(i)}_1)|\notag\\
&\leq C\Big(\frac{\theta_i}{\bar{\theta}} + \frac{\log (n)}{n\bar{\theta}^2}\Big)  \Vert (\tilde{H}^{(i)})^{-\frac 12} (\tilde{\xi}^{(i)}_1-w\hat{\xi}_1)\Vert_\infty + C\Big(\frac{1}{n\bar{\theta}^2} \sqrt{\frac{\theta_i}{\bar{\theta}}}+ \frac{1}{n\bar{\theta}^2} \sqrt{\frac{\log (n)}{ n\bar{\theta}^2 }}\,\Big)\Vert \overline{\xi}^{(i)}_1-w\hat{\xi}_1\Vert  
\end{align*}
Further with $\tilde{H}^{(i)} (i,i) \asymp n\bar{\theta} (\theta_i \vee \bar{\theta})$, we have 
\begin{align} \label{conditioningterm-3}
 &\quad\frac{|W(i)(\tilde{H}^{(i)})^{-1/2}( X-I_n)(\overline{\xi}^{(i)}_1 -\tilde{ \xi}^{(i)}_1)| }{\sqrt{\tilde{H}^{(i)}(i,i)}} \notag\\ 
 &\leq C \widetilde{\kappa_i} n\bar{\theta} \|(\tilde{H}^{(i)})^{-1/2}(w\hat{\xi}_1-\tilde{\xi}^{(i)}_1)\|_\infty +C {(n\bar{\theta}^2)^{-\frac 32}}   \|\overline{\xi}^{(i)}_1-w\hat{\xi}_1\|
\end{align}
over the event $E$.
Now, plugging \eqref{conditioningterm-4} and \eqref{conditioningterm-3} into \eqref{21120901} and combining all $i$'s,  we thus finish the proof of  (\ref{Laplacian-entry-3-2}). 

\subsubsection{Proof of (\ref{Laplacian-entry-3-3})}
\label{app:subsub_techproof3}
 Note that $\overline{\xi}^{(i)}_1$ is the first eigenvector of $(\tilde{H}^{(i)})^{-1/2}\tilde{A}^{(i)}(\tilde{H}^{(i)})^{-1/2}$. The eigen-gap between $\tilde{\lambda}^{(i)}_1$ and $|\bar{\lambda}_2^{(i)}|$ is of order $K^{-1}\lambda_1(PG)\asymp 1$ in light of Weyl's inequality
\begin{align}\label{21121401}
\max_{i} | \overline{\lambda}^{(i)}_i - \tilde{\lambda}^{(i)}_i | \leq \Vert (\tilde{H}^{(i)})^{-\frac 12} (\tilde{A}^{(i)}- \Omega) (\tilde{H}^{(i)})^{-\frac 12}\Vert \leq C\sqrt{\frac{\log (n)}{n\bar{\theta}^2}}
\end{align}
and $K^{-1}\lambda_1(PG) \gg \sqrt{ {\log (n)}/{n\bar{\theta}^2}} $. Similarly, the eigengap between $\hat{\lambda}_1$ and $|\bar{\lambda}_2^{(i)}|$ is of order $K^{-1}\lambda_1(PG)$.  We claim that ${\rm sgn} (\hat{\xi}_1'\overline{\xi}^{(i)}_1)  = {\rm sgn} (\hat{\xi}_1'\tilde{\xi}^{(i)}_1) \equiv w $. Notice that  $|(\hat{\xi}_1')\overline{\xi}^{(i)}_1 -  (\hat{\xi}_1')\tilde{\xi}^{(i)}_1| \leq \Vert \overline{\xi}^{(i)}_1 - \tilde{\xi}^{(i)}_1 \Vert  =o(1)$. This, together with the fact  that $|(\hat{\xi}_1')\tilde{\xi}^{(i)}_1 |>c$ for some constant $c\in (0,1)$, implies that ${\rm sgn} (\hat{\xi}_1'\overline{\xi}^{(i)}_1 ) = {\rm sgn} (\hat{\xi}_1'\tilde{\xi}^{(i)}_1)$.
%
%
Moreover, 
\[
\|\overline{\xi}^{(i)}_1- w\hat{\xi}_1\|\leq { K} \lambda_1^{-1} (PG)\| \underbrace{\big((\tilde{H}^{(i)})^{-1/2}\tilde{A}^{(i)} (\tilde{H}^{(i)})^{-1/2}- H^{-1/2}AH^{-1/2})}_{=: \tilde{\Delta}(i)\equiv \tilde{\Delta}} \hat{\xi}_1  \|
\]
Recall $X = (\tilde{H}^{(i)})^{1/2}H^{-1/2}$. It is seen that
\begin{align*} 
\tilde{\Delta} & =  (\tilde{H}^{(i)})^{-\frac 12}(\tilde{A}^{(i)} -A)(\tilde{H}^{(i)})^{-\frac12} + \big( (\tilde{H}^{(i)})^{-\frac12}A(\tilde{H}^{(i)})^{-\frac12}- H^{-\frac12}AH^{-\frac12}\big)\cr
& =  - (\tilde{H}^{(i)})^{-\frac12}(e_iW(i) + W(i)' e_i')(\tilde{H}^{(i)})^{-\frac12} + 
(\tilde{H}^{(i)})^{-\frac12}A( (\tilde{H}^{(i)})^{-\frac12}-H^{-\frac12}) +( (\tilde{H}^{(i)})^{-\frac12} - H^{-1/2})AH^{-\frac12}\cr
& = - (\tilde{H}^{(i)})^{-\frac12}(e_iW(i)+ W(i)'e_i')(\tilde{H} ^{(i)})^{-\frac12} + (\tilde{H}^{(i)})^{-\frac12}A(\tilde{H}^{(i)})^{-\frac12}(I_n-X) + (X^{-1}-I_n)H^{-\frac12}AH^{-\frac12}. 
\end{align*}
By definition, $H^{-1/2}AH^{-1/2}\hat{\xi}_1=\hat{\lambda}_1\hat{\xi}_1$. It follows that
\[
\tilde{\Delta}\hat{\xi}_1 = - (\tilde{H}^{(i)})^{-\frac12}(e_iW(i)+W(i)'e_i')(\tilde{H}^{(i)})^{-\frac12} \hat{\xi}_1
+ (\tilde{H}^{(i)})^{-\frac12}A(\tilde{H}^{(i)})^{-\frac12}(I_n-X) \hat{\xi}_1 + \hat{\lambda}_1 (X^{-1}-I_n)\hat{\xi}_1. 
\]
As a result,{\small
\begin{align} \label{leave-effect-1}
\|\overline{\xi}^{(i)}_1-w\hat{\xi}_1\|& \leq { K } \lambda_1^{-1} (PG)\frac{|W(i)(\tilde{H}^{(i)})^{-\frac12}\hat{\xi}_1|}{\sqrt{\tilde{H}^{(i)}(i,i)}} + { K } \lambda_1^{-1} (PG)\frac{\|(\tilde{H}^{(i)})^{-\frac12}W(i)'\|}{\sqrt{\tilde{H}^{(i)}(i,i)}}\cdot |\hat{\xi}_1(i)| + C\|(I_n-X)\hat{\xi}_1\|\cr
&\leq  {K } \lambda_1^{-1} (PG)\frac{|W(i)(\tilde{H}^{(i)})^{-\frac12}\hat{\xi}_1|}{\sqrt{\tilde{H}^{(i)}(i,i)}} +\frac{C {K }\lambda_1^{-1} (PG) }{\sqrt{n\bar{\theta}^2}}|\hat{\xi}_1(i)| + C\|(I_n-X)\hat{\xi}_1\|,\cr
&\leq  C{ K } \lambda_1^{-1} (PG) \bigg(\frac{|W(i)(\tilde{H}^{(i)})^{-\frac12}\hat{\xi}_1|}{\sqrt{\tilde{H}^{(i)}(i,i)}} +\frac{1 }{\sqrt{n\bar{\theta}^2}}|\tilde{\xi}^{(i)}_1(i)| +\frac{1 }{\sqrt{n\bar{\theta}^2}}|w\hat{\xi}_1(i)-\tilde{\xi}^{(i)}_1(i)| \bigg) + C\|(I_n-X)\hat{\xi}_1\|,\cr
&\leq C{ K }\lambda_1^{-1} (PG) \bigg( \frac{|W(i)(\tilde{H}^{(i)})^{-\frac12}\hat{\xi}_1|}{\sqrt{\tilde{H}^{(i)}(i,i)}} +\frac{\kappa_i}{\sqrt{\log(n)}} +\frac{1}{\sqrt{n\bar{\theta}^2}}|w\hat{\xi}_1(i)- \tilde{\xi}^{(i)}_1(i)| \bigg) + C\|(I_n-X)\hat{\xi}_1\|, 
\end{align}
}
where in the first line we have used 
$\|(\tilde{H}^{(i)})^{-1/2}A(\tilde{H}^{(i)})^{-1/2}\|\leq CK^{- 1}\lambda_1(PG)$ and $\|X-I_n\|\leq 1/2$, in the second line we have used the estimate 
\begin{align}\label{est:WtH}
{(\tilde{H}^{(i)}(i,i))^{-\frac12}}\| (\tilde{H}^{(i)})^{-\frac12}W(i)'\|&\leq  \sqrt{W(i)(\tilde{H}^{(i)})^{-1} W(i)'} \Big/ \sqrt{\tilde{H}^{(i)}(i,i)}\notag\\
&\leq \frac{\sqrt{(\theta_i/\bar{\theta})\vee (\log (n)/n\bar{\theta}^2)}}{\sqrt{n\bar{\theta}(\bar{\theta} \vee \theta_i)}}
\leq C(n\bar{\theta}^2)^{-\frac12}
\end{align}
by simply using Bernstein inequality to $W(i)(\tilde{H}^{(i)})^{-1} W(i)'$,
 and in the last line we have used \eqref{order-xi}. 

We consider the first term in \eqref{leave-effect-1}. Note that 
\begin{align*}
 \frac{|W(i)(\tilde{H}^{(i)})^{-\frac12}\hat{\xi}_1|}{\sqrt{\tilde{H}^{(i)}(i,i)}} & \leq  \frac{|W(i)(\tilde{H}^{(i)})^{-\frac12}\tilde{\xi}^{(i)}_1|}{\sqrt{\tilde{H}^{(i)}(i,i)}} +  \frac{|W(i)(\tilde{H}^{(i)})^{-\frac12}(\overline{\xi}^{(i)}_1 - \tilde{\xi}^{(i)}_1)|}{\sqrt{\tilde{H}^{(i)}(i,i)}} +  \frac{\|W(i)(\tilde{H}^{(i)})^{-\frac12}\|\|w \hat{\xi}_1-\overline{\xi}^{(i)}_1\|}{\sqrt{\tilde{H}^{(i)}(i,i)}}\cr
 &\leq   \frac{|W(i)(\tilde{H}^{(i)})^{-\frac12}\tilde{\xi}^{(i)}_1|}{\sqrt{\tilde{H}^{(i)}(i,i)}} +  \frac{|W(i) (\tilde{H}^{(i)})^{-\frac12}(\overline{\xi}^{(i)}_1 - \tilde{\xi}^{(i)}_1)|}{\sqrt{\tilde{H}^{(i)}(i,i)}} + \frac{C}{\sqrt{n\bar{\theta}^2}} \| w\hat{\xi}_1-\overline{\xi}^{(i)}_1\|,
\end{align*}
In \eqref{2112080100} and \eqref{conditioningterm-4}, we have seen that the first two terms are bounded by
\begin{align*}
\widetilde{\kappa}_i \bigl(1+n\bar{\theta}\|(\tilde{H}^{(i)})^{-\frac12}( w \hat{\xi}_1-\tilde{\xi}^{(i)}_1)\|_\infty\bigr) + \frac{\log(n)}{n\bar{\theta}^2}\|\overline{\xi}^{(i)}_1- w \hat{\xi}_1\| 
\end{align*}
up to some constant.
Combining the above gives
\begin{align} \label{leave-effect-7}
\frac{|W(i)(\tilde{H}^{(i)})^{-\frac12}\hat{\xi}_1|}{\sqrt{\tilde{H}^{(i)}(i,i)}} 
& \leq C
\widetilde{\kappa}_i \bigl(1+n\bar{\theta}\|(\tilde{H}^{(i)})^{-\frac12}( w \hat{\xi}_1-\tilde{\xi}^{(i)}_1)\|_\infty\bigr) + \Big(\frac{\log(n)}{n\bar{\theta}^2} +\frac{1}{\sqrt{n\bar{\theta}^2}}\Big)\|\overline{\xi}^{(i)}_1- w \hat{\xi}_1\| .
\end{align}
We plug it into \eqref{leave-effect-1} and move all terms of $\|\overline{\xi}^{(i)}_1-w\hat{\xi}_1\|$ to the left hand side. It follows that
\beq \label{leave-effect-8}
\|\overline{\xi}^{(i)}_1-w\hat{\xi}_1\|  
\leq CK\lambda_1^{-1} (PG)\bigg(
\widetilde{\kappa}_i \bigl(1+n\bar{\theta}\|(\tilde{H}^{(i)})^{-\frac12}( w \hat{\xi}_1-\tilde{\xi}^{(i)}_1)\|_\infty\bigr) +  \frac{ |w\hat{\xi}_1(i)-\tilde{\xi}^{(i)}_1(i)|}{\sqrt{n\bar{\theta}^2}} \bigg)+ C\|(I_n-X)\hat{\xi}_1\| 
\eeq
Below, we bound $\|(I_n-X)\hat{\xi}_1\|$. Note that
\beq \label{leave-effect-2}
\|(I_n-X)\hat{\xi}_1\|\leq \|(I_n-X)\tilde{\xi}^{(i)}_1\| + \|(I_n-X)(w\hat{\xi}_1 - \tilde{\xi}^{(i)}_1)\|, 
\eeq
where
\begin{align*}
\|(I_n-X)\tilde{\xi}^{(i)}_1\|^2 & \leq  \sum_{j=1}^n |X(j,j)-1|^2(\tilde{\xi}^{(i)}_1(j))^2=:{ (J_1)}\cr
\|(I_n-X)(w\hat{\xi}_1-\tilde{\xi}^{(i)}_1)\|^2 & \leq \|(\tilde{H}^{(i)})^{-\frac12}(w\hat{\xi}_1-\tilde{\xi}^{(i)}_1)\|^2_\infty\cdot \sum_{j=1}^n |X(j,j)-1|^2\tilde{H}^{(i)}(j,j)\notag\\
&=: \|(\tilde{H}^{(i)})^{-\frac12}( w\hat{\xi}_1-\tilde{\xi}^{(i)}_1)\|^2_\infty\cdot {  (J_2)}. 
\end{align*}
Recall the bound of $|X(j,j)-1|$ in (\ref{est.X-I.entry}).
It follows that over the event $E$,
\begin{align*}
(J_1) & \leq C\sum_{j=1}^n \frac{|A(i,j) +\theta_i\theta_j+ \theta_i\bar{\theta} + \frac{\log (n)}{n}|^2(\tilde{\xi}^{(i)}_1(j))^2}{[\tilde{H}^{(i)}(j,j)]^2}\notag\\
&\leq C\sum_{j=1}^n\big( A(i,j) +\theta_i\theta_j + \theta_i\bar{\theta} +\frac{\log (n)}{n} \big)\frac{ (\tilde{\xi}^{(i)}_1(j))^2}{[\tilde{H}^{(i)}(j,j)]^2},\cr
(J_2) & \leq C\sum_{j=1}^n \frac{|A(i,j)+\theta_i\theta_j +\theta_i\bar{\theta} + \frac{\log (n)}{n} |^2}{\tilde{H}^{(i)}(j,j)}\notag\\
&\leq C\sum_{j=1}^n \big( A(i,j) +\theta_i\theta_j+ \theta_i\bar{\theta} +\frac{\log (n)}{n} \big)\frac{1}{\tilde{H}^{(i)}(j,j)}, 
\end{align*}
where we again use the fact that $A(i,j)\in \{0,1\}$. We shall bound the two terms similarly, using the Bernstein's inequality (Theorem \ref{thm:Bern_ineq}). For $(J_1)$,
\begin{itemize}
\item The mean is bounded by (up to some constant)
$$ \sum_{j=1}^n \Big(\theta_i\theta_j + \theta_i\bar{\theta}  + \frac{\log (n)}{n}\Big)\frac{1}{(n\bar{\theta})^3 (\bar{\theta}\vee \theta_j)}\leq \frac{1}{n\bar{\theta}^2 } \cdot \frac{n\bar{\theta}\theta_i+ \log (n)}{(n\bar{\theta})^2};$$
\item The variance is bounded by (up to some constant)
$$\sum_{j=1}^n \frac{\theta_i\theta_j}{(n\bar{\theta})^6 (\bar{\theta}\vee \theta_j)^2} \leq  \frac{1}{(n\bar{\theta}^2)^2}\cdot \frac{\theta_i}{(n\bar{\theta})^3};$$
\item Each individual term is bounded by (up to some constant) $\frac{|(\tilde{\xi}^{(i)}_1(j))^2|}{[\tilde{H}^{(i)}(j,j)]^2}\leq \frac{1}{n\bar{\theta}^2}\cdot \frac{1}{(n\bar{\theta})^2}$.
\end{itemize}
We then have
\begin{align} \label{leave-effect-3}
(J_1) & \leq \frac{C}{n\bar{\theta}^2}\frac{\theta_i}{n\bar{\theta}} \biggl( 1+ \frac{\log (n)}{n\bar{\theta}\theta_i}\biggr) .
\end{align}
For $(J_2)$, 
\begin{itemize}
\item The mean is bounded by (up to some constant)
$$ \sum_{j=1}^n \Big(\theta_i\theta_j + \theta_i\bar{\theta}  + \frac{\log (n)}{n}\Big)\frac{1}{n\bar{\theta}(\bar{\theta}\vee \theta_j)}\leq \frac{\theta_i}{\bar{\theta}} + \frac{\log (n)}{n\bar{\theta}^2};$$
\item The variance is bounded by (up to some constant)
$$ \sum_{j=1}^n \frac{\theta_i\theta_j }{(n\bar{\theta})^2(\bar{\theta}\vee \theta_j)^2}\leq \frac{1}{n\bar{\theta}^2}\cdot \frac{\theta_i}{\bar{\theta}};$$
\item Each individual term is bounded by (up to some constant) $ 
\frac{1}{\tilde{H}(j,j)}\leq \frac{C}{n\bar{\theta}^2}$. 
\end{itemize}
It follows that
\beq \label{leave-effect-4}
(J_2)\leq \frac{C\theta_i}{\bar{\theta}} + \frac{C\log (n)}{n\bar{\theta}^2}
\eeq
Plugging \eqref{leave-effect-3}-\eqref{leave-effect-4} into \eqref{leave-effect-2}, we find out that
\begin{align} \label{leave-effect-5}
\|(I_n-X)\hat{\xi}_1\|\leq C
\frac{\widetilde{\kappa}_i}{\sqrt{\log(n)}}\bigl(1+n\bar{\theta}\|(\tilde{H}^{(i)})^{-\frac12}(w\hat{\xi}_1-\tilde{\xi}^{(i)}_1)\|_\infty\bigr) 
\end{align}
We plug \eqref{leave-effect-5} into \eqref{leave-effect-8}, together with the assumption $K^{-1}\lambda_1(PG) \asymp 1$,  to get
\begin{align} \label{leave-effect-final}
\|\overline{\xi}^{(i)}_1- w\hat{\xi}_1\| &\leq  C
\widetilde{\kappa}_i \bigl(1+n\bar{\theta}\|(\tilde{H}^{(i)})^{-1/2}(w\hat{\xi}_1-\tilde{\xi}^{(i)}_1)\|_\infty\bigr) + \frac{ 1}{\sqrt{n\bar{\theta}^2}}|w\hat{\xi}_1(i)-\tilde{\xi}^{(i)}_1(i)|
\end{align}
over the event $E$, which proved (\ref{Laplacian-entry-3-3}) by considering all $i$'s altogether.

 \subsection{Proof of the second claim in Theorem~\ref{thm:eigenvector}} \label{sub:2-Keig}
 In this section, we show the proof of (\ref{result-2}). Similarly to the proof of (\ref{result-1}), we streamline the proof into the following lemmas. In addition to the notations in the end of Section \ref{sec:Intro}, below we will use $\Vert \cdot\Vert_{2\to \infty}$ to denote the matrix $2\to \infty$ norm, i.e., the maximum row-wise $\ell^2$-norm of a matrix. Specifically, for any matrix $A \in \mathbb{R}^{n\times m}$ and vector  $x\in \mathbb{R}^{m}$,  $\Vert A\Vert_{2\to \infty}:= \max_{\Vert x\Vert=1} \Vert Ax\Vert_{\infty} = \max_{i} \Vert A(i)\Vert$.
   
\begin{lem}\label{lem:upbd_key2K}
Suppose the assumptions in Theorem \ref{thm:eigenvector} hold. Recall $\kappa_t:=\sqrt{\frac{\log (n)}{n\bar{\theta}^2 }} \sqrt{\frac{\theta_t}{n\bar{\theta}}}$ for $1\leq t \leq n$.  With probability $1-o(n^{-3})$, simultaneously for  $1\leq i,t \leq n$,  
\begin{align}\label{eq:t1_2K}
\Vert  \tilde{\Xi}^{(i)}_1 (t) O_2^{(i)} - \Xi_1 (t)\Vert \leq C  K^{\frac 32} \beta_n^{-1}\kappa_t \Big(1\wedge \sqrt{\frac{\theta_t}{\bar{\theta}}}\, \Big), 
\end{align}
 for some orthogonal matrices $ O_2^{(i)}\in \mathbb{R}^{K-1, K-1}$.
\end{lem}

\begin{lem}\label{lem:upbd_key11_2K}
 Under the assumptions in Theorem \ref{thm:eigenvector}. 
 With probability $1-o(n^{-3})$, simultaneously for  $1\leq i \leq n$,  
\begin{align}
&\Vert  \hat{\Xi}_1(i) - \tilde{\Xi}^{(i)}_1(i)O_3^{(i)})\Vert \leq CK^{\frac 32} \beta_n^{-1}\kappa_i + C K \beta_n^{-1} \Vert e_i'\Delta \hat{\Xi}_1\Vert , \label{Laplacian-entry-1_2K}\\
&\Vert e_i' \Delta \hat{\Xi}_1\Vert  \leq  \Vert e_i' \Delta \tilde{\Xi}^{(i)}_1 \Vert + \Vert e_i' \Delta (\overline{\Xi}^{(i)}_1-\tilde{\Xi}^{(i)}_1 O_4^{(i)})\Vert + \frac{C}{\sqrt{n\bar{\theta}^2 }}  \Vert \hat{\Xi}_1-\overline{\Xi}^{(i)}_1O_5^{(i)}\Vert  , \label{Laplacian-entry-2_2K}
\end{align}
for some orthogonal matrices $ O_4^{(i)},  O_5^{(i)}\in \mathbb{R}^{K-1, K-1}$ and $O_3^{(i)}:= O_4^{(i)}O_5^{(i)}$, where  $\Delta\equiv \Delta{(i)}:= (\tilde{H}^{(i)})^{-1/2} W {H}^{-1/2}  $ for short.
\end{lem}

\begin{lem}\label{lem:techB1_2K}
Under the assumptions in Lemma \ref{lem:upbd_key11_2K}. Recall the notation of $\widetilde{\kappa}_i$ in (\ref {Laplacian-entry-3-1}) for $1\leq i\leq n$.  With probability $1-o(n^{-3})$, simultaneously for  $1\leq i \leq n$,  
\begin{align}  
\Vert e_i' \Delta \tilde{\Xi}^{(i)}_1 \Vert &\leq CK^{\frac 12} \widetilde{\kappa}_i\label{Laplacian-entry-3-1_2K}\\
\Vert e_i' \Delta (\overline{\Xi}^{(i)}_1-\tilde{\Xi}^{(i)}_1 O_4^{(i)})\Vert &\leq  CK^{\frac 12} \Big( \widetilde{\kappa}_i\big(1 +  n\bar{\theta} \Vert (\tilde{H}^{(i)})^{-\frac 12} (   \hat{\Xi}_1 -\tilde{\Xi}^{(i)}_1O_3^{(i)})\Vert_{2\to \infty }  \big)  + \frac{ \log (n)}{n\bar{\theta}^2}\Vert  \hat{\Xi}_1- \overline{\Xi}^{(i)}_1O_5^{(i)} \Vert\Big),\label{Laplacian-entry-3-2_2K}\\
\Vert  \hat{\Xi}_1-\overline{\Xi}^{(i)}_1O_5^{(i)}\Vert &\leq C  K^{\frac 32}\beta_n^{-1} \,\widetilde{\kappa_i} \Big(1+  n\bar{\theta}  \Vert \tilde{H}^{-\frac 12} (   \hat{\Xi}_1 -\tilde{\Xi}^{(i)}_1O_3^{(i)})\Vert_{2\to \infty } \Big) + \frac{C{K\beta_n^{-1}}  }{\sqrt{n\bar{\theta}^2} }\Vert \hat{\Xi}_1(i) - \tilde{\Xi}^{(i)}_1(i) O_3^{(i)}\Vert. \label{Laplacian-entry-3-3_2K}
\end{align}
\end{lem}
 
In the sequel, we will prove the second claim in Theorem~\ref{thm:eigenvector} (i.e., (\ref{result-2})) based on the above lemmas. The proofs of the lemmas are postponed to the next three subsections.
\begin{proof} [Proof of  (\ref{result-2}) ]
Plugging Lemma \ref{lem:techB1_2K} into (\ref{Laplacian-entry-2_2K}), we first have with probability $1- o(n^{-3})$, simultaneously for all $1\leq i \leq n$,
\begin{align*}
\Vert e_i' \Delta \hat{\Xi}_1\Vert  \leq CK^{\frac 12} \Big( \tilde{\kappa}_i\big(1 +  n\bar{\theta}  \Vert (\tilde{H}^{(i)})^{-\frac 12} (   \hat{\Xi}_1 -\tilde{\Xi}^{(i)}_1O_3^{(i)})\Vert_{2\to \infty }  \big)  + \frac{C{K\beta_n^{-1}} \sqrt{\log (n)} }{n\bar{\theta}^2 }\Vert \hat{\Xi}_1(i) - \tilde{\Xi}^{(i)}_1(i) O_3^{(i)}\Vert\Big)
\end{align*} 
which, further substituted to (\ref{Laplacian-entry-1_2K}), implies that 
\begin{align*}
\Vert  \hat{\Xi}_1(i) - \tilde{\Xi}^{(i)}_1(i)O_3^{(i)})\Vert \leq &C K^{\frac32}\beta_n^{-1} \widetilde{\kappa}_i\big(1 +  n\bar{\theta}  \Vert (\tilde{H}^{(i)})^{-\frac 12} (   \hat{\Xi}_1 -\tilde{\Xi}^{(i)}_1O_3^{(i)})\Vert_{2\to \infty }  \big)  \notag\\
&+ \frac{CK^3\sqrt{\log (n)}}{\beta_n^2 n\bar{\theta}^2}\Vert  \hat{\Xi}_1(i) - \tilde{\Xi}^{(i)}_1(i)O_3^{(i)})\Vert.
\end{align*}
Since $K^3\beta_n^{-2}\log (n) /n\bar{\theta}^2=o(1)$, we then arrive at 
\begin{align*}
\Vert  \hat{\Xi}_1(i) - \tilde{\Xi}^{(i)}_1(i)O_3^{(i)})\Vert \leq C K^{\frac32}\beta_n^{-1} \widetilde{\kappa}_i\big(1 +  n\bar{\theta}  \Vert (\tilde{H}^{(i)})^{-\frac 12} (   \hat{\Xi}_1 -\tilde{\Xi}^{(i)}_1O_3^{(i)})\Vert_{2\to \infty }  \big).
\end{align*}
Set $\tilde{O}_1\equiv \tilde{O}_1^{(i)}:=(O_3^{(i)}) ' O_2^{(i)}$. Using Lemma \ref{lem:upbd_key2K} and let $t=i$, we will see that 
\begin{align*}
\Vert  \hat{\Xi}_1(i) - {\Xi}_1(i)\tilde{O}_1'\Vert &\leq \Vert  \hat{\Xi}_1(i) - \tilde{\Xi}^{(i)}_1(i)O_3^{(i)})\Vert + \Vert  \tilde{\Xi}^{(i)}_1 (i) O_2^{(i)} - \Xi_1 (i)\Vert \notag\\
&\leq C K^{\frac32}\beta_n^{-1} \widetilde{\kappa}_i\big(1 +  n\bar{\theta}  \Vert (\tilde{H}^{(i)})^{-\frac 12} (   \hat{\Xi}_1 -\tilde{\Xi}^{(i)}_1O_3^{(i)})\Vert_{2\to \infty }  \big)
\end{align*}
over the event $E$. Suppose that $\hat{\Xi}_1' \Xi_1$ has the singular value decomposition (SVD) $\hat{\Xi}_1' \Xi_1 = U'\cos \Theta V$, we  define $O_1={\rm  sgn} ( \hat{\Xi}_1' \Xi_1 ):= U'V$. Using sine-theta theorem (i.e., (\ref{sine-theta:perturb1}) and (\ref{sine-theta:perturb2})), we can derive 
\begin{align*}
\Vert O_1- \tilde{O}_1\Vert  &\leq  \Vert \hat{\Xi}_1' \Xi_1  - O_1\Vert 
+  \Vert \hat{\Xi}_1' \Xi_1  - \tilde{O}_1\Vert \notag\\
&\leq \Vert \hat{\Xi}_1' \Xi_1  - O_1\Vert 
+  \Vert \hat{\Xi}_1  - \tilde{\Xi}^{(i)}_1O_3^{(i)}\Vert + \Vert (\tilde{\Xi}^{(i)}_1)' \Xi_1- O_2^{(i)}\Vert\notag\\
&\leq CK \beta_n^{-1} \Big( \Vert L_0- L\Vert  + \Vert \tilde{L}^{(i)} - L \Vert + \Vert \tilde{L}^{(i)} - L_0 \Vert \Big)\notag\\
&\leq CK\beta_n^{-1} \sqrt{\frac{\log (n)}{n\bar{\theta}^2}},
\end{align*}
by which, we will obtain 
\begin{align}\label{21121602}
\Vert  \hat{\Xi}_1(i) - {\Xi}_1(i){O}_1'\Vert &\leq \Vert  \hat{\Xi}_1(i) - {\Xi}_1(i)\tilde{O}_1'\Vert   + \Vert \Xi_1(i)\Vert\cdot \Vert O_1- \tilde{O}_1\Vert \notag\\
&\leq C K^{\frac32}\beta_n^{-1} \widetilde{\kappa}_i\big(1 +  n\bar{\theta}  \Vert (\tilde{H}^{(i)})^{-\frac 12} (   \hat{\Xi}_1 -\tilde{\Xi}^{(i)}_1O_3^{(i)})\Vert_{2\to \infty }  \big)
\end{align}
and 
\begin{align}\label{21121601}
\Vert H_0^{-\frac 12} {\Xi}_1({O}_1- \tilde{O}_1)') \Vert_{2\to \infty }\leq \Vert H_0^{-\frac 12} {\Xi}_1 \Vert_{2\to \infty }\cdot \Vert {O}_1- \tilde{O}_1\Vert \leq CK^{\frac 32}\beta_n^{-1} \frac{1}{n\bar{\theta}}\sqrt{\frac{\log (n)}{n\bar{\theta}^2}}
\end{align}
Here to obtain the above two inequalities, we used  the second estimate of (\ref{order-xi}).

Applying Lemma \ref{lem:upbd_key2K} again together
 with (\ref{est.tildeYY}), (\ref{21121601}), it is easy to deduce that 
\begin{align*}
&\Vert (\tilde{H}^{(i)})^{-\frac 12} (   \hat{\Xi}_1 -\tilde{\Xi}^{(i)}_1O_3^{(i)})\Vert_{2\to \infty } \notag\\
 &\leq C\Vert H_0^{-\frac 12}(\hat{\Xi}_1 - {\Xi}_1\tilde{O}_1') \Vert_{2\to \infty }  + C\Vert  H_0^{-\frac 12}({\Xi}_1(O_2^{(i)})'- \tilde{\Xi}^{(i)}_1)O_3^{(i)} \Vert_{2\to \infty }\notag\\
&\leq C\Vert H_0^{-\frac 12}(\hat{\Xi}_1- {\Xi}_1{O}_1') \Vert_{2\to \infty } +  C\Vert H_0^{-\frac 12} {\Xi}_1({O}_1- \tilde{O}_1)') \Vert_{2\to \infty } +C \Vert  H_0^{-\frac 12}({\Xi}_1(O_2^{(i)})'- \tilde{\Xi}^{(i)}_1)\Vert_{2\to \infty }\notag\\
&\leq C\Vert H_0^{-\frac 12}(\hat{\Xi}_1- {\Xi}_1{O}_1') \Vert_{2\to \infty }  + \frac{CK^{\frac 32} \beta_n^{-1}}{n\bar{\theta}}\sqrt{\frac{\log (n)}{n\bar{\theta}^2}}
\end{align*}
Thereby, according to the condition $K^3\beta_n^{-2}\log (n) /n\bar{\theta}^2=o(1)$, (\ref{21121602}) can be further improved to 
\begin{align}\label{21121502}
\Vert  \hat{\Xi}_1(i) - {\Xi}_1(i){O} _1'\Vert &\leq  C K^{\frac32}\beta_n^{-1} \widetilde{\kappa}_i \big(1 +  n\bar{\theta} \Vert H_0^{-\frac 12}(\hat{\Xi}_1- {\Xi}_1{O}_1') \Vert_{2\to \infty }  \big).
\end{align}
Next, we multiply both sides of the above inequality by $H_0^{-\frac 12}(i,i)$ and take the maximum over $i$ since $\hat{\Xi}_1- \Xi O_1'$ is independent of $i$, it yields that, 
\begin{align}\label{21121501}
 \Vert H_0^{-\frac 12}(\hat{\Xi}_1- {\Xi}_1O_1') \Vert_{ 2\to \infty }&= \max_{i} \Vert e_i' H_0^{-\frac12} (\hat{\Xi}_1O_1 - {\Xi}_1)\Vert \notag\\
&\leq CK^{\frac 32} \beta_n^{-1} (n\bar{\theta})^{-1}\sqrt{\frac{\log (n) }{n\bar{\theta}^2}} +  o(\Vert H_0^{-\frac 12}(\hat{\Xi}_1- {\Xi}_1O_1' \Vert_{2\to \infty})
\end{align}
 Rearranging both sides of (\ref{21121501}), we can conclude that 
\begin{align*}
 \Vert H_0^{-\frac 12}(\hat{\Xi}_1- {\Xi}_1O_1') \Vert_{2\to \infty}\leq CK^{\frac 32} \beta_n^{-1} (n\bar{\theta})^{-1}\sqrt{\frac{\log (n) }{n\bar{\theta}^2}}. 
\end{align*}
which, further substituted into (\ref{21121502}), yields   (\ref{result-2}) due to the condition $K^3\beta_n^{-2}\log (n) /n\bar{\theta}^2=o(1)$.
\end{proof}

\subsection{Proof of Lemma \ref{lem:upbd_key2K}} \label{subsec:pr_Lem_upbd_key2K}

We state the proof of Lemma \ref{lem:upbd_key2K} which is quite similar to Lemma  \ref{lem:upbd_key1} with additional attention to the non-commutative multiplication of matrices.
Fix the index $i$, we start with the perturbation from $L_0$ to $\tilde{L}^{(i)}$.
\begin{align*}
\tilde{\Xi}^{(i)}_1 \tilde{\Lambda}^{(i)}_1 = \tilde{L}^{(i)} \tilde{\Xi}^{(i)}_1 &= \big( H_0^{\frac 12}(\tilde{H}^{(i)})^{-\frac 12} \big) L_0 \big( H_0^{\frac 12}(\tilde{H}^{(i)})^{-\frac 12} \big) \tilde{\Xi}^{(i)}_1= \tilde{Y} \lambda_1 \xi_1\xi_1' \tilde{Y} \tilde{\Xi}^{(i)}_1 + \tilde{Y} \Xi_1\Lambda_1 \Xi_1' \tilde{Y} \tilde{\Xi}^{(i)}_1
\end{align*}
by recalling the definition $\tilde{Y}=  H_0^{\frac 12}(\tilde{H}^{(i)})^{-\frac 12}$. Then, for each $1\leq t\leq n $
\begin{align} \label{2021071401}
\tilde{\Xi}^{(i)}_1(t) = \tilde{Y}(t,t) \lambda_1 \xi_1(t) \xi_1' \tilde{Y}\tilde{\Xi}^{(i)}_1 (\tilde{\Lambda}^{(i)}_1)^{-1}+ \tilde{Y} (t,t)\Xi_1(t)\Lambda_1 \Xi_1' \tilde{Y} \tilde{\Xi}^{(i)}_1 (\tilde{\Lambda}^{(i)}_1)^{-1}.
\end{align}
Recall (\ref{prop:eig.tL}), over the event $E$, we first crudely bound $ \Vert \lambda_1 (\tilde{\Lambda}^{(i)}_1)^{-1} \Vert $ by $\beta_n^{-1}\lambda_1(PG)$.
 Then, using the estimate (\ref{est:tY-In}), we can crudely bound the first term on the RHS of (\ref{2021071401}) by
\begin{align} \label{21071402}
 \Vert \tilde{Y}(t,t) \lambda_1 \xi_1(t) \xi_1' \tilde{Y} \tilde{\Xi}^{(i)}_1 (\tilde{\Lambda}^{(i)}_1)^{-1}\Vert  &\leq
 C \beta_n^{-1} \lambda_1(PG)\Big(\Vert \tilde{Y} - I_n\Vert |\xi_1(t)| + \Vert \xi_1'  \tilde{\Xi}^{(i)}_1\Vert |\xi_1(t)| \Big)
 \notag \\
 &\leq C \beta_n^{-1}\lambda_1(PG)\kappa_i \Big(1 \wedge \sqrt{\frac{\theta_i}{\bar{\theta}}} \, \Big)
\end{align}
over the event $E$, where we used the first estimate in Lemma \ref{lem:order} and sin-theta theorem for $\Vert \xi_1'  \tilde{\Xi}^{(i)}_1\Vert $ that 
\begin{align*}
\Vert \xi_1'  \tilde{\Xi}^{(i)}_1\Vert \leq C K\lambda_1^{-1}(PG) \Vert\tilde{L}^{(i)} - L_0 \Vert \leq  C \Vert \tilde{Y} - I_n\Vert. 
\end{align*}
For the second term on the RHS of  (\ref{2021071401}), we have
\begin{align*}
\Vert \tilde{Y}(t,t) \Xi_1(t)\Lambda_1 \Xi_1' \tilde{Y} \tilde{\Xi}^{(i)}_1 (\tilde{\Lambda}^{(i)}_1)^{-1}
- \Xi_1(t) \Lambda_1 \Xi_1'  \tilde{\Xi}^{(i)}_1 (\tilde{\Lambda}^{(i)}_1)^{-1}\Vert
\leq  C \beta_n^{-1} \lambda_1(PG) \Vert \tilde{Y}- I_n\Vert \Vert\Xi_1(t)\Vert
\end{align*}
and 
\begin{align*}
 \Xi_1(t) \Lambda_1 \Xi_1'  \tilde{\Xi}^{(i)}_1 (\tilde{\Lambda}^{(i)}_1)^{-1} =  \Xi_1(t)\Xi_1'  L_0\,  \tilde{\Xi}^{(i)}_1 (\tilde{\Lambda}^{(i)}_1)^{-1} =  \Xi_1(t)\Xi_1'\tilde{\Xi}^{(i)}_1 + \Xi_1(t)\Xi_1'  (L_0-\tilde{L}^{(i)})  \tilde{\Xi}^{(i)}_1 (\tilde{\Lambda}^{(i)}_1)^{-1}.
\end{align*}
By singular value decomposition (SVD), we write  $\Xi_1'\tilde{\Xi}^{(i)}_1= U \cos \varTheta \, V' $ for some orthogonal matrices $U, V$ and diagonal matrix $\cos\varTheta$ all of which are $i$-dependent. Setting $ O_2^{(i)}=\big( {\rm sgn}(\Xi_1'\tilde{\Xi}_1^{(i)})\big)':= VU'$ which is an orthogonal matrix, then we obtain that
\begin{align} \label{21121201}
\Vert \Xi_1'\tilde{\Xi}^{(i)}_1- ({O}_2^{(i)})'\Vert \leq C(K\beta_n^{-1} \Vert\tilde{L}^{(i)} - L_0 \Vert)^2\leq CK\beta_n^{-1} \Vert\tilde{L}^{(i)} - L_0 \Vert. 
\end{align}
Here we used the fact that $K\beta_n^{-1} \Vert\tilde{L}^{(i)} - L_0 \Vert\leq C \beta_n^{-1}\lambda_1(PG) \Vert \tilde{Y}- I_n\Vert = o(1)$.
Further we crudely bound 
\begin{align*}
 \Vert \Xi_1(t)\Xi_1'  (L_0-\tilde{L}^{(i)})  \tilde{\Xi}^{(i)}_1 (\tilde{\Lambda}^{(i)}_1)^{-1} \Vert &\leq CK\beta_n^{-1} \Vert\tilde{L}^{(i)} - L_0 \Vert \Vert \Xi_1(t)\Vert\notag\\
 &\leq C \beta_n^{-1}\lambda_1(PG) \Vert \tilde{Y} - I_n\Vert  \Vert  \Xi_1(t)\Vert
\end{align*}
Hence,
\begin{align}\label{21071403}
\Vert \tilde{Y}(t,t) \Xi_1(t)\Lambda_1 \Xi_1' \tilde{Y} \tilde{\Xi}^{(i)}_1 (\tilde{\Lambda}^{(i)}_1)^{-1} - \Xi_1(t) ({O}_2^{(i)})' \Vert &\leq C \beta_n^{-1}\lambda_1(PG) \Vert \tilde{Y} - I_n\Vert  \Vert  \Xi_1(t)\Vert \notag \\
&\leq C \sqrt K\beta_n^{-1}\lambda_1(PG)\kappa_i \Big(1 \wedge \sqrt{\frac{\theta_i}{\bar{\theta}}} \, \Big)
\end{align}
over the event $E$.

Plugging in (\ref{21071402}) and (\ref{21071403}) back  to (\ref{2021071401}), we simply conclude that 
\begin{align*}
\Vert  \tilde{\Xi}^{(i)}_1 (t) O_2^{(i)}- \Xi_1 (t)\Vert 
\leq C \sqrt K \beta_n^{-1}\lambda_1(PG)\kappa_i \Big(1 \wedge \sqrt{\frac{\theta_i}{\bar{\theta}}} \, \Big)
\end{align*}
over the event $E$. Combining all $i$'s and $t$'s together and noting $\mathbb{P}(E)= 1- o(n^{-3})$. This finished the proof of Lemma \ref{lem:upbd_key2K} by further noticing that $\lambda_1(PG)\asymp K$.

\subsection{Proof of Lemma \ref{lem:upbd_key11_2K}} 
In this section, we prove Lemma \ref{lem:upbd_key11_2K}.

Let us fix the index $i$.
The proof of (\ref{Laplacian-entry-2_2K}) is straightforward by the decomposition 
$$\hat{\Xi}_1=\tilde{\Xi}_1^{(i)} O_4^{(i)}O_5^{(i)} +  (\overline{\Xi}_1^{(i)}- \tilde{\Xi}_1^{(i)} O_4^{(i)}) O_5^{(i)} + \hat{\Xi}_1 - \overline{\Xi}_1^{(i)}O_5^{(i)}$$
where the two orthogonal matrices $O_4^{(i)}, O_5^{(i)}$ will be specified later. 
We further bound
\begin{align*}
\Vert e_i' \Delta (\hat{\Xi}_1-\overline{\Xi}_1^{(i)}O_5^{(i)})\Vert & \leq \frac{1}{\sqrt{\tilde{H}^{(i)}(i,i)}} \Vert W(i) (\tilde{H}^{(i)})^{-\frac 12} \Vert \, \Vert X\Vert \,  \Vert \hat{\Xi}_1-\overline{\Xi}^{(i)}_1O_5^{(i)}\Vert  \notag\\
&\leq \frac{C}{\sqrt{n\bar{\theta}^2 }}  \Vert \hat{\Xi}_1-\overline{\Xi}_1O_5^{(i)}\Vert  
\end{align*}
over the event $E$, by writing $\Delta= (\tilde{H}^{(i)})^{-1/2} W (\tilde{H}^{(i)})^{-1/2} X$ and using the fact $\Vert X\Vert\leq C$ and $\Vert W(i) (\tilde{H}^{(i)})^{-\frac 12} \Vert \leq \sqrt{\theta_i/\bar{\theta}}\,\vee \sqrt{\log (n)/n\bar{\theta}^2}$ (see (\ref{est:WtH})) over the event $E$. This, together with the trivial identities $\Vert e_i'\Delta \tilde{\Xi}_1^{(i)} O_4^{(i)}O_5^{(i)}\Vert = \Vert e_i'\Delta \tilde{\Xi}_1^{(i)} \Vert$ and $\Vert e_i'\Delta  (\overline{\Xi}_1^{(i)}- \tilde{\Xi}_1^{(i)} O_4^{(i)}) O_5^{(i)}\Vert = \Vert e_i'\Delta  (\overline{\Xi}_1^{(i)}- \tilde{\Xi}_1^{(i)} O_4^{(i)})  \Vert$ implies (\ref{Laplacian-entry-2_2K}). 

We then turn to show (\ref{Laplacian-entry-1_2K}). Note that
\begin{align*}
\hat{\Xi}_1\hat{\Lambda}_1 = L \, \hat{\Xi}_1 = X (\tilde{H}^{(i)})^{-\frac 12} A(\tilde{H}^{(i)})^{-\frac 12} X \hat{\Xi}_1 = X (\tilde{H}^{(i)})^{-\frac 12} \Omega (\tilde{H}^{(i)})^{-\frac 12} X \hat{\Xi}_1 + X(\tilde{H}^{(i)})^{-\frac 12} (A-\Omega) (\tilde{H}^{(i)})^{-\frac 12} X \hat{\Xi}_1
\end{align*}
by the notation $X= (\tilde{H}^{(i)})^{\frac 12} H^{-\frac 12}$.
Then, 
\begin{align}\label{21071404}
 \hat{\Xi}_1(i)= &X(i,i) \tilde{\lambda}^{(i)}_1 \tilde{\xi}^{(i)}_1(i) \big(\tilde{\xi}^{(i)}_1\big)' X \hat{\Xi}_1 \hat{\Lambda}_1^{-1}+ X(i,i) \tilde{\Xi}^{(i)}_1(i) \tilde{\Lambda}^{(i)} (\tilde{\Xi}^{(i)}_1)' X \hat{\Xi}_1\hat{\Lambda}_1^{-1} 
 \notag\\
 &+  X(i,i) e_i'( \tilde{H}^{(i)})^{-\frac 12} (A-\Omega) (\tilde{H}^{(i)})^{-\frac 12} X \hat{\Xi}_1\hat{\Lambda}_1^{-1}.
\end{align}
Recall the estimate $\Vert X - I_n \Vert\leq C \sqrt{\log (n)}/\sqrt{n\bar{\theta}^2}$ following from  Lemma \ref{lem:noisetildeH} and the properties of eigenvalues and eigenvectors of  $\tilde{L}^{(i)}$ in Lemma \ref{lem:est.etL}. 
%
Then, for the first term on the RHS of (\ref{21071404}), we have 
\begin{align}\label{21112302}
\Vert X(i,i) \tilde{\lambda}^{(i)}_1 \tilde{\xi}^{(i)}_1(i) \big(\tilde{\xi}^{(i)}_1\big)' X \hat{\Xi}_1\hat{\Lambda}_1^{-1}\Vert \leq C K^{-1} \lambda_1(PG)\Vert \hat{\Lambda}_1^{-1}\Vert \big|\tilde{\xi}^{(i)}_1(i)\big| \Big(\Vert X- I_n\Vert + \Vert \big(\tilde{\xi}^{(i)}_1\big)' \hat{\Xi}_1 \Vert\Big).
\end{align}
 Recall that $\hat{\lambda}_j$'s  for  $1\leq j \leq K$ share the same asymptotic as $\lambda_j$'s in (\ref{order-lambda}) over the event $E$.  By sin-theta theorem and (\ref{l2norm_LtL}), we have the bound 
\begin{align}\label{21112301}
\Vert \big(\tilde{\xi}^{(i)}_1\big)' \hat{\Xi}_1 \Vert \leq C {K} \lambda_1^{-1}(PG)\Vert L- \tilde{L}^{(i)} \Vert \leq C  \Big(\sqrt{\frac{\log (n)}{n\bar{\theta}^2}} + \frac{K\lambda_1^{-1}(PG)}{\sqrt{n \bar{\theta}^2}}\Big).
\end{align}
Thus, plugging (\ref{21112301}), (\ref{prop:eigv.tL}) together with $\Vert X - I_n \Vert\leq C \sqrt{\log (n)}/\sqrt{n\bar{\theta}^2}$ into (\ref{21112302}),  we arrive at 
\begin{align} \label{21071406}
\Vert X(i,i) \tilde{\lambda}^{(i)}_1 \tilde{\xi}^{(i)}_1(i) \big(\tilde{\xi}^{(i)}_1\big)' X \hat{\Xi}_1\hat{\Lambda}_1^{-1}\Vert\leq C K  {\beta_n^{-1}} \kappa_i  \Big( 1\wedge \sqrt{\frac{\theta_i}{\bar{\theta}}}\, \Big),
\end{align}
where we used the trivial bound $\lambda_1(PG)\leq CK$. 

To estimate the other two term in (\ref{21071404}), we need the assistance of $\overline{\Xi}$, the eigenspace of $(\tilde{H}^{(i)})^{-\frac 12} \tilde{A}^{(i)}(\tilde{H}^{(i)})^{-\frac 12}$, which is counterpart to $\tilde{\Xi}^{(i)}_1$ and $\hat{\Xi}_1$.  Recall that $\tilde{A}^{(i)}= \Omega + \tilde{W}^{(i)}-{\rm diag}(\Omega)$ where $\tilde{W}^{(i)}$ is obtained by zeroing-out $i$-th row and column of $W$. Similarly to (\ref{21121201}), we can then claim that there exists an orthogonal matrix $ O_4^{(i)}$ by sin-theta theorem such that 
\begin{align}\label{21071801}
\Vert  \overline{\Xi}^{(i)}_1- \tilde{\Xi}^{(i)}_1O_4^{(i)}  \Vert  \leq  CK\beta_n^{-1} \Vert (\tilde{H}^{(i)})^{-\frac 12} \tilde{A}^{(i)}(\tilde{H}^{(i)})^{-\frac 12} - (\tilde{H}^{(i)})^{-\frac 12} \Omega (\tilde{H}^{(i)})^{-\frac 12} \Vert \leq  \frac{C K\beta_n^{-1}}{\sqrt{{n\bar{\theta}^2}}}
\end{align}
over the event $E$, where $O_4^{(i)}= {\rm sgn} ((\tilde{\Xi}^{(i)}_1)'\overline{\Xi}^{(i)}_1 )$. We will also need an orthogonal matrix $O_5\equiv O_5(i):= {\rm sgn } ((\overline{\Xi}^{(i)}_1)'  \hat{\Xi}_1)$. Again by sin-theta theorem,
\begin{align}\label{210723est6}
\Vert (\overline{\Xi}^{(i)}_1)'  \hat{\Xi}_1 - O_5 \Vert^{\frac 12} &\leq CK\beta_n^{-1} \Vert  (\tilde{H}^{(i)})^{-\frac 12} \tilde{A}^{(i)} (\tilde{H}^{(i)})^{-\frac 12} - H^{-\frac 12} A H^{-\frac 12}\Vert \notag\\
&\leq C \beta_n^{-1}\lambda_1(PG) \Vert X-I_n\Vert + CK\beta_n^{-1} \Vert (\tilde{H}^{(i)})^{-\frac 12} (\tilde{A}^{(i)}- A) (\tilde{H}^{(i)})^{-\frac 12}\Vert\notag\\
&\leq C K\beta_n^{-1} \sqrt{\frac{\log (n)}{n\bar{\theta}^2}}.
\end{align}
Here we used $\Vert(\tilde{H}^{(i)})^{-\frac 12} A(\tilde{H}^{(i)})^{-\frac 12} \Vert\asymp\Vert(\tilde{H}^{(i)})^{-\frac 12} \Omega(\tilde{H}^{(i)})^{-\frac 12} \Vert \asymp K^{-1}\lambda_1(PG)$ to get $K$ canceled for the first term of second line above.
 We then introduce the shorthand notation  $O_3^{(i)}=O_4^{(i)}O_5^{(i)}$.
And for the second term on the RHS of (\ref{21071404}), similarly to (\ref{21071403}),  we get 
\begin{align}\label{21071407}
&\quad \Vert X(i,i) \tilde{\Xi}^{(i)}_1(i) \tilde{\Lambda}^{(i)} (\tilde{\Xi}^{(i)}_1)' X \hat{\Xi}_1\hat{\Lambda}_1^{-1}- \tilde{\Xi}^{(i)}_1(i) O_3^{(i)} \Vert \notag\\
&\leq C\big( \beta_n^{-1}\lambda_1(PG)\Vert X - I_n \Vert + K \beta_n^{-1} \Vert L- \tilde{L}^{(i)} \Vert+ \Vert (\tilde{\Xi}^{(i)})_1'  \hat{\Xi}_1 - O_3^{(i)} \Vert \big)\Vert \tilde{\Xi}^{(i)}_1(i)\Vert 
\notag\\
&\leq C { K^{\frac 32}\beta_n^{-1}} \kappa_i + C\Vert (\tilde{\Xi}^{(i)}_1)'  \hat{\Xi}_1 - O_3^{(i)} \Vert \Vert \tilde{\Xi}^{(i)}_1(i)\Vert 
\end{align}
over the event $E$, where we recall that $\kappa_i= \sqrt{\log (n)/n\bar{\theta}^2}\cdot \sqrt{\theta_i/n\bar{\theta}} $. Moreover, we have 
\begin{align*}
\Vert (\tilde{\Xi}^{(i)}_1)'  \hat{\Xi}_1 - O_3^{(i)} \Vert  \leq \Vert  \overline{\Xi}^{(i)}_1- \tilde{\Xi}^{(i)}_1O_4^{(i)}  \Vert + \Vert (\overline{\Xi}^{(i)}_1)'  \hat{\Xi}_1 - O_5^{(i)} \Vert
\end{align*}
which with (\ref{21071801}), (\ref{210723est6}) and (\ref{prop:eigv.tL}) leads to 
\begin{align}\label{21112304}
\Vert X(i,i) \tilde{\Xi}^{(i)}_1(i) \tilde{\Lambda}^{(i)} (\tilde{\Xi}^{(i)})_1' Y \hat{\Xi}_1\hat{\Lambda}_1^{-1}- \tilde{\Xi}^{(i)}_1(i) O_3^{(i)} \Vert  & \leq C K^{\frac 32} \beta_n^{-1} \kappa_i 
\end{align}
Combining (\ref{21071406}) and (\ref{21112304}) back into (\ref{21071404}), we get 
\begin{align}\label{2107230501}
\Vert  \hat{\Xi}_1(i) - \tilde{\Xi}^{(i)}_1(i)O_3^{(i)}\Vert \leq 
C {K^{\frac 32}\beta_n^{-1} }\kappa_i+  \Vert X(i,i)e_i' (\tilde{H}^{(i)})^{-\frac 12} (A-\Omega) (\tilde{H}^{(i)})^{-\frac 12} X \hat{\Xi}_1\hat{\Lambda}_1^{-1}\Vert
\end{align}
In the sequel, we proceed to the second term on the RHS above. First, using the trivial bound $|X(i,i)|\leq 2$ and $\Vert \hat{\Lambda}_1\Vert^{-1}\leq K\beta_n^{-1}$, we have 
\begin{align*}
 &\quad \Vert X(i,i)e_i' (\tilde{H}^{(i)})^{-\frac 12} (A-\Omega) (\tilde{H}^{(i)})^{-\frac 12} X \hat{\Xi}_1\hat{\Lambda}_1^{-1}\Vert \notag\\
 &\leq C K \beta_n^{-1} \Vert e_i' (\tilde{H}^{(i)})^{-\frac 12} (A-\Omega) (\tilde{H}^{(i)})^{-\frac 12} X \hat{\Xi}_1\Vert\notag\\
 &\leq C K \beta_n^{-1}\Big( \Vert e_i' (\tilde{H}^{(i)})^{-\frac 12} W (\tilde{H}^{(i)})^{-\frac 12} X \hat{\Xi}_1\Vert + \Vert e_i' (\tilde{H}^{(i)})^{-\frac 12} {\rm diag}(\Omega) (\tilde{H}^{(i)})^{-\frac 12} X \hat{\Xi}_1\Vert\Big)
\end{align*}
We can simply get the bound 
\begin{align*}
\Vert e_i' (\tilde{H}^{(i)})^{-\frac 12} {\rm diag}(\Omega) (\tilde{H}^{(i)})^{-\frac 12} X \hat{\Xi}_1\Vert &= \Vert (\tilde{H}^{(i)})^{-1}(i,i)\Omega(i,i) X(i,i) \hat{\Xi}_1(i) \Vert  \notag\\
&\leq C \, \frac{\theta_i^2}{n\bar{\theta}(\bar{\theta}\vee \theta_i)} \, \Vert\hat{\Xi}_1(i) \Vert \notag\\
&\leq \frac{\sqrt K}{\sqrt{\log (n)}}\kappa_i  \Big(1\wedge \sqrt{\frac{\theta_i}{\bar{\theta}}}\, \Big)
\end{align*}
This leads to 
\begin{align}\label{211125305}
\Vert X(i,i)e_i' (\tilde{H}^{(i)})^{-\frac 12} (A-\Omega) (\tilde{H}^{(i)})^{-\frac 12} X \hat{\Xi}_1\hat{\Lambda}_1^{-1}\Vert \leq 
C {K^{\frac 32}\beta_n^{-1} }\kappa_i+ C K \beta_n^{-1} \Vert e_i' \Delta \hat{\Xi}_1\Vert 
\end{align}
over the event $E$ satisfying $\mathbb{P}(E) = 1- o(n^{-3})$.
Combining (\ref{211125305}) and (\ref{2107230501}) and considering all $i$'s, we then conclude the proof of (\ref{Laplacian-entry-1_2K}).

\subsection{Proof of  Lemma \ref{lem:techB1_2K}} \label{subsec:pr_Lem_techB1_2K}
The proof of  Lemma \ref{lem:techB1_2K} is rather complicated. We will show the three claims (i.e.,  (\ref{Laplacian-entry-3-1_2K})- (\ref{Laplacian-entry-3-3_2K}))  separately in the following three parts.

\subsubsection{Proof of (\ref{Laplacian-entry-3-1_2K})}
Write $\Delta=(\tilde{H}^{(i)})^{-\frac 12} W (\tilde{H}^{(i)})^{-\frac 12}X $, we first crudely have
\begin{align}\label{21112401}
\Vert e_i' \Delta \tilde{\Xi}^{(i)}_1 \Vert &\leq \Vert e_i' (\tilde{H}^{(i)})^{-\frac 12} W (\tilde{H}^{(i)})^{-\frac 12} \tilde{\Xi}^{(i)}_1\Vert  + \Vert e_i' (\tilde{H}^{(i)})^{-\frac 12} W (\tilde{H}^{(i)})^{-\frac 12}(X- I_n) \tilde{\Xi}^{(i)}_1\Vert 
\end{align}
We start with the first term on the RHS of (\ref{21112401}).
\begin{align*}
\Vert e_i' (\tilde{H}^{(i)})^{-\frac 12} W (\tilde{H}^{(i)})^{-\frac 12} \tilde{\Xi}^{(i)}_1\Vert = \Big\Vert\frac{1}{\sqrt{\tilde{H}^{(i)}(i,i)}} W(i)(\tilde{H}^{(i)})^{-\frac 12}\tilde{\Xi}^{(i)}_1\Big\Vert =\Big\Vert  \frac{1}{\sqrt{\tilde{H}^{(i)}(i,i)}} \sum_{t=1}^n \frac{W(i,t)}{\sqrt{\tilde{H}^{(i)}(t,t)}} \tilde{\Xi}^{(i)}_1(t)  \Big\Vert.
\end{align*}
Thanks to the independence between $\tilde{\Xi}^{(i)}_1$ and $W(i)$, we can estimate $\sum_{t=1}^n \frac{W(i,t)}{\sqrt{\tilde{H}^{(i)}(t,t)}} \tilde{\Xi}^{(i)}_1(t)$ componentwisely by Bernstein inequality with respect to the randomness of $W(i)$. For each $2\leq p\leq K$, we can bound the variance of $\sum_{t=1}^n \frac{W(i,t)}{\sqrt{\tilde{H}^{(i)}(t,t)}} \tilde{\xi}^{(i)}_p(t)$ by 
\begin{align*}
{\rm var} \Big( \sum_{t=1}^n \frac{W(i,t)}{\sqrt{\tilde{H}^{(i)}(t,t)}} \tilde{\xi}^{(i)}_p(t)\Big) \asymp \sum_{t=1}^n \frac{\theta_i\theta_t}{\tilde{H}^{(i)}(t,t)} \big(\tilde{\xi}^{(i)}_p(t)\big)^2
&\leq C  \sum_{t=1}^n \frac{\theta_i\theta_t}{n\bar{\theta}(\theta_t \vee \bar{\theta})} \big(\tilde{\xi}^{(i)}_p(t)\big)^2 \leq C \frac{\theta_i}{n\bar{\theta}}.
\end{align*}
 Each individual summand can be bounded by $C/n\bar{\theta}$ over the event $E$. As a result, 
\begin{align*}
\Big|\sum_{t=1}^n \frac{W(i,t)}{\sqrt{\tilde{H}^{(i)}(t,t)}} \tilde{\xi}^{(i)}_p(t)\Big| \leq C \Big(\sqrt{\frac{\theta_i \log (n)}{n\bar{\theta}}} + \frac{\log (n)}{n\bar{\theta}}\Big)\leq C \frac{\sqrt{\log (n)}}{n\bar{\theta}} \sqrt{n\bar{\theta} \theta_i \vee \log (n)}
\end{align*}
Further with $\tilde{H}^{(i)} (i,i)\asymp n\bar{\theta}(\theta_i \vee \bar{\theta})$, we finally conclude that 
\begin{align} \label{21071502term1}
\Vert e_i' (\tilde{H}^{(i)})^{-\frac 12} W (\tilde{H}^{(i)})^{-\frac 12} \tilde{\Xi}^{(i)}_1\Vert=\Big\Vert  \frac{1}{\sqrt{\tilde{H}^{(i)}(i,i)}} \sum_{t=1}^n \frac{W(i,t)}{\sqrt{\tilde{H}^{(i)}(t,t)}} \tilde{\Xi}^{(i)}_1(t)  \Big\Vert \leq C K^{\frac 12}\, \widetilde{\kappa}_i.
\end{align}
Next, regarding the term $\Vert e_i' (\tilde{H}^{(i)})^{-\frac 12} W (\tilde{H}^{(i)})^{-\frac 12}(X- I_n) \tilde{\Xi}^{(i)}_1\Vert $, using the estimate (\ref{est.X-I.entry}),  we can  derive 
\begin{align}\label{21112402term11}
\frac{ \Vert  W(i)  (\tilde{H}^{(i)})^{-\frac 12} (X- I_n)\tilde{\Xi}^{(i)}_1\Vert  }{\sqrt{\tilde{H}^{(i)}(i,i)}} &\leq \frac{C}{\sqrt{\tilde{H}^{(i)}(i,i)}}  \sum_{t=1, t\neq i}^n |W(i,t)|\, \frac{A(i,t) +\theta_i\theta_t+ \theta_i \bar{\theta} + \frac{\log (n)}{n}}{|\tilde{H}^{(i)}(t,t)|^{\frac 32}} \Vert \tilde{\Xi}^{(i)}_1(t) \Vert 
\notag\\
&\leq \frac{C}{\sqrt{\tilde{H}^{(i)}(i,i)}} \sum_{t=1, t\neq i}^n\, \frac{A(i,t)  +\theta_i\theta_t + \theta_i \bar{\theta}+ \frac{\log (n)}{n}}{|\tilde{H}^{(i)}(t,t)|^{\frac 32}} \Vert \tilde{\Xi}^{(i)}_1(t) \Vert \notag\\
& \leq C \sqrt K \frac{\widetilde{\kappa}_i}{\sqrt{\log (n)}}
\end{align}
where the last step is analogous to how we get (\ref{2112080101}) by Bernstein's inequality and one can refer to the details in Section \ref{app:subsub_techproof1}. Combining (\ref{21071502term1}) and (\ref{21112402term11}) into (\ref{21112401}), and considering all $i$'s, we thus conclude  (\ref{Laplacian-entry-3-1_2K}).

\subsubsection{Proof of (\ref{Laplacian-entry-3-2_2K})}
The proof is similar to the proof of (\ref{Laplacian-entry-3-2})
in Section \ref{app:subsub_techproof2}.
First, by definition, we bound 
\begin{align}\label{211124010}
\Vert e_i' \Delta (\overline{\Xi}^{(i)}_1-\tilde{\Xi}^{(i)}_1 O_4^{(i)})\Vert \leq &\Vert e_i' (\tilde{H}^{(i)})^{-\frac 12} W (\tilde{H}^{(i)})^{-\frac 12} (\overline{\Xi}^{(i)}_1-\tilde{\Xi}^{(i)}_1 O_4^{(i)})\Vert  \notag\\
&+ \Vert e_i' (\tilde{H}^{(i)})^{-\frac 12} W (\tilde{H}^{(i)})^{-\frac 12}(X- I_n) (\overline{\Xi}^{(i)}_1-\tilde{\Xi}^{(i)}_1 O_4^{(i)})\Vert. 
\end{align}
We rewrite the first term on the RHS by 
\begin{align*}
\Vert e_i' (\tilde{H}^{(i)})^{-\frac 12} W (\tilde{H}^{(i)})^{-\frac 12} (\overline{\Xi}^{(i)}_1-\tilde{\Xi}^{(i)}_1 O_4^{(i)})\Vert= \Big\Vert  \frac{1}{\sqrt{\tilde{H}^{(i)}(i,i)}}  W(i) (\tilde{H}^{(i)})^{-\frac 12}(   \overline{\Xi}^{(i)}_1 -\tilde{\Xi}^{(i)}_1O_4^{(i)} )\Big\Vert. 
\end{align*}
According to the definition of $\overline{\Xi}^{(i)}_1$, $ \overline{\Xi}^{(i)}_1 -\tilde{\Xi}^{(i)}_1O_4^{(i)}$ is also independent of $W(i)$. Then, analogously to the previous section, restricted to the randomness of $W(i)$, we bound the variance of  each component  of $W(i) (\tilde{H}^{(i)})^{-\frac 12}(   \overline{\Xi}^{(i)}_1 -\tilde{\Xi}^{(i)}_1O_4^{(i)} )$ by 
\begin{align*}
\sum_{t=1}^n \theta_i \theta_t \left(\frac{  \big(\overline{\Xi}^{(i)}_1(t) -\tilde{\Xi}^{(i)}_1 (t)O_4^{(i)} \big)e_p}{\sqrt{\tilde{H}^{(i)}(t,t)}}\right)^2 \leq \frac{C\theta_i}{n \bar{\theta}}.
\end{align*}
Here to obtain the RHS upper bound, we used an elementary derivation
\begin{align*}
\sum_{t} (\tilde{\Xi}^{(i)}_1 (t)O_4^{(i)} e_p)^2 = e_p' (O_4^{(i)})' (\tilde{\Xi}^{(i)}_1)' \sum_{t}e_t e_t' \tilde{\Xi}^{(i)}_1 O_4^{(i)} e_p= e_p' (O_4^{(i)})' (\tilde{\Xi}^{(i)}_1)' \tilde{\Xi}^{(i)}_1 O_4^{(i)} e_p = 1.
\end{align*}
There is some ambiguity over the dimension of $e_p$ and $e_t$. $e_p$ shall be of dimension $K-2$ while $e_t$ is of dimension $n$.
Further,  each summand in the $p$-th component of $W(i) (\tilde{H}^{(i)})^{-\frac 12}(   \overline{\Xi}^{(i)}_1 -\tilde{\Xi}^{(i)}_1O_4^{(i)} )$ is bounded by $C \Vert (\tilde{H}^{(i)})^{-\frac 12} (   \overline{\Xi}^{(i)}_1 -\tilde{\Xi}^{(i)}_1O_4^{(i)} )e_p\Vert_\infty$.  Thus, for each $1\leq p \leq K-1$, 
\begin{align*}
\big|W(i) (\tilde{H}^{(i)})^{-\frac 12}(   \overline{\Xi}^{(i)}_1 -\tilde{\Xi}^{(i)}_1O_4^{(i)} ) e_p \big| \leq C \sqrt{\frac{\theta_i \log (n)}{n\bar{\theta}}} + C \log (n)  \Vert (\tilde{H}^{(i)})^{-\frac 12} (   \overline{\Xi}^{(i)}_1 -\tilde{\Xi}^{(i)}_1O_4^{(i)} )e_p\Vert_\infty;
\end{align*}
and therefore, 
\begin{align*}
\Vert W(i) (\tilde{H}^{(i)})^{-\frac 12}(   \overline{\Xi}^{(i)}_1 -\tilde{\Xi}^{(i)}_1O_4^{(i)} )\Vert  &\leq C\bigg( \sum_{p=1}^{K-1} \Big(\sqrt{\frac{\theta_i \log (n)}{n\bar{\theta}}} +  \log (n)  \Vert (\tilde{H}^{(i)})^{-\frac 12} (   \overline{\Xi}^{(i)}_1 -\tilde{\Xi}^{(i)}_1O_4^{(i)} )e_p\Vert_\infty\Big)^2\bigg)^{\frac 12} \notag\\
& \leq C \sqrt{K\log (n)} \sqrt{\frac{\theta_i}{n\bar{\theta}}} + C\log (n) \Big(\sum_{p=1}^{K-1}\Vert (\tilde{H}^{(i)})^{-\frac 12} (   \overline{\Xi}^{(i)}_1 -\tilde{\Xi}^{(i)}_1O_4^{(i)} )e_p\Vert_\infty^2\Big)^{\frac 12}.
\end{align*}
We further have 
\begin{align*}
\Big(\sum_{p=1}^{K-1}\Vert (\tilde{H}^{(i)})^{-\frac 12} (   \overline{\Xi}^{(i)}_1 -\tilde{\Xi}^{(i)}_1O_4^{(i)} )e_p\Vert_\infty^2\Big)^{\frac 12}
&\leq \Big(\sum_{p=1}^{K-1}\Vert (\tilde{H}^{(i)})^{-\frac 12} (   \overline{\Xi}^{(i)}_1 -\tilde{\Xi}^{(i)}_1O_4^{(i)} \Vert^2_{2\to \infty} \Big)^{\frac 12}\notag\\
& \leq  \sqrt{K}  \Vert (\tilde{H}^{(i)})^{-\frac 12} (   \hat{\Xi}_1 -\tilde{\Xi}^{(i)}_1O_4^{(i)} O_5^{(i)})\Vert_{2\to \infty} +\frac{C\sqrt K}{\sqrt {n\bar{\theta}^2}} \Vert  \hat{\Xi}_1- \overline{\Xi}^{(i)}_1O_5^{(i)} \Vert.
\end{align*}
 Thus, over the event $E$,
\begin{align}\label{21071502term2}
&\quad \Vert e_i' (\tilde{H}^{(i)})^{-\frac 12} W (\tilde{H}^{(i)})^{-\frac 12} (\overline{\Xi}^{(i)}_1-\tilde{\Xi}^{(i)}_1 O_4^{(i)})\Vert \notag\\
&\leq  C \sqrt K \, \kappa_i+ C\sqrt K \frac{\log (n)}{\sqrt{n\bar{\theta}(\bar{\theta}\vee \theta_i)}} \, \Vert (\tilde{H}^{(i)})^{-\frac 12} (   \hat{\Xi}_1 -\tilde{\Xi}^{(i)}_1O_4^{(i)} O_5^{(i)})\Vert_{2\to \infty}+ \frac{C\sqrt K\, \log (n)}{n\bar{\theta}^2}\Vert  \hat{\Xi}_1- \overline{\Xi}^{(i)}_1O_5^{(i)} \Vert
\end{align}
Next, for the second term of (\ref{211124010}),  using the estimate (\ref{est.X-I.entry}), we have {\small
\begin{align}\label{21112501}
&\quad \frac{\Vert  W(i) (\tilde{H}^{(i)})^{-\frac 12}(X- I_n)(  \overline{\Xi}_1 -\tilde{\Xi}^{(i)}_1O_4^{(i)}   )\Vert}{\sqrt{\tilde{H}^{(i)}(i,i)}} \notag\\
&\leq \frac{C}{\sqrt{\tilde{H}^{(i)}(i,i)}}  \sum_{t=1, t\neq i}^n |W(i,t)|\, \frac{A(i,t) +\theta_i \theta_t+ \theta_i \bar{\theta} + \log (n)/n}{\tilde{H}^{(i)}(t,t)} \, 
\frac{\Vert \overline{\Xi}_1(t) - \tilde{\Xi}^{(i)}_1(t) O_4^{(i)}\Vert}{\sqrt{\tilde{H}^{(i)}(t,t)}}
\notag\\
&\leq \frac{C}{\sqrt{\tilde{H}^{(i)}(i,i)}} \sum_{t=1, t\neq i}^n \frac{A(i,t) +\theta_i \theta_t+ \theta_i \bar{\theta} +\frac{ \log (n)}{n}}{\tilde{H}^{(i)}(t,t)} \bigg(\big\Vert (\tilde{H}^{(i)})^{-\frac 12}(  \hat{\Xi}_1 -\tilde{\Xi}^{(i)}_1O_4^{(i)}  O_5^{(i)} )\big\Vert_{2\to\infty} + \frac{\Vert \hat{\Xi}_1(t)- \overline{\Xi}^{(i)}_1(t)O_5^{(i)}\Vert}{\sqrt{\tilde{H}^{(i)}(t,t)}}\bigg)
\end{align} 
}
Similarly to the derivations of upper bounds of (\ref{211125sub1}) and (\ref{211125sub2}), we bound the two sums on the RHS of (\ref{21112501}) corresponding to the two terms in the parenthesis separately as follows:
\begin{align*}
&\quad \frac{1}{\sqrt{\tilde{H}^{(i)}(i,i)}} \sum_{t=1, t\neq i}^n \frac{A(i,t) + \theta_i \bar{\theta} + \log (n)/n}{\tilde{H}^{(i)}(t,t)} \big\Vert (\tilde{H}^{(i)})^{-\frac 12}(  \hat{\Xi}_1 -\tilde{\Xi}^{(i)}_1O_4^{(i)}  O_5^{(i)} )\big\Vert_{2\to \infty }  \notag\\
 &\leq \frac{\theta_i/\bar{\theta} + \log (n)/n\bar{\theta}^2}{\sqrt{n\bar{\theta} (\bar{\theta}\vee \theta_i)}}\, \big\Vert (\tilde{H}^{(i)})^{-\frac 12}(  \hat{\Xi}_1 -\tilde{\Xi}^{(i)}_1O_4^{(i)}  O_5^{(i)} )\big\Vert_{2\to \infty } \notag\\
&\leq \Big(\frac{1}{\bar{\theta}} \, \sqrt{\frac{\theta_i}{n\bar{\theta}}} + \frac{\log (n)}{(n\bar{\theta}^2)^{\frac 32}}\Big) \big\Vert (\tilde{H}^{(i)})^{-\frac 12}(  \hat{\Xi}_1 -\tilde{\Xi}^{(i)}_1O_4^{(i)}  O_5^{(i)} )\big\Vert_{2\to \infty }\end{align*}
and 
\begin{align*}
&\frac{1}{\sqrt{\tilde{H}^{(i)}(i,i)}} \sum_{t=1, t\neq i}^n \frac{A(i,t) + \theta_i \bar{\theta} + \log(n)/n}{(\tilde{H}^{(i)}(t,t))^{3/2}}  \Vert \hat{\Xi}_1(t)- \overline{\Xi}^{(i)}_1(t)O_5^{(i)}\Vert \notag\\
&\leq \frac{1}{\sqrt{\tilde{H}^{(i)}(i,i)}} \bigg(\sum_{t=1, t\neq i}^n \frac{\big(A(i,t) + \theta_i \bar{\theta} + \log(n)/n\big)^2}{(\tilde{H}^{(i)}(t,t))^3} \bigg)^{\frac 12} \Big( \sum_{t=1, t\neq i}^n \Vert \hat{\Xi}_1(t)- \overline{\Xi}^{(i)}_1(t)O_5^{(i)}\Vert ^2 \Big)^{\frac 12}\notag\\
&\leq C(n\bar{\theta}^2)^{-\frac 32}\Big({\rm tr}(\hat{\Xi}_1- \overline{\Xi}^{(i)}_1O_5^{(i)})'(\hat{\Xi}_1- \overline{\Xi}^{(i)}_1O_5^{(i)})\Big)^{\frac 12} \notag\\
&\leq C\sqrt K (n\bar{\theta}^2)^{-\frac 32} \Vert \hat{\Xi}_1- \overline{\Xi}^{(i)}_1O_5^{(i)} \Vert, 
\end{align*}
over the event $E$, in which, we applied (\ref{211125sub4}) and (\ref{211125sub3}). We plug the above two estimates into (\ref{21112501}) and conclude that over the event $E$,
\begin{align*}
&\quad \frac{\Vert  W(i) (\tilde{H}^{(i)})^{-\frac 12}(X- I_n)(  \overline{\Xi}^{(i)}_1 -\tilde{\Xi}^{(i)}_1O_4^{(i)}   )\Vert}{\sqrt{\tilde{H}^{(i)}(i,i)}} \notag\\
&\leq C
\widetilde{\kappa_i} n\bar{\theta} \, \Vert (\tilde{H}^{(i)})^{-\frac 12} (   \hat{\Xi}_1 -\tilde{\Xi}^{(i)}_1O_4^{(i)} O_5^{(i)})\Vert_{2\to \infty } + \sqrt K (n\bar{\theta}^2)^{-\frac 32} \Vert \hat{\Xi}_1- \overline{\Xi}^{(i)}_1O_5^{(i)} \Vert
\end{align*}
This, together with (\ref{21071502term2}), concludes the proof of (\ref{Laplacian-entry-3-2_2K}) for fixed $i$, by the fact that $\log (n)/\sqrt{n\bar{\theta}^2}\leq \widetilde{\kappa}_i\cdot n\bar{\theta}$. Combining all $i$'s and the fact $\mathbb{P}(E) = 1- o(n^{-3})$, we finish the proof.

\subsubsection{Proof of (\ref{Laplacian-entry-3-3_2K})}

By sin-theta theorem and the fact that  the eigen-gap is of the order $O(K^{-1} \beta_n)$ in light of Weyl's inequality (see (\ref{21121401})),  analogously to (\ref{leave-effect-1}), we first have  {\small
\begin{align} \label{211125012}
&\quad \Vert  \hat{\Xi}_1 - \overline{\Xi}^{(i)}_1  O_5^{(i)}  \Vert \notag\\
&\leq  {K\beta_n^{-1}}\Vert  \big( H^{-\frac 12} A H^{-\frac 12} - (\tilde{H}^{(i)})^{-\frac 12} \tilde{A}^{(i)} (\tilde{H}^{(i)})^{-\frac 12} \big)  \hat{\Xi}_1\Vert \notag\\
& \leq  { K\beta_n^{-1}} \Big(\Vert (I_n- X^{-1}) H^{-\frac 12} A H^{-\frac 12}  \hat{\Xi}_1\Vert + \Vert (\tilde{H}^{(i)})^{-\frac 12} A (\tilde{H}^{(i)})^{-\frac 12} (X- I_n) \hat{\Xi}_1\Vert + \Vert ( \tilde{H}^{(i)})^{-\frac 12} (A- \tilde{A}^{(i)}) (\tilde{H}^{(i)})^{-\frac 12} \hat{\Xi}_1\Vert  \Big)\notag\\
& \leq C { K\beta_n^{-1}}  \Big(  \Vert (X- I_n)  \hat{\Xi}_1\hat{\Lambda}_1\Vert +   \Vert (\tilde{H}^{(i)})^{-\frac 12} A (\tilde{H}^{(i)})^{-\frac 12} (X - I_n) \hat{\Xi}_1\Vert + \Vert  (\tilde{H}^{(i)})^{-\frac 12} (e_i W(i) + W(i)' e_i') (\tilde{H}^{(i)})^{-\frac 12} \hat{\Xi}_1\Vert \Big).
\end{align}
}
We start with a simple derivation,
\begin{align*}
&\quad \Vert (\tilde{H}^{(i)})^{-\frac 12} A (\tilde{H}^{(i)})^{-\frac 12} ( X - I_n) \hat{\Xi}_1\Vert \notag\\
& \leq \Vert (\tilde{H}^{(i)})^{-\frac 12} \Omega (\tilde{H}^{(i)})^{-\frac 12} \Vert \Vert (X - I_n) \hat{\Xi}_1\Vert +  
\Vert (\tilde{H}^{(i)})^{-\frac 12} (A-\Omega) (\tilde{H}^{(i)})^{-\frac 12} \Vert \Vert (X - I_n) \hat{\Xi}_1\Vert \notag\\
&\leq C  { K^{-1}}\lambda_1(PG)\Vert (X - I_n) \hat{\Xi}_1\Vert;
\end{align*}
Second, we have
\begin{align*}
\Vert  (\tilde{H}^{(i)})^{-\frac 12}W(i)' e_i' (\tilde{H}^{(i)})^{-\frac 12} \hat{\Xi}_1\Vert &=  \tilde{H}^{(i)}(i,i)^{-\frac 12} \Vert  \hat{\Xi}_1(i) (\tilde{H}^{(i)})^{-\frac 12}W(i)'  \Vert  \notag\\
& \leq \frac{C}{\sqrt{n\bar{\theta}^2}} \Vert \hat{\Xi}_1(i)\Vert\notag\\
& \leq \frac{C\sqrt K\, \kappa_i}{\sqrt{\log (n)}} + \frac{C}{\sqrt{n\bar{\theta}^2} }\Vert \hat{\Xi}_1(i) - \tilde{\Xi}^{(i)}_1(i) O_3^{(i)}\Vert.
\end{align*}
where in the second step we used (\ref{est:WtH}) and we decomposed $\hat{\Xi}_1(i)$ as $\tilde{\Xi}^{(i)}O_3^{(i)} + \hat{\Xi}_1(i) - \tilde{\Xi}^{(i)}_1(i) O_3^{(i)}$ and employed (\ref{prop:eigv.tL}) in the last step. Thus, we further bound the RHS of (\ref{211125012}) as 
\begin{align}\label{211125200}
\Vert  \hat{\Xi}_1 - \overline{\Xi}^{(i)}_1  O_5^{(i)}  \Vert
 \leq &C {\beta_n^{-1} } \lambda_1(PG)\Vert (X - I_n)   \hat{\Xi}_1\Vert + C  {K\beta_n^{-1}} \Vert \tilde{H}^{(i)}(i,i)^{-\frac 12} W(i)  (\tilde{H}^{(i)})^{-\frac 12} \hat{\Xi}_1\Vert \notag\\
& + \frac{ C { K^{\frac 32}\beta_n^{-1}} \, \kappa_i}{\sqrt{\log (n)}} + \frac{C{K\beta_n^{-1}}  }{\sqrt{n\bar{\theta}^2} }\Vert \hat{\Xi}_1(i) - \tilde{\Xi}^{(i)}_1(i) O_3^{(i)}\Vert.
\end{align}
In the sequel, we analyze the first two terms on the RHS above. For $\Vert (X - I_n)   \hat{\Xi}_1\Vert$, similarly to (\ref{leave-effect-2}),  we decompose $\hat{\Xi}_1$ and get that 
\begin{align*}
\Vert (X - I_n)   \hat{\Xi}_1\Vert\leq  \Vert (X - I_n)   \tilde{\Xi}^{(i)}_1\Vert + \Vert (X - I_n)   (\hat{\Xi}_1 - \tilde{\Xi}^{(i)}_1O_3^{(i)})\Vert
\end{align*}
Then, we replicate the derivations for the two terms in (\ref{leave-effect-2}) with $\tilde{\xi}^{(i)}_1$, $\hat{\xi}_1$ replaced by $\tilde{\Xi}^{(i)}_1$, $\hat{\Xi}_1$ and $w$ replaced by $O_3'$ to get 
\begin{align*}
 \Vert (X - I_n)   \tilde{\Xi}^{(i)}_1\Vert^2 &\leq C \sum_{j=1}^n\big( A(i,j)+\theta_i\theta_j + \theta_i\bar{\theta} + \frac{\log(n)}{n} \big)\frac{\Vert \tilde{\Xi}^{(i)}_1(j)\Vert^2 }{[\tilde{H}^{(i)}(j,j)]^2} \leq \frac{ CK\widetilde{\kappa}_i^2}{\log(n)}
  \notag\\
\Vert (X - I_n)   (\hat{\Xi}_1 - \tilde{\Xi}^{(i)}_1O_3^{(i)})\Vert^2 &\leq C \|(\tilde{H}^{(i)})^{-1/2}(\hat{\Xi}_1-\tilde{\Xi}^{(i)}_1O_3^{(i)})\|^2_{2\to \infty }\sum_{j=1}^n \frac{ A(i,j) +\theta_i\theta_j + \theta_i\bar{\theta} + \frac{\log (n)}{n}}{\tilde{H}^{(i)}(j,j)}\notag\\
&\leq C\|(\tilde{H}^{(i)})^{-1/2}(\hat{\Xi}_1-\tilde{\Xi}^{(i)}_1O_3^{(i)}) \|^2_{2\to \infty }\cdot \frac{n\bar{\theta}\theta_i + \log(n)}{n\bar{\theta}^2}
\end{align*}
over the event $E$. More detailed steps can be referred to derivations from  (\ref{leave-effect-2})-(\ref{leave-effect-4}). We thereby arrive at 
\begin{align}\label{21112520}
\Vert (X - I_n)   \hat{\Xi}_1\Vert 
&\leq  C \frac{\sqrt K }{\sqrt{\log (n)}}\, \widetilde{\kappa}_i \Big(1+  n\bar{\theta} \Vert (\tilde{H}^{(i)})^{-\frac 12} (   \hat{\Xi}_1 -\tilde{\Xi}^{(i)}_1O_3^{(i)})\Vert_{2\to \infty }\Big) 
\end{align}

Now we turn to study the term $ \Vert \tilde{H}^{(i)}(i,i)^{-\frac 12} W(i)  (\tilde{H}^{(i)})^{-\frac 12} \hat{\Xi}_1\Vert $. Using (\ref{21071502term1}), (\ref{21071502term2}) and (\ref{est:WtH}), we can deduce that
{\small
\begin{align}\label{21112521}
 \Vert \tilde{H}^{(i)}(i,i)^{-\frac 12} W(i) ( \tilde{H}^{(i)})^{-\frac 12} \hat{\Xi}_1\Vert  &\leq  \Vert \tilde{H}^{(i)}(i,i)^{-\frac 12} W(i)  (\tilde{H}^{(i)})^{-\frac 12} \tilde{\Xi}^{(i)}_1\Vert  +  \Vert \tilde{H}^{(i)}(i,i)^{-\frac 12} W(i)  (\tilde{H}^{(i)})^{-\frac 12} (  \overline{\Xi}^{(i)}_1 -\tilde{\Xi}^{(i)}_1O_4^{(i)}   ) \Vert \notag\\
 & \quad +  \Vert \tilde{H}^{(i)}(i,i)^{-\frac 12} W(i)  (\tilde{H}^{(i)})^{-\frac 12}(\hat{\Xi}_1- \overline{\Xi}^{(i)}_1O_5^{(i)}) \Vert\notag\\
 &\leq  C \sqrt K \, \widetilde{\kappa}_i \Big(1+  n\bar{\theta}  \Vert (\tilde{H}^{(i)})^{-\frac 12} (   \hat{\Xi}_1 -\tilde{\Xi}^{(i)}_1O_3^{(i)})\Vert_{2\to \infty } \Big)  \notag\\
 &\quad + C\Big(\sqrt K\, \frac{ \log (n)}{n\bar{\theta}^2} + \frac{1}{\sqrt{n\bar{\theta}^2}}\Big)\Vert  \hat{\Xi}_1- \overline{\Xi}^{(i)}_1O_5^{(i)} \Vert
\end{align}
}
over the event $E$.
Combining (\ref{21112520}) and (\ref{21112521}) into (\ref{211125200}) and putting all terms equipped with factor $\Vert  \hat{\Xi}_1- \overline{\Xi}^{(i)}_1O_5^{(i)} \Vert$ to the LHS, under the condition $K^3 \beta_n^{-2}\log (n)/n\bar{\theta}^2 = o(1)$ and $\lambda_1(PG)\leq CK$,  we finally see that
\begin{align*}
\Vert  \hat{\Xi}_1- \overline{\Xi}^{(i)}_1O_5^{(i)} \Vert \leq C  K^{\frac 32}\beta_n^{-1} \, \widetilde{\kappa}_i \Big(1+  n\bar{\theta} \Vert (\tilde{H}^{(i)})^{-\frac 12} (   \hat{\Xi}_1 -\tilde{\Xi}^{(i)}_1O_3^{(i)})\Vert_{2\to \infty } \Big) + \frac{C{K\beta_n^{-1}}  }{\sqrt{n\bar{\theta}^2} }\Vert \hat{\Xi}_1(i) - \tilde{\Xi}^{(i)}_1(i) O_3^{(i)}\Vert
\end{align*}
over the event $E$. Thus we complete the proof by considering all $i$'s.

\section{Rate of Mixed-SCORE-Laplacian}
We prove the error rate of Mixed-SCORE-Laplacian in this Section. In detail, in Section~\ref{subsec:pf_lem_hatR} we prove Lemma~\ref{lem:hatR}; in Section \ref{subsec: pf_thm_uppbd_pi}, we prove the first claim of Theorem~\ref{thm:uppbd_pi}; in Section \ref{subsec:pf_cor_rates_wunw}, we briefly state the proofs of Corollary~\ref{cor:rates-generalLoss} and the second claim of Theorem~\ref{thm:uppbd_pi}, as these arguments directly stem from Theorem~\ref{thm:uppbd_pi}.   

\subsection{Proof of Lemma \ref{lem:hatR}} \label{subsec:pf_lem_hatR}
Fix the choice of $\hat{\xi}_1$ such that $w=1$ in (\ref{result-1}). Choose the orthogonal matrix $O_1$ appeared in Theorem \ref{thm:eigenvector}. By definition, 
\begin{align*}
\Vert O_1'\hat{r}_i- r_i \Vert  = \Vert e_i' \big(\hat{\Xi}_1 O_1/ \hat{\xi}_1(i)-\Xi_1 /\xi_1(i)\big)\Vert & \leq \Vert e_i' (\hat{\Xi}_1 O_1- \Xi_1)/\hat{\xi}_1(i)\Vert + \Vert \Xi_1(i)\Vert \Big| \frac{1}{\hat{\xi_1}(i)}-\frac{1}{{\xi_1}(i)}\Big| 
\end{align*}
Employing Theorem \ref{thm:eigenvector} with Lemma \ref{lem:order}, for $i\in S_n(c_0)$, we have 
\begin{align*}
\Vert e_i' (\hat{\Xi}_1 O_1- \Xi_1)/\hat{\xi}_1(i)\Vert \leq C\frac{\Vert e_i' (\hat{\Xi}_1 O_1- \Xi_1)\Vert}{\xi_1(i)} \leq C\sqrt{\frac{K^3\log(n)}{n\bar{\theta}(\bar{\theta}\wedge\theta_i)\beta_n^2 }} 
\end{align*}
and 
\begin{align*}
\Vert \Xi_1(i)\Vert \Big| \frac{1}{\hat{\xi}_1(i)}-\frac{1}{{\xi}_1(i)}\Big| \leq C \frac{\Vert \Xi_1(i)\Vert}{\xi_1(i)}\cdot \frac{|\hat{\xi}_1(i)- \xi_1(i)|}{\xi_1(i)}  \leq  C\sqrt{\frac{K^3\log(n)}{n\bar{\theta}(\bar{\theta}\wedge\theta_i)\beta_n^2 }} 
\end{align*}
with probability $1-o(n^{-3})$ simultaneously for  $ i \in S_n(c_0)$  .
Combining the above inequalities, we immediately get (\ref{bound-hatR}) simultaneously for $ i \in S_n(c_0)$, with probability $1-o(n^{-3})$.

\subsection{Proof of the first claim in Theorem \ref{thm:uppbd_pi}} \label{subsec: pf_thm_uppbd_pi}

This theorem has two claims, one is about the node-wise error, and the other is about the bound for the $\ell^1$-loss. The second claim follows easily from the first claim, and its proof is relegated to Section~\ref{subsec:pf_cor_rates_wunw}. We now prove the first claim about the node-wise error. 

We only focus on $i\in \hat{S}_n(c)$ (see  \eqref{keepSet-VS}). For $i\notin \hat{S}_n(c)$, since we take trivial estimator  $K^{-1} \mathbf{1}_K$, the estimation error is then  trivially bounded by some constant. 
Recall the definition, for $i\in \hat{S}_n(c)$,
\begin{align*}
\hat{\pi}_i^*(k)=\max\{\hat{w}_i(k)/\hat{b}_1(k),\, 0\}, \qquad  \hat{\pi}_i=\hat{\pi}_i^*/\|\hat{\pi}_i^*\|_1 
\end{align*}
and correspondingly in the oracle case, $\pi_i=\pi_i^*/\|\pi_i^*\|_1, \pi_i^*=[\diag(b_1)]^{-1}w_i$. We shall study the errors of  $\hat{w}_i$'s and $\hat{b}_1$ compared to $w_i$'s and $b_1$ separately. 

Suppose we are under the high probability $1- o(n^{-3})$ event in which Lemma~\ref{lem:hatR} holds. We refrain ourselves from stating the high probability in the following derivations.
We first study $\hat{w}_i$'s.
Thanks to the choice of a variant of successive projection in our vertex hunting algorithm, referring to Lemma 3.1 of \cite{Mixed-SCORE}, it is easy to deduce that 
\begin{align}\label{est:vertex}
 \Vert \mathsf{P}\widehat{V} O_1- V\Vert_{2\to \infty}  \leq C \max_{i\in \hat{S}_n^*(c, \gamma)} \Vert O_1'\hat{r}_i- r_i \Vert \leq C \sqrt{\frac{K^3\log (n)}{n\bar{\theta}^2 \beta_n^2}}.
\end{align}
for some $K\times K$ permutation matrix $\mathsf{P}$, where we denote by $V= (v_1, v_2, \cdots, v_K)'$ and $\widehat{V}=(\hat{v}_1, \hat{v}_2, \cdots, \hat{v}_K)' $.
In our Mixed-SCORE-Laplacian algorithm, $\hat{w}_i$'s are solved from
\begin{align*}
\hat{Q} \hat{w}_i  = \left(
\begin{array}{c}
1\\
O_1'\hat{r}_i
\end{array}
\right), \qquad \hat{Q}:=  \left(
\begin{array}{cccc}
1& 1& \cdots & 1\\
O_1'\hat{v}_1 & O_1'\hat{v}_2 & \cdots& O_1'\hat{v}_K
\end{array}
\right) 
\end{align*}
Here, a little different from original linear system, we multiply $\hat{r}_i$ and $\hat{v}_1, \cdots,, \hat{v}_K$ by $O_1'$ on the left. Analogously, for the oracle case, 
\begin{align*}
{Q} {w}_i  = \left(
\begin{array}{c}
1\\
{r}_i
\end{array}
\right), \qquad {Q}:=  \left(
\begin{array}{cccc}
1& 1& \cdots & 1\\
{v}_1 & {v}_2 & \cdots& {v}_K
\end{array}
\right). 
\end{align*}
Note that since $v_j$'s, $\hat{v}_j$'s for $1\leq j \leq K$ are the vertices, we easily get that both $\hat{Q}$ and $Q$ are of full-rank.  Then, 
\begin{align}\label{21081503}
\Vert \mathsf{P}\hat{w}_i- w_i\Vert &=\left \Vert (\hat{Q}\mathsf{P}')^{-1}  \left(
\begin{array}{c}
1\\
O_1'\hat{r}_i
\end{array}
\right)
- Q^{-1} \left(
\begin{array}{c}
1\\
r_i
\end{array}
\right)
\right\Vert  \notag\\
&\leq  \left \Vert \Big((\hat{Q}\mathsf{P}')^{-1} - Q^{-1}\Big)\left(
\begin{array}{c}
1\\
{r}_i
\end{array}
\right)
\right\Vert   + \left\Vert  \hat{Q}^{-1}  \left[ \left(
\begin{array}{c}
1\\
O_1\hat{r}_i
\end{array}
\right)-
\left(
\begin{array}{c}
1\\
{r}_i
\end{array}
\right) \right]
\right\Vert.
\end{align}
For the first term on the RHS of (\ref{21081503}), we have
\begin{align*}
 \left\Vert \Big( (\hat{Q}\mathsf{P}')^{-1}  - Q^{-1}\Big)\left(
\begin{array}{c}
1\\
{r}_i
\end{array}
\right) \right\Vert  = \left\Vert \hat{Q}^{-1}  (\hat{Q}\mathsf{P}'- Q ) Q^{-1}
\left(
\begin{array}{c}
1\\
{r}_i
\end{array}
\right) \right\Vert = \big\Vert \hat{Q}^{-1} \big\Vert  \, \big\Vert ( \hat{Q} \mathsf{P}' - Q ) w_i\big\Vert,
\end{align*}
and 
\begin{align*}
 \big\Vert ( \hat{Q} \mathsf{P}'- Q  ) w_i\big\Vert = \Vert (O_1'\widehat{V}'\mathsf{P}'- V')w_i\Vert \leq 
 \Vert \mathsf{P}\widehat{V} O_1- V\Vert_{2\to \infty} 
\end{align*}
If we can claim that $\Vert \hat{Q}^{-1}\Vert\leq C$, then we are done with the bound of the first term. Notice that one easily check
\begin{align*}
\Vert  \hat{Q}\mathsf{P}'- Q \Vert \leq  \sqrt K \Vert \mathsf{P}\widehat{V} O_1- V\Vert_{2\to \infty} = o(\sqrt K)
\end{align*}
since $\frac{K^3 \log(n)}{n\bar{\theta}^2\beta_n^2}\to 0$ as $n\to\infty$.
Suppose that $\Vert  Q^{-1}\Vert \asymp K^{-\frac 12}$, then immediately $\Vert \hat{Q}^{-1} \Vert \asymp K^{-\frac 12}$. To claim that $\Vert  Q^{-1}\Vert \asymp K^{-\frac 12}$, we use the identity 
\begin{align} \label{formula:v_k}
v_k(t)= \frac{b_t(k) }{b_1(k)}, \quad  2\leq t \leq K,
\end{align}
which can be easily verified with some elementary derivations from the definition of $R$ and the fact $\Xi= H_0^{-\frac 12}\Theta \Pi B $ with $B= (b_1, \cdots, b_K)$ (see the proof of Lemma~\ref{lem:simplex} in Section \ref{sec:pf_simplex}).
We will see that
\begin{align*}
Q = B'{\rm diag} (1/b_1(1), \cdots, 1/b_1(K));
\end{align*}
 And due to $b_1(k)\asymp 1$ (claimed in the Proof of Lemma \ref{lem:order}), we then obtain that 
\begin{align*}
\Vert Q^{-1} \Vert = \Vert {\rm diag} (b_1(1), \cdots, b_1(K)) B^{-1}\Vert \asymp \Vert B^{-1} \Vert. 
\end{align*}
Further recall that $BB'= (\Pi' \Theta H_0^{-1} \Theta \Pi )^{-1}$. Hence, $\lambda_{\min}(BB')= 1/ \lambda_{\max} (\Pi' \Theta H_0^{-1} \Theta \Pi) \asymp K$, which leads to $\Vert Q^{-1} \Vert \asymp \Vert B^{-1} \Vert \asymp K^{-\frac 12}$. As a consequence, 
\begin{align*}
\left\Vert \Big( (\hat{Q}\mathsf{P}')^{-1}  - Q^{-1}\Big)\left(
\begin{array}{c}
1\\
{r}_i
\end{array}
\right) \right\Vert   \leq C K^{-\frac 12} \Vert \mathsf{P}\widehat{V} O_1- V\Vert_{2\to \infty}.
\end{align*}
Next, for the second term on the RHS of (\ref{21081503}), one simply bounds it by $\Vert \hat{Q}^{-1} \Vert \Vert O_1\hat{r}_i - r_i\Vert\leq C K^{-\frac 12} \Vert O_1\hat{r}_i - r_i\Vert $. Combining these two estimates into (\ref{21081503}), with the aids of (\ref{est:vertex}) and   Lemma~\ref{lem:hatR}, we conclude that 
\begin{align}\label{est:hw-w}
\Vert \mathsf{P}\hat{w}_i- w_i\Vert \leq C K^{-\frac 12} \Vert O_1\hat{r}_i - r_i\Vert\leq C \sqrt{\frac{K^2\log (n)}{n\bar{\theta} (\bar{\theta}\wedge \theta_i)\beta_n^2}}
. 
\end{align}
Next, we study the error between $1/e_k'\mathsf{P}\hat{b}_1$ and ${b}_1^{-1} (k)$. Here to the end of this section, with a little ambiguity of notation, we denote $\{e_k\}_{k=1}^K$ for the standard basis of $\mathbb{R}^K$.
By definition, since $P$ is a permutation matrix,  
\begin{align*}
\bigg|\frac{1}{\big(e_k'\mathsf{P}\hat{b}_1\big)^2}-\frac{1}{\big({b}_1(k)\big)^2}\bigg| \leq \big| \hat{\lambda}_1- \lambda_1\big| + \big| e_k'\mathsf{P} \widehat{V} \hat{\Lambda}_1\widehat{V}'\mathsf{P}' e_k -   {v}_k' {\Lambda}_1 {v}_k\big|.
\end{align*}
The eigenvalue difference is simply bounded by $\sqrt{\log (n)/n\bar{\theta}^2}$ by Weyl' inequality, which has been previously shown in entry-wise eigenvector analysis. To bound the second term above, we first claim $\Vert v_k\Vert\leq C\sqrt K$. To see this, using (\ref{formula:v_k}) and $b_1(k)\asymp 1$, 
$$\Vert v_k\Vert \leq  C \Vert e_k' B\Vert \leq   C\Vert BB'\Vert^{\frac 12} \leq C\sqrt K .$$
We can then derive that 
\begin{align*}
&\quad  \big| e_k'\mathsf{P} \hat{V} \hat{\Lambda}_1\hat{V}'\mathsf{P}' e_k -   {v}_k' {\Lambda}_1 {v}_k\big| \notag\\ 
&\leq \big| e_k'(\mathsf{P}\widehat{V} O_1- V)O_1' \hat{\Lambda}_1 O_1 O_1'\widehat{V}' \mathsf{P}' e_k\big| +  \big| v_kO_1'\hat{\Lambda}_1O_1(O_1'\widehat{V}' \mathsf{P}'- V)e_k  \big| + \big|{v}_k' (O_1'\hat{\Lambda}_1 O_1- \Lambda_1){v}_k\big| \notag\\
 &\leq C K^{-\frac 12} |\lambda_2(PG)|\Vert \mathsf{P}\widehat{V} O_1- V\Vert_{2\to \infty}  + \big|{v}_k' (O_1'\hat{\Lambda}_1 O_1- \Lambda_1){v}_k\big|
\end{align*}
Here in the last step, we used the trivial bound $\Vert \hat{\Lambda}_1\Vert\leq K^{-1} |\lambda_2(PG)|$. We further estimate the second term above. Notice that $O_1 = {\rm sgn}(\hat{\Xi}_1' \Xi_1)$ shown up in Theorem \ref{thm:eigenvector}. By $L_0 \Xi_1= \Xi_1 \Lambda_1$, ${L} \hat{\Xi}_1= \hat{\Xi}_1 \hat{\Lambda}_1$, and sine-theta theorem (\ref{sine-theta:perturb1}), 
\begin{align*}
 \big|{v}_k' (O_1'\hat{\Lambda}_1 O_1- \Lambda_1){v}_k\big| &\leq  \big|{v}_k' (O_1- \hat{\Xi}_1' \Xi_1)'\hat{\Lambda}_1 O_1{v}_k\big| + \big|{v}_k' \Xi_1'({L} - L_0) \hat{\Xi}_1 O_1{v}_k\big|\notag\\
 & \qquad + \big|{v}_k' \Lambda_1(O_1- \hat{\Xi}_1' \Xi_1)' O_1{v}_k\big| \notag\\
 & \leq C |\lambda_2(PG)|\Vert O_1- \hat{\Xi}_1' \Xi_1\Vert + K \Vert {L}- L_0\Vert \notag\\
 & \leq C |\lambda_2(PG)| (K\beta_n^{-1} \Vert {L}- L_0\Vert)^2 + K \Vert {L}- L_0\Vert\notag\\
 & \leq C \sqrt{\frac{K^3\log (n)}{n\bar{\theta}^2 \beta_n^2}}
\end{align*}
where we also used $\Vert{L}- L_0\Vert \leq C\sqrt{\log (n)/n\bar{\theta}^2}$ and $K^{\frac 32}\beta_n^{-1} \sqrt{\log (n)/n\bar{\theta}^2}= o(1)$. Next, using (\ref{est:vertex}) and the last inequality in Condition~\ref{reg-conds}(b), we bound
\begin{align*}
 K^{-\frac 12} |\lambda_2(PG)|\Vert \mathsf{P}\widehat{V} O_1- V\Vert_{2\to \infty}
& \leq C \sqrt{\frac{K^2\log (n)}{n\bar{\theta}^2 \beta_n^2}}\, . 
\end{align*}
It follows then
\begin{align*}
\big| e_k'\mathsf{P} \hat{V} \hat{\Lambda}_1\hat{V}'\mathsf{P}' e_k -   {v}_k' {\Lambda}_1 {v}_k\big| \leq C \sqrt{\frac{K^3\log (n)}{n\bar{\theta}^2 \beta_n^2}}\, .
\end{align*}
As a consequence,
\begin{align}\label{est:hb-b}
\Big|\frac{1}{(\mathsf{P}\hat{b}_1)(k)} - \frac{1}{ b_1(k)} \Big| \leq C\sqrt{\frac{K^3\log (n)}{n\bar{\theta}^2 \beta_n^2}}
\end{align}
since $b_1(k)\asymp 1$.

Now, we are able to study $\mathsf{P}\hat{\pi}_i^*$ and further $\mathsf{P}\hat{\pi}_i$ by (\ref{est:hw-w}) and (\ref{est:hb-b}). 
If $(\mathsf{P}\hat{w}_i)(k)\leq 0$, trivially we have 
$$\big| (\mathsf{P}\hat{\pi}_i^*)(k) - \pi_i^*(k)  \big| = \pi^*_i(k)\asymp w_i(k)\leq |(\mathsf{P} \hat{w}_i)(k) - {w}_i(k) |
$$
For the case that $(\mathsf{P}\hat{w}_i)(k)> 0$, we get the bound
\begin{align*}
\big| (\mathsf{P}\hat{\pi}_i^*)(k) - \pi_i^*(k)  \big| =\left| \frac{(\mathsf{P}\hat{w}_i)(k)}{(\mathsf{P}\hat{b}_1)(k)} - \frac{{w}_i(k)}{{b}_1(k)} \right| &\leq \big| (\mathsf{P}\hat{w}_i)(k)\big|\, \Big|\frac{1}{(\mathsf{P}\hat{b}_1)(k)} - \frac{1}{ b_1(k)} \Big| + \frac{|(\mathsf{P}\hat{w}_i)(k) - {w}_i(k) |}{|b_1(k)|}
\end{align*}
Moreover, taking sum over $k$ for both sides above, 
\begin{align*}
\Vert\mathsf{P} \hat{\pi}_i^*- \pi_i^*\Vert_1 &\leq C \max_{k}{\Big|\frac{1}{(\mathsf{P}\hat{b}_1)(k)} - \frac{1}{ b_1(k)} \Big|} + \frac{\Vert \mathsf{P}\hat{w}_i - w_i\Vert_1}{\min_k |b_1(k)|} \leq C \sqrt{\frac{K^3\log (n)}{n\bar{\theta} (\bar{\theta}\wedge \theta_i)\beta_n^2}}
\end{align*}
Here we used the Cauchy-Schwarz inequality $\Vert\mathsf{P} \hat{w}_i - w_i\Vert_1\leq \sqrt K\,  \Vert \mathsf{P}\hat{w}_i - w_i\Vert$ and further applied  (\ref{est:hw-w}) and (\ref{est:hb-b}).
As a result, 
\begin{align*}
\big| (\mathsf{P}\hat{\pi}_i)(k) - \pi_i(k)  \big| = \left| \frac{(\mathsf{P}\hat{\pi}_i^*)(k)}{ \Vert \mathsf{P}\hat{\pi}_i^*\Vert_1 }- \frac{\pi_i^*(k) }{\Vert \pi_i^*\Vert_1}\right|\leq \big| (\mathsf{P}\hat{\pi}_i^*)(k)\big|\, \frac{\Vert \mathsf{P}\hat{\pi}_i^*- \pi_i^*\Vert_1}{ \Vert \mathsf{P} \hat{\pi}_i^*\Vert_1\Vert \pi_i^*\Vert_1 } + \frac{|(\mathsf{P}\hat{\pi}_i^*)(k)- \pi_i^*(k)  |}{\Vert \pi_i^*\Vert_1}.
\end{align*}
And summing up over $k$ for both sides,  we can further have 
\begin{align*}
\Vert \mathsf{P}\hat{\pi}_i- \pi_i\Vert_1 \leq \frac{\Vert \mathsf{P}\hat{\pi}_i^*- \pi_i^*\Vert_1}{ \Vert \pi_i^*\Vert_1 } \leq C   \sqrt{\frac{K^3\log (n)}{n\bar{\theta} (\bar{\theta}\wedge \theta_i)\beta_n^2}}
\end{align*}
since $\Vert \pi_i^*\Vert_1 = \sum_{k} w_i(k)/b_1(k)\asymp 1$ by $b_1(k)\asymp 1$ for all $1\leq k\leq K$.
Therefore, we finished the proof.

\subsection{Proofs of Corollary~\ref{cor:rates-generalLoss} and the second claim of Theorem~\ref{thm:uppbd_pi}}\label{subsec:pf_cor_rates_wunw}
The proofs oare straightforward by employing Theorem \ref{thm:uppbd_pi}.  We shortly claim it below.
\begin{proof}[Proof of the second claim of Theorem~\ref{thm:uppbd_pi}]
Recall the definition of the $\ell^1$-loss ${\cal L}(\hat{\Pi},\Pi)$ in \eqref{Loss}. Employing the node-wise errors in Theorem \ref{thm:uppbd_pi} and taking average, we see that 
\begin{align*}
{\cal L}(\hat{\Pi},\Pi)  \leq C\sqrt{\log(n)} \int\min\Bigl\{ \frac{err_n}{\sqrt{t\wedge 1}},\, 1\Bigr\} dF_n(t),
\end{align*}
with probability $1-o(n^{-3})$. Further by the trivial bound ${\cal L}(\hat{\Pi},\Pi)  \leq 2$,  the high probability  error rate implies the  expected $\ell^1$-loss rate, i.e., the $err_n(\theta)$ in (\ref{OptRate}). This finishes the proof.
\end{proof}

\begin{proof}[Proof of Corollary \ref{cor:rates-generalLoss}]
Recall the loss metric ${\cal L}(\hat{\Pi},\Pi; p,q)$  in \eqref{generalLoss}. We crudely  bound 
\begin{align*}
\|T\hat{\pi}_i-\pi_i\|_q^q \leq C_q \Vert \hat{\pi}_i-\pi_i \Vert_1^ q   
\end{align*}
where $C_q$ is some constant depending on $q$. Combining it with the node-wise error rate in Theorem \ref{thm:uppbd_pi} gives Corollary~\ref{cor:rates-generalLoss}.

For the special case $p=1/2$ and $q=1$, we further bound 
\begin{align}\label{for:Lws}
{\cal L}^{w}(\hat{\Pi},\Pi) &=  \min_{T}\Bigl\{\frac{1}{n}\sum_{i=1}^n (\theta_i/\bar{\theta})^{1/2} \|T\hat{\pi}_i-\pi_i\|_1\Bigr\} \notag\\
&\leq  \min_{T}\Bigl\{\frac{1}{n}\sum_{i\in S_1} (\theta_i/\bar{\theta})^{1/2} \|T\hat{\pi}_i-\pi_i\|_1\Bigr\} +  \min_{T}\Bigl\{\frac{1}{n}\sum_{i\in S_2} (\theta_i/\bar{\theta})^{1/2} \|T\hat{\pi}_i-\pi_i\|_1\Bigr\}
\end{align}
where we recall the definition of $S_1, S_2$ in (\ref{def:S1S2}).
For the first term,  we use  Cauchy-Schwarz inequality and get 
\begin{align*}
\frac{1}{n}\sum_{i\in S_1} (\theta_i/\bar{\theta})^{1/2} \|T\hat{\pi}_i-\pi_i\|_1
&\leq  \Big( \frac 1n\sum_{i\in S_1} \theta_i/\bar{\theta}\Big)^{\frac 12} \Big( \frac 1n\sum_{i\in S_1}  \|T\hat{\pi}_i-\pi_i\|_1^2\Big)^{\frac 12} \notag\\
& \leq  C \sqrt{\log (n)}\, err_n
\end{align*}
with probability $1-o(n^{-3})$.
Plugging in the above inequality into (\ref{for:Lws}), and applying  the error rate in Theorem \ref{thm:uppbd_pi} separately for  $i\in S_2$, especially noticing that $\theta_i/\bar{\theta} \leq c\,  err_n^2 \log (n)$ for $i\notin \hat{S}_n(c)$, one can easily obtain 
\begin{align*}
{\cal L}^{w}(\hat{\Pi},\Pi) \leq  C \sqrt{\log (n)}\, err_n
\end{align*}
with probability $1-o(n^{-3})$. Further with trivial bound ${\cal L}^{w}(\hat{\Pi},\Pi) \leq C$, we then conclude the proof.
\end{proof}

\section{Least-favorable configurations and proof of the lower bound}  \label{sec:suppLB}
The key of proving the lower bound arguments in Theorem~\ref{thm:LB} and Theorem~\ref{thm:LB2} is to carefully construct the least-favorable configurations (LFC). 
The LFC for these two theorems are different. We start from the less complicated one, the LFC for the weighted $\ell^1$-loss ${\cal L}^w(\hat{\Pi},\Pi)$, and then modify it to construct the LFC for the standard $\ell^1$-loss ${\cal L}(\hat{\Pi},\Pi)$. The following notation is useful. 

\begin{de} \label{def:ParamClass}
Given $(n,K,  \beta_n)$, $\theta\in\mathbb{R}^n$ and $P\in\mathbb{R}^{K\times K}$, 
let ${\cal Q}_n(K,\theta,P, \beta_n)$ denote the collection of eligible membership matrices $\Pi$ such that Condition~\ref{reg-conds} is satisfied. 
\end{de}

First, we construct the LFC for proving the lower bound in Theorem~\ref{thm:LB2}. We take a special form of $P$,
\beq \label{P-star}
P^* = \beta_n I_K + (1-\beta_n){\bf 1}_K{\bf 1}_K', \quad \text{where $0<\beta_n<c<1$,} 
\eeq
and construct a collection of $\Pi$. We need a well-known result. 
\begin{lem}[Varshamov-Gilbert bound for packing numbers] \label{lem:packing}
For any $s\geq 8$, there exist $J\geq 2^{s/8}$ and $\omega^{(0)}, \omega^{(1)},\ldots,\omega^{(J)}\in\{0,1\}^s$ such that $\omega^{(0)}={\bf 0}_s$ and $\|\omega^{(j)}-\omega^{(k)}\|_1\geq s/8$, for all $0\leq j<k\leq J$. 
\end{lem}

In Theorem~\ref{thm:LB2}, we assume $F_n(err_n)\leq \check{c}$, for a constant $\check{c}\in (0,1)$. 
Let $c = \frac{1+\check{c}}{2}\in (0,1)$. Let $n_1=\lfloor K^{-1} cn\rfloor $and $n_0= n- Kn_1$. We set 
\beq \label{LBconstruct-1}
\Pi^*= \Big( \underbrace{\frac 1K {\bf 1}_K, \cdots, \frac 1K {\bf 1}_K}_{\substack{n_0}},\,  \underbrace{{\bf e}_1, \cdots, {\bf e}_1}_{\substack{n_1}}, \, \cdots, \underbrace{{\bf e}_K, \cdots, {\bf e}_K}_{\substack{n_1}}\Big)'. 
\eeq
Without loss of generality, we can assume that those  $\theta_i$'s corresponding to the pure nodes in $\Pi^*$ contains the top $\lfloor  (c- \check{c})n\rfloor$ degrees and they are evenly assigned to different communities such that the average degrees of the pure nodes in different communities are of the same order; we can also assume that the first $n_0$ $\theta_i$'s satisfy $\theta_i/\bar{\theta} \geq err_n^2$, by the assumption of $F_n(err_n^2)\leq \check{c}$. Note that we can always find a permutation to achieve such $\theta$ and re-construct  $\Pi^*$ correspondingly.
Let $m=\lfloor n_0/2\rfloor$ and $r =\lfloor K/2\rfloor$. We apply Lemma~\ref{lem:packing} to $s=mr$ to get $\omega^{(0)}, \omega^{(1)},\ldots,\omega^{(J)}$, where $J\geq 2^{(mr/8)}$. We re-arrange each $\omega^{(j)}$ to an $m\times r$ matrix row-wisely, denoted as $H^{(j)}$, and then construct $\Gamma^{(j)}\in\mathbb{R}^{n\times K}$ whose nonzero entries only appear in the top left  $(2m)\times (2r)$ block:
\beq \label{LBconstruct-2}
\Gamma^{(j)} = \left[\begin{array}{ccc}
H^{(j)} & - H^{(j)} & {\bf 0}_{m\times 1}\\
- H^{(j)} & H^{(j)} & {\bf 0}_{m\times 1}\\
{\bf 0}_{1\times r} & {\bf 0}_{1\times r} &0 \\
\hdashline
{\bf 0}_{n_1\times r} & {\bf 0}_{n_1\times r} & {\bf 0}_{n_1\times 1}\\
\end{array}\right], \qquad 0\leq j\leq J. 
\eeq
In \eqref{LBconstruct-2}, if $K$ is an even number, then the last column (consisting of zero entries) disappears; similarly, if $n_0$ is an even number, then the last row above the dashed line (consisting of zero entries) disappears. Let $\Theta=\diag(\theta_1,\theta_2,\ldots,\theta_n)$. We construct $\Pi^{(0)},\Pi^{(1)},\ldots,\Pi^{(J)}$ by 
\beq \label{LBconstruct-3}
\Pi^{(j)} = \Pi^* + \gamma_n \Theta^{-\frac12}\Gamma^{(j)}, \qquad \mbox{where}\quad \gamma_n= c_0K^{\frac12}(n\bar{\theta}\beta_n^2)^{-\frac12}, \quad\mbox{for $0\leq j\leq J$},
\eeq
where $c_0>0$ is a properly small constant. The following theorem is proved in the next section. 

\begin{thm} \label{thm:least-favorable1}
Fix $c_1$-$c_4$ in Condition~\ref{reg-conds} and $\check{c}$ in Theorem~\ref{thm:LB2}. 
Given any $(n,K, \alpha_n, \beta_n)$ and $\theta\in \mathbb{R}^n$ such that $F_n(err_n^2)\leq \check{c}$, let $P^*$ be as in \eqref{P-star}, and construct $\Pi^{(0)}, \Pi^{(1)},\ldots,\Pi^{(J)}$ as in \eqref{LBconstruct-1}-\eqref{LBconstruct-3}. When $c_0$ in \eqref{LBconstruct-3} is  properly small, the following statements are true. 
\begin{itemize} \itemsep 5pt
\item For any constant $c_5>0$, let ${\cal Q}_n(\theta,P^*)={\cal Q}_n(K, \theta,P^*, c_5\beta_n)$ (see Definition~\ref{def:ParamClass}). There exists a properly small $c_5$ such that 
$\Pi^{(j)}$ is contained in ${\cal Q}_n(\theta,P^*)$, for $0\leq j\leq J$.  
\item There exists a constant $C_1>0$ such that ${\cal L}^w(\Pi^{(j)}, \Pi^{(k)})\geq C_1err_n$, for all $0\leq j<k\leq J$. 
\item Let ${\cal P}_j$ be the probability measure of a DCMM model with $(\theta,P^*,\Pi^{(j)})$ and let $\mathrm{KL}(\cdot,\cdot)$ denote the Kullback-Leibler divergence. There exists a constant $\epsilon_1\in (0,1/8)$ such that  $\sum_{1\leq j\leq J} \mathrm{KL}(\mathcal{P}_j,\mathcal{P}_0)\leq (1/8-\epsilon_1) J\log(J)$.  
\end{itemize}
Furthermore, $\inf_{\hat{\Pi}}\sup_{\Pi\in {\cal Q}_n(\theta,P^*)}\mathbb{E}{\cal L}^w(\hat{\Pi},\Pi)\geq Cerr_n$. 
\end{thm}
\noindent
Theorem~\ref{thm:LB2} follows immediately from Theorem~\ref{thm:least-favorable1}.

\medskip

Next, we construct the LFC for proving the lower bound in Theorem~\ref{thm:LB}. We still take $P^*$ in \eqref{P-star} and construct a collection of $\Pi$. Compared with the previous case, the targeted lower bound now depends on $F_n(\cdot)$, so that the construction is more sophisticated. We separate two cases according to whether the following holds: 
\beq \label{cond-NoTrivial}
\int_0^{err_n^2}dF_n(t)\leq C\int_{err_n^2}^\infty \frac{err_n}{\sqrt{t_n\wedge 1}}dF_n(t). 
\eeq
To understand \eqref{cond-NoTrivial}, note that $\theta_i/\bar{\theta}\leq err_n^2$ is equivalent to $n\bar{\theta}\theta_i\beta_n^2\leq K^3\log(n)$. For such a node $i$, the best estimator is the naive estimator $\hat{\pi}_i^{naive}=\frac{1}{K}{\bf 1}_K$. In \eqref{cond-NoTrivial}, the left hand side is the total contribution of these nodes in ${\cal L}(\hat{\Pi},\Pi)$, and the right hand side is the contribution of remaining nodes. 
Therefore, \eqref{cond-NoTrivial} guarantees that the rate of convergence of the unweighted $\ell^1$-loss is driven by those nodes for which we can indeed construct non-trivial estimators of $\pi_i$ from data.
When \eqref{cond-NoTrivial} is violated, the lower bound can be proved by similar techniques but simpler least-favorable configurations. The details of this case is relegated in the next section and below we will focus on the case that (\ref{cond-NoTrivial}) holds.  Note that all examples in Section~\ref{subsec:theta-class} satisfy \eqref{cond-NoTrivial}. 

We need a technical lemma about the property of $F_n(\cdot)$ that satisfies the requirement in Theorem~\ref{thm:LB}. It is proved in the next section.
\begin{lem} \label{lem:Fn}
Fix $\rho>0$ and $a_0\in (0,1)$. 
Given any $\theta\in {\cal G}_n(\varrho, a_0)$ (see Definition~\ref{def:thetaClass}), recall that $F_n(\cdot)$ is the empirical distribution associated with $\eta_i=\theta_i/\bar{\theta}$, $1\leq i\leq n$. Let $\tilde{F}_n$ be the empirical distribution associated with $\tilde{\eta}_i=\eta_i\wedge 1$. For any $c>0$ and $\epsilon\in (0,1)$, define
\[
\tau_n(c,\epsilon) =\inf \Bigl\{t>0: \, \int_{err_n^2}^t d\tilde{F}_n(x)\geq (1-\epsilon)\int_{err_n^2}^{c} d\tilde{F}_n(x) \Bigr\}. 
\]
If $F_n(\cdot)$ satisfies \eqref{cond-NoTrivial}, then there exists a number $c_n>err_n^2$ and a constant $\tilde{a}_0\in (0,1)$ such that  $F_n(c_n)\leq 1-\tilde{a}_0$ and 
 \beq \label{Fn-key-claim}
\int_{\tau_n(c_n, 1/8)}^{c_n} \frac{1}{\sqrt{t\wedge 1}}d F_n(t) + \frac{\lceil n\cdot \varpi_n \rceil}{n\sqrt{\tau_n(c_n, 1/8)\wedge 1}}\;\geq \;\tilde{a}_0 \int_{err_n^2}^{\infty}\frac{1}{\sqrt{t\wedge 1}}d F_n(t),
 \eeq
 where $ \varpi_n =[F_n(c_n) - F_n(err_n^2-)]/ 8-[ F_n (c_n) - F_n (\tau_n(c_n, 1/8))]$ and for any distribution function $F(\cdot)$, we define $F(x-) = \lim_{\omega \to 0} F(x- \omega)$. 
\end{lem} 

We now construct a collection of $\Pi$ using $c_n$ in Lemma~\ref{lem:Fn}. 
We re-order $\theta_i$'s such that 
\beq \label{LBconstruct-4}
\theta_{(1)} \leq \theta_{(2)}\leq\ldots\leq \theta_{(n)}.
\eeq
From the way $\tilde{\eta}_i$'s are defined, this ordering also implies that $\tilde{\eta}_{(1)}\leq \tilde{\eta}_{(2)}\leq\ldots\leq\tilde{\eta}_{(n)}$. 
Let $c_n$ be as in Lemma~\ref{lem:Fn}. Define
\[
s_n=\max\{1\leq i\leq n: \, \tilde{\eta}_{(i)}\leq err_n^2\}, \qquad n_0 = \max\{1\leq i\leq n: \, \tilde{\eta}_{(i)}\leq c_n\}-s_n.
\]
It follows from the definition of $\tilde{F}_n$ that $n_0$ is approximately the total number of $\tilde{\eta}_i$'s such that $err_n^2< \tilde{\eta}_i\leq c_n$. 
It can be derived from the definition of $\tau_n(c_n,1/8)$ and $\varpi_n$ that $n[F_n(c_n) - F_n(\tau_n(c_n, 1/8))]  + \lceil n \varpi_n\rceil = n_0 - \lfloor 7n_0/8\rfloor$. Combining these claims with \eqref{Fn-key-claim} gives
\[
\frac{1}{n}\sum_{\lfloor 7n_0/8\rfloor<i- s_n\leq n_0}\frac{1}{\sqrt{\eta_{(i)}\wedge 1}}\; \gtrsim \; \int_{err_n^2}^{\infty}\frac{1}{\sqrt{t\wedge 1}}d F_n(t). 
\]
We multiply $err_n$ on both hand sides. 
By the condition \eqref{cond-NoTrivial}, we have $err_n \int_{err_n^2}^{\infty} \frac{1}{\sqrt{t\wedge 1}}d F_n(t)\geq C^{-1} \int \min\{\frac{err_n}{\sqrt{t\wedge 1}},1\}dF_n(t)$, which yields a lower bound for the right hand side. 
For the left hand side, we plug in $\eta_i=\theta_i/\bar{\theta}$. 
It follows that
\[
\sqrt{\frac{K^3}{n\bar{\theta}\beta_n^2}}\cdot \frac{1}{n} \sum_{\lfloor 7n_0/8\rfloor<i-s_n\leq n_0}\frac{1}{\sqrt{\theta_{(i)}\wedge \bar{\theta}}}\; \gtrsim \;  \int \min\Bigl\{\frac{err_n}{\sqrt{t\wedge 1}}, 1\Bigr\}d F_n(t). 
\]
Notice that for each individual $i$ such that $\lfloor 7n_0/8\rfloor<i-s_n\leq n_0$, its contribution to the left hand side sum above is negligible since $n^{-1}/\sqrt{\theta_{(i)}} \leq n^{-1}/ \sqrt{\bar{\theta} err_n^2} = n^{-1/2}K^{-3/2} \bar{\theta}^{1/2} \beta_n= o(1)$. Therefore, we can remove finitely many $i$ from the left hand side sum without changing the inequality above. This further implies 
\[
\sqrt{\frac{K^3}{n\bar{\theta}\beta_n^2}}\cdot \frac{1}{n} \sum_{\lfloor 7n_0/8\rfloor +2 <i-s_n\leq n_0}\frac{1}{\sqrt{\theta_{(i)}\wedge \bar{\theta}}}\; \gtrsim \;  \int \min\Bigl\{\frac{err_n}{\sqrt{t\wedge 1}}, 1\Bigr\}d F_n(t). 
\] 
Let ${\cal M}_0$ be the index set of the nodes ordered between $s_n$ and $s_n+n_0$ in \eqref{LBconstruct-4}. Let $\gamma_n= c_0K^{\frac12}(n\bar{\theta}\beta_n^2)^{-\frac12}$. The above implies that
\beq \label{LBconstruct-5}
\gamma_n\inf_{\substack{{\cal M}\subset{\cal M}_0,\\|{\cal M}|\geq n_0/8  -2 }}\biggl\{\frac{1}{n}\sum_{i\in {\cal M}}\frac{1}{\sqrt{\theta_i\wedge \bar{\theta}}}\biggr\} \; \gtrsim \;  K^{-1} \int \min\Bigl\{\frac{err_n}{\sqrt{t\wedge 1}}, 1\Bigr\}d F_n(t). 
\eeq
The set ${\cal M}_0$ plays a key role in the construction of the least-favorable configurations. 
We now re-arrange nodes by putting nodes in ${\cal M}_0$ as the first $n_0$ nodes, with the last $n-n_0$ nodes ordered in a way such that the average degrees of the pure nodes in different communities of $\Pi^*$  are of the same order (such an ordering always exists). After node re-arrangement, we construct $\Pi^*$ and $\Gamma^{(0)},\Gamma^{(1)},\ldots,\Gamma^{(J)}$ in the same way as in \eqref{LBconstruct-1}-\eqref{LBconstruct-2}. Let $\tilde{\theta}_i=\theta_i\wedge \bar{\theta}$ and 
$\widetilde{\Theta}=\diag(\tilde{\theta}_1,\tilde{\theta}_2,\ldots,\tilde{\theta}_n)$. Let  
\beq \label{LBconstruct-6}
\Pi^{(j)} = \Pi^* + \gamma_n \widetilde{\Theta}^{-\frac12}\Gamma^{(j)},\qquad \mbox{for}\quad 0\leq j\leq J, 
\eeq
where $\gamma_n= c_0K^{\frac12}(n\bar{\theta}\beta_n^2)^{-\frac12}$ is the same as in \eqref{LBconstruct-5}. The following theorem is an analog of Theorem~\ref{thm:least-favorable1} for the unweighted $\ell^1$-loss and is proved in next section.

\begin{thm} \label{thm:least-favorable2}
Fix $c_1$-$c_4$ in Condition~\ref{reg-conds} and $(\varrho,a_0)$ in Theorem~\ref{thm:LB}. 
Given $(n,K, \beta_n)$ and $\theta\in {\cal G}_n(\varrho, a_0)$,  let $P^*$ be as in \eqref{P-star}, and construct $\Pi^{(0)}, \Pi^{(1)},\ldots,\Pi^{(J)}$ as in \eqref{LBconstruct-6}. When $c_0$ in \eqref{LBconstruct-3} is  properly small, the following statements are true. 
\begin{itemize} \itemsep 5pt
\item For any constant $c_5>0$, let ${\cal Q}_n(\theta,P^*)={\cal Q}_n(K, \theta,P^*, c_5\beta_n)$. There exists a properly small $c_5$ such that 
$\Pi^{(j)}$ is contained in ${\cal Q}_n(\theta,P^*)$, for $0\leq j\leq J$.  
\item There exists a constant $C_2>0$ such that ${\cal L}(\Pi^{(j)}, \Pi^{(k)})\geq C_2\int\min\{\frac{err_n}{\sqrt{t\wedge 1}}, 1\}dF_n(t)$, for all $0\leq j<k\leq J$. 
\item Let ${\cal P}_j$ be the probability measure of a DCMM model with $(\theta,P^*,\Pi^{(j)})$ and let $\mathrm{KL}(\cdot,\cdot)$ denote the Kullback-Leibler divergence. There exists a constant $\epsilon_2\in (0,1/8)$ such that  $\sum_{1\leq j\leq J} \mathrm{KL}(\mathcal{P}_j,\mathcal{P}_0)\leq (1/8-\epsilon_2) J\log(J)$.  
\end{itemize}
Furthermore, $\inf_{\hat{\Pi}}\sup_{\Pi\in {\cal Q}_n(\theta, P^*)}\mathbb{E}{\cal L}(\hat{\Pi},\Pi)\geq C\int\min\{\frac{err_n}{\sqrt{t\wedge 1}}, 1\}dF_n(t)$. 
\end{thm}
\noindent
Theorem~\ref{thm:LB} follows immediately from Theorem~\ref{thm:least-favorable2}. 

\smallskip
\noindent
{\bf Remark}: In Theorems~\ref{thm:least-favorable1}-\ref{thm:least-favorable2}, we fix $P=P^*$ and prove the lower bounds by taking supreme over a class of $\Pi$. Such lower bounds are not only $\theta$-specific but also $P$-specific, and they are stronger than the $\theta$-specific lower bounds in Theorems~\ref{thm:LB} and \ref{thm:LB2}. In Section \ref{subsec:low_P_extend} of the supplementary material, we show that we can prove such $P$-specific lower bounds for an arbitrary $P$ if one of the following holds as $n\to\infty$: 
(a) $(K,P)$ are fixed; (b) $(K,P)$ can depend on $n$, but $K\leq C$ and $P{\bf 1}_K\propto {\bf 1}_K$;
(c) $(K,P)$ can depend on $n$, and $K$ can be unbounded, but $P{\bf 1}_K\propto {\bf 1}_K$ and $|\lambda_2(P)|\leq C\beta_n=o(1)$.

\section{Proofs in lower bound analysis}\label{sec:lowerbd_pf}
In this section, we complete the proofs of lower bounds, i.e., Theorems~\ref{thm:LB} and \ref{thm:LB2}. To this end, we will show the proofs of Theorems~ \ref{thm:least-favorable1}-\ref{thm:least-favorable2}   and Lemma \ref{lem:Fn} stated in Section \ref{sec:suppLB}. We organize this section  as follows: 
In Section \ref{a.subsec:lb1}, we provide the proof of Theorem \ref{thm:least-favorable1} regarding weighted loss metric ${\cal L}^w(\hat{\Pi},\Pi)$. In Section \ref{a.subsec:lb2}, we claim Lemma \ref{lem:Fn} and prove Theorem \ref{thm:least-favorable2} under the condition (\ref{cond-NoTrivial}). The proof of Theorem \ref{thm:least-favorable2} with (\ref{cond-NoTrivial}) violated is relatively simpler and we state it in Section \ref{sub:pr_vio} for completeness.  In Section \ref{subsec:low_P_extend}, we shortly show how to extend the lower bounds to $P$-specific case under some certain additional assumptions. This supports our arguments in the Remark in the end of Section \ref{sec:suppLB}. 


Throughout this section, we will use $\mathcal{C}_{p, k}$ to denote  the index set collecting indices of  the pure nodes in $k$-th community for $1\leq k \leq K$.

\subsection{Proof of Theorem \ref{thm:least-favorable1}} \label{a.subsec:lb1}

We begin with the proof of the  first claim. We first verify $\Pi^{(j)}\in {\cal Q}_n(\theta,P^*)$, for every $0\leq j\leq J$ which are constructed in \eqref{LBconstruct-1}-\eqref{LBconstruct-3}. By the  definition of perturbation matrix $\Gamma^{(j)}$'s in (\ref{LBconstruct-2}), and the fact that $\gamma_n/\sqrt{\theta_i}\leq c_0/K $ for all $1\leq i\leq n_0$ due to $\theta_i/\bar{\theta} \geq err_n^2$ for all $1\leq i\leq n_0$, it is easy to see that $\Pi^{(j)}$'s are indeed membership matrices when choosing small $c_0$. Next, we check Condition~\ref{reg-conds}.  Note that Condition~\ref{reg-conds}(d) and the last inequality in Condition~\ref{reg-conds}(a) immediately hold because of the construction of $\Pi^*$.  By definition, $G^{(j)} =K (\Pi^{(j)} )' \Theta H_0^{-1} \Theta \Pi^{(j)}$ and 
\begin{align}\label{eq:Gj-G*}
\Vert G^{(j)}  - G^*  \Vert \leq 2\Vert  G^*\Vert^{\frac 12} \Vert K \gamma_n^2 (\Gamma^{(j)})'  \Theta^{\frac 12} H_0^{-1} \Theta^{\frac 12}    \Gamma^{(j)}  \Vert^{\frac 12 }    +  \Vert K \gamma_n^2 (\Gamma^{(j)})'  \Theta^{\frac 12} H_0^{-1} \Theta^{\frac 12}    \Gamma^{(j)}  \Vert.
\end{align}
Elementary computations lead to 
\begin{align*}
  G^*  = \Big(\sum_{i=1}^{n_0} \frac{\theta_i^2}{H_0(i,i)}\Big) \frac 1K \mathbf{1}_K\mathbf{1}_K' + K {\rm diag} \Big( \sum_{i\in \mathcal{C}_{p,1}}\frac{\theta_i^2}{H_0(i,i)} \, , \cdots, \sum_{i\in \mathcal{C}_{p,K}}\frac{\theta_i^2}{H_0(i,i)} \Big)\, . 
\end{align*} 
By our construction and assumptions on $\Pi^*$ and $\theta$, 
it can be derived from $\sum_{i=1}^n \theta_i^2/H_0(i,i) \asymp 1$ that 
\begin{align*}
K \sum_{i\in \mathcal{C}_{p, k}}\frac{\theta_i^2}{H_0(i,i)} \asymp 1
 \end{align*}
 for all $1\leq k \leq K$. It follows that $\Vert G^*\Vert\leq c $ and $\Vert (G^*)^{-1}\Vert\leq c$ for some constant $c$.  Furthermore, one can also derive 
 \begin{align}\label{22011001}
 \Vert  K\gamma_n^2 (\Gamma^{(j)})'  \Theta^{\frac 12} H_0^{-1} \Theta^{\frac 12}    \Gamma^{(j)} \Vert \leq c_0^2 err_n^2 = o(1)
 \end{align}
 following from $\sum_{i=1}^n \theta_i/H_0(i,i) \leq 1/\bar{\theta}$  and 
$
|e_i \Gamma^{(j)} x| \leq \sqrt K 
$
for all $1\leq i\leq n$ and any unit vector $x\in \mathbb{R}^K$. Therefore,  by Weyl's inequality and (\ref{eq:Gj-G*}), we can conclude that the first two inequalities in Condition~\ref{reg-conds}(a) hold for all $G^{(j)}$'s.  Further with our choice of special $P^*$ which satisfies that ${\bf 1}_K' P  =\big(K - (K-1)\beta_n \big){\bf1}_K'> (1-c) K {\bf 1}_K$, $\lambda_1(P^*) \asymp K$, and $\lambda_k(P^*)= \beta_n$ for all $2\leq k \leq K$,  the requirements that $|\lambda_K(PG)|\geq c_5\beta_n$ for some $c_5>0$ hold for $P^*G^*$ and $P^* G^{(j)}$'s.  The eigengap condition, Condition~\ref{reg-conds}(b), holds for $P^*G^*$. Condition~\ref{reg-conds}(b) also holds for $P^* G^{(j)}$'s, as a result of the Weyl's inequality, where  
\begin{align*}
|\lambda_1(P^*G^{(j)})  \lambda_2(P^*G^{(j)})| \leq \sigma_1(P^*G^{(j)}) \sigma_2(P^*G^{(j)}) \leq C K \beta_n 
\end{align*}
and $\lambda_1(P^*G^{(j)}) \asymp K$. Here we use $\sigma_1(P^*G^{(j)}), \sigma_2(P^*G^{(j)})$ to denote the first and second largest singular values of $P^*G^{(j)}$; and the right hand side upper bound is due to the expression $P^* G^{(j)} = \beta_n G^{(j)} + K(1-\beta_n) \frac 1K \mathbf{1}_K\mathbf{1}_K' G^{(j)} $ with the fact that  $\Vert G^{(j)}\Vert\leq c $ and $\Vert (G^{(j)})^{-1}\Vert\leq c$ for some constant $c$.
%

   Lastly, we claim that Condition~\ref{reg-conds}(c) holds for all $G^{(j)}$'s. Using \textit{Perron's theorem}, we obtain that the first right eigenvector of $P^* G^{(j)}$ is positive for all $1\leq j\leq J$. In particular, for $P^*G^*$, all of its entries are positive and  $\asymp 1$. We then claim that $P^* G^{(j)}(i,k) \asymp 1$ for all $1\leq i,k\leq K$. To see this, for each $1\leq i,k\leq K$, we first write    \begin{align*}
   P^*G^{(j)}  (i,k) - P^* G^*(i,k)  &= K \gamma_n e_i' P^* (\Gamma^{(j)})' \Theta^{\frac 12} H_0^{-1} \Theta \Pi^* e_k  + K \gamma_n e_i' P^* (\Pi^*)' \Theta H_0^{-1} \Theta^{\frac 12} \Gamma^{(j)} e_k \notag\\
    & \quad + K \gamma_n^2 e_i' P^* (\Gamma^{(j)})' \Theta^{\frac 12} H_0^{-1} \Theta^{\frac 12} \Gamma^{(j)}e_k.
   \end{align*}
   Note that $P^*(\Gamma^{(j)})' = \beta_n (\Gamma^{(j)})'$ by the definition of $\Gamma^{(j)}$ such that $\mathbf{1}_K'(\Gamma^{(j)})' = 0 $. We thus easily bound the first and third terms on the RHS above by 
   \begin{align*}
&  | K \gamma_n e_i' P^* (\Gamma^{(j)})' \Theta^{\frac 12} H_0^{-1} \Theta \Pi^* e_k |\leq \beta_n \Vert G^*\Vert ^{\frac 12} \Vert K \gamma_n^2 (\Gamma^{(j)})'  \Theta^{\frac 12} H_0^{-1} \Theta^{\frac 12}    \Gamma^{(j)} \Vert^{\frac 12} \leq c \beta_n err_n\notag\\
 &|K \gamma_n^2 e_i' P^* (\Gamma^{(j)})' \Theta^{\frac 12} H_0^{-1} \Theta^{\frac 12} \Gamma^{(j)} e_k | \leq \beta_n \Vert K \gamma_n^2 (\Gamma^{(j)})'  \Theta^{\frac 12} H_0^{-1} \Theta^{\frac 12}    \Gamma^{(j)} \Vert \leq c\beta_n err_n^2
   \end{align*}
   which are both of order $o(1)$. For the second term, by the definition of $\Pi^*$, we have 
   \begin{align*}
  | K \gamma_n e_i' P^* (\Pi^*)' \Theta H_0^{-1} \Theta^{\frac 12} \Gamma^{(j)} e_k | & \leq  |K\gamma_n \beta_n e_i'(\Pi^*)' \Theta H_0^{-1} \Theta^{\frac 12} \Gamma^{(j)} e_k  | + |K\gamma_n (1-\beta_n)  \mathbf{1}_n' \Theta H_0^{-1} \Theta^{\frac 12} \Gamma^{(j)} e_k|  \notag\\
  & \leq c \beta_n err_n + c  K \gamma_n \sum_{i=1}^n \frac{\sqrt{\theta_i}}{n\bar{\theta}} \notag\\
  & \leq c \, err_n
   \end{align*}
   where we used Cauchy-Schwarz inequality $\sum_{i=1}^n \sqrt{\theta_i}\leq \sqrt n (\sum_{i=1}^n \theta_i )^{1/2}= n \sqrt{\bar{\theta}}$ in the last step. Therefore, it follows from the above discussions that 
   \begin{align*}
   P^*G^{(j)}  (i,k) = P^*G^*(i,k)  + \big( P^*G^{(j)}  (i,k) - P^*G^*(i,k) \big)  = P^*G^*(i,k)  + o(1) \asymp 1
   \end{align*} 
for all $1\leq i,k \leq K$. 
 As a result,  $\min_{i,k}P^*G^{(j)}(i,k)/ \max_{i,k}P^*G^{(j)}(i,k) >c$ for some constant $c>0$. Then, $\eta^{(j)}_1$, the first right eigenvector of $P^*G^{(j)}$, satisfies
\begin{align}\label{22010901}
\frac{\min_{k}  \eta^{(j)}_1(k)} {\max_{k}  \eta^{(j)}_1(k)}\geq \frac{\min_{i,k}P^*G^{(j)}(i,k) \sum_{k}  \eta^{(j)}_1(k) }{ \max_{i,k}P^*G^{(j)}(i,k)\sum_{k}  \eta^{(j)}_1(k) } >c.
\end{align}
This proves  Condition~\ref{reg-conds}(c).
We then conclude the proof of first statement.

Next, we proceed to prove the second statement, the pairwise difference between $\Pi^{(j)}$'s under the weighted loss metric. By definition, 
\begin{align*}
{\cal L}^w(\Pi^{(j)}, \Pi^{(k)}) &= \frac 1n \sum_{i=1}^n (\theta_i/\bar{\theta})^{\frac 12} \Vert \pi_i^{(j)}-  \pi_i^{(k)}\Vert_1 = \frac {c_0}{n} \sum_{i=1}^{n_0} \frac{\sqrt{K}}{ \beta_n\sqrt{n\bar{\theta}^2} }\Vert { e}_i \big(\Gamma^{(j)}- \Gamma^{(k)}\big)\Vert_1 \notag\\
&=  \frac{c_0\sqrt{K}}{ \beta_n\sqrt{n\bar{\theta}^2} } \cdot \frac4n \Vert \omega^{(j)}- \omega^{(k)}\Vert_1 \geq C_1 err_n,
\end{align*}
for some constant $C_1>0$.

In the last part, we prove the third claim in Theorem \ref{thm:least-favorable1} regarding the KL divergence statement.  Note that 
\begin{align*}
KL(\mathcal{P}_\ell,\mathcal{P}_0)=\sum_{1\leq i<j\leq n}\Omega^{(\ell)}_{ij}\log(\Omega^{(\ell)}_{ij}/\Omega^{(0)}_{ij}) + \big(1- \Omega^{(\ell)}_{ij}\big )  \log\frac{1- \Omega^{(\ell)}_{ij}}{1- \Omega^{(0)}_{ij}}.
\end{align*} 
Notice that $\Omega_{ij}^{(\ell)} = \Omega_{ij}^{(0)}$ for $n_0< i<j \leq n$. Only the pairs satisfying $0<i< j\leq n_0$ and $0<i\leq n_0<j\leq n$ have the contributions. We then write $KL(\mathcal{P}_\ell,\mathcal{P}_0)= (I)+ (II)$ where 
\begin{align}
&(I): = \sum_{0<i< j\leq n_0}\Omega^{(\ell)}_{ij}\log(\Omega^{(\ell)}_{ij}/\Omega^{(0)}_{ij}) + \big(1- \Omega^{(\ell)}_{ij}\big )  \log\frac{1- \Omega^{(\ell)}_{ij}}{1- \Omega^{(0)}_{ij}}\label{def:I}\\
&(II):=  \sum_{0<i\leq n_0< j\leq n}\Omega^{(\ell)}_{ij}\log(\Omega^{(\ell)}_{ij}/\Omega^{(0)}_{ij}) + \big(1- \Omega^{(\ell)}_{ij}\big )  \log\frac{1- \Omega^{(\ell)}_{ij}}{1- \Omega^{(0)}_{ij}} \label{def:II}
\end{align}
We begin with the estimate of $(I)$. For simplicity, we write $\Gamma^{(\ell)}= \big(\Gamma_1^{(\ell)}, \cdots, \Gamma_n^{(\ell)}\big)'$ for $\ell=0, \cdots, J$. By definition,  for $0<i< j\leq n_0$, 
\begin{align*}
&\Omega_{ij}^{(0)} = \theta_i \theta_j \Big( \frac{1}{K^2} {\bf 1}_K' P{\bf 1}_K\Big)  = \theta_i\theta_j \big(1- (1- 1/K) \beta_n\big)<1;
\end{align*}
and 
\begin{align*}
\Omega_{ij}^{(\ell)} &= \theta_i \theta_j \big(\frac 1K {\bf 1}_K + \frac{\gamma_n}{\sqrt{\theta_i}}\Gamma_i^{(\ell)} \big)' P \big(\frac 1K {\bf 1}_K + \frac{\gamma_n}{\sqrt{\theta_j}}\Gamma_j^{(\ell)} \big) \notag\\
&=\theta_i\theta_j \big(1- (1- 1/K) \beta_n\big) + \theta_i\theta_j \frac{\gamma_n^2\beta_n}{\sqrt{\theta_i\theta_j}} \big(\Gamma_i^{(\ell)}\big)' \Gamma_j^{(\ell)} \notag\\
&= \Omega_{ij}^{(0)} \Big(1+ \Delta_{ij}^{(\ell)}\Big), \qquad \qquad j \neq 0, 
\end{align*}
in which, $\Delta_{ij}^{(\ell)}:= (\gamma_n^2/\sqrt{\theta_i\theta_j})\,\cdot\beta_n \big(\Gamma_i^{(\ell)}\big)' \Gamma_j^{(\ell)}/ [1-(1-1/K)\beta_n]$. Here we used the identity $\mathbf{1}_K ' \Gamma^{(\ell)}_j =0$ for all $1\leq j\leq n $, $1\leq \ell \leq J$. Further by $1-(1-1/K)\beta_n >c$ for some constant $c>0$and the assumption that the first $n_0$ $\theta_i$'s satisfy $\theta_i/\bar{\theta} \geq err_n^2$, we notice that 
\begin{align*}
\max_{0<i,j\leq n_0} |\Delta_{ij}^{(\ell)}|\leq  C\frac{c_0^2K^2}{\beta_n\sqrt{(n\bar{\theta}\theta_i) (n\bar{\theta}\theta_j)}} \leq C c_0^2K^{-1} \beta_n 
\end{align*}
and $ \Omega_{ij}^{(0)}  \Delta_{ij}^{(\ell)} < c (1-  \Omega_{ij}^{(0)}) $ for some constant $c\in (0,1)$ when $c_0$ is small enough. 
Choosing sufficiently small $c_0$, we  have the Taylor expansions 
\begin{align*}
\Omega^{(\ell)}_{ij}\log(\Omega^{(\ell)}_{ij}/\Omega^{(0)}_{ij}) =  \Omega_{ij}^{(0)} \big(1+ \Delta_{ij}^{(\ell)}\big) \log \big(1+ \Delta_{ij}^{(\ell)} \big) = \Omega_{ij}^{(0)} \big( \Delta_{ij}^{(\ell)} + \frac 12 (\Delta_{ij}^{(\ell)})^2 + O((\Delta_{ij}^{(\ell)})^3)\big)
\end{align*}
and 
\begin{align*}
\big(1- \Omega^{(\ell)}_{ij}\big )  \log\frac{1- \Omega^{(\ell)}_{ij}}{1- \Omega^{(0)}_{ij}} &= \big(1- \Omega_{ij}^{(0)}  - \Omega_{ij}^{(0)} \Delta_{ij}^{(\ell)} \big) \log \Big(1-\frac{\Omega_{ij}^{(0)}}{1- \Omega^{(0)}_{ij}} \Delta_{ij}^{(\ell)} \Big)\notag\\
&=- \Omega^{(0)}_{ij} \Delta_{ij}^{(\ell)} + \frac{(\Omega^{(0)}_{ij} )^2}{2(1-\Omega^{(0)}_{ij} ) } ( \Delta_{ij}^{(\ell)} )^2 + O\Big(\frac{(\Omega^{(0)}_{ij} \Delta_{ij}^{(\ell)} )^3}{(1-\Omega^{(0)}_{ij} )^2}\Big)
\end{align*}
Combining the above two equations together into (\ref{def:I}), we arrive at
\begin{align}\label{deri:I1}
(I) & \leq  \sum_{0<i<j\leq n_0} \frac{\Omega^{(0)}_{ij}}{(1-\Omega^{(0)}_{ij})} ( \Delta_{ij}^{(\ell)} )^2 = \sum_{0<i<j\leq n_0}\frac{\gamma_n^4 \beta_n^{2} [\big(\Gamma_i^{(\ell)}\big)' \Gamma_j^{(\ell)}]^2}{[1-(1-1/K)\beta_n](1- \Omega^{(0)}_{ij})} \notag\\
& \leq C \gamma_n^4\beta_n^{2} n_0^2K^2 \leq Cc_0^4 n_0 K\cdot \frac{K^3}{\beta_n^2n\bar{\theta}^2} \leq  Cc_0^4 n_0 K
\end{align}
Here we used $K^3/\beta_n^2(n\bar{\theta}^2)\leq c$ for some sufficiently small $c>0$ and the crude bound $|\big(\Gamma_i^{(\ell)}\big)' \Gamma_j^{(\ell)}| \leq K$, which follows from the definition of $\Gamma^{(\ell)}$ in (\ref{LBconstruct-2}). 

In the sequel, we turn to study $(II)$. Since $0<i\leq n_0< j\leq n$, we suppose that $j\in \mathcal{C}_{p,\hat{j}}$ for some $1\leq \hat{j}\leq K$. Then, 
\begin{align*}
&\Omega_{ij}^{(0)} = \theta_i \theta_j \Big( \frac{1}{K} {\bf 1}_K' P{\bf e}_{\hat{j}}\Big)= \theta_i\theta_j \big(1- (1- 1/K) \beta_n\big);
\end{align*}
and 
\begin{align*}
\Omega_{ij}^{(\ell)} &= \theta_i \theta_j \big(\frac 1K {\bf 1}_K + \frac{\gamma_n}{\sqrt{\theta_i}}\Gamma_i^{(\ell)} \big)' P {\bf e}_{\hat{j}} \notag\\
&=\theta_i\theta_j \big(1- (1- 1/K) \beta_n\big) + \theta_i\theta_j \frac{\gamma_n\beta_n}{\sqrt{\theta_i}} \big(\Gamma_i^{(\ell)}\big)' {\bf e}_{\hat{j}} \notag\\
&= \Omega_{ij}^{(0)} \Big(1+ \widetilde{\Delta}_{ij}^{(\ell)}\Big)
\end{align*}
with $\widetilde{\Delta}_{ij}^{(\ell)}:= ( \gamma_n/\sqrt{\theta_i} ) \cdot \beta_n \big(\Gamma_i^{(\ell)}\big)' {\bf e}_{\hat{j}} / \big(1- (1- 1/K) \beta_n\big)$. Similarly, one can easily check that 
\begin{align*}
 \max_{i,j} |\widetilde{\Delta}_{ij}^{(\ell)}| \leq \frac{Cc_0\sqrt K}{\sqrt{n\bar{\theta}\theta_i}} \leq Cc_0 K^{-1} \beta_n
\end{align*}
by our assumption that the first $n_0$ $\theta_i$'s satisfy $\theta_i/\bar{\theta} \geq err_n^2$.
Therefore, in the same way as (\ref{deri:I1}), we can derive 
\begin{align}\label{deri:II1}
(II) & \leq  \sum_{0<i\leq n_0<j\leq n} \frac{\Omega^{(0)}_{ij}}{(1-\Omega^{(0)}_{ij})}
(\widetilde{ \Delta}_{ij}^{(\ell)} )^2 =  \sum_{\begin{subarray}{c}0<i\leq n_0,\\ j\in \mathcal{C}_{p, \hat{j}}, 1\leq \hat{j}\leq K \end{subarray}}\frac{ \theta_j\gamma_n^2 \beta_n^2 [\big(\Gamma_i^{(\ell)}\big)' {\bf e}_{\hat{j}}]^2}{[1-(1-1/K)\beta_n](1- \Omega^{(0)}_{ij})} \notag\\
& \leq C\Big(\sum_{j=n_0+1}^n \theta_j\Big) \gamma_n^2 \beta_n^2 n_0\leq C c_0^2n_0K. 
\end{align}
We now combine (\ref{deri:I1}) and (\ref{deri:II1}). They imply that  
\begin{align*}
\sum_{\ell=1}^J KL(\mathcal{P}_\ell,\mathcal{P}_0) \leq C c_0^2 J nK.
\end{align*}
Here $C$ is a constant independent of choice of $c_0$ and $n$. At the same time, since $J\geq 2^{\lfloor n_0/2\rfloor \times\lfloor K/2\rfloor/8}$, we obtain that $\log J\geq c nK$ for some constant $c>0$ not relying on the other parameters. By properly choosing $c_0$, we can always find a constant $\epsilon_1\in (0, 1/8)$ such that $C c_0^2 J nK\leq Cc_0^2/c \cdot J\log (J)\leq (1/8- \epsilon_1)J\log (J) $. This finishes the proof of  the last claim. Furthermore,  with standard techniques of lower bound analysis (e.g., \cite[Theorem 2.5]{tsybakov2009introduction1}), we ultimately obtain the lower bound stated in Theorem \ref{thm:least-favorable1}.

\subsection{Proof of Lemma \ref{lem:Fn} and Theorem \ref{thm:least-favorable2}} \label{a.subsec:lb2}

\begin{proof} [Proof of Lemma \ref{lem:Fn}]
Recall  Definition~\ref{def:thetaClass}. For such $c_n$, $\varrho$ and $a_0$,  we see that if $\tau_n(c_n, 1/8)\geq \varrho c_n$, then by definition of $\tau_n(c_n, 1/8)$ and $\varpi_n $, 
\begin{align*}
\int_{\tau_n(c_n, 1/8)}^{c_n} \frac{1}{\sqrt{t\wedge 1}} {\rm d} F_n(t)  + \frac{\lceil n\cdot \varpi_n \rceil}{n\sqrt{\tau_n(c_n, 1/8)\wedge 1}}\; &\geq \frac{1}{8\sqrt{c_n\wedge 1}}  \tilde{F}_n(c_n) \notag\\
&\geq \frac{\sqrt{\varrho}}{8\sqrt{\varrho c_n\wedge 1}}  \tilde{F}_n(c_n)\geq \frac{\sqrt{\varrho}}{8}\int_{\varrho c_n}^{c_n} \frac{1}{\sqrt{t\wedge 1}}{\rm d} F_n(t)\;. 
\end{align*}
By Definition~\ref{def:thetaClass}, we conclude 
\begin{align*}
\int_{\tau_n(c_n, 1/8)}^{c_n} \frac{1}{\sqrt{t\wedge 1}} {\rm d} F_n(t) \geq  \tilde{a}_0 \int_{err_n^2}^{\infty}\frac{1}{\sqrt{t\wedge 1}} {\rm d} F_n(t), \quad \tilde{a}_0:=  \frac{\sqrt{\varrho}}{8} a_0.
\end{align*}
In the case that $\tau_n(c_n, 1/8)< \rho c_n$, trivially, (\ref{Fn-key-claim}) holds with $\tilde{a}_0= a_0$.
\end{proof}
With the help of Lemma \ref{lem:Fn}, we are able to prove Theorem \ref{thm:least-favorable2} below.

\begin{proof}[Proof of Theorem \ref{thm:least-favorable2}]
Since the least-favorable configurations $\Pi^*$ and $\Pi^{(j)}$'s are quite similar to  those for the weighted loss metric, only with slightly different perturbation scales. Such differences will not affect the regularity conditions. In fact,  one can simply verify the  regularity conditions in the same manner as the first part of the proof of Theorem \ref{thm:least-favorable1}. We thus conclude the first statement without details.

Next, for the pairwise difference under unweight loss metric, by definition, 
\begin{align*}
{\cal L}(\Pi^{(j)},\Pi^{(k)}) = \frac 1n \sum_{i=1}^n \Vert \pi_i^{(j)} - \pi_i^{(k)}\Vert _1 &= \frac 1n \sum_{i=1}^{n_0} \frac{\gamma_n}{\sqrt{\theta_{i} \wedge \bar{\theta}}} \Vert {\bf e}_i'\big( \Gamma^{(j)} - \Gamma^{(k)}\big)\Vert_1 
\end{align*}
Note that $  \Vert H^{(j)} - H^{(k)}\Vert_1 \geq \lfloor n_{0}/2\rfloor \times\lfloor K/2\rfloor/8$. At least $  \lfloor n_{0}/2\rfloor/8 $ rows of $ H^{(j)} - H^{(k)}$ will contribute to the RHS term above. Since the construction of $\Gamma^{(j)}$ based on $H^{(j)}$, it is not hard to see 
that 
\begin{align*}
\frac {\gamma_n}{n} \sum_{i=1}^{n_0} \frac{1}{\sqrt{\theta_{i} \wedge \bar{\theta}}} \Vert {\bf e}_i'\big( \Gamma^{(j)} - \Gamma^{(k)}\big)\Vert_1 & \geq \frac {\gamma_n}{n} \min_{\substack{{\cal M}\subset{\cal M}_0,\\|{\cal M}|= 2\lfloor\lfloor n_{0}/2\rfloor /8\rfloor}} \sum_{i\in \mathcal{M}}  \frac{K-1}{\sqrt{\theta_{i} \wedge \bar{\theta}}} \notag\\
& \geq \frac {\gamma_n}{n} \min_{\substack{{\cal M}\subset{\cal M}_0,\\|{\cal M}|= \lfloor n_{0}/8 \rfloor -1}} \sum_{i\in \mathcal{M}}  \frac{K-1}{\sqrt{\theta_{i} \wedge \bar{\theta} }} \notag\\
& \gtrsim \;   \int \min\Bigl\{\frac{err_n}{\sqrt{t\wedge 1}}, 1\Bigr\} { \rm d} F_n(t)
\end{align*}
where the last step is due to (\ref{LBconstruct-5}). This concludes the second statement. 

In the end, we briefly state the proof of  the KL divergence bound since it is quite analogous to the counterpart proof  of Theorem  \ref{thm:least-favorable1}. We again define $(I)$ and $(II)$ as (\ref{def:I})-(\ref{def:II}), and bound them separately. Thanks to  the slight difference on the perturbation scale, one can simply mimic the proof  of Theorem  \ref{thm:least-favorable1} and obtain the following bounds under current settings.
\begin{align*}
(I) & \leq  \sum_{1\leq i< j \leq n_0} \frac{\Omega^{(0)}_{ij}}{(1-\Omega^{(0)}_{ij})}
( \Delta_{ij}^{(\ell)} )^2 = \sum_{ 1\leq i< j \leq n_0}\frac{\frac{\theta_i}{\theta_i\wedge \bar{\theta} } \, \frac{\theta_j}{\theta_j\wedge \bar{\theta} }\, \gamma_n^4 \beta_n^{2} [\big(\Gamma_i^{(\ell)}\big)' \Gamma_j^{(\ell)}]^2}{[1-(1-1/K)\beta_n](1- \Omega^{(0)}_{ij})} \notag\\
& \leq C \gamma_n^4\beta_n^{2} n_{0}^2K^2 \leq Cc_0^4 n_{0} K\cdot \frac{K^3}{\beta_n^2n\bar{\theta}^2} \leq  Cc_0^4 n_{0} K
\end{align*}
 and 
 \begin{align*}
(II) & \leq  \sum_{0<i\leq n_0< j\leq n} \frac{\Omega^{(0)}_{ij}}{(1-\Omega^{(0)}_{ij})}
(\widetilde{ \Delta}_{ij}^{(\ell)} )^2 = \sum_{\begin{subarray}{c}0<i\leq n_0,\\ j\in \mathcal{C}_{p, \hat{j}}, 1\leq \hat{j}\leq K \end{subarray}}\frac{ \frac{\theta_i}{\theta_i\wedge \bar{\theta} }\,  \theta_j \gamma_n^2 \beta_n^2 [\big(\Gamma_i^{(\ell)}\big)' {\bf e}_{\hat{j}}]^2}{[1-(1-1/K)\beta_n](1- \Omega^{(0)}_{ij})} \notag\\
& \leq C\Big(\sum_{j=n_0+1}^n \theta_j\Big) \gamma_n^2 \beta_n^2 n_{0}\leq C c_0^2n_{0}K. 
\end{align*}
Here to obtain the two upper bounds above, we used an estimate 
\begin{align*}
\sum_{1\leq i\leq n_0} \frac{\theta_i}{\theta_i\wedge \bar{\theta}} \leq 
\sum_{1\leq i\leq n_0} \Big(1+ \frac{\theta_i}{\bar{\theta}} \Big)\leq 2 c n_0
\end{align*}
for some constant $c>0$, where the last step is due to our ordering of $\theta_i$'s and  $\theta_i/\bar \theta \leq c_n \leq C$ for some constant $C>0$, $1\leq i \leq n_0$, which follows from Definition \ref{def:thetaClass} and self-normalization of $F_n(\cdot)$ (i.e., $\int t dF_n(t) = 1$).
As a result, 
 \begin{align*}
\frac{1}{J}\sum_{\ell=1}^J KL(\mathcal{P}_\ell,\mathcal{P}_0) \leq C c_0^2 n_{0}K\leq \tilde{C} c_0^2 \log J.
\end{align*}
Properly  choosing sufficiently  small $c_0$, we thus complete the third statement.
Furthermore, by standard techniques of lower bound analysis (e.g., \cite[Theorem 2.5]{tsybakov2009introduction1}), we ultimately obtain the lower bound stated in Theorem \ref{thm:least-favorable2}.

\end{proof}

\subsection{Proof of Theorem~\ref{thm:least-favorable2} without (\ref{cond-NoTrivial})  } \label{sub:pr_vio}
In this section, we show the proof of  Theorem~\ref{thm:least-favorable2} in the case that (\ref{cond-NoTrivial})  violates. We will need a distinct sequence of   least-favorable configurations. We still order $\theta_i$'s as (\ref{LBconstruct-4}). But we define 
\beq
n_0 = \max\{1\leq i \leq n: \theta_i/\bar{\theta} \leq err_n^2\} \notag
\eeq 
which means $n_0$ is the total number of $\eta_i$'s such that $0<\eta_i\leq err_n^2$. For the remaining $n-n_0$ nodes, we order them in the way that the average degrees of the pure nodes in different communities of $\Pi^*$  are of the same order as before. $\Pi^*$ and $\Gamma^{(0)},\Gamma^{(1)},\ldots,\Gamma^{(J)}$are constructed in the same way as in \eqref{LBconstruct-1}-\eqref{LBconstruct-2}.  Different from (\ref{LBconstruct-6}), let 
\beq \label{def:Pi(j)_vio}
\Pi^{(j)} = \Pi^* + {c_0} K^{-1} \Gamma^{(j)},\qquad \mbox{for}\quad 0\leq j\leq J.  
\eeq
First,
following the first part of proof of Theorem \ref{thm:least-favorable1}, we will see that $G^*, P^*G^*$ satisfy the regularity conditions in Condition~\ref{reg-conds}. Especially, in this case,  we still have 
$$ K \sum_{i\in \mathcal{C}_{p, k}}\frac{\theta_i^2}{H_0(i,i)} \asymp \int_{err_n^2}^{\infty} t d F_n(t)\geq  \int_{0}^{\infty}  t d F_n(t) - err_n^2 \asymp 1 $$
 Furthermore, one can derive
\[
 \Vert  c_0^2 K^{-2}(\Gamma^{(j)})'  \Theta H_0^{-1} \Theta    \Gamma^{(j)}  \Vert \leq c_0^2. 
\]
Similarly to the analog in the proof of Theorem~\ref{thm:least-favorable1}
, by choosing properly small $c_0$, we have the regularity conditions hold for $P^*G^{(j)}$'s as well. Since the proofs are quite similar, we hence omit the details.

Second, under the construction (\ref{def:Pi(j)_vio}), 
\begin{align*}
{\cal L}(\Pi^{(j)},\Pi^{(k)}) = \frac 1n \sum_{i=1}^n \Vert \pi_i^{(j)} - \pi_i^{(k)}\Vert _1 &= \frac {c_0}{nK} \sum_{i=1}^{n_0}\Vert {\bf e}_i'\big( \Gamma^{(j)} - \Gamma^{(k)}\big)\Vert_1 \notag\\
&= \frac{4c_0}{nK} \Vert H^{(j)}- H^{(k)}\Vert_1 \notag\\
&\geq C\frac{n_0}{n}
\end{align*}
for some constant $C$ not relying on the other parameters. Notice that $n_0/n = \int_{0}^{err_n^2} {\rm d} F_n(t)$. Since (\ref{cond-NoTrivial})  violates, we thus conclude that 
$$
{\cal L}(\Pi^{(j)},\Pi^{(k)})>C  \int_{0}^{err_n^2} {\rm d} F_n(t)\geq C_3\int\min\{\frac{err_n}{\sqrt{t\wedge 1}}, 1\}{\rm d} F_n(t), 
$$
for some constant $C_3>0$.
Third, we claim the KL divergence in the same way as previously, $KL(\mathcal{P}_\ell,\mathcal{P}_0)= (I)+ (II)$ and  $(I)$, $(II)$ are defined in (\ref{def:I})-(\ref{def:II}). By our least-favorable configurations (\ref{def:Pi(j)_vio}), we bound 
\begin{align*}
(I) & \leq  \sum_{1\leq i< j \leq n_0} \frac{\Omega^{(0)}_{ij}}{(1-\Omega^{(0)}_{ij})}
( \Delta_{ij}^{(\ell)} )^2 = \sum_{ 1\leq i< j \leq n_0}\frac{\theta_i \theta_j  c_0^4K^{-4} \beta_n^2 [\big(\Gamma_i^{(\ell)}\big)' \Gamma_j^{(\ell)}]^2}{[1-(1-1/K)\beta_n](1- \Omega^{(0)}_{ij})} \notag\\
& \leq Cc_0^4 \Big(\sum_{i=1}^{n_0} \theta_i\Big)^2 \beta_n^{2} K^{-2} \leq Cc_0^4 n_{0} K\cdot \frac{K^3}{\beta_n^2n\bar{\theta}^2} \leq  Cc_0^4 n_{0} K
\end{align*}
where in this case $\Delta_{ij}^{(\ell)}=c_0^2K^{-2} \beta_n  \big(\Gamma_i^{(\ell)}\big)' \Gamma_j^{(\ell)}/(1-(1-1/K)\beta_n)$, and we used the fact that $\theta_i\leq K^3\beta_n^{-2}/(n\bar{\theta})$ for all $1\leq i \leq n_0$ to obtain  the second inequality on the second row;
 and 
 \begin{align*}
(II) & \leq  \sum_{0<i\leq n_0< j\leq n} \frac{\Omega^{(0)}_{ij}}{(1-\Omega^{(0)}_{ij})}
(\widetilde{ \Delta}_{ij}^{(\ell)} )^2 = \sum_{\begin{subarray}{c}0<i\leq n_0,\\ j\in \mathcal{C}_{p, \hat{j}}, 1\leq \hat{j}\leq K \end{subarray}}\frac{\theta_i  \theta_j c_0^2 K^{-2} \beta_n^2 [\big(\Gamma_i^{(\ell)}\big)' {\bf e}_{\hat{j}}]^2}{[1-(1-1/K)\beta_n](1- \Omega^{(0)}_{ij})} \notag\\
& \leq Cc_0^2\Big(\sum_{i=1}^{n_0} \theta_i\Big) \Big(\sum_{j=n_0+1}^n \theta_j\Big)K^{-2} \beta_n^2 \leq C c_0^2n_{0}K
\end{align*}
where $\widetilde{\Delta}_{ij}^{(\ell)}:=c_0K^{-1}\cdot \beta_n \big(\Gamma_i^{(\ell)}\big)' {\bf e}_{\hat{j}} / \big(1- (1- 1/K) \beta_n\big)$ for this case. Combining the upper bounds for $(I)$ and $(II)$, we finally get 
\beq
 \sum_{1\leq \ell\leq J} \mathrm{KL}(\mathcal{P}_\ell,\mathcal{P}_0)\leq C c_0^2 J n_0K\leq (1/8-\epsilon_3) J\log(J)
\eeq
 for a constant $\epsilon_3\in (0,1/8)$,  by choosing sufficiently small $c_0$ and noting $n_0K \asymp \log (J)$.
 
 In conclusion, we proved the analogs of the three claims in Theorem \ref{thm:least-favorable2} when (\ref{cond-NoTrivial})  violates. Further  by standard techniques of lower bound analysis (e.g., \cite[Theorem 2.5]{tsybakov2009introduction1}), we ultimately obtain the lower bound.

\subsection{Extension to $P$-specific lower bounds} \label{subsec:low_P_extend}
In this subsection, we study the $P$-specific  lower bounds of ${\cal L}^w(\hat{\Pi},\Pi)$ and ${\cal L}(\hat{\Pi},\Pi)$ for arbitrary $P$ if  one of the following condition holds as $n\to \infty$:
\begin{itemize}
\item [(a)] $(K,P)$ are fixed;
\item  [(b)] $(K,P)$ can depend on $n$, but $K\leq C$ and $P{\bf 1}_K\propto {\bf 1}_K$;
\item [(c)] $(K,P)$ can depend on $n$, and $K$ can be unbounded, but $P{\bf 1}_K\propto {\bf 1}_K$ and $|\lambda_2(P)|\leq C\beta_n=o(1)$. 
\end{itemize}
Since the proofs are quite analogous to the case of the special $P$ in the manuscript, in the sequel, we point out the key differences compared to the proofs for Theorems~\ref{thm:least-favorable1}-\ref{thm:least-favorable2}, and shortly state how to adapt the proofs in the previous subsections to the current cases.
\begin{itemize}
\item[(a)] If  $K=K_0$ and  $P=P_0$, for a fixed integer $K_0\geq 2$ and a fixed matrix $P_0$, we can simplify the construction of $\Gamma^{(j)}$'s and hence the configurations $\Pi^{(j)}$'s. More specifically, we apply Lemma~\ref{lem:packing} to $n_0$ to get $\omega^{(0)}, \omega^{(1)}, \dots, \omega^{(J)}$, where $J\geq 2^{n_0/8}$. We insert $\omega^{(j)}$'s into $n$-dim vectors $\gamma^{(j)}$'s such that 
\beq
(\gamma^{(j)})' = \big( \, (\omega^{(j)})', \mathbf{0}_{1\times (n-n_0)} \big). 
\eeq
Let $\eta\in \mathbb{R}^{K_0}$ be a nonzero vector such that 
\beq
\eta' \mathbf{1}_{K_0}= 0, \qquad \eta' P_0 \mathbf{1}_{K_0} = 0, \qquad \Vert \eta\Vert_1\asymp K_0 
\eeq
Such $\eta$ always exists by solving certain linear system. Based on these notations, we re-define $\Gamma^{(j)} = \gamma^{(j)}\eta'$ and re-define $\Pi^{(j)}$ correspondingly as (\ref{LBconstruct-3}) for weighted loss, (\ref{LBconstruct-6}) or (\ref{def:Pi(j)_vio}) for unweighted loss. The verifications of regularity conditions and pairwise difference between the configurations can be claimed in the same way as in the proofs of Theorems~\ref{thm:least-favorable1}-\ref{thm:least-favorable2}. The most distinguishing part appears in the KL divergence. Especially, for $0<i<j\leq n_0$, 
\beq
\Omega_{ij}^{(\ell)} =  \Omega_{ij}^{(0)} \Big(1+ \Delta_{ij}^{(\ell)}\Big), \quad \Delta_{ij}^{(\ell)}\propto \gamma^{(\ell)}(i) \gamma^{(\ell)}(j) \eta'P\eta  \label{def:delta_in_cor1}
\eeq
where the coefficients we did not specify for $\Delta_{ij}$'s rely on the perturbation scale we take from (\ref{LBconstruct-3}), or (\ref{LBconstruct-6}), or (\ref{def:Pi(j)_vio}). Similarly for $0<i\leq n_0<j\leq n$, if $j\in \mathcal{C}_{p, \hat{j}}$,
\beq
\Omega_{ij}^{(\ell)} =  \Omega_{ij}^{(0)} \Big(1+ \widetilde{\Delta}_{ij}^{(\ell)}\Big), \quad \widetilde{\Delta}_{ij}^{(\ell)}\propto \gamma^{(\ell)}(i)  \eta'P e_{\hat{j}}. \label{def:delta_in_cor2}
\eeq
Nevertheless, in this case, $\eta'P\eta  \asymp 1$ and $|\eta'P e_{\hat{j}}|\leq C$. In particular, $\beta_n\asymp 1$. All of these facts lead to similar derivations on upper bounds of $(I)$ and $(II)$ (see definitions in  (\ref{def:I})-(\ref{def:II})). One can claim the desired upper bounded for KL divergence for the least-favorable configurations we constructed here. One can conclude the proof by mimicking the proofs of Theorems~\ref{thm:least-favorable1}-\ref{thm:least-favorable2}.

\item[(b)] If both $(K,P)$ may depend on $n$, but they satisfy that $K\leq C$ and $P{\bf 1}_K\propto {\bf 1}_K$.  We take the same simplified least-favorable configurations as in Case (a). The regularity conditions and pairwise difference can be claimed likewisely. ${\bf 1}_K$ is an eigenvector of $P$. We can take special $\eta$, the eigenvector associated  to the smallest eigenvalue (in magnitude) of $P$. In (\ref{def:delta_in_cor1}) and (\ref{def:delta_in_cor2}), we have $\eta'P\eta  \asymp \beta_n$ and $|\eta'P e_{\hat{j}}| \leq  C\beta_n$, which fit the arguments in the proofs of  KL divergence for  Theorems~\ref{thm:least-favorable1}-\ref{thm:least-favorable2}. Thereby, we can prove the KL divergence in the same way as the proofs of Theorems~\ref{thm:least-favorable1}-\ref{thm:least-favorable2}.

\item[(c)]If  both $(K,P)$ may depend on $n$ and $K$ can be unbounded, but $P{\bf 1}_K\propto {\bf 1}_K$ and $|\lambda_2(P)|\leq C\beta_n=o(1)$. 
We adopt the same least-favorable configurations in Section \ref{sec:suppLB} correspondingly to Theorems~\ref{thm:least-favorable1}-\ref{thm:least-favorable2}. The different parts only appear in the quantities involving $P$. Notice that in this case, ${\bf 1}_K $ is the eigenvector associated to the largest eigenvalue of $P$, 
\beq
(\Gamma_i^{(\ell)})' P {\bf 1}_K \propto (\Gamma_i^{(\ell)})'  {\bf 1}_K = 0 , \qquad (\Gamma_i^{(\ell)})' P\Gamma_j^{(\ell)}\asymp \beta_n
\eeq
since the other eigenvalues of $P$ are asymptotically of order $\beta_n$. These two estimates exactly coincide with the ones in the proofs of  Theorems~\ref{thm:least-favorable1}-\ref{thm:least-favorable2}. Then, all the arguments in  the proofs of  Theorems~\ref{thm:least-favorable1}-\ref{thm:least-favorable2} can be directly applied in this setting. We thereby conclude our proof. 
\end{itemize}

\section{Analysis for a general $b$} \label{sec:choose-b}

In Section~\ref{subsec:Methods}, we explained the rationale of choosing $b=1/2$. An important claim there is that for a general $b$, subject to a column permutation of $\hat{\Pi}$, it holds simultaneously that 
\beq \label{SNR-in-supp}
\|\hat{\pi}_i-\pi_i\|_1\lesssim  \frac{C\sqrt{\log(n)}}{\sqrt{n\theta_i\bar{\theta}}} \cdot \delta(b,F_n), \qquad\mbox{where } \delta(b, F_n)= \frac{\sqrt{\int t^{3-4b}dF_n(t)}}{\int t^{2-2b}dF_n(t)}. 
\eeq
In this section, we discuss why \eqref{SNR-in-supp} is true. 

Let $H$ be the same as in \eqref{Def:Laplacian}, and let $H_0=\mathbb{E}H$. 
Given a fixed $b\geq 0$, let 
\[
L= H^{-b} AH^{-b}, \qquad L_0 = H_0^{-b}\Omega H_0^{-b}. 
\]
For $1\leq k\leq K$, let $\hat{\lambda}_k$ be the $k$th largest eigenvalue (in magnitude) of $L$, and let 
$\hat{\xi}_k\in\mathbb{R}^n$ be the corresponding eigenvector. Similarly, let $\lambda_k$ be the $k$th largest eigenvalue (in magnitude) of $L_0$, and let 
$\xi_k\in\mathbb{R}^n$ be the corresponding eigenvector. 
From the proof of Theorem~\ref{thm:uppbd_pi}, we can see that the node-wise error $\|\hat{\pi}_i-\pi_i\|_1$ is closely related to the following quantity: 
\beq \label{ER_i}
\mathrm{ER}_i:=\max_{1\leq k\leq K}\biggl\{\frac{|\hat{\xi}_k(i)-\xi_k(i)|}{\xi_1(i)}\biggr\}, 
\eeq
subject to a rotation matrix $O_1\in\mathbb{R}^{(K-1)\times (K-1)}$ applied to the columns of $\hat{\Xi}_1=[\hat{\xi}_2,\hat{\xi}_3,\ldots,\hat{\xi}_K]$. 
To study $\mathrm{ER}_i$, we introduce a proxy to $\hat{\xi}_k$. Note that $\hat{\lambda}_k\hat{\xi}_k=H^{-b}AH^{-b}\hat{\xi}_k$, from the  definition of eigenvalues and eigenvectors. It implies $\hat{\xi}_k=\hat{\lambda}_k^{-1}H^{-b}AH^{-b}\hat{\xi}_k$. 
We replace $(H, \hat{\xi}_k)$ on the right hand side by $(H_0, \xi_k)$ to obtain
\beq \label{xi-star}
\hat{\xi}_k^* := {\lambda}_k^{-1}H_0^{-b}AH_0^{-b}\xi_k, \qquad 1\leq k\leq K. 
\eeq
We then similarly define 
\beq \label{SNR_i_star}
\mathrm{ER}^*_i:=\max_{1\leq k\leq K}\biggl\{\frac{|\hat{\xi}^*_k(i)-\xi_k(i)|}{\xi_1(i)}\biggr\}. 
\eeq
The difference between $\mathrm{ER}_i^*$ and $\mathrm{ER}_i$ is governed by $\zeta:=\hat{\xi}^*_1-\hat{\xi}_1\in\mathbb{R}^n$. How to control the effect of $\zeta$ is the central topic of entry-wise eigenvector analysis. In Section~\ref{sec:Entrywise-proof} and Section~\ref{sec:EntrywiseAnalysis-in-supp}, we give rigorous analysis in the case of $b=1/2$. The analysis of a general $b$ follows a similar vein and is omitted here (the regularity conditions may be slightly different). From now on, we assume that $|\mathrm{ER}_i-\mathrm{ER}^*_i|$ is negligible and focus on studying $\mathrm{ER}_i^*$ in \eqref{SNR_i_star}.

\begin{lem} \label{lem:SNR-general-b}
Consider the DCMM model in \eqref{mod-DCMM}-\eqref{mod-DCMM2}. Write $G_0=[{\rm tr} (\Theta M_0^2\Theta)]^{-1}\Pi'\Theta M_0^2\Theta\Pi$ where $M_0 = H_0^{-b}$. As $n\to\infty$, suppose $K$ is fixed and $PG_0$ converges to a fixed irreducible non-singular matrix that has distinct eigenvalues. 
Denote by $\theta_{\max}$ and $\theta_{\min}$ the maximum and minimum of $\theta_i$'s. 
We assume $n\bar{\theta}\theta_{\min}/\log(n)\to\infty$ and 
set $\tau=0$ in the definition of $H$ (see \eqref{Def:Laplacian}). 
Let $\hat{\xi}_1^*, \ldots,\hat{\xi}_K^*$ be defined as  in \eqref{xi-star}. With probability $1-o(n^{-3})$, simultaneously for all $1\leq i\leq n$, 
\begin{align} \label{ER-star-bound}
\mathrm{ER}^*_i \leq  C\Biggl(\sqrt{ \frac{\log(n)}{\theta_i}}\times \frac{\sqrt{\sum_{j=1}^n \theta_j^{3-4b}}}{\sum_{j=1}^n \theta_j^{2-2b}} + \frac{\log(n)}{\theta_i} \times \frac{\theta_{\max}^{1-2b}}{\sum_{j=1}^n \theta_j^{2-2b}}\Biggr). 
\end{align}
\end{lem}

We have made some strong assumptions in Lemma~\ref{lem:SNR-general-b}, e.g., $K$ is fixed, $PG_0$ converges to a fixed matrix with distinct eigenvalues, and $n\bar{\theta}\theta_{\min}/\log(n)\to\infty$. These assumptions are only for convenience, as we want to avoid re-defining those regularity conditions in Section~\ref{subsec:RegConds} for a general $b$. However, there is no technical hurdle of extending  Lemma~\ref{lem:SNR-general-b} to allow for weaker conditions similar to those in Section~\ref{subsec:RegConds}.

We now look into the right hand side of \eqref{ER-star-bound} and write it as $\mathrm{ER}^*_i\leq C(\omega_i^*+\widetilde{\omega}_i)$, where $\omega^*_{i}$ and $\widetilde{\omega}_{i}$ represent the two terms in the brackets, respectively. By Definition~\ref{Def-CDF},  $\sum_{i=1}^n\theta_i^\gamma=\bar{\theta}^\gamma\int t^{\gamma}dF_n(t)$, for any $\gamma\in\mathbb{R}$. Consequently, 
\beq \label{SNR-supp-1}
\omega^*_i = \sqrt{ \frac{\log(n)}{\theta_i}}\times \frac{\sqrt{\sum_{j=1}^n \theta_j^{3-4b}}}{\sum_{j=1}^n \theta_j^{2-2b}} = \frac{C\sqrt{\log(n)}}{\sqrt{n\theta_i\bar{\theta}}}\times \underbrace{\frac{\sqrt{\int t^{3-4b}dF_n(t)}}{\int t^{2-2b}dF_n(t)}}_{\delta(b, F_n)}.
\eeq
In addition, note that $\sqrt{\sum_j \theta_j^{3-4b}}\leq\sqrt{ \theta_{\max}^{2-4b}\sum_j\theta_j}\leq \theta_{\max}^{1-2b}\sqrt{n\bar{\theta}}$. We then obtain:   
\beq \label{SNR-supp-2}
\frac{\widetilde{\omega}_i}{\omega^*_i} = \frac{\sqrt{\theta_i}}{\sqrt{\log(n)}}\times \frac{\sqrt{\sum_{j=1}^n \theta_j^{3-4b}}}{\theta_{\max}^{1-2b}}\leq  \frac{\sqrt{n\bar{\theta}\theta_i}}{\sqrt{\log(n)}}. 
\eeq 
As long as $n\bar{\theta}\theta_i/\log(n)\geq C$, $\widetilde{\omega}_i$ is dominated by $\omega^*_i$, simultaneously for all fixed $b$. Furthermore, $\omega^*_i$ is minimized at $b=1/2$. 
It suggests that $b=1/2$ is the universally best choice (as claimed in Section~\ref{subsec:Methods} of the main text).

\subsection{Proof of Lemma \ref{lem:SNR-general-b}}
We recall that $\hat{\xi}_k^*$ is a proxy to $\xi_k$. It is tedious to bound the difference between $\hat{\xi}_k^*$ and $\xi_k$ and handle the rotation matrix $O_1\in\mathbb{R}^{(K-1)\times (K-1)}$. In the special case of $b=1/2$, such analysis is detailed in Section~\ref{sec:Entrywise-proof} and Section~\ref{sec:EntrywiseAnalysis-in-supp}. For a general $b$, we skip this step but only study $\hat{\xi}_k^*$. Since the statement of  Lemma \ref{lem:SNR-general-b} is only about $\hat{\xi}_k^*$, the proof is much shorter.

Recall the proxy 
$\hat{\xi}_k^* := \lambda_k^{-1} H_0^{-b}AH^{-b}_0 \xi_k$, where $H_0(i,i) = \mathbb{E} d_i\asymp n\bar{\theta} \theta_i$. Under the assumptions on $G_0 $ and $PG_0$, it can be observed that the eigenvalues $\lambda_1, \cdots, \lambda_K$  are well separated by a order of $\sum_{i=1}^n \theta_i^{2-2b} /(n\bar{\theta})^{2b}$ and 
\[ |\lambda_ k|\asymp  \sum_{j=1}^n \theta_j^{2-2b} /(n\bar{\theta})^{2b} \quad \text{ for all $1\leq k \leq K$.}\]
In the same manner to prove Lemma \ref{lem:order}, under the assumptions in Lemma \ref{lem:SNR-general-b}, we can also claim that 
\begin{align*}
\xi_1(i) \asymp \frac{\theta_i^{1- b}}{ \sqrt{\sum_{j=1}^n  \theta_i^{2- 2b}} }, \quad \max_{1\leq k \leq K} |\xi_k (i )| \leq C \frac{\theta_i^{1- b}}{ \sqrt{ \sum_{j=1}^n  \theta_i^{2- 2b} } }
\end{align*}
by replacing $H_0$ by $M_0^2$. We skip the details for simplicity. Based on the above estimates, we can derive 
\begin{align}\label{eq:2023061401}
\big|\hat{\xi}_k^*(i) - \xi_k(i) \big| &\lesssim  \frac{(n\bar{\theta})^{2b} M_0(i,i)}{\sum_{j=1}^n \theta_j^{2-2b} } \, {\sum_{j=1}^n \big(A(i,j) - \Omega(i,j) \big)M_0(j,j)\xi_k(j)} \notag\\
& \lesssim   \frac{(n\bar{\theta})^{b}}{\big(\sum_{j=1}^n \theta_j^{2-2b} \big)\theta_i^b} \, {\sum_{j\neq i}^n W(i,j)M_0(j,j)\xi_k(j)} +  \frac{\theta_i^{3- 3b}}{\big(\sum_{j=1}^n \theta_j^{2-2b}\big)^{\frac 32}}\,. 
\end{align}
for any fixed $1\leq k \leq K$.
To proceed, we apply Bernstein inequality (i.e., Theorem \ref{thm:Bern_ineq}) to the summation on the RHS above and get 
\begin{align*}
& \quad {\sum_{j\neq i}^n W(i,j)M_0(j,j)\xi_k(j)}\notag\\
& \leq C \sqrt{\sum_{j=1}^n \theta_i \theta_j (n\bar{\theta}\theta_j)^{-2b} \frac{\theta_j^{2-2b}}{\sum_{j=1}^n \theta_j^{2-2b}} \cdot \log (n) } \, + C\max_{j } \bigg|(n\bar{\theta}\theta_j)^{-b} \frac{\theta_j^{1-b} } {\sqrt{\sum_{j=1}^n \theta_j^{2-2b}} }\bigg| \cdot \log (n)  \notag\\
& \leq C\sqrt {\theta_i \log (n)} \, \frac{(\sum_{j=1}^n \theta_j^{3-4b})^{1/2}}{(n\bar{\theta})^b (\sum_{j=1}^n \theta_j^{2-2b})^{1/2}} + \frac{C\max_j \theta_j^{1-2b}}{(n\bar{\theta})^b (\sum_{j=1}^n \theta_j^{2-2b})^{1/2}} \cdot  \log (n)
\end{align*}
with probability at least $1- o(n^{-3})$.  Substituting the above inequality back into (\ref{eq:2023061401}), together with the estimate of $\xi_1(i)$, we obtain that 
\begin{align*}
\frac{\max_{1\leq k\leq K} \big|\hat{\xi}_k^*(i) - \xi_k(i) \big|}{\xi_1(i)}\lesssim \frac{\sqrt{\log (n) }}{\sqrt{\theta_i}}\, \frac{\big(\sum_{j=1}^n \theta_j^{3-4b}\big)^{1/2}}{\sum_{j=1}^n \theta_j^{2-2b}} + \frac{\log n}{\theta_i} \, \frac{\max_j \theta_j^{1-2b} }{\sum_{j=1}^n \theta_j^{2-2b}}
\end{align*}
with probability at least $1- o(n^{-4})$. By combining this inequality for all $i$, we conclude our proof. 

\section{Relaxed condition on $P$} \label{supp:relaxP}
In Section~\ref{subsec:RegConds}, we introduced one regularity condition on $P$ in Condition~\ref{reg-conds} (b). Specifically, we assume that 
\begin{align*}
\min_{1\leq k\leq K} \Big\{\sum_{1\leq \ell\leq K}P(k, \ell) \Big\} \geq c_2 K. 
\end{align*} 
As we explained in the main context of Section~\ref{subsec:RegConds}, this condition can be relaxed to 
\begin{align}\label{cond:relaxedP}
\min_{1\leq k\leq K} \Big\{\sum_{1\leq \ell\leq K}P(k, \ell) \Big\} \geq c_2 \alpha_n, \quad \text{where $\alpha_n \in [1, K]$.} 
\end{align} 
 Our theory still holds under certain modification of the regularity conditions in this case. We discuss more details in this section. 
 
 Let $\alpha_n \in [1, K]$. We re-define 
 \begin{align}\label{def:newG}
 G = \alpha_n \cdot \Pi'\Theta D_{\theta}^{-1}\Theta \Pi
 \end{align}
 The following theorem holds. 
 \begin{thm}\label{thm:eigen_relaxedP}
 Consider the DCMM model \eqref{mod-DCMM}-\eqref{mod-DCMM2}, where Condition~\ref{reg-conds}(a), the first inequality of (b), and (c) are satisfied for $G$ defined in (\ref{def:newG}). 
In addition, suppose (\ref{cond:relaxedP}) holds and $K^2\alpha_n \log(n)/(n\bar{\theta}^2\beta_n^2)\to 0$ as $n\to\infty$. 
With probability $1-o(n^{-3})$,  there exists $\omega\in \{1, -1\}$ and  an orthogonal matrix $O_1\in \mathbb{R}^{(K-1)\times(K-1)}$ such that simultaneously for all $1\leq i\leq n$, 
\begin{align}
|\omega\hat{\xi}_1(i)-\xi_1(i)| & \leq C \sqrt{\frac{K^2\theta_i\log(n)}{n^2\bar{\theta}^3 \alpha_n}}\Biggl(1+ \sqrt{\frac{\log(n)}{n\bar{\theta}\theta_i (\alpha_n/K)}}\, \Biggr),\label{result-1}\\
\Vert e_i' (\hat{\Xi}_1 O_1 - \Xi_1)\Vert&\leq  C \sqrt{\frac{K^2\alpha_n \theta_i\log(n)}{n^2\bar{\theta}^3 \beta_n^2}}\Biggl(1+ \sqrt{\frac{\log(n)}{n\bar{\theta}\theta_i(\alpha_n/K) }} \, \Biggr).  \label{result-2}
\end{align}
 \end{thm}
Based on the above theorem, we have the node-wise error and rate of MSL below. 
 \begin{thm}
Suppose the assumptions in Theorem~\ref{thm:eigen_relaxedP} hold, and additionally, Condition~\ref{reg-conds}(d) is satisfied.  Let $\hat{\Pi}$ be the output of Algorithm~\ref{alg:MSL}.
 With probability $1-o(n^{-3})$, there exists a permutation $T$ on $\{1,2,\ldots,K\}$, such that simultaneously for all $1\leq i\leq n$, 
\beq \label{Nodewise-Err}
\Vert T\hat{\pi}_i- \pi_i\Vert_1 \leq  C \min\left\{ \sqrt{\frac{K^2\alpha_n\log (n)}{n\bar{\theta} (\bar{\theta}\wedge \theta_i)\beta_n^2}},\;\; 1\right\}, 
\eeq
In addition, let 
${\cal L}(\hat{\Pi},\Pi)$ be the $\ell^1$-loss in \eqref{Loss}. Then, 
\[
\mathbb{E}{\cal L}(\hat{\Pi},\Pi) \leq C \sqrt{\log(n)}\cdot  \frac{1}{n}\sum_{i=1}^n \min\left\{ \frac{K\sqrt{\alpha_n}}{ \beta_n \sqrt{n\bar{\theta} (\bar{\theta}\wedge \theta_i)}},\;\; 1\right\}. 
\] 

 \end{thm}

\bibliographystyle{chicago}
\bibliography{DCMM_Laplacian}

\end{document}